\documentclass[10pt]{amsart}
\usepackage[utf8]{inputenc}
\usepackage[T1]{fontenc}
\usepackage{amsfonts,amsmath,amssymb,amsthm,tikz}
\usepackage[shortlabels]{enumitem}

\usetikzlibrary{decorations.pathreplacing}
\usepackage{float} 
\usepackage[pagebackref]{hyperref}

\usepackage{tikz} 
\usepackage{tikz-3dplot}

\usepackage[normalem]{ulem}

\usepackage[a4paper]{geometry}

\newtheorem{theorem}{Theorem}[section]
\newtheorem{lemma}[theorem]{Lemma}
\newtheorem{proposition}[theorem]{Proposition}
\newtheorem{corollary}[theorem]{Corollary}

\theoremstyle{definition}
\newtheorem{definition}[theorem]{Definition}

\newtheorem{remark}[theorem]{Remark}
\newtheorem{example}[theorem]{Example}

\numberwithin{equation}{section}

\makeatletter
\def\subsection{\@startsection{subsection}{2}%
 \z@{.5\linespacing\@plus.7\linespacing}{.3\linespacing}%
 {\normalfont\bfseries}}
\makeatother

\newcommand{\Cc}{\mathbb{C}} 
\newcommand{\Pp}{\mathbb{P}}
\newcommand{\Rr}{\mathbb{R}}
\newcommand{\Nn}{\mathbb{N}}
\newcommand{\Zz}{\mathbb{Z}}
\newcommand{\Qq}{\mathbb{Q}}

\newcommand{\Kk}{\mathbb{K}}

\newcommand{\calc}{\mathcal{C}}
\newcommand{\cald}{\mathcal{D}}
\newcommand{\cale}{\mathcal{E}}
\newcommand{\calo}{\mathcal{O}}
\newcommand{\calp}{\mathcal{P}}
\newcommand{\calv}{\mathcal{V}}

\newcommand{\ord}{\mathrm{ord}} 
\newcommand{\divis}{\mathrm{div}}
\newcommand{\maxid}{\mathfrak{m}}

\newcommand{\de}{\mathbf{i}}
\newcommand{\ex}{\mathbf{e}}
\newcommand{\ic}{\mathbf{c}}
\newcommand{\cZ}{\mathcal{Z}}

\newcommand{\red}[1]{\textcolor{red}{#1}}

\newcommand{\blue}[1]{\textcolor{blue}{#1}}

\newcommand{\violet}[1]{\textcolor{violet}{#1}}
\newcommand{\orange}[1]{\textcolor{orange}{#1}}

\usepackage{todonotes}

\title{Lotuses as computational architectures}
\author[Garc\'{\i}a Barroso]{Evelia R. Garc\'{\i}a Barroso}
\author[Gonz\'alez P\'erez]{Pedro D. Gonz\'alez P\'erez}
\author[Popescu-Pampu]{Patrick Popescu-Pampu}

\email{ergarcia@ull.es}
\email{pgonzalez@mat.ucm.es}
\email{patrick.popescu-pampu@univ-lille.fr}

\address{(Evelia R. Garc\'{\i}a Barroso) Universidad de La Laguna. IMAULL.
Departamento de Matem\'aticas, Estad\'{\i}stica e I.O., Apartado de Correos 456.
38200, La Laguna, Tenerife, Espa\~na}
\address{(Pedro D. Gonz\'alez P\'erez) Universidad Complutense de Madrid. Instituto de Matem\'atica Interdisciplinar y Departamento de \'Algebra,  Geometr\' \i a y Topolog\'\i a.,
 Plaza de las Ciencias 3, Madrid 28040, Espa\~na}
\address{(Patrick Popescu-Pampu) Universit\'e de Lille, CNRS, Laboratoire Paul Painlev\'e, 59000 Lille, France}


\thanks{\emph{Acknowledgments}. 
The authors gratefully acknowledge the support of Universidad de La Laguna (Tenerife, Spain), where part of this work was done (Spanish grant PID2019-105896GB-I00 funded by
MCIN/AEI/10.13039/501100011033). 
This work was also supported by the Labex CEMPI (ANR-11-LABX-0007-01), the ANR SINTROP (ANR-22-CE40-0014) and the Spanish grant of MCIN 
(PID2020-114750GB-C32/AEI/10.13039/501100011033). The third author is grateful to the scientific committee of the conference {\em 115 AM: Algebraic and topological interplay of algebraic varieties} hold at Jaca in 2023 between June 12 and 17 in honor of the 60th Birthday of Enrique Artal and the 55th Birthday of Alejandro Melle for having invited him to give a mini-course about lotuses. This text originates from it. We are grateful to the anonymous referee and to Jules Chenal for their remarks and to Antoni Rangachev for his explanations about the delta invariant. We thank warmly Christopher-Lloyd Simon for his careful reading of a previous version of this paper.}

\date{February 24, 2025}

\begin{document}

\begin{abstract}
    \emph{Lotuses} are certain types of finite contractible simplicial complexes, obtained by identifying vertices of polygons subdivided by diagonals. As we explained in a previous paper, each time one resolves a complex reduced plane curve singularity by a sequence of toroidal modifications with respect to suitable local coordinates, one gets a naturally associated lotus, which allows to unify the classical trees used to encode the combinatorial type of the singularity. In this paper we explain how to associate a lotus to each \emph{constellation of crosses}, which is a finite constellation of infinitely near points endowed with compatible germs of normal crossings divisors with two components, and how this lotus may be seen as a \emph{computational architecture}. Namely, if the constellation of crosses is associated to an embedded resolution of a complex reduced plane curve singularity $A$, one may compute progressively as vertex and edge weights on the lotus the log-discrepancies of the exceptional divisors, the orders of vanishing on them of the starting coordinates, the multiplicities of the strict transforms of the branches of $A$, the orders of vanishing of a defining function of $A$, the associated Eggers-Wall tree, the delta invariant and the Milnor number of $A$, etc. We illustrate these computations using three recurrent examples. Finally,  we describe the changes to be done when one works in positive characteristic.
\end{abstract}

{\bf This paper will appear in {\em Algebraic and Topological Interplay of Algebraic Varieties.
A tribute to the work of E. Artal and A. Melle}. Contemporary Maths., Amer. Math. Soc.}
\bigskip

\maketitle

\medskip
{\em \hfill We dedicate this paper to Enrique Artal Bartolo and Alejandro Melle Hern\'andez. \smallskip} 

\tableofcontents

\section{Introduction}
\label{sec:Intro}

The third author introduced in \cite{PP 11} certain types of simplicial complexes of dimension at most two, which he called {\em kites},  showing that they unified the Enriques diagrams and the dual graphs associated to finite sequences of blowups of infinitely near points of a smooth point on a complex analytic  surface. Kites were constructed by gluing {\em lotuses} and {\em ropes}. He had the idea of doing such a construction after having learnt from the paper \cite[Section 6.4]{DV 98} of de Jong and van Straten that Christophersen \cite[Section 1.1]{C 91} and  Stevens \cite[Section 1]{S 91bis} had represented {\em certain} finite sequences of blowups by polygons triangulated using diagonals (see also \cite[Remark 8.16]{PP 11}). Namely, he wanted to associate a generalization of such a triangulated polygon to {\em any} finite sequence of blowups. This led him to glue several triangulated polygons in a tree-like fashion, producing  what we call in this paper a {\em lotus} (a name initially suggested by Teissier when the third author drew him by hand an analog of Figure \ref{fig:Unilotus}).

A typical example of lotus is shown in Figure \ref{fig:exlotus}:

 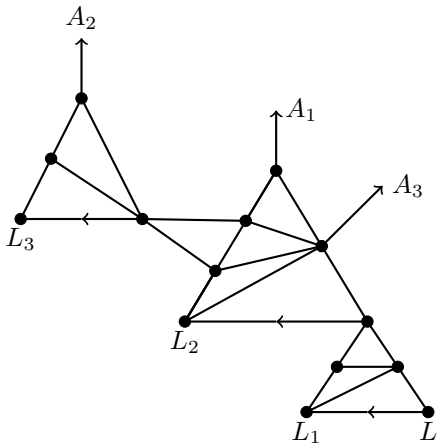
\begin{figure}[h!]
    \begin{center}
\begin{tikzpicture}[scale=0.8]

 \draw [->, color=black, thick](1,0) -- (0,0);
    \draw [-, color=black, thick](0,0) -- (-1,0);
    \draw [-, color=black, thick](0,1.5) -- (-1,0);
      \draw [-, color=black, thick](0,1.5) -- (1,0);
       \draw [-, color=black, thick](-0.5,3/4) -- (0.5,3/4);
       \draw [-, color=black, thick](-1,0) -- (0.5,3/4);
       \draw [-, color=black, thick](-5.2,4.2) -- (-3.7,3.2);

        \draw [->, color=black, thick](0,1.5) -- (-1.5,1.5);
   \draw [-, color=black, thick](-1.5,1.5) -- (-3,1.5);
     \draw [-, color=black, thick](0,1.5) -- (-1.5,4);
      \draw [-, color=black, thick](-1.5,4) -- (-3,1.5);
    \draw [->, color=black, thick](-1.5,4) -- (-1.5,5);
     \draw [-, color=black, thick](-3/4,11/4) -- (-3,1.5);
     \draw [-, color=black, thick](-3/4,11/4) -- (-2.5,2.34);
       \draw [-, color=black, thick](-3/4,11/4) -- (-2,3.17);
       \draw [->, color=black, thick](-3/4,11/4) -- (1/4,15/4);
        \draw [-, color=black, thick](-2,3.17) -- (-1.5,4);
        \draw [-, color=black, thick](-3,1.5) -- (-2.5,2.34);
        \draw [-, color=black, thick](-3,1.5) -- (-2.5,2.34);
        \draw [-, color=black, thick](-3.7,3.2) -- (-2,3.17);
        \draw [-, color=black, thick](-3.7,3.2) -- (-2.5,2.34);
        \draw [->, color=black, thick](-3.7,3.2) -- (-4.7,3.2);
        \draw [-, color=black, thick](-5.7,3.2) -- (-4.7,3.2);
           \draw [->, color=black, thick](-4.7,5.2) -- (-4.7,6.2);
            \draw [-, color=black, thick](-3.7,3.2) -- (-4.7,5.2);
            \draw [-, color=black, thick](-5.7,3.2) -- (-4.7,5.2);
       
  \node[draw,circle, inner sep=1.5pt,color=black, fill=black] at (-1,0){};
   \node[draw,circle, inner sep=1.5pt,color=black, fill=black] at (1,0){};
     \node[draw,circle, inner sep=1.5pt,color=black, fill=black] at (0,1.5){};
      \node[draw,circle, inner sep=1.5pt,color=black, fill=black] at (-0.5,3/4){};
       \node[draw,circle, inner sep=1.5pt,color=black, fill=black] at (0.5,3/4){};
        \node[draw,circle, inner sep=1.5pt,color=black, fill=black] at (-3,1.5){};
         \node[draw,circle, inner sep=1.5pt,color=black, fill=black] at (-1.5,4){};
         \node[draw,circle, inner sep=1.5pt,color=black, fill=black] at (-3/4,11/4){};
         \node[draw,circle, inner sep=1.5pt,color=black, fill=black] at (-2.5,2.34){};
         \node[draw,circle, inner sep=1.5pt,color=black, fill=black] at (-2,3.17){};
          \node[draw,circle, inner sep=1.5pt,color=black, fill=black] at (-3.7,3.2){};
          \node[draw,circle, inner sep=1.5pt,color=black, fill=black] at (-5.7,3.2){};
           \node[draw,circle, inner sep=1.5pt,color=black, fill=black] at (-4.7,5.2){};
           \node[draw,circle, inner sep=1.5pt,color=black, fill=black] at (-5.2,4.2){};
       
   \node [below, color=black] at (-1,0) {$L_{1}$};
\node [below, color=black] at (1,0) {$L$};
\node [below, color=black] at (-3,1.5) {$L_{2}$};
\node [right, color=black] at (1/4,15/4) {$A_{3}$};
\node [right, color=black] at (-1.5,5) {$A_{1}$};
 \node [below, color=black] at (-5.7,3.2) {$L_{3}$};
 \node [above, color=black] at (-4.7,6.2) {$A_{2}$};

 \end{tikzpicture}
\end{center}
 \caption{An example of lotus}
\label{fig:exlotus}
   \end{figure}

The authors of the present paper simplified in \cite{GBGPPP 20} the constructions and terminology of \cite{PP 11}, droping the terms {\em kite} and {\em rope} and speaking only about {\em lotuses}. They showed how to associate a lotus to each toroidal resolution process of plane curve singularities, obtained by considering successive Newton polygons of the singularity and of its strict transforms on the intermediate surfaces of the process. They emphasized the fact that:
\begin{quote}
    {\em The lotus may be seen as the space-time representation of the evolution of dual graphs of total transforms of a suitable completion of the given plane curve singularity.}
\end{quote}
They also explained how to get, starting from the lotus, the dual graph of the associated embedded resolution, its Enriques diagram and the Eggers-Wall tree of the completed plane curve singularity. By contrast with the previous types of graphs, which are all endowed with supplementary structures equivalent to decorations by integers, lotuses are devoid of such decorations. They only have several distinguished oriented edges, as shown in Figure \ref{fig:exlotus}. 

Overviews of the paper \cite{GBGPPP 20} may be found in \cite[Section 7.1]{GBGPPP 20} and \cite{PP 23}.

The aim of this paper is to explain that lotuses are also very convenient {\em computational architectures}. Namely, they allow to determine very easily by hand many basic combinatorial invariants of a  reduced plane curve singularity $A$, when its combinatorial type is given through a process of embedded resolution by blowups of infinitely near points: 
  \begin{itemize}
      \item 
         the {\em dual graph} weighted by the {\em self-intersection numbers of the irreducible components of the exceptional divisor} of the embedded resolution (see Section \ref{sec:dualgr}); 
      \item 
         the {\em log-discrepancies} of those components (see Section \ref{sec:ldov}); 
      \item 
          the {\em orders of vanishing of a defining function of $A$} along those components (see Sections \ref{sec:ldov} and \ref{sec:ovint});
    \item 
         the {\em multiplicities of the strict transforms} of $A$ at all the infinitely near points blown up during the embedded resolution process (see Section \ref{sec:multstr});
      \item 
          the {\em intersection numbers of the branches of $A$} (see Sections \ref{sec:ovint} and \ref{sec:intfromEW});
      \item 
          the {\em Eggers-Wall tree of $A$} relative to the initial vertex of the lotus, therefore also examples of Newton-Puiseux series defining a plane curve singularity with the given combinatorial type (see Section \ref{sec:EWfromlot});
      \item 
          conversely, lotuses may be constructed from Eggers-Wall trees (see Section \ref{sec:lotfromEW});
      \item 
          the {\em generating sequences of the semigroups} of those branches (see Section \ref{sec:computgensg});
       \item  
          the {\em delta invariant} and the {\em Milnor number} of $A$ (see Section \ref{sec:deltaMiln}). 
  \end{itemize}

  The computations performed using lotuses are based on the fact that all these invariants of plane curve singularities may be seen also as invariants of the associated lotuses, measuring the sizes of suitable parts of them.  We illustrate each type of computation using three recurrent examples, one of which is the lotus shown in Figure \ref{fig:exlotus}, the other two being sublotuses of it. The simplest of those sublotuses is associated to the standard cusp singularity, which is the fundamental example on which any general method of singularity theory is usually first illustrated. We recommend to the reader who wants to learn  performing such computations to treat each example as an exercise: just make the computations on the corresponding lotus and compare the result with our figure. 

Plane curve singularities remain an important experimental ground for studying phenomena concerning singularities of arbitrary dimension or codimension. Our experience taught us that lotuses are very practical for  experiments and may lead to discoveries. For instance, the main theorems of  Castellini's thesis \cite{C 15} had been discovered and are formulated using lotuses. They concern delta-constant deformations of plane curve singularities obtained by the method of A'Campo from \cite{A 75}. We believe that lotuses could be used fruitfully in the study of other problems about deformations of singularities.  

We tried to make this text self-contained, in order to help learning from it a way of  thinking geometrically about the basic invariants of plane curve singularities. For this reason, we recall the definitions of the combinatorial invariants mentioned above and of their properties which are useful in the computations using lotuses. We also indicate some proofs, when we find them particularly enlightening. 

\medskip
Let us outline the structure of this article. The first paragraph of each section describes more precisely its content. The recurrent notations are included inside \boxed{\rm boxes}. For simplicity of reference, we do the same with the rules explaining how to make computations on the lotus (see Propositions \ref{prop:graphprox},  \ref{prop:dualgraph} and Corollaries \ref{cor:spacetimerep}, \ref{cor:reclord}, \ref{cor:multedges}, \ref{cor:firstmethint}, \ref{cor:secmethint}, \ref{cor:lotustoEW}, \ref{cor:mingenseq}, \ref{cor:deltamincomp}).

In Section \ref{sec:properaction} we explain basic facts about effective divisors and their intersection numbers on smooth complex analytic surfaces. In Section \ref{sec:plcsembres} we introduce vocabulary concerning the embedded resolutions of plane curve singularities by blowups of points.  

By contrast with the paper \cite{GBGPPP 20}, in which lotuses were constructed starting from Newton polygons, here we build them starting from finite sequences of blowups of infinitely near points of a smooth point $O$ of a complex analytic surface $S$, with the supplementary datum of a {\em cross} (two smooth branches intersecting transversally) at each point being blown up.  Those crosses must satisfy a suitable compatibility condition (see Definition \ref{def:constcrosses}). We say that such a finite set of infinitely near points endowed with crosses and the information of which infinitely near points have to be blown up is a {\em finite active constellation of crosses}. The vocabulary needed to speak with precision about them is explained in Section \ref{sec:activeconst} and their lotuses are defined in Section \ref{sec:lotusactive}. One may also associate a lotus to every finite set of positive rational numbers. Such {\em Newton lotuses} are the building blocks of all lotuses and are presented in Section \ref{sec:Nlot}.

With the exception of Section \ref{sec:Nlot}, Sections \ref{sec:dualgr}--\ref{sec:deltaMiln} explain how to use the previous lotuses in order to compute the basic invariants of plane curve singularities of the list above. 

In the last two sections \ref{sec:charp} and \ref{sec:existNP}, we explain how to adapt our setting to reduced formal plane curve singularities defined over algebraically closed fields of positive characteristic. Those sections contain several new contributions, namely, a definition of {\em Eggers-Wall trees} for all such singularities, possibly devoid of Newton-Puiseux roots (see Definition \ref{Ew-p-lotus}), a second definition of {\em Newton-Puiseux-Eggers-Wall trees} when all their branches have Newton-Puiseux roots (see Definition \ref{def:EWp}), a characterisation of the singularities having this property (see Corollary \ref{cor:NP-p}) and a comparison of the two kinds of trees (see Corollary \ref{cor:compartrees}).

\vfill

\medskip
\section{Actions of proper morphisms on effective divisors}  \label{sec:properaction}

In this section we recall basic facts about effective divisors on smooth complex analytic surfaces: {\em intersection numbers}, {\em direct images} and {\em total transforms} by proper morphisms (see Definition \ref{def:transfdiv}), the {\em projection formula} (see Theorem \ref{thm:projform}) and an application to {\em blowups} (see Proposition \ref{prop:relblowup}). 

\medskip
Excepted in the last two sections \ref{sec:charp} and \ref{sec:existNP} of the paper,  we work in complex analytic geometry, speaking for instance of {\em biholomorphisms} and {\em bimeromorphic maps}, but our considerations may be transposed without difficulties to  formal geometry over an arbitrary algebraically closed field of characteristic zero.

If $\Sigma$ is a smooth complex analytic surface, a {\bf divisor} on $\Sigma$ is a finite linear combination with integral coefficients of closed irreducible curves on $\Sigma$.
 Its {\bf support} is equal to the union of the irreducible curves appearing with non-zero coefficients in this linear combination.  A divisor is called {\bf compact} if its support is compact. It is called {\bf effective} if its coefficients are non-negative. If $h$ is a holomorphic function on $\Sigma$, we denote by $\boxed{Z(h)}$ its associated effective divisor, which is the zero locus of $h$, each of its irreducible components being counted with its multiplicity.

Recall that a continuous map $\pi$ between Hausdorff spaces is called {\bf proper} if the inverse image $\pi^{-1}(K)$ of any compact subset $K$ of the target is again compact. Divisors may be pushed forward and pulled back by proper morphisms between smooth surfaces:

\begin{definition}
   \label{def:transfdiv}
    Let $\pi : \Sigma_1 \to \Sigma_2$ be a proper morphism between two smooth complex analytic surfaces and let $D_1, D_2$ be two divisors on $\Sigma_1$ and $\Sigma_2$ respectively. 
    
        \noindent  $\bullet$
           If $D_1$ is irreducible, its {\bf direct image} $\boxed{\pi_*(D_1)}$ is: 
          \[\deg(\pi|_{D_1} : D_1 \to \pi(D_1)) \ \pi(D_1). \]
        Here $\pi(D_1)$ is seen as a reduced subspace of $\Sigma_2$, which is either a point or a closed irreducible curve, by the properness of $\pi$, and $\boxed{\deg(\pi|_{D_1} : D_1 \to \pi(D_1))} \ \in \Zz_{\geq 0}$ denotes the degree of the restriction of $\pi$ to $D_1$. Then  $\pi_*$ is extended by linearity to arbitrary effective divisors on $\Sigma_1$. 
        
         \noindent $\bullet$
           The {\bf pull-back} or {\bf total transform} $\boxed{\pi^*(D_2)}$ of $D_2$ by $\pi$ is defined by gluing the divisors $Z(\pi^* h)$ of local defining functions $h$ of $D_2$. Here $\boxed{\pi^* h} := h \circ \pi$ is the {\bf pull-back of $h$ by $\pi$}. 
  
\end{definition}

If $D_1$ is irreducible and $\pi(D_1)$ is a point, then $\deg(\pi|_{D_1} : D_1 \to \pi(D_1)) =0$, therefore $\pi_*(D_1)=0$. The direct image of any compact divisor is compact. Similarly,  the properness of $\pi$ implies that the pull-back of any compact divisor is also compact.

If $D, E$ are two divisors on a smooth surface $\Sigma$ and at least one of them is compact, then one may associate to them an {\bf intersection number} $\boxed{D \cdot E} \in \Zz$ (see \cite[Section II.10]{BHPV 04} for the complex analytic setting and \cite[Theorem V.1.1]{H 77} for the algebraic geometric setting).

Direct images and total transforms of divisors are connected by the following  {\bf projection formula} (see \cite[Proposition 8.3]{F 84} for a more general statement):

\begin{theorem}  
   \label{thm:projform}
    Let $\pi : \Sigma_1 \to \Sigma_2$ be a proper morphism between two smooth complex analytic surfaces and $D_1, D_2$ be two divisors on $\Sigma_1$ and $\Sigma_2$ respectively, at least one of which is compact. Then:
      \[ D_1 \cdot \pi^*(D_2) = \pi_*(D_1) \cdot D_2. \] 
\end{theorem}

We will apply this theorem in the proofs of Propositions \ref{prop:relblowup} and \ref{prop:intbr}.

For us, particularly important proper morphisms are the {\em blowups} of points of smooth surfaces (see \cite[Section 8.4]{BK 86}, \cite[Section 3.2]{W 04},  \cite[Section 1.2.4]{GBGPPP 20}):

\begin{definition}
    \label{def:blowup}
    Let $P$ be a point of the smooth surface $\Sigma$. Its associated {\bf blowup morphism}  is the unique proper bimeromorphic morphism $\boxed{\pi_P} : \boxed{\Sigma_P} \to \Sigma$ which is an isomorphism above $\Sigma \smallsetminus \{P\}$ and which is defined in local coordinates $(x,y)$ at $P$ as the projection onto $\Cc^2$ of the closure of the graph of the projectivisation map $\Cc^2 \dashrightarrow \Cc\Pp^1$. Its {\bf exceptional divisor} $\boxed{E_P}$ is the preimage $\pi_P^{-1}(P) \hookrightarrow \Sigma_P$. If $C \hookrightarrow \Sigma$ is a closed reduced complex curve, its {\bf strict transform} $\boxed{C_P}$ by the blowup morphism $\pi_P$ is the closure of $\pi_P^{-1}(C \smallsetminus \{P\})$ inside $\Sigma_P$. 
\end{definition}

By working in local coordinates $(x,y)$, one may see that the second projection $\Cc^2 \times \Cc\Pp^1 \to \Pp^1$ establishes an isomorphism $E_P \simeq \Cc\Pp^1$. Keeping the notations above, one has the following behaviour 
of intersection numbers, which will be essential for us in Section \ref{sec:dualgr}:

\begin{proposition}   
   \label{prop:relblowup}
    Assume that the reduced curve $C \hookrightarrow \Sigma$ is {\rm compact} and {\rm smooth at $P$}. Then the self-intersection number $C\cdot C$ of $C$ in $\Sigma$ and the self-intersection numbers $E_P\cdot E_P$, $C_P \cdot C_P$ of $E_P$ and $C_P$ in $\Sigma_P$ satisfy the following relations:
      \begin{enumerate}[(a)]
          \item 
              $E_P \cdot E_P = -1$.
          \item 
             $C_P \cdot C_P = C \cdot C -1$.
      \end{enumerate}
\end{proposition}

\begin{proof}
  We give the proof of this basic property of blowups in order to illustrate the power of the {\em projection formula} (see Theorem \ref{thm:projform}). 
  We will use the following consequences of the {\em smoothness of $C$ at $P$} :
    \begin{equation}
        \label{eq:totransfsm}
       \pi_P^*(C) = E_P + C_P. 
    \end{equation}
     \begin{equation}
        \label{eq:transv}
         E_P \cdot C_P = 1. 
    \end{equation}
    Both equalities are in fact equivalent to the smoothness of $C$ at $P$ and may be proved by working in local coordinates $(x,y)$ such that $C$ is equal to $Z(x)$ in a neighborhood of $P$. Note that \eqref{eq:transv}   means that $C_P$ intersects $E_P$ transversally  at a single point.

  \begin{enumerate}[(a)]
          \item 
            By applying the projection formula to the proper morphism $\pi_P$, the divisor $C_P$ on $\Sigma_P$ and the divisor $C$ on $\Sigma$, we get $ C_P \cdot \pi_P^*(C) = (\pi_P)_*(C_P) \cdot C$. 
            Therefore, by \eqref{eq:totransfsm}:
                \begin{equation}   \label{eq:firstint}
                   C_P \cdot (E_P + C_P) = C \cdot C.
                \end{equation}
            By applying the projection formula to the same morphism, but replacing $C_P$ by the total transform $\pi_P^*(C)$ of $C$, we get:
            $\pi_P^*(C) \cdot \pi_P^*(C) = (\pi_P)_*(\pi_P^*(C)) \cdot C$.
            Therefore, again by \eqref{eq:totransfsm}, as $(\pi_P)_*(\pi_P^*(C)) = C$:           \begin{equation}  
                  \label{eq:secondint}
                   (E_P + C_P) \cdot (E_P + C_P) = C \cdot C.
             \end{equation}
            By substracting (\ref{eq:firstint}) from (\ref{eq:secondint}), we get $E_P \cdot (E_P + C_P) = 0$, thus by \eqref{eq:transv}:
               \[ E_P \cdot E_P = - E_P \cdot C_P = -1.\]

          \item 
          By (\ref{eq:firstint}), we have $C_P \cdot C_P = C \cdot C - C_P \cdot E_P$. 
        Using again the equality \eqref{eq:transv}, we get the desired relation: 
          \[C_P \cdot C_P = C \cdot C -1. \]         
      \end{enumerate}
\end{proof}

\medskip
\section{Plane curve singularities and their embedded resolutions}  \label{sec:plcsembres}

In this section we recall the notions of {\em abstract resolution} and {\em embedded resolution} of a plane curve singularity (see Definition \ref{def:resol}), the existence of a {\em minimal embedded resolution}, and we illustrate it with the example of the standard {\em cusp} (see Example \ref{ex:cuspembres}).

\medskip
Throughout the paper, $(S, O)$ denotes a germ of {\em smooth} complex analytic surface, with local ring $\boxed{\mathcal{O}_{S, O}}$, consisting of the germs of holomorphic functions on $(S,O)$. Let $\boxed{\maxid_{S,O}}$ be its maximal ideal, consisting of those germs of functions which vanish at the point $O$.  A {\bf local coordinate system} on $(S,O)$ is a pair $(x,y)$ which generates the maximal ideal $\maxid_{S,O}$. It identifies the local ring $\mathcal{O}_{S, O}$ with the complex algebra $\Cc\{x,y\}$ of convergent power series in two variables.

\begin{definition} \label{def:plcurvesing}
       A {\bf plane curve singularity} is a non-zero germ of effective divisor $A$ 
       on a smooth germ of surface $(S, O)$. A {\bf defining function} of  $A$ 
       is a function germ $f \in \calo_{S,O}$ such that $A = Z(f)$. 
       Such a function germ is unique up to multiplication by a unit of $\calo_{S,O}$. The quotient $\boxed{\calo_A} := \mathcal{O}_{S, O}/ (f)$ is the local ring of $A$. 
       One has a canonical {\bf restriction morphism} $\calo_{S,O} \to \calo_A$.  
       A {\bf branch} is an irreducible plane curve singularity.
\end{definition}

One of the main methods of study of plane curve singularities is to {\em resolve} them. One distinguishes two types of {\em resolutions}:

\begin{definition}
   \label{def:resol}
    Let $A \hookrightarrow (S,O)$ be a {\em reduced} plane curve singularity. 

    \noindent $\bullet$ 
        A {\bf model} of $(S,O)$ is a germ along $E_{\pi}$ of a morphism $\boxed{\pi : (S_{\pi}, E_{\pi}) \to (S,O)}$ obtained as a composition of blowups of points above $O$. Here $E_{\pi}$ denotes the {\bf exceptional divisor} of $\pi$, that is, the preimage $\pi^{-1}(O)$, seen as a reduced divisor. The sum $\boxed{A_{\pi}}$ of irreducible components of the total transform $\pi^* (A)$ which are not contained in $E_{\pi}$ is called the {\bf strict transform} of $A$ by $\pi$. 
         
    \noindent $\bullet$ 
         An {\bf abstract resolution of $A$} is a proper bimeromorphic morphism  $\pi : A_{\pi} \to A$ such that  $A_{\pi}$ is smooth. Notice that if $A$ is not irreducible, then $A_{\pi}$ is a {\bf multigerm}, that is, a finite disjoint union of germs. 

    \noindent $\bullet$ 
         An {\bf embedded resolution of $A$} is a model  $\pi : S_{\pi} \to S$ such that the total transform $\pi^* (A)$ of $A$ on $S_{\pi}$ in the sense of Definition \ref{def:transfdiv}  is a normal crossings divisor.  

 If $A \hookrightarrow (S,O)$ is a {\em not necessarily reduced} plane curve singularity, a {\bf resolution} of it of either kind is a resolution of its reduction, each branch of its strict transform being endowed with the same multiplicity as its image in $A$.
\end{definition}

If $\pi : S_{\pi} \to S$ is an embedded resolution of $A$, then the restriction $A_{\pi} \to A$ of $\pi$ to the strict transform $A_{\pi}$ of $A$ by $\pi$ is an abstract resolution of $A$. The simplest non-trivial model of $(S,O)$ is the blowup of $O$ of Definition \ref{def:blowup}.  

Given a reduced plane curve singularity, one may get an embedded resolution of it by a finite sequence of blowups of points (see \cite[Section 8.4, page 496, Theorem 9]{BK 86} for a complex analytic proof and \cite[Theorem V.3.9]{H 77} for a proof in algebraic geometry):

\begin{theorem}  \label{thm:blowupprocess}
   Let $A \hookrightarrow (S,O)$ be a reduced plane curve singularity. Then there exists an embedded resolution of $A$ obtained as a finite composition of blowups of points. 
\end{theorem}

There is always a {\bf minimal embedded resolution}, in the sense that any other embedded resolution factors canonically through it. It is obtained by blowing up only points of the strict transforms of $A$: at each step, blow up only the points at which the total transform of $A$ is {\em not} locally a normal crossings divisor.

\begin{example} \label{ex:cuspembres}
    Assume that $(x,y)$ is a local coordinate system on $(S,O)$ and that the plane curve singularity $A \hookrightarrow (S,O)$ is defined by the equation $f(x,y) = 0$ in this coordinate system, where:
      \[ f(x,y) = y^2 - x^3.\]
    That is, $A$ is a {\bf cusp singularity} at $O$. Its {\em minimal} embedded resolution process by blowups of points is illustrated in Figure \ref{fig:embrescusp} (see \cite[Section 1.2.5]{GBGPPP 20} for the corresponding computations). The irreducible components of the successive exceptional divisors are represented as black segments and the strict transforms of $A$ as orange arcs or segments labeled again by $A$. The points lying on the successive strict transforms of $A$ are labeled $O_0 := O, O_1, O_2, O_3$. Only the first three ones are blown up, the exceptional divisor created by the blowup of $O_j$ being denoted $E_j$. We keep labeling by $E_j$ its strict transforms by ulterior blowups. After a single blowup, the strict transform of $A$ is already smooth, which means that one has already got an {\em abstract resolution} of $A$. But this is not an {\em embedded resolution}, because the total transform of $A$ is not a normal crossings divisor: the strict transform is tangent to the exceptional divisor. One needs two more blowups in order to get a normal crossings divisor. 
\end{example}


   
     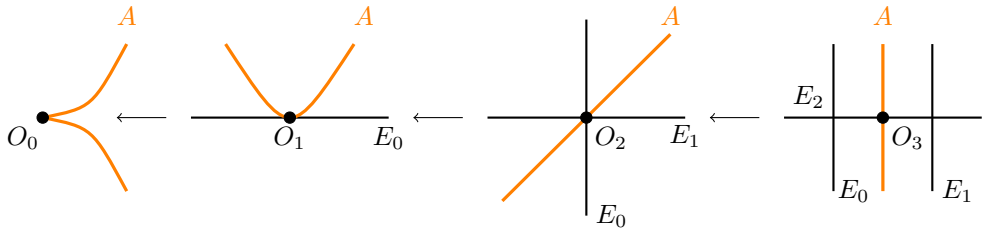
\begin{figure}[h!]
    \begin{center}
\begin{tikzpicture}[scale=0.65]

\begin{scope}[shift={(1,0)}]
 \draw[very thick, color=orange](0,0) .. controls (1,0.2) ..(1.7,1.5); 
\draw[very thick, color=orange](0,0) .. controls (1,-0.2) ..(1.7,-1.5); 
\node [above, color=orange] at (1.7,1.7) {$A$};

\node[draw,circle, inner sep=1.5pt,color=black, fill=black] at (0,0){};
\node [below, color=black] at (-0.4,0) {$O_0$};
\end{scope}


\begin{scope}[shift={(6,0)}]
     \draw [-, color=black, thick](-2,0) -- (2,0);
\draw[very thick, color=orange] (-1.3,1.5) .. controls (0,-0.5) .. (1.3,1.5);
\node [above, color=orange] at (1.5, 1.7) {$A$};

\node [below, color=black] at (0,0) {$O_1$};
\node [below, color=black] at (2,0) {$E_0$};
\node[draw,circle, inner sep=1.5pt,color=black, fill=black] at (0,0){};

\end{scope}


\begin{scope}[shift={(12,0)}]
     \draw [-, color=black, thick](-2,0) -- (2,0);
\draw [-, color=black, thick] (0,-2)--(0,2);
\node [right, color=black] at (0,-2) {$E_{0}$};
\node [above, color=black] at (2,-0.8) {$E_1$};
\draw[very thick, color=orange](-1.7,-1.7) -- (1.7, 1.7); 
\node [above, color=orange] at (1.7,1.7) {$A$};

\node[draw,circle, inner sep=1.5pt,color=black, fill=black] at (0,0){};
\node [below, color=black] at (0.5,0) {$O_2$};
\end{scope}


\begin{scope}[shift={(18,0)}]
     \draw [-, color=black, thick](-2,0) -- (2,0);
\draw [-, color=orange, very thick] (0,-1.5)--(0,1.5);
\node [above, color=orange] at (0,1.7) {$A$};
\node [above, color=black] at (-1.5,0) {$E_2$};
\node [right, color=black] at (-1.1,-1.5) {$E_0$};
\node [right, color=black] at (1,-1.5) {$E_1$};
\node[draw,circle, inner sep=1.5pt,color=black, fill=black] at (0,0){};
\node [below, color=black] at (0.5,0) {$O_3$};
\draw[thick, color=black](1,-1.5) -- (1, 1.5); 
\draw[thick, color=black](-1,-1.5) -- (-1, 1.5); 
\end{scope}


      \draw[<-](2.5,0)--(3.5,0);
     
      \draw[<-](8.5,0)--(9.5,0);
     
      \draw[<-](14.5,0)--(15.5,0);

\end{tikzpicture}
\end{center}
 \caption{The process of minimal embedded resolution of a cusp singularity $A$}
\label{fig:embrescusp}
   \end{figure}

\medskip
\section{Active constellations of crosses}  \label{sec:activeconst}

The terminology introduced in this section is necessary in order to speak with precision about the embedded resolution processes of plane curve singularities and in order to associate them lotuses in Section \ref{sec:lotusactive}. 
We first define {\em models of $(S,O)$, infinitely near points of $O$, constellations of infinitely near points} and their natural partial order relations (see Definition \ref{def:infnear}). Then we define {\em active constellations} and their {\em models} (see Definition \ref{def:actconst}). We extend these notions to {\em active constellations of crosses} (see Definition \ref{def:constcrosses}). Finally, we define active constellations of crosses {\em adapted to a plane curve singularity} and the {\em completion} of the curve singularity relative to them (see Definition \ref{def:adaptedactconst}).

\medskip
The points which are blown up during a process of embedded resolution of a plane curve singularity on $(S,O)$ are {\em infinitely near} $O$ on various models of $(S,O)$ in the sense of Definition \ref{def:resol}, and their set is a {\em constellation} (a terminology we borrow from \cite{CGL 96}), in the following sense:

\begin{definition} 
    \label{def:infnear}
    Consider a smooth germ of surface $(S,O)$. 
 
     \begin{enumerate}

    \item \label{toappear}
         Let $\pi_i : (S_{\pi_i}, E_{\pi_i}) \to (S,O)$ for $i \in \{1, 2\}$ be two models of $(S,O)$. An irreducible component  $E_1$ of $E_{\pi_1}$ is said {\bf to appear on $S_{\pi_2}$} if the bimeromorphic transform of $E_1$ on $S_{\pi_2}$ is also an irreducible component of $E_{\pi_2}$ (and not a point of it).

     \item  \label{iteminfnear}
        An {\bf infinitely near point} of $O$ is either $O$ or a point of the exceptional divisor of a model of $(S,O)$. Two such points, on two models, are considered to be the same, if the associated bimeromorphic maps between the two models are biholomorphisms in their neighborhoods. 

    \item 
        If $P$ is an infinitely near point of $O$, then we denote by $\boxed{S^P}$ the minimal model containing $P$.

    \item \label{passthroughin}
        If $P$ is an infinitely  near point of $O$ and $A$ is a curve singularity on $(S,O)$, we say that {\bf $A$ passes through $P$} or that {\bf $P$ lies on $A$} if the strict transform of $A$ on $S^P$ contains $P$.

    \item  \label{proxnot}
        If $O_1$ and $O_2$ are two infinitely near points of $O$, then one says that $O_2$ {\bf is proximate to}  $O_1$, written $\boxed{O_2 \rightarrow O_1}$, if $O_2$ belongs to the strict transform of the irreducible rational curve created by blowing up $O_1$. If moreover there is no point $O_3$ such that $O_2 \rightarrow O_3 \rightarrow O_1$, one says that $O_1$ {\bf is the parent of $O_2$}.

     \item
        A (finite or infinite) {\bf constellation (above $O$)} is a set $\calc$ of infinitely near points of $O$, closed under the operation of taking the parent of a point different from $O$.

    \item  \label{itempartorder}
        The {\bf constellation of all infinitely near points of $O$} is denoted by $\boxed{\calc_{S, O}}$. It is endowed with a strict partial order relation $\boxed{\preceq}$ generated by the proximity binary relation. That is,  whenever $P, Q \in \calc_{S, O}$, one has $P \preceq Q$ if and only if there exists a chain $Q = P_n \rightarrow P_{n-1} \rightarrow \cdots \rightarrow P_1 = P$. If it exists, this chain is unique. The relation $P \preceq Q$ means that {\bf $Q$ is infinitely near $P$ in the sense of \eqref{iteminfnear}}.  

    \item 
       Consider $P, Q \in \calc_{S, O}$ such that $P \preceq Q$. We denote by $\boxed{\pi^Q_P : S^Q \to S^P}$ the associated morphism of minimal models of $(S,O)$ containing $Q$  and $P$ respectively.  It is obtained by blowing up recursively the points $P_1, \dots, P_{n-1}$ of the chain above. 
 \end{enumerate}
\end{definition}

Note that the proximity binary relation $\rightarrow$ is not a strict partial order relation,  because it is not transitive. The simplest example of this fact may be obtained by varying slightly the choice of infinitely near points of Example \ref{ex:cuspembres}, illustrated in Figure \ref{fig:embrescusp}, whose notations we keep. Namely, start from the origin $O_0$ of $\Cc^2$,  choose an arbitrary point $O_1$ on the exceptional divisor $E_0$ of the blowup of $\Cc^2$ at $O_0$, then a point $O_2' \neq O_2$ on the exceptional divisor $E_1$ of the blowup at $O_1$. The inequality $O_2' \neq O_2$ is equivalent to the fact that $O_2'$ does not belong to the strict transform of $E_0$. Therefore $O_2' \rightarrow O_1 \rightarrow O_0$, but $O_2'$ is not proximate to $O_0$. 

Given a finite constellation endowed with the restriction of the partial order of Definition \ref{def:infnear} \eqref{itempartorder}, we may construct various models from it by blowing up all its non-maximal infinitely near points and also some of its maximal points. In order to specify which maximal points have to be blown up, we declare {\em active} all the elements of the constellation which we blow up:

\begin{definition}    
\label{def:actconst}
    Let $(S,O)$ be a smooth germ of surface. 

     \begin{enumerate}
    \item 
       Let $\calc$ be a finite constellation above $O$. An {\bf active subset} of it is any subset of $\calc$ which contains all non-maximal elements of the subposet $(\calc, \preceq)$ of $(\calc_{S, O}, \preceq)$. A constellation endowed with a preferred active subset is called an {\bf active constellation}.  Each constellation may be seen canonically as an active constellation by declaring active only its non-maximal points. The points of an active constellation which are not active are called {\bf inactive}. 

    \item  \label{itemodelconst}
        Let $\calc$ be a finite active constellation above $O$. We denote by $\boxed{\pi_{\calc}: S_{\calc} \to S}$ the composition of blowups of all active elements of $\calc$ and by $\boxed{E_{\calc}}$ its exceptional divisor. $S_{\calc}$ is called the {\bf model of the active constellation} $\calc$. 
    \end{enumerate}
\end{definition}


   
     \begin{figure}[h!]
    \begin{center}
\begin{tikzpicture}[scale=0.65]

\begin{scope}[shift={(2,0)}]

\node[draw,circle, inner sep=1.5pt,color=black, fill=black] at (0,0){};
\node [below, color=black] at (-0.4,0) {$O_0$};
\end{scope}


\begin{scope}[shift={(6,0)}]
     \draw [-, color=black, thick](-2,0) -- (2,0);
\node [below, color=black] at (0,0) {$O_1$};
\node [below, color=black] at (2,0) {$E_0$};
\node[draw,circle, inner sep=1.5pt,color=blue, fill=black] at (0,0){};

\end{scope}


\begin{scope}[shift={(12,0)}]
     \draw [-, color=black, thick](-2,0) -- (2,0);
\draw [-, color=black, thick] (0,-2)--(0,2);
\node [right, color=black] at (0,-2) {$E_{0}$};
\node [above, color=black] at (2,-0.8) {$E_1$};
\node[draw,circle, inner sep=1.5pt,color=black, fill=black] at (0,0){};
\node [below, color=black] at (0.5,0) {$O_2$};
\end{scope}


\begin{scope}[shift={(18,0)}]
     \draw [-, color=black, thick](-2,0) -- (2,0);
\node [above, color=black] at (-1.5,0) {$E_2$};
\node [right, color=black] at (-1.1,-1.5) {$E_0$};
\node [right, color=black] at (1,-1.5) {$E_1$};
\node[draw,circle, inner sep=1.5pt,color=black, fill=black] at (0,0){};
\node [below, color=black] at (0.5,0) {$O_3$};
\draw[thick, color=black](1,-1.5) -- (1, 1.5); 
\draw[thick, color=black](-1,-1.5) -- (-1, 1.5); 
\end{scope}


      \draw[<-](2.5,0)--(3.5,0);
     
      \draw[<-](8.5,0)--(9.5,0);
     
      \draw[<-](14.5,0)--(15.5,0);

\end{tikzpicture}
\end{center}
 \caption{The blowup process leading to the model of the active constellation $\{O_0, \dots, O_3\}$ in which $O_3$ is {\em inactive}}
\label{fig:assconstcusp}
   \end{figure}
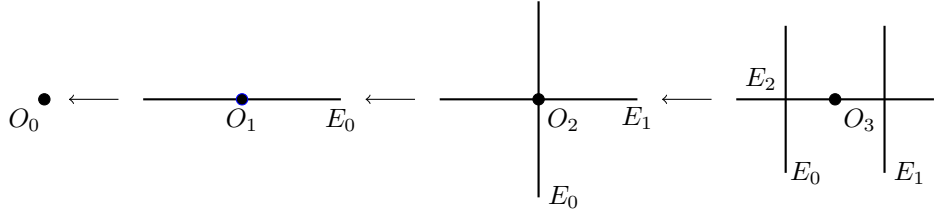


   
     \begin{figure}[h!]
    \begin{center}
\begin{tikzpicture}[scale=0.65]

\begin{scope}[shift={(2,0)}]

\node[draw,circle, inner sep=1.5pt,color=black, fill=black] at (0,0){};
\node [below, color=black] at (-0.4,0) {$O_0$};
\end{scope}


\begin{scope}[shift={(6,0)}]
     \draw [-, color=black, thick](-2,0) -- (2,0);
\node [below, color=black] at (0,0) {$O_1$};
\node [below, color=black] at (2,0) {$E_0$};
\node[draw,circle, inner sep=1.5pt,color=blue, fill=black] at (0,0){};

\end{scope}


\begin{scope}[shift={(12,0)}]
     \draw [-, color=black, thick](-2,0) -- (2,0);
\draw [-, color=black, thick] (0,-2)--(0,2);
\node [right, color=black] at (0,-2) {$E_{0}$};
\node [above, color=black] at (2,-0.8) {$E_1$};
\node[draw,circle, inner sep=1.5pt,color=black, fill=black] at (0,0){};
\node [below, color=black] at (0.5,0) {$O_2$};
\end{scope}


\begin{scope}[shift={(18,0)}]
     \draw [-, color=black, thick](-2,0) -- (2,0);
\node [above, color=black] at (-1.5,0) {$E_2$};
\node [right, color=black] at (-1.1,-1.5) {$E_0$};
\node [right, color=black] at (1,-1.5) {$E_1$};
\node[draw,circle, inner sep=1.5pt,color=black, fill=black] at (0,0){};
\node [below, color=black] at (0.5,0) {$O_3$};
\draw[thick, color=black](1,-1.5) -- (1, 1.5); 
\draw[thick, color=black](-1,-1.5) -- (-1, 1.5); 
\end{scope}



\begin{scope}[shift={(18,-6)}]
     \draw [-, color=black, thick](-2,0) -- (2,0);
\node [above, color=black] at (-1.5,0) {$E_2$};
\node [right, color=black] at (-1.1,-1.5) {$E_0$};
\node [right, color=black] at (1,-1.5) {$E_1$};
\node [right, color=black] at (0,-1.5) {$E_3$};
\draw[thick, color=black](1,-1.5) -- (1, 1.5); 
\draw[thick, color=black](-1,-1.5) -- (-1, 1.5); 

\draw[thick, color=black](0,-1.5) -- (0, 1.5); 
\end{scope}


      \draw[<-](2.5,0)--(3.5,0);
     
      \draw[<-](8.5,0)--(9.5,0);
     
      \draw[<-](14.5,0)--(15.5,0);

      \draw[<-](18,-3)--(18,-4);

\end{tikzpicture}
\end{center}
 \caption{The blowup process leading to the model of the active constellation $\{O_0, \dots, O_3\}$ in which $O_3$ is {\em active}}
\label{fig:assconstcuspbis}
   \end{figure}
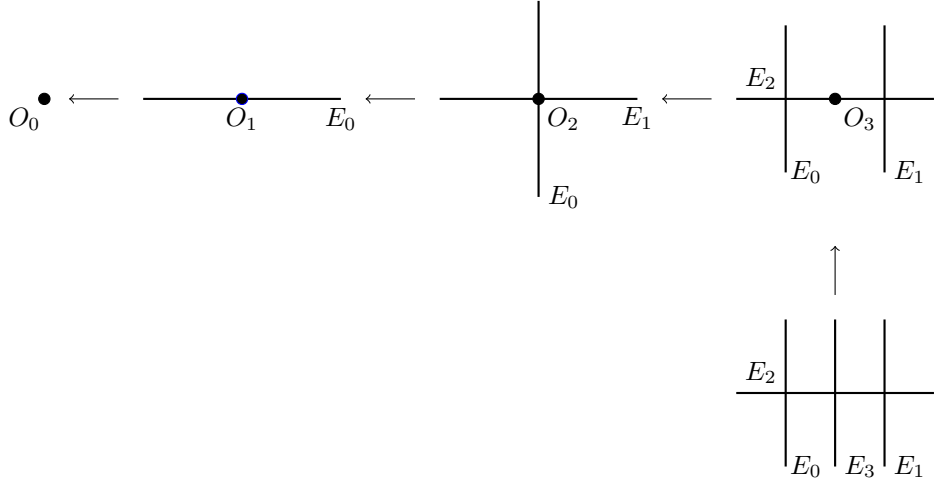

\begin{example}  \label{ex:assconstex}
    We return to Example \ref{ex:cuspembres},  
    that is, we consider again the minimal embedded resolution process of a cusp singularity $A$, represented in Figure \ref{fig:embrescusp}. $\calc = \{O_0, O_1, O_2, O_3\} $ is the constellation of infinitely near points lying on $A$ and its strict transforms till one achieves the minimal embedded resolution of $A$. Their proximity relations are $O_1 \to O_0$, $O_2 \to O_0$, $O_2 \to O_1$ and $O_3 \to O_2$. The point $O_3$ is the only maximal point of $(\calc, \preceq)$. It may be either {\em inactive} or {\em active}, the other points of $\calc$ being by contrast necessarily active. In Figures \ref{fig:assconstcusp} and \ref{fig:assconstcuspbis}  we represent the associated evolutions of the exceptional divisors during the blowup process of all active points, leading to the model of the active constellation $\calc$, in the sense of Definition \ref{def:actconst} \eqref{itemodelconst}. 
    \end{example}

Usually, one performs the blowup process leading to the minimal embedded resolution of a reduced plane curve singularity by computing the strict transforms in coordinate charts in which the exceptional divisors are coordinate axes. That is, each time one blows up an infinitely near point of $O$, there is also at that point an accompanying germ of normal crossings divisor with two components. We call such germs {\em crosses} and a constellation whose points are accompanied by crosses a {\em constellation of crosses} (this notion, as well as that of {\em active constellation}, seems to be new):

\begin{definition}  
    \label{def:constcrosses}
    Let $(S,O)$ be a smooth germ of surface. 
     \begin{enumerate}[(a)]
        \item  \label{orcross}
          A {\bf cross} on $(S,O)$ is a reduced
           germ of normal crossings divisor 
           with two branches. 
           An {\bf ordered cross} on $(S,O)$  is a cross endowed with 
           a total order of its branches.

        \item A {\bf constellation of crosses $\boxed{\hat{\calc}}$ (above $O$)} is a constellation $\calc$ above $O$ such that for each point $P \in \calc$ is given a cross $\boxed{X_P}$ on $(S^P, P)$ with the property that if $P, Q$  belong to $\calc$, and $P \preceq Q$ then:
          \begin{equation}  \label{eq:inclcrosses}
          ((\pi^Q_P)^{-1} (X_P), Q)  \subseteq   X_Q . 
          \end{equation}
        Here $((\pi^Q_P)^{-1} (X_P), Q)$ denotes the germ at $Q$ of the reduced total transform $(\pi^Q_P)^{-1} (X_P)$ of $X_P$ on $S^Q$.  

         \item The constellation of crosses $\hat{\calc}$ is called {\bf active} if the underlying constellation $\calc$ is active in the sense of Definition \ref{def:actconst}. 

        \item \label{totransfactcross}
          If $\hat{\calc}$ is a finite active constellation of crosses, we denote by $\boxed{X_{\hat{\calc}}}$ the total transform of all the crosses of $\hat{\calc}$ on the model $S_{\calc}$, and we call it the {\bf total transform of $\hat{\calc}$}. One has the inclusion $E_{\calc} \subset X_{\hat{\calc}}$, where $E_{\calc}$ is the exceptional divisor of the model of the finite active constellation underlying $\hat{\calc}$, in the sense of Definition \ref{def:actconst} \eqref{itemodelconst}. 
     \end{enumerate}
\end{definition}

\begin{remark}  \label{rem:ordcross}
    The notion of {\em ordered cross} is needed when one wants to speak of the Newton polygon of an effective divisor relative to a cross, as in \cite[Definition 1.4.14]{GBGPPP 20}. As Newton polygons were essential in the treatment of lotuses described in that paper, there {\em crosses} were assumed to be ordered. By contrast, in the present paper, crosses will be assumed unordered, unless otherwise stated. 
    More precisely, the only ordered crosses will be those corresponding to the base edges of lotuses, in the sense of Definition \ref{def:vocablotus}.
\end{remark}

Note that if $(\calc, (X_P)_{P \in \calc})$ is a constellation of crosses and if $Q \in \calc$ is different from $O$, then $X_Q$ contains the germ at $Q$ of the exceptional divisor of $\pi^Q_O$. This results from the particular case 
  \[ ((\pi^Q_O)^{-1} (X_O), Q)  \subseteq   X_Q \]
of relation (\ref{eq:inclcrosses}). Therefore, $X_Q$ is an {\em admissible cross} in the following sense:

\begin{definition}  \label{def:admcross}
    Let $\pi: (S_{\pi}, E_{\pi}) \to (S,O)$ be a model of $(S, O)$. Let $Q \in E_{\pi}$. An {\bf admissible cross} on $(S_{\pi}, Q)$ is a cross which contains the germ $(E_{\pi}, Q)$.
\end{definition}

Let us recall the following standard notion which will simplify the statement of Proposition \ref{prop:fincc} (see \cite[Remark 7.8]{GBGPPP 19} for historical information about this terminology):

\begin{definition} 
   \label{def:curvetta}
    Let $\pi : (S_{\pi}, E_{\pi})  \to (S,O)$ be a model of $(S,O)$ and let $Q$ be a smooth point of $E_{\pi}$. A {\bf curvetta} at $Q$ is a germ of smooth curve on $(S_{\pi}, Q)$ transversal to $E_{\pi}$. 
\end{definition}

Therefore, an admissible cross in the sense of Definition \ref{def:admcross} consists either of the germ $(E_{\pi}, Q)$ of the exceptional divisor $E_{\pi}$ of $\pi$ at a singular point $Q$ of $E_{\pi}$ or of the germ of the sum $E_{\pi} + C$ at a smooth point of $E_{\pi}$, where $C$ is a curvetta at $Q$.

{\em Finite} constellations of crosses may be alternatively described as resulting from the following kind of recursive processes:

\begin{proposition} \label{prop:fincc}
   Any finite active constellation of crosses may be obtained recursively by starting from a cross $X_O$ on $(S,O)$, by declaring $O$ inactive and by performing a finite number of times (possibly $0$) one of the following elementary  extension steps of a given active constellation of crosses  $\hat{\cald}$ with underlying active constellation $\cald$: 
  \begin{enumerate}[(a)]

      \item \label{stepactivate} 
        Choose an inactive point of $\cald$ and activate it. 
      
      \item \label{stepcross} 
        Choose a point $Q \in E_{\cald}$ which does not belong to $\cald$ and: 
         \begin{enumerate}[(b1)] 
            \item \label{donttouch} If $(X_{\hat{\cald}}, Q)$ is a cross, then define $X_Q := (X_{\hat{\cald}}, Q)$.
            
            \item \label{addcurvetta} If $(X_{\hat{\cald}}, Q)$ has only one component, which is necessarily the germ $(E_{\cald}, Q)$ at $Q$ of the exceptional divisor $E_{\cald}$, then add a curvetta $L_Q$ at $Q$ and define $X_Q := (E_{\cald}, Q) + L_Q$. 
         \end{enumerate}
            
    \noindent The resulting constellation of crosses is $\hat{\cald} \cup \{(Q, X_Q)\}$, its active points being those of $\hat{\cald}$.
  \end{enumerate}
  \end{proposition}

\begin{example}  \label{ex:assconstcrossex}
   Let us go back to Example \ref{ex:assconstex}. 
    In Figure \ref{fig:extconstcrosscusp} is represented an active constellation of crosses whose underlying constellation $\calc= \{ O_0, \dots, O_3 \}$ is that whose blowup process is shown in Figure \ref{fig:assconstcusp}. That is, the infinitely near point $O_3$ is declared inactive. Each cross is drawn using bold blue segments, excepted the last one which has an orange branch, representing the final strict transform of the starting branch $A$ of the embedded resolution of the cusp. This strict transform is a curvetta on $S_{\calc}$, in the sense of Definition \ref{def:curvetta}. The intermediate strict transforms of $A$ are drawn as dashed orange arcs or segments.
    
    Let us describe how to obtain $\calc$ by iterating steps of type \ref{stepactivate} or \ref{stepcross} of  Proposition \ref{prop:fincc}: 
      \begin{itemize}
        \item Start from the active constellation   of crosses $(O_0, X_{O_0} := L + L_1)$, in which $O_0$ is inactive.
        \item Perform step \ref{stepactivate} on    $O_0$, that is, activate it. 
        \item Choose $O_1$ as the intersection      point of the exceptional divisor $E_0$ of the blowup of $O_0$ and the strict transform of $L_1$. Then perform step \ref{donttouch}, that is, take $X_{O_1} := E_0 + L_1$. 
        \item Perform step \ref{stepactivate} on    $O_1$, that is, activate it. 
        \item Choose $O_2$ as the intersection      point of the exceptional divisor $E_1$ of the blowup of $O_1$ and the strict transform of $E_0$. Then perform step \ref{donttouch}, that is, take $X_{O_2} := E_0 + E_1$. 
        \item Perform step \ref{stepactivate} on    $O_2$, that is, activate it. 
        \item Choose $O_3$ as the intersection      point of the exceptional divisor $E_2$ of the blowup of $O_2$ and the strict transform of the branch $A$. Then perform step \ref{addcurvetta}, by taking $X_{O_3} := E_2 + A$. 
      \end{itemize}
\end{example}


   
     \begin{figure}[h!]
    \begin{center}
\begin{tikzpicture}[scale=0.65]

\begin{scope}[shift={(1,0)}]
\node [below, color=black] at (-0.4,0) {$O_0$};
\node [above, color=black] at (-1,0) {$L_{1}$};
\node [above, color=black] at (0,1) {$L$};
  \draw [-, color=blue, line width=2.5pt](-1,0) -- (1,0);
    \draw [-, color=blue, line width=2.5pt](0,-1) -- (0,1);
    \node[draw,circle, inner sep=1.5pt,color=black, fill=black] at (0,0){};
    
    \draw[dashed, color=orange, line width=1pt](0,0) .. controls (1,0.2) ..(1.7,1.5); 
\draw[dashed, color=orange, line width=1pt](0,0) .. controls (1,-0.2) ..(1.7,-1.5); 
\end{scope}


\begin{scope}[shift={(6,0)}]
    
\node [below, color=black] at (2,0) {$E_0$};

\node [above, color=black] at (1.5,1) {$L$};
\node [above, color=black] at (0,1) {$L_{1}$};
  \draw [-, color=blue, line width=2.5pt](-1,0) -- (1,0);
    \draw [-, color=blue, line width=2.5pt](0,-1) -- (0,1);
    \draw [-, color=black, thick](1.5,-1) -- (1.5,1);
     \draw [-, color=black, thick](-2,0) -- (2,0);
     \node[draw,circle, inner sep=1.5pt,color=blue, fill=black] at (0,0){};
\node [below, color=black] at (-0.4,0) {$O_1$};

\draw[dashed, color=orange, line width=1pt] (-1.3,1.5) .. controls (0,-0.5) .. (1.3,1.5);

\end{scope}


\begin{scope}[shift={(12,0)}]
     \draw [-, color=black, thick](-2,0) -- (2,0);
\draw [-, color=black, thick] (0,-2)--(0,2);
\node [below, color=black] at (0,-2) {$E_{0}$};
\node [above, color=black] at (2,-0.8) {$E_1$};
 \draw [-, color=blue, line width=2.5pt](-1,0) -- (1,0);
    \draw [-, color=blue, line width=2.5pt](0,-1) -- (0,1);
    \draw [-, color=black, thick](-1,-1.5) -- (1,-1.5);
    \draw [-, color=black, thick](-1.4,-1) -- (-1.4,1);
\node[draw,circle, inner sep=1.5pt,color=black, fill=black] at (0,0){};
\node [below, color=black] at (0.5,0) {$O_2$};
\node [right, color=black] at (1,-1.5) {$L$};
\node [above, color=black] at (-1.4,1) {$L_1$};

\draw[dashed, color=orange, line width=1pt](-1.3,-1.3) -- (1.3, 1.3); 
\end{scope}


\begin{scope}[shift={(18,0)}]
     \draw [-, color=black, thick](-2,0) -- (2,0);
\node [above, color=black] at (-1.8,0) {$E_2$};
\node [below, color=black] at (-1.2,-2) {$E_0$};
\node [right, color=black] at (1.4,-1.5) {$E_1$};
 \draw [-, color=blue, line width=2.5pt](-1,0) -- (1,0);
\draw [-, color=black, thick](-2,-1.5) -- (-0.3,-1.5);
\node [left, color=black] at (-2,-1.5) {$L$};
\node[draw,circle, inner sep=1.5pt,color=black, fill=black] at (0,0){};
\node [below, color=black] at (0.5,0) {$O_3$};
\draw[thick, color=black](1.4,-1.5) -- (1.4, 1.5); 
\draw[thick, color=black](-1.2,-2) -- (-1.2, 2); 
\draw [-, color=black, thick](0.6,1) -- (2.2,1);
\node [right, color=black] at (2.2,1) {$L_1$};

\draw [-, color=orange, line width=2.5pt](0,-1) -- (0,1);
\node [above, color=black] at (0,1) {$A$};

\end{scope}


      \draw[<-](2.5,0)--(3.5,0);
     
      \draw[<-](8.5,0)--(9.5,0);
     
      \draw[<-](14.5,0)--(15.5,0);

\end{tikzpicture}
\end{center}
 \caption{Extension of the active constellation of Figure \ref{fig:assconstcusp} into an active  constellation of crosses}
\label{fig:extconstcrosscusp}
   \end{figure}
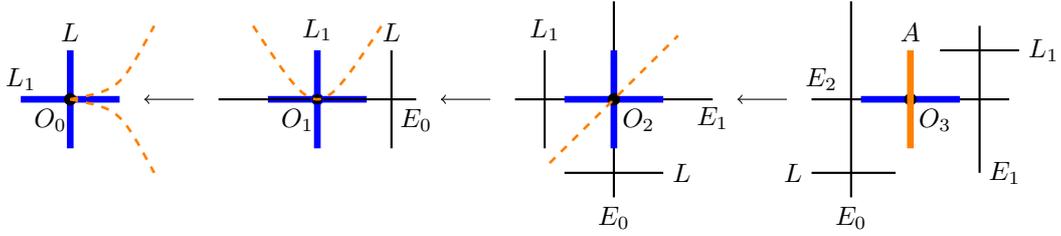

In the sequel we will be mainly interested in active constellations of crosses {\em adapted} to reduced plane curve singularities, in the following sense:

\begin{definition}  
  \label{def:adaptedactconst}
    Let $A \hookrightarrow (S,O)$ be a reduced plane curve singularity. Consider the constellation $\calc_A$ of infinitely near points lying on the strict transforms of $A$ on all the models lying between $S$ and an 
    embedded resolution $\pi : S_{\pi} \to S$ of $A$. Look at it as an active constellation, by declaring all its maximal points to be inactive. We say that an active constellation of crosses $\hat{\calc}_A$ with underlying active constellation $\calc_A$ is {\bf adapted to $A$} if one has the inclusion:
    \[ \pi^{-1}(A) \subseteq X_{\hat{\calc}_A}. \]
    That is, we ask that during the blowup process dictated by $\calc_A$, whenever  the total transform of $A$ is a cross at a point $P \in \calc_A$, the cross $X_P$ of $\hat{\calc}_A$ coincides with it. In this case, the {\bf completion of $A$ with respect to $\hat{\calc}_A$} is the projection  $\boxed{\hat{A}} := \pi ( X_{\hat{\calc}_A}) \hookrightarrow (S,O)$.
\end{definition}

The curve singularity $A$ is contained in $\hat{A}$ but it may have other branches corresponding to components of crosses which are not branches of $A$.  For instance, the active constellation of crosses of Example \ref{ex:assconstcrossex} is adapted to the cusp singularity $A$ and the completion of $A$ with respect to it is the sum of $A$ and the germs at $O$ of coordinate axes (see also Example \ref{ex:reconstrlotus} below). We introduced a very similar notion of {\em completion} in \cite[Definition 1.4.15]{GBGPPP 20}. Note that, unlike completions of local rings or metric spaces, completions of plane curve singularities are not canonical. 

Completions $\hat{A}$ of plane curve singularities $A$ will be important in Proposition \ref{prop:fromlotustoEW} and its Corollary \ref{cor:lotustoEW} and will lead us to introduce {\em complete Eggers-Wall trees} in Definition \ref{def:EW-complete}.

\medskip
\section{The lotus of a finite active constellation of crosses}  \label{sec:lotusactive}

In this section we explain how to associate a {\em lotus} to each finite active constellation of crosses (see Definition \ref{def:lotusfacc} and Proposition \ref{prop:interprlotus}). This notion is illustrated with three examples (see Figures \ref{fig:asslotuscusp}, \ref{fig:lotusbranch}, \ref{fig:lotus3branches}), which will then serve repeatedly in the remaining sections to illustrate various uses of lotuses as computational architectures.

\medskip
One associates classically a {\em dual graph} to every curve lying on a smooth surface which has only normal crossing singularities (see \cite{PP 22} for a historical presentation of the use of such graphs in singularity theory and Definition \ref{def:dualgraph} for a precise formulation). The simplest dual graphs are associated to crosses: 

\begin{definition}  \label{def:dualsegm}
 The {\bf dual segment of a cross} is an affine compact segment whose vertices are associated to the components of the cross (see Figure \ref{fig:dual segment}).
\end{definition}

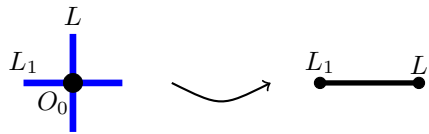
\begin{figure}[h!]
    \begin{center}
\begin{tikzpicture}[scale=0.65]

\begin{scope}[shift={(1,0)}]
\node [below, color=black] at (-0.4,0) {$O_0$};
\node [above, color=black] at (-1,0) {$L_{1}$};
\node [above, color=black] at (0,1) {$L$};
  \draw [-, color=blue, line width=2.5pt](-1,0) -- (1,0);
    \draw [-, color=blue, line width=2.5pt](0,-1) -- (0,1);
    \node[draw,circle, inner sep=2.5pt,color=black, fill=black] at (0,0){};
    \end{scope}
    \draw[->][thick, color=black](3,0) .. controls (4,-0.5) ..(5,0); 
    \begin{scope}[shift={(7,0)}]
    \draw [-, color=black, line width=2pt](-1,0) -- (1,0);
    \node[draw,circle, inner sep=1.5pt,color=black, fill=black] at (-1,0){};
     \node[draw,circle, inner sep=1.5pt,color=black, fill=black] at (1,0){};
     \node [above, color=black] at (-1,0) {$L_{1}$};
\node [above, color=black] at (1,0) {$L$};
    \end{scope}
    \end{tikzpicture}
    \end{center}
     \caption{A cross and it dual segment}
\label{fig:dual segment}
    \end{figure}

The following definition is the central one of the paper. It associates a {\em lotus} -- a special type of two-dimensional simplicial complex -- to every finite active constellation of crosses:

  \begin{definition}  \label{def:lotusfacc}
      Let $\hat{\calc}$ be a finite active  constellation of crosses. Its {\bf lotus} $\boxed{\Lambda(\hat{\calc})}$ is the two-dimensional simplicial complex constructed recursively as follows, starting from the dual segment of the cross $X_O$ of $\hat{\calc}$ at $O$ and following the elementary extension steps of a given active constellation of crosses $\hat{\cald}$ explained in  Proposition \ref{prop:fincc}, whose notations we keep using:
      
      \begin{enumerate}[(A)]

      \item \label{(A)}
         When an inactive point $P$ of $\hat{\cald}$ is activated as in step \ref{stepactivate}, consider a two-dimensional simplex called a {\bf petal}, whose vertices are labeled by the branches of the cross $X_P$ and by the exceptional divisor $E_P$ of the blowup of $P$. The {\bf base edge} of this  petal is the edge labeled by the branches of $X_P$. Then, glue affinely the base edge of the petal with 
      the dual segment of the cross $X_P$ seen in $\Lambda(\hat{\cald})$, identifying vertices with the same label.

      \item When a point $Q$ is chosen on $E_{\cald} \smallsetminus \cald$ as in step \ref{stepcross}, we consider two cases:

        \begin{enumerate}[(B1)]  \label{lotuscaseB}
            \item \label{(B1)}
            If $X_Q = (X_{\hat{\cald}})_Q$, then the dual segment of $(X_{\hat{\cald}})_Q$ is already an edge of the lotus $\Lambda(\hat{\cald})$.
            
            \item \label{(B2)}
            If $X_Q = (E_{\cald}, Q) + L_Q$, then glue the dual segment of the cross $X_Q$ to the lotus $\Lambda(\hat{\cald})$ by identifying the vertices corresponding  to $(E_{\cald}, Q)$. 
      \end{enumerate}
  \end{enumerate}
  \end{definition}

\begin{example}  \label{ex:reconstrlotus}
    In Figure \ref{fig:asslotuscusp} is represented the recursive construction of the lotus of the finite active constellation of crosses from Figure \ref{fig:extconstcrosscusp}, which illustrates Example \ref{ex:assconstcrossex}.  We have drawn as bold segments the edges of the lotus which are the dual segments of the crosses of the constellation. Starting from the dual segment $[LL_1]$ of the cross $X_O$, one performs successively three steps of type \ref{(A)} and one step of type \ref{(B2)}. 
\end{example}


   
     \begin{figure}[h!]
    \begin{center}
\begin{tikzpicture}[scale=0.8]

\begin{scope}[shift={(-4,0)}]
 \draw [-, color=black, line width=2.5pt](-1,0) -- (1,0);
  \node[draw,circle, inner sep=1.5pt,color=black, fill=black] at (-1,0){};
   \node[draw,circle, inner sep=1.5pt,color=black, fill=black] at (1,0){};
   \node [below, color=black] at (-1,0) {$L_{1}$};
\node [below, color=black] at (1,0) {$L$};
\end{scope}

\begin{scope}[shift={(0,0)}]
\draw [fill=pink!25](-1,0) -- (1,0)--(0,1.5);
 \draw [-, color=black, line width=2.5pt](-1,0) -- (1,0);
  \draw [-, color=black, thick](-1,0) -- (0,1.5);
  \draw [-, color=black, thick](1,0) -- (0,1.5);
  \node[draw,circle, inner sep=1.5pt,color=black, fill=black] at (-1,0){};
   \node[draw,circle, inner sep=1.5pt,color=black, fill=black] at (1,0){};
    \node[draw,circle, inner sep=1.5pt,color=black, fill=black] at (0,1.5){};
   \node [below, color=black] at (-1,0) {$L_{1}$};
\node [below, color=black] at (1,0) {$L$};
\node [right, color=black] at (0,1.5) {$E_{0}$};
\end{scope}

\begin{scope}[shift={(4,0)}]
\draw [fill=pink!25](-1,0) -- (1,0)--(0,1.5);
\draw [fill=pink!50](-0.7,2.7)--(-1,0)--(0,1.5);
 \draw [-, color=black, line width=2.5pt](-1,0) -- (1,0);
  \draw [-, color=black, line width=2.5pt](-1,0) -- (0,1.5);
  \draw [-, color=black, thick](1,0) -- (0,1.5);
   \draw [-, color=black, thick](-0.7,2.7) -- (0,1.5);
     \draw [-, color=black, thick](-0.7,2.7) -- (-1,0);
  \node[draw,circle, inner sep=1.5pt,color=black, fill=black] at (-1,0){};
   \node[draw,circle, inner sep=1.5pt,color=black, fill=black] at (1,0){};
    \node[draw,circle, inner sep=1.5pt,color=black, fill=black] at (0,1.5){};
     \node[draw,circle, inner sep=1.5pt,color=black, fill=black] at (-0.7,2.7){};
   \node [below, color=black] at (-1,0) {$L_{1}$};
\node [below, color=black] at (1,0) {$L$};
\node [right, color=black] at (0,1.5) {$E_{0}$};
\node [left, color=black] at (-1,3) {$E_{1}$};
\end{scope}

\begin{scope}[shift={(8,0)}]
\draw [fill=pink!25](-1,0) -- (1,0)--(0,1.5);
\draw [fill=pink!50](-0.7,2.7)--(-1,0)--(0,1.5);
\draw [fill=pink!75](-0.7,2.7)--(1,3)--(0,1.5);
 \draw [-, color=black, line width=2.5pt](-1,0) -- (1,0);
  \draw [-, color=black, line width=2.5pt](-1,0) -- (0,1.5);
  \draw [-, color=black, thick](1,0) -- (0,1.5);
   \draw [-, color=black, line width=2.5pt](-0.7,2.7) -- (0,1.5);
     \draw [-, color=black, thick](-0.7,2.7) -- (-1,0);
     \draw [-, color=black, thick](0,1.5) -- (1,3);
     \draw [-, color=black, thick](-0.7,2.7) -- (1,3);
  \node[draw,circle, inner sep=1.5pt,color=black, fill=black] at (-1,0){};
   \node[draw,circle, inner sep=1.5pt,color=black, fill=black] at (1,0){};
    \node[draw,circle, inner sep=1.5pt,color=black, fill=black] at (0,1.5){};
     \node[draw,circle, inner sep=1.5pt,color=black, fill=black] at (-0.7,2.7){};
     \node[draw,circle, inner sep=1.5pt,color=black, fill=black] at (1,3){};
   \node [below, color=black] at (-1,0) {$L_{1}$};
\node [below, color=black] at (1,0) {$L$};
\node [right, color=black] at (0,1.5) {$E_{0}$};
\node [left, color=black] at (-0.7,2.7) {$E_{1}$};
\node [right, color=black] at (1,3) {$E_{2}$};
\end{scope}

\begin{scope}[shift={(12,0)}]
\draw [fill=pink!25](-1,0) -- (1,0)--(0,1.5);
\draw [fill=pink!50](-0.7,2.7)--(-1,0)--(0,1.5);
\draw [fill=pink!75](-0.7,2.7)--(1,3)--(0,1.5);
 \draw [-, color=black, line width=2.5pt](-1,0) -- (1,0);
  \draw [-, color=black, line width=2.5pt](-1,0) -- (0,1.5);
  \draw [-, color=black, thick](1,0) -- (0,1.5);
   \draw [-, color=black, line width=2.5pt](-0.7,2.7) -- (0,1.5);
     \draw [-, color=black, thick](-0.7,2.7) -- (-1,0);
      \draw [-, color=black, thick](0,1.5) -- (1,3);
     \draw [-, color=black, thick](-0.7,2.7) -- (1,3);
      \draw [-, color=black, line width=2.5pt](1,3) -- (0.7,4);
  \node[draw,circle, inner sep=1.5pt,color=black, fill=black] at (-1,0){};
   \node[draw,circle, inner sep=1.5pt,color=black, fill=black] at (1,0){};
    \node[draw,circle, inner sep=1.5pt,color=black, fill=black] at (0,1.5){};
     \node[draw,circle, inner sep=1.5pt,color=black, fill=black] at (-0.7,2.7){};
     \node[draw,circle, inner sep=1.5pt,color=black, fill=black] at (1,3){};
    \node[draw,circle, inner sep=1.5pt,color=black, fill=black] at (0.7,4){};
   \node [below, color=black] at (-1,0) {$L_{1}$};
\node [below, color=black] at (1,0) {$L$};
\node [right, color=black] at (0,1.5) {$E_{0}$};
\node [left, color=black] at (-0.7,2.7) {$E_{1}$};
\node [right, color=black] at (1,3) {$E_{2}$};
\node [above, color=black] at (0.7,4) {$A$};
\end{scope}

\end{tikzpicture}
\end{center}
 \caption{From left to right: construction of the lotus of the active constellation of crosses of Figure \ref{fig:extconstcrosscusp},  according to the algorithm described in Definition \ref{def:lotusfacc}}
\label{fig:asslotuscusp}
   \end{figure}
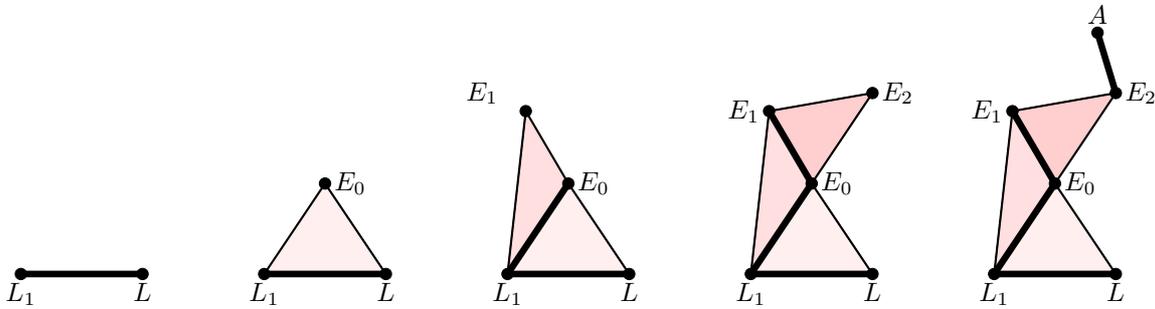

The lotus of a finite active constellation of crosses may be also described as follows:

  \begin{proposition}
    \label{prop:interprlotus}
      Let $\hat{\calc}$ be a finite active constellation of crosses. Then its lotus is the simplicial complex obtained in the following way:
        \begin{itemize}
            \item Its vertices are the irreducible components of $X_{\hat{\calc}}$. 

             \item Two vertices are joined by an edge if and only if the strict transforms of the corresponding components intersect on one of the surfaces $S^P$, when $P$ varies among the elements of the constellation $\hat{\calc}$. 

             \item If three points are pairwise joined by edges in the resulting graph, then they are the vertices of a unique two-dimensional simplex.
        \end{itemize}
  \end{proposition}

Therefore:

\medskip
\begin{center}
\fbox{
\begin{minipage}{0.75\textwidth}
   \begin{corollary} 
      \label{cor:spacetimerep}
   The lotus of a finite active constellation of crosses is a space-time representation of the evolution of the dual graphs of total transforms of the crosses of the constellation, as one performs the blowups leading to its model.
   \end{corollary}
\end{minipage}
}
\end{center}
\medskip

We will use the following terminology about certain subsets of the lotus of an active constellation:

  \begin{definition}  
     \label{def:vocablotus}
      Let $\hat{\calc}$ be a finite active constellation of crosses and let 
      $\Lambda := \Lambda(\hat{\calc})$ be its lotus, in the sense of Definition \ref{def:lotusfacc}. The edges of $\Lambda$ are distinguished as follows into types: 

     \noindent $\bullet$ 
          A {\bf base edge} is the dual segment of the cross $X_O$ of $\hat{\calc}$ at $O$ or of a  cross obtained by step \ref{(B2)} of Definition \ref{def:lotusfacc}. Each base edge is {\em oriented}, either conventionally in what concerns the dual segment of $X_O$, or from its vertex representing an exceptional curve towards the vertex representing a curvetta, in the sense of Definition \ref{def:curvetta}. In this way, the cross $X_O$ and the crosses obtained by step \ref{(B2)} of Definition \ref{def:lotusfacc} are {\em ordered}, in the sense of Definition \ref{def:constcrosses} \ref{orcross}.
             
     \noindent $\bullet$ 
         A {\bf lateral edge} is the dual segment of a cross obtained as the germ of the total transform of $\hat{\calc}$ at a singular point of it, provided that it is not a base edge (see Definition \ref{def:constcrosses}\ref{totransfactcross}).

     \noindent $\bullet$ 
         An {\bf internal edge} is the dual segment of a cross at an active point of $\hat{\calc}$. 

   The {\bf  lateral boundary}  
   $\boxed{\partial_+ \Lambda}$ of the lotus  
   $\Lambda$ is  the union of its lateral edges  and  of its base edges which are dual to the crosses at inactive points.

   The {\bf initial vertex} of $\Lambda$, denoted $\boxed{L}$,  is the vertex of the dual segment of $X_O$ which is the starting vertex of the chosen orientation. A {\bf base vertex} is a vertex of the lotus $\Lambda$ corresponding to an element of the divisor $X_{\hat{\calc}}$ which is a curvetta. The other vertices of $\Lambda$ are called {\bf lateral vertices}. A {\bf rupture vertex} is a vertex of $\Lambda$ whose removal disconnects the lotus. A {\bf membrane} of $\Lambda$ is the closure inside $\Lambda$ of a connected component of the complement of the set of rupture vertices of $\Lambda$. A membrane contains a unique base edge of $\Lambda$. The {\bf initial vertex} of a membrane is the starting vertex of its unique base edge.
  \end{definition}

We leave to the reader the task of  proving by induction on the number of steps described in Definition \ref{def:lotusfacc} needed to construct the lotus $\Lambda(\hat{\calc})$ that the types  {\em base/lateral/internal} define a partition of its set of edges. Similarly, the types {\em base/lateral} define a partition of the set of vertices. The rupture vertices are particular lateral vertices.

  There is a bijection from the set of active points of an active constellation of crosses $\hat{\calc}$ and the set of lateral vertices of its lotus $\Lambda(\hat{\calc})$. This bijection sends each active point $P$ of $\calc$ to the vertex representing the irreducible component $E_P$ (born by blowing up $P$) of the exceptional divisor $E_{\calc}$ introduced in Definition \ref{def:constcrosses}. One may prove  by induction that:

  \medskip
\begin{center}
\fbox{
\begin{minipage}{0.75\textwidth}
\begin{proposition}
    \label{prop:graphprox}
   The total subgraph of the $1$-skeleton of $\Lambda(\hat{\calc})$ joining its lateral vertices corresponds under the bijection $P \mapsto E_P$ to the graph of the proximity binary relation of Definition \ref{def:infnear} \eqref{proxnot} on the set of active points of $\hat{\calc}$. By definition, this graph joins by an edge each pair of proximate active points of the constellation. 
\end{proposition}
\end{minipage}
}
\end{center}
\medskip

\begin{example}  
  \label{ex:typeselements}
    Consider the final lotus $\Lambda$ of the construction process described in Figure \ref{fig:asslotuscusp}. It has three petals, filled in shades of pink. Its base edges are $[L L_1]$ and $[E_2 A]$ and its lateral edges are $[L E_0]$, $[E_0 E_2]$, $[E_2 E_1]$ and $[E_1 L_1]$, therefore its lateral boundary  is: 
     \[\partial_+ \Lambda = [L E_0] \cup [E_0 E_2] \cup [E_2 E_1] \cup [E_1 L_1] \cup [E_2 A]. \]
    Its internal edges are $[L_1 E_0]$ and $[E_0 E_1]$. Its base vertices are $L, L_1$ and $A$. Its lateral vertices are $E_0, E_1, E_2$. The triangle with vertices $E_0, E_1, E_2$ is the total subgraph of the $1$-skeleton of $\Lambda$ joining its lateral vertices. It represents therefore the proximity binary relation on the set $\{O_0, O_1, O_2\}$ of active points of the active constellation of crosses of Example \ref{ex:assconstcrossex}. Indeed: $O_1 \rightarrow O_0$, $O_2 \rightarrow O_0$ and $O_2 \rightarrow O_1$.
    The only rupture vertex is $E_2$. 
     There are two membranes, the segment $[E_2 A]$ and the closed triangulated polygon with vertices $L, L_1, E_1, E_2$ and $E_0$ whose base edge is $[L L_1]$.
\end{example}

 \begin{figure}[h!]
    \begin{center}
\begin{tikzpicture}[scale=1]

 \draw [->, color=black, thick](1,0) -- (0,0);
    \draw [-, color=black, thick](0,0) -- (-1,0);
    \draw [-, color=black, thick](0,1.5) -- (-1,0);
      \draw [-, color=black, thick](0,1.5) -- (1,0);
       \draw [-, color=black, thick](-0.5,3/4) -- (0.5,3/4);
       \draw [-, color=black, thick](-1,0) -- (0.5,3/4);

        \draw [->, color=black, thick](0,1.5) -- (-1.5,1.5);
   \draw [-, color=black, thick](-1.5,1.5) -- (-3,1.5);
     \draw [-, color=black, thick](0,1.5) -- (-1.5,4);
      \draw [-, color=black, thick](-1.5,4) -- (-3,1.5);
    \draw [->, color=black, thick](-1.5,4) -- (-1.5,5);
     \draw [-, color=black, thick](-3/4,11/4) -- (-3,1.5);
     \draw [-, color=black, thick](-3/4,11/4) -- (-2.5,2.34);
       \draw [-, color=black, thick](-3/4,11/4) -- (-2,3.17);

  \node[draw,circle, inner sep=1.5pt,color=black, fill=black] at (-1,0){};
   \node[draw,circle, inner sep=1.5pt,color=black, fill=black] at (1,0){};
     \node[draw,circle, inner sep=1.5pt,color=black, fill=black] at (0,1.5){};
      \node[draw,circle, inner sep=1.5pt,color=black, fill=black] at (-0.5,3/4){};
       \node[draw,circle, inner sep=1.5pt,color=black, fill=black] at (0.5,3/4){};
        \node[draw,circle, inner sep=1.5pt,color=black, fill=black] at (-3,1.5){};
         \node[draw,circle, inner sep=1.5pt,color=black, fill=black] at (-1.5,4){};
         \node[draw,circle, inner sep=1.5pt,color=black, fill=black] at (-3/4,11/4){};
         \node[draw,circle, inner sep=1.5pt,color=black, fill=black] at (-2.5,2.34){};
         \node[draw,circle, inner sep=1.5pt,color=black, fill=black] at (-2,3.17){};

   \node [below, color=black] at (-1,0) {$L_{1}$};
   \node [below, color=black] at (1,0) {$L$};
   \node [below, color=black] at (-3,1.5) {$L_{2}$};
   \node [right, color=black] at (-1.5,5) {$A_{1}$};

 \end{tikzpicture}
\end{center}
 \caption{The lotus of an active constellation of crosses adapted to a branch $A_1$}
 \label{fig:lotusbranch}
   \end{figure}
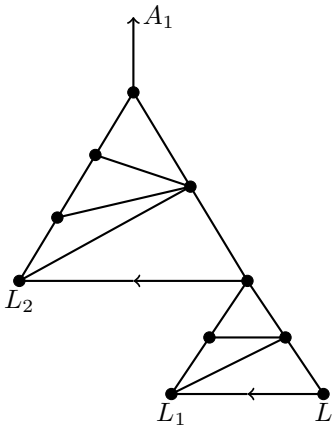

 \begin{figure}[h!]
    \begin{center}
\begin{tikzpicture}[scale=1]

 \draw [->, color=black, thick](1,0) -- (0,0);
    \draw [-, color=black, thick](0,0) -- (-1,0);
    \draw [-, color=black, thick](0,1.5) -- (-1,0);
      \draw [-, color=black, thick](0,1.5) -- (1,0);
       \draw [-, color=black, thick](-0.5,3/4) -- (0.5,3/4);
       \draw [-, color=black, thick](-1,0) -- (0.5,3/4);
       \draw [-, color=black, thick](-5.2,4.2) -- (-3.7,3.2);

        \draw [->, color=black, thick](0,1.5) -- (-1.5,1.5);
   \draw [-, color=black, thick](-1.5,1.5) -- (-3,1.5);
     \draw [-, color=black, thick](0,1.5) -- (-1.5,4);
      \draw [-, color=black, thick](-1.5,4) -- (-3,1.5);
    \draw [->, color=black, thick](-1.5,4) -- (-1.5,5);
     \draw [-, color=black, thick](-3/4,11/4) -- (-3,1.5);
     \draw [-, color=black, thick](-3/4,11/4) -- (-2.5,2.34);
       \draw [-, color=black, thick](-3/4,11/4) -- (-2,3.17);
       \draw [->, color=black, thick](-3/4,11/4) -- (1/4,15/4);
        \draw [-, color=black, thick](-2,3.17) -- (-1.5,4);
        \draw [-, color=black, thick](-3,1.5) -- (-2.5,2.34);
        \draw [-, color=black, thick](-3,1.5) -- (-2.5,2.34);
        \draw [-, color=black, thick](-3.7,3.2) -- (-2,3.17);
        \draw [-, color=black, thick](-3.7,3.2) -- (-2.5,2.34);
        \draw [->, color=black, thick](-3.7,3.2) -- (-4.7,3.2);
        \draw [-, color=black, thick](-5.7,3.2) -- (-4.7,3.2);
           \draw [->, color=black, thick](-4.7,5.2) -- (-4.7,6.2);
            \draw [-, color=black, thick](-3.7,3.2) -- (-4.7,5.2);
            \draw [-, color=black, thick](-5.7,3.2) -- (-4.7,5.2);
       
  \node[draw,circle, inner sep=1.5pt,color=black, fill=black] at (-1,0){};
   \node[draw,circle, inner sep=1.5pt,color=black, fill=black] at (1,0){};
     \node[draw,circle, inner sep=1.5pt,color=black, fill=black] at (0,1.5){};
      \node[draw,circle, inner sep=1.5pt,color=black, fill=black] at (-0.5,3/4){};
       \node[draw,circle, inner sep=1.5pt,color=black, fill=black] at (0.5,3/4){};
        \node[draw,circle, inner sep=1.5pt,color=black, fill=black] at (-3,1.5){};
         \node[draw,circle, inner sep=1.5pt,color=black, fill=black] at (-1.5,4){};
         \node[draw,circle, inner sep=1.5pt,color=black, fill=black] at (-3/4,11/4){};
         \node[draw,circle, inner sep=1.5pt,color=black, fill=black] at (-2.5,2.34){};
         \node[draw,circle, inner sep=1.5pt,color=black, fill=black] at (-2,3.17){};
          \node[draw,circle, inner sep=1.5pt,color=black, fill=black] at (-3.7,3.2){};
          \node[draw,circle, inner sep=1.5pt,color=black, fill=black] at (-5.7,3.2){};
           \node[draw,circle, inner sep=1.5pt,color=black, fill=black] at (-4.7,5.2){};
           \node[draw,circle, inner sep=1.5pt,color=black, fill=black] at (-5.2,4.2){};
       
   \node [below, color=black] at (-1,0) {$L_{1}$};
\node [below, color=black] at (1,0) {$L$};
\node [below, color=black] at (-3,1.5) {$L_{2}$};
\node [right, color=black] at (1/4,15/4) {$A_{3}$};
\node [right, color=black] at (-1.5,5) {$A_{1}$};
 \node [below, color=black] at (-5.7,3.2) {$L_{3}$};
 \node [above, color=black] at (-4.7,6.2) {$A_{2}$};

 \end{tikzpicture}
\end{center}
 \caption{The lotus of an active constellation of crosses adapted to a sum $A_1 + A_2 + A_3$ of three branches}
\label{fig:lotus3branches}
   \end{figure}
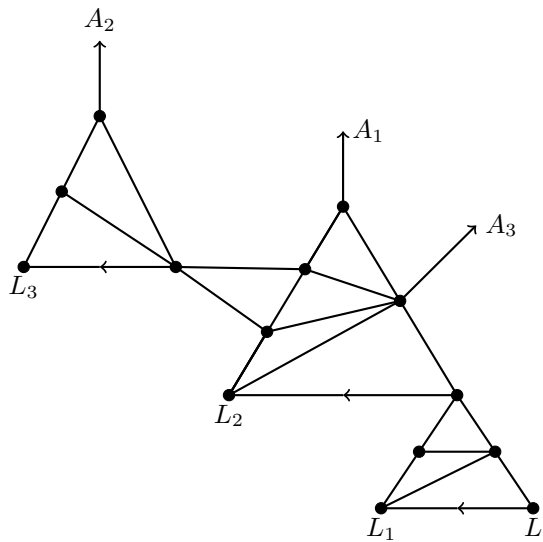

   \begin{example}  \label{ex:lotusbranch}
    In Figure \ref{fig:lotusbranch} is represented the lotus of an active constellation of crosses adapted to a branch $A_1$, in the sense of Definition \ref{def:adaptedactconst}. Its base vertices are $L, L_1, L_2$ and $A_1$. It has $2$ rupture vertices, $3$ base edges, which wear arrows, $9$ lateral edges and $5$ internal ones.  In Figure \ref{fig:lotus3branches} is represented the lotus of an active constellation of crosses adapted to a sum $A_1 + A_2 + A_3$ of three branches. Its base vertices are $L, L_1, L_2, L_3, A_1, A_2, A_3$. It has $5$ rupture vertices, $6$ base edges, which wear arrows, $13$ lateral edges and $7$ internal ones. It is not obvious to find defining series of branches having active constellation of crosses whose lotuses are shown in Figures \ref{fig:lotusbranch} and \ref{fig:lotus3branches}. We will explain in Example \ref{ex:EWtoseries} how to get such series. 
\end{example}

In the sequel, we will illustrate all the computations of invariants on the final lotus of Figure \ref{fig:asslotuscusp} and in the lotuses of Figures \ref{fig:lotusbranch} and \ref{fig:lotus3branches}. Starting from Figure \ref{fig:lotusbranch}, we do not fill any more the petals of the lotuses, as they are easily visible. In both Figures \ref{fig:lotusbranch} and \ref{fig:lotus3branches} we have represented with arrowheads the vertices corresponding to the branches of the plane curve singularity to which the given active constellation of crosses is adapted, in the sense of Definition \ref{def:adaptedactconst}. This is a standard convention in the representation of dual graphs of total transforms of plane curve singularities by their embedded resolutions. We will use this convention throughout the paper.

\begin{remark}
    \label{rem:semipet}
      One may also associate a type of lotus to a constellation which is not endowed with crosses, by gluing not only {\em petals} and {\em dual segments} but also {\em semipetals}, as was explained by the third author in \cite{PP 11}. There he called the resulting object a {\em sail}, a terminology we use neither here nor in \cite[Section 1.5.5]{GBGPPP 20}, where such generalized lotuses were examined in the context of embedded resolutions of plane curve singularities. We work with constellation of crosses instead of just constellations precisely because their lotuses are simpler, in the sense that they do not need semipetals in order to be constructed. 
\end{remark}

\medskip
\section{How to see the weighted dual graph inside the lotus}  
\label{sec:dualgr}

In this section we explain how to see the weighted dual graph of the total transform of a finite active constellation of crosses inside the associated lotus (see Proposition \ref{prop:dualgraph}).

\medskip
Each configuration of smooth curves intersecting normally in a smooth complex surface has an associated {\em dual graph}, which generalizes the dual segment of a cross introduced in Definition \ref{def:dualsegm}:

\begin{definition}  \label{def:dualgraph}
    A reduced normal crossing divisor on a smooth complex surface is called {\bf simply normal crossing} if all its irreducible components are smooth. The {\bf weighted dual graph} $\boxed{G(D)}$ of a simply normal crossing divisor $D$ is constructed in the following way: 
    
     \noindent $\bullet$ 
           its {\bf vertices} correspond bijectively with the irreducible components of $D$; 

     \noindent $\bullet$ 
            the {\bf edges} between two given vertices are the dual segments of the crosses contained in the union of the components of $D$ corresponding to these two vertices;

     \noindent $\bullet$ 
           each vertex corresponding to a compact irreducible component of $D$ is {\bf weighted} by the self-intersection number of that component in the ambient surface. 
\end{definition}

   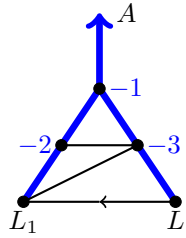
\begin{figure}[h!]
    \begin{center}
\begin{tikzpicture}[scale=1]

 \draw [->, color=black, thick](1,0) -- (0,0);
    \draw [-, color=black, thick](0,0) -- (-1,0);
    \draw [-, color=blue, line width=2.5pt](0,1.5) -- (-1,0);
      \draw [-, color=blue, line width=2.5pt](0,1.5) -- (1,0);
       \draw [-, color=black, thick](-0.5,3/4) -- (0.5,3/4);
       \draw [-, color=black, thick](-1,0) -- (0.5,3/4);
        \draw [->, color=blue, line width=2.5pt](0,1.5) -- (0,2.5);

  \node[draw,circle, inner sep=1.5pt,color=black, fill=black] at (-1,0){};
   \node[draw,circle, inner sep=1.5pt,color=black, fill=black] at (1,0){};
     \node[draw,circle, inner sep=1.5pt,color=black, fill=black] at (0,1.5){};
      \node[draw,circle, inner sep=1.5pt,color=black, fill=black] at (-0.5,3/4){};
       \node[draw,circle, inner sep=1.5pt,color=black, fill=black] at (0.5,3/4){};
\node [below, color=black] at (-1,0) {$L_{1}$};
\node [below, color=black] at (1,0) {$L$};
\node [right, color=black] at (0.1,2.5) {$A$};
 \node [right, color=blue] at (0,1.5) {$\mathbf -1$};   
  \node [left, color=blue] at (-0.5,3/4) {$\mathbf -2$};
   \node [right, color=blue] at (0.5,3/4) {$\mathbf -3$};

  \end{tikzpicture}
\end{center}
 \caption{The weighted dual graph of the active constellation of crosses of Example \ref{ex:assconstcrossex}}
\label{fig:dualgrcusp}
   \end{figure}


 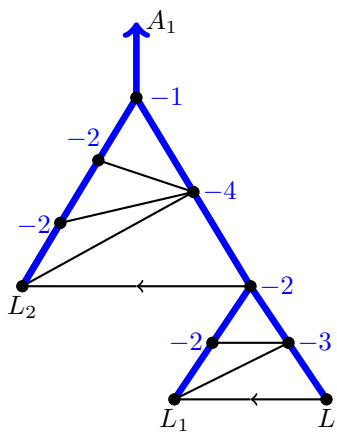
\begin{figure}[h!]
    \begin{center}
\begin{tikzpicture}[scale=1]

 \draw [->, color=black, thick](1,0) -- (0,0);
    \draw [-, color=black, thick](0,0) -- (-1,0);
    \draw [-, color=blue, line width=2.5pt](0,1.5) -- (-1,0);
      \draw [-, color=blue, line width=2.5pt](0,1.5) -- (1,0);
       \draw [-, color=black, thick](-0.5,3/4) -- (0.5,3/4);
       \draw [-, color=black, thick](-1,0) -- (0.5,3/4);

        \draw [->, color=black, thick](0,1.5) -- (-1.5,1.5);
   \draw [-, color=black, thick](-1.5,1.5) -- (-3,1.5);
     \draw [-, color=blue, line width=2.5pt](0,1.5) -- (-1.5,4);
      \draw [-, color=black, thick](-1.5,4) -- (-3,1.5);
    \draw [->, color=blue, line width=2.5pt](-1.5,4) -- (-1.5,5);
     \draw [-, color=black, thick](-3/4,11/4) -- (-3,1.5);
     \draw [-, color=black, thick](-3/4,11/4) -- (-2.5, 2.34);
       \draw [-, color=black, thick](-3/4,11/4) -- (-2, 3.17);
      
        \draw [-, color=blue, line width=2.5pt](-2,3.17) -- (-1.5,4);
        \draw [-, color=blue, line width=2.5pt](-3,1.5) -- (-2.5, 2.34);
        \draw [-, color=blue, line width=2.5pt](-3,1.5) -- (-2.5,2.34);

        \draw [-, color=blue, line width=2.5pt](-2, 3.17) -- (-2.5,2.34);

  \node[draw,circle, inner sep=1.5pt,color=black, fill=black] at (-1,0){};
   \node[draw,circle, inner sep=1.5pt,color=black, fill=black] at (1,0){};
     \node[draw,circle, inner sep=1.5pt,color=black, fill=black] at (0,1.5){};
      \node[draw,circle, inner sep=1.5pt,color=black, fill=black] at (-0.5,3/4){};
       \node[draw,circle, inner sep=1.5pt,color=black, fill=black] at (0.5,3/4){};
        \node[draw,circle, inner sep=1.5pt,color=black, fill=black] at (-3,1.5){};
         \node[draw,circle, inner sep=1.5pt,color=black, fill=black] at (-1.5,4){};
         \node[draw,circle, inner sep=1.5pt,color=black, fill=black] at (-3/4,11/4){};
         \node[draw,circle, inner sep=1.5pt,color=black, fill=black] at (-2.5,2.34){};
         \node[draw,circle, inner sep=1.5pt,color=black, fill=black] at (-2,3.17){};

   \node [below, color=black] at (-1,0) {$L_{1}$};
\node [below, color=black] at (1,0) {$L$};
\node [below, color=black] at (-3,1.5) {$L_{2}$};
\node [right, color=black] at (-1.5,5) {$A_{1}$};
 \node [right, color=blue] at (0,1.5) {$\mathbf -2$};   
  \node [left, color=blue] at (-0.5,3/4) {$\mathbf -2$};
   \node [right, color=blue] at (0.5,3/4) {$\mathbf -3$};
     \node [right, color=blue] at (-1.45,4) {$\mathbf -1$}; 
    \node [right, color=blue] at (-3/4,11/4) {$\mathbf -4$};

   \node [left, color=blue] at (-2.5,2.3) {$\mathbf -2$};
   
     \node [above, color=blue] at (-2.2,3.2) {$\mathbf -2$};
 \end{tikzpicture}
\end{center}
 \caption{The weighted dual graph of the active constellation of crosses whose lotus is shown in Figure  \ref{fig:lotusbranch}}
\label{fig:dualgrbranch}
   \end{figure}


 \begin{figure}[h!]
    \begin{center}
\begin{tikzpicture}[scale=1]

 \draw [->, color=black, thick](1,0) -- (0,0);
    \draw [-, color=black, thick](0,0) -- (-1,0);
    \draw [-, color=blue, line width=2.5pt](0,1.5) -- (-1,0);
      \draw [-, color=blue, line width=2.5pt](0,1.5) -- (1,0);
       \draw [-, color=black, thick](-0.5,3/4) -- (0.5,3/4);
       \draw [-, color=black, thick](-1,0) -- (0.5,3/4);
       \draw [-, color=black, thick](-5.2,4.2) -- (-3.7,3.2);

        \draw [->, color=black, thick](0,1.5) -- (-1.5,1.5);
   \draw [-, color=black, thick](-1.5,1.5) -- (-3,1.5);
     \draw [-, color=blue, line width=2.5pt](0,1.5) -- (-1.5,4);
      \draw [-, color=black, thick](-1.5,4) -- (-3,1.5);
    \draw [->, color=blue, line width=2.5pt](-1.5,4) -- (-1.5,5);
     \draw [-, color=black, thick](-3/4,11/4) -- (-3,1.5);
     \draw [-, color=black, thick](-3/4,11/4) -- (-2.5,2.34);
       \draw [-, color=black, thick](-3/4,11/4) -- (-2,3.17);
       \draw [->, color=blue, line width=2.5pt](-3/4,11/4) -- (1/4,15/4);
        \draw [-, color=blue, line width=2.5pt](-2,3.17) -- (-1.5,4);
        \draw [-, color=blue, line width=2.5pt](-3,1.5) -- (-2.5,2.34);
        \draw [-, color=blue, line width=2.5pt](-3,1.5) -- (-2.5,2.34);
        \draw [-, color=blue, line width=2.5pt](-3.7,3.2) -- (-2,3.17);
        \draw [-, color=blue, line width=2.5pt](-3.7,3.2) -- (-2.5,2.34);
        \draw [->, color=black, thick](-3.7,3.2) -- (-4.7,3.2);
        \draw [-, color=black, thick](-5.7,3.2) -- (-4.7,3.2);
           \draw [->, color=blue, line width=2.5pt](-4.7,5.2) -- (-4.7,6.2);
            \draw [-, color=blue, line width=2.5pt](-3.7,3.2) -- (-4.7,5.2);
            \draw [-, color=blue, line width=2.5pt](-5.7,3.2) -- (-4.7,5.2);
       
  \node[draw,circle, inner sep=1.5pt,color=black, fill=black] at (-1,0){};
   \node[draw,circle, inner sep=1.5pt,color=black, fill=black] at (1,0){};
     \node[draw,circle, inner sep=1.5pt,color=black, fill=black] at (0,1.5){};
      \node[draw,circle, inner sep=1.5pt,color=black, fill=black] at (-0.5,3/4){};
       \node[draw,circle, inner sep=1.5pt,color=black, fill=black] at (0.5,3/4){};
        \node[draw,circle, inner sep=1.5pt,color=black, fill=black] at (-3,1.5){};
         \node[draw,circle, inner sep=1.5pt,color=black, fill=black] at (-1.5,4){};
         \node[draw,circle, inner sep=1.5pt,color=black, fill=black] at (-3/4,11/4){};
         \node[draw,circle, inner sep=1.5pt,color=black, fill=black] at (-2.5,2.34){};
         \node[draw,circle, inner sep=1.5pt,color=black, fill=black] at (-2,3.17){};
          \node[draw,circle, inner sep=1.5pt,color=black, fill=black] at (-3.7,3.2){};
          \node[draw,circle, inner sep=1.5pt,color=black, fill=black] at (-5.7,3.2){};
           \node[draw,circle, inner sep=1.5pt,color=black, fill=black] at (-4.7,5.2){};
           \node[draw,circle, inner sep=1.5pt,color=black, fill=black] at (-5.2,4.2){};
       
   \node [below, color=black] at (-1,0) {$L_{1}$};
\node [below, color=black] at (1,0) {$L$};
\node [below, color=black] at (-3,1.5) {$L_{2}$};
\node [right, color=black] at (1/4,15/4) {$A_{3}$};
\node [right, color=black] at (-1.5,5) {$A_{1}$};
 \node [below, color=black] at (-5.7,3.2) {$L_{3}$};
 \node [above, color=black] at (-4.7,6.2) {$A_{2}$};

 \node [right, color=blue] at (0,1.5) {$\mathbf -2$};   
  \node [left, color=blue] at (-0.5,3/4) {$\mathbf -2$};
   \node [right, color=blue] at (0.5,3/4) {$\mathbf -3$};
  
     \node [right, color=blue] at (-1.45,4) {$\mathbf -1$}; 
    \node [right, color=blue] at (-3/4,11/4) {$\mathbf -4$}; 
    
    \node [right, color=blue] at  (-4.7,5.2){$\mathbf -1$}; 
    
    \node [below, color=blue] at (-4,3.2) {$\mathbf -3$};  
  
   \node [left, color=blue] at (-5.2,4.2) {$\mathbf -2$};

   \node [left, color=blue] at (-2.6,2.3) {$\mathbf -3$};
   
     \node [above, color=blue] at (-2.2,3.2) {$\mathbf -3$};
 \end{tikzpicture}
\end{center}
 \caption{The weighted dual graph of the active constellation of crosses whose lotus is shown in Figure  \ref{fig:lotus3branches}}
\label{fig:dualgr3branches}
   \end{figure}
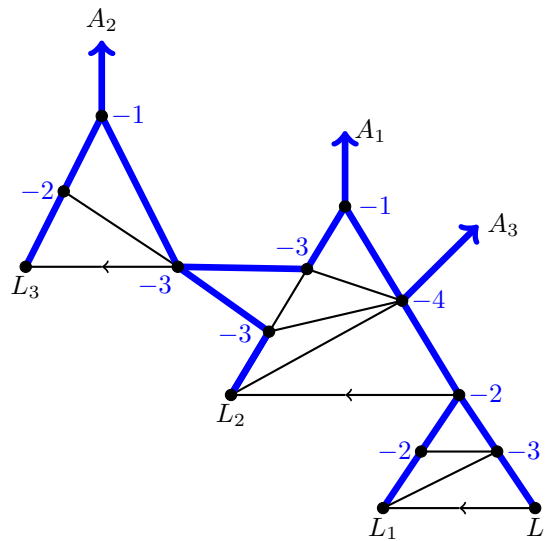

We are interested in the dual graphs of the total transforms of finite active constellation of crosses, in the sense of Definition \ref{def:constcrosses}. Their irreducible components are obtained by blowing up infinitely near points of the base point $O$ of the initial cross. Therefore, their weights may be computed recursively using Proposition \ref{prop:relblowup}. One obtains the following way of seeing the corresponding dual graph inside the associated lotus:

\medskip
\begin{center}
\fbox{
\begin{minipage}{0.75\textwidth}
   \begin{proposition}  
   \label{prop:dualgraph}
    Let $\hat{\calc}$ be a finite active constellation of crosses and let $\Lambda(\hat{\calc})$ be its associated lotus. Then the dual graph $G(X_{\hat{\calc}})$ of the total transform $X_{\hat{\calc}}$ of $\hat{\calc}$ in the sense of Definition \ref{def:constcrosses} \ref{totransfactcross} is isomorphic to 
    the lateral boundary $\partial_+ \Lambda( \hat{\calc} )$ of $\Lambda( \hat{\calc} )$ in the sense of Definition \ref{def:vocablotus}. 
    The weight of a lateral vertex of the dual graph is equal to the opposite of the number of petals incident to it, when we see it as a vertex of the lotus. 
\end{proposition}
\end{minipage}
}
\end{center}
\medskip

\begin{example}   
   \label{ex:dualgraphcuspres}
    Consider again the finite active constellation of crosses from Example \ref{ex:assconstcrossex}. 
    The dual graph of its total transform is represented in Figure \ref{fig:dualgrcusp} using bold blue segments. Near each lateral vertex in the sense of Definition \ref{def:vocablotus} is shown the corresponding weight, which is equal to the opposite of the number of petals incident to it, as explained in Proposition \ref{prop:dualgraph}. 
    In Figures \ref{fig:dualgrbranch} and \ref{fig:dualgr3branches} are represented the dual graphs of the lotuses of Figures \ref{fig:lotusbranch} and \ref{fig:lotus3branches} respectively, obtained in the same way. 
\end{example}

\medskip
\section{Log-discrepancies and orders of vanishing of the first axis} \label{sec:ldov}

We start this section by recalling basic facts about {\em semivaluations} and {\em valuations} of commutative rings (see Definition \ref{def:semival}), illustrated by the semivaluations of $\calo_{S,O}$ which will be important for us. We define then the notion of {\em log-discrepancy} of a branch on a model of $(S,O)$ (see Definition \ref{def:logdiscr}). Finally, we explain how to compute on the lotus $\Lambda(\hat{\calc})$ of a finite active constellation of crosses $\hat{\calc}$ both the log-discrepancy $\lambda(E)$ and the order of vanishing $\ord_E(L)$ corresponding to each vertex $E$ of the lotus (see Corollary \ref{cor:reclord}).

\medskip
In the sequel we consider the standard addition and  total order of $\Rr_{\geq 0} = [0, + \infty)$ to be extended as usual to the set $[0, + \infty] := \Rr_{\geq 0} \cup \{+ \infty\}$. 

The following notion of {\em semivaluation} will allow us to look at branches on $(S,O)$, at infinitely near points of $O$ and at irreducible components of exceptional divisors of models of $(S,O)$ as particular cases of a single concept (see Remark \ref{rem:annvpt}): 

\begin{definition}  
   \label{def:semival}
    Let $R$ be a commutative ring. A {\bf semivaluation of 
    $R$} is a function $\nu : R \to [0, \infty]$ such that:
    \begin{enumerate}
        \item \label{itemprod} 
          For every $f, g \in R$, $\nu(fg) = \nu(f) + \nu(g)$.
        
         \item \label{itemsum} 
            For every $f, g \in R$, $\nu(f+g) \geq \min \{ \nu(f) , \nu(g) \}$.

         \item \label{itemone} 
             $\nu(1) =0$. 

         \item \label{itemzero} 
             $\nu(0) = + \infty$. 
    \end{enumerate}
   A semivaluation $\nu$ is called a {\bf valuation} if $0$ is the only element $f$ of $R$ such that $\nu(f) = + \infty$. 
\end{definition}

As a first manipulation of these axioms, let us show that if $u \in R$ is a unit and $\nu$ a semivaluation of $R$, then $\nu(u) =0$. By the definition of units, there exists $v \in R$ such that $1 = u \cdot v$.  Therefore, by conditions \eqref{itemone} and \eqref{itemprod}:
   \[ 0 = \nu(1) = \nu(u) + \nu(v).\]
As $\nu$ takes values in $[0, \infty]$, we deduce that both $\nu(u)$ and $\nu(v)$ vanish. In particular, $\nu(u) =0$, as announced.

Let now $\nu$ be a semivaluation of the local ring $\mathcal{O}_{S, O}$. Consider a plane curve singularity $A \hookrightarrow (S,O)$. Define:
  \[  \boxed{\nu(A)}:= \nu(f_A) \in \Rr_{\geq 0}, \] 
where $f_A \in \mathcal{O}_{S, O}$ is a defining germ of $A$. This real number is independent of the choice of the defining germ $f_A$. Indeed, if $g_A$ is a second such germ, then $g_A = u \cdot f_A$, where $u$ is a unit of the ring $\mathcal{O}_{S, O}$. Therefore, by condition \eqref{itemprod} of Definition \ref{def:semival}:
  \[ \nu(g_A) = \nu(u) + \nu(f_A) = \nu(f_A),\]
where the second equality is a consequence of the fact that $u$ is a unit of $\mathcal{O}_{S, O}$.

\begin{definition}  \label{def:dival}
     Let $\pi : (S_{\pi}, E_{\pi})  \to (S,O)$ be a model of $(S,O)$ in the sense of 
     Definition \ref{def:resol} and let $C$ be a branch on $S_{\pi}$ at a point of $E_{\pi}$ 
     (the branch $C$ may be included in $E_{\pi}$). We denote by:
          \[ \boxed{\ord_C} : \mathcal{O}_{S, O} \to \Zz_{\geq 0} \cup \{+ \infty\}\]
     the {\bf divisorial valuation} of $\mathcal{O}_{S, O}$ which associates to $f \in \mathcal{O}_{S, O}$ the order of vanishing on $C$ of the pull-back $\pi^* f$. In other words, $\ord_C(f)$ is the multiplicity of $C$ as a component of the principal effective divisor on the surface $S_{\pi}$ defined by the function $\pi^* f$. 
     Notice that if $f\ne 0$, then $\ord_C (f)$ is a nonnegative integer.  
     This shows that $\ord_C$ is indeed a valuation.
\end{definition}

 If $C$ is a germ at some point of $E_{\pi}$ of an irreducible component $E$ of the exceptional divisor $E_{\pi}$, then $\ord_C(f)$ depends only on $E$ and not on the point at which the germ $C$ is taken. For this reason, it will be also denoted $\boxed{\ord_E(f)}$. This integer is unchanged if one replaces $E$ by a strict transform on any other model on which $E$ appears (in the sense of Definition \ref{def:infnear} (\ref{toappear})). Therefore, the knowledge of the valuation $\ord_E$ is equivalent to the knowledge of $E$ and of its strict transforms on all the models on which it appears. It is also equivalent to the knowledge of the infinitely near point from which this class of irreducible exceptional divisors arises by blowup. 

Similarly, if $C$ is not included in the exceptional divisor $E_{\pi}$, then it is the strict transform of the branch $\pi(C)$ on $S$ and $\ord_C(f) = \ord_{\pi(C)}(f)$.

\begin{remark}
 \label{rem:annvpt}
     We have reached the announced viewpoint from which one may look at branches, at irreducible exceptional divisors (more precisely, at bimeromorphic classes of such branches and divisors on the various models of $(S,O)$) and at infinitely near points of $O$ as particular cases of the same concept: that of semivaluation of the local ring $\mathcal{O}_{S, O}$. For a thorough study of the space of semivaluations of $\calo_{S,O}$, one may consult Favre and Jonsson's book \cite{FJ 04}. In this paper, that space will appear in the statement of Proposition \ref{prop:ew-emb}. 
\end{remark}

For any  branch $A \hookrightarrow (S,O)$,  besides the  divisorial valuation $\ord_A$ of Definition \ref{def:dival}, one has also an associated semivaluation which is not a valuation:

\begin{definition}
      \label{def:intsv}
      Let $A \hookrightarrow (S,O)$ be a branch. 
The {\bf intersection semivaluation} 
     \[ \boxed{I_A} : \mathcal{O}_{S, O} \to \Zz_{\geq 0} \cup \{+ \infty\}\]
is defined by $I_A (D) := (A\cdot D)_O$ for any curve singularity $D \hookrightarrow (S, O)$ (see Definition \ref{def:intnumber}). 
\end{definition}

The semivaluation $I_A$ is not a valuation because $I_A(A) = \infty$. Intersection semivaluations will be used in the statement of Proposition \ref{prop:ew-emb}.

 Let us define now the notion of {\em log-discrepancy} of a divisorial semivaluation. This important concept appears for instance in the definition of the {\em topological zeta function} of a plane curve singularity (see Wall \cite[Section 8.4]{W 04} or Veys \cite{V 24}), and for us it will play a crucial role in the statements of Propositions \ref{prop:ew-emb} and \ref{prop:fromlotustoEW}: 

\begin{definition}  
    \label{def:logdiscr}
    Let $\pi : (S_{\pi}, E_{\pi})  \to (S,O)$ be a model of $(S,O)$ and let $C \hookrightarrow S_{\pi}$ be a branch at a point of $E_{\pi}$. 
    Let $\omega$ be a germ of non-vanishing differential form of degree $2$ on $(S,O)$. Consider the divisor $\divis(\pi^* \omega)$ of its pull-back to the model $S_{\pi}$. The {\bf log-discrepancy} $\boxed{\lambda(C)}$ of $C$ is defined as:
      \[ 1 + \ord_C(\divis(\pi^* \omega)).\]
\end{definition}

Note that $\lambda(C)$ depends only on the valuation $\ord_C$, that is,   it is independent of the choice of non-vanishing differential form $\omega$. One has $\lambda(C) \geq 1$, with equality if and only if $C$ is not contained in the exceptional divisor $E_{\pi}$, that is, if and only if the valuation $\ord_C$ is defined by a branch on $(S,O)$. 

The following statement may be proved by working in local coordinates. It shows that the orders of vanishing of the smooth branch $L$ and the log-discrepancies of the various branches on models of $(S,O)$ may be computed in the same way recursively during a process of blowups of points, the only difference lying in their values taken on branches already present on $(S,O)$:

\begin{proposition}  
   \label{prop:addlotus}
      Let $\hat{\calc}$ be a finite active constellation of crosses over $O$, in the sense of Definition \ref{def:constcrosses}, with initial cross $X_O = L + L_1$. 
    Let $P \in \calc$ 
     and let $X_P = F_1 + F_2$ be its corresponding  cross on $(S^P, P)$.   
    Then:
       \begin{enumerate}
          \item 
              $\ord_{E_P}(L) = \ord_{F_1}(L) + \ord_{F_2}(L),$
          \item   \label{addlogdis}
              $\lambda(E_P) = \lambda(F_1) + \lambda(F_2).$
      \end{enumerate}
    Moreover, for every branch $C$ on $(S,O)$, one has:
    \begin{enumerate}
         \item[(3)]
            $\ord_{C}(L) = 
          \left\{ \begin{array}{cc}
              0 & \mbox{ if } L \neq C,\\
              1 & \mbox{ if } L = C.
          \end{array}\right.$
        \item[(4)]
          $\lambda(C) = 1.$
      \end{enumerate}   
\end{proposition}

\begin{remark}   
   \label{rem:ldconven}
     Relation \eqref{addlogdis} of Proposition \ref{prop:addlotus} is the reason why {\em log-discrepancies} are  more convenient for computations than the {\em discrepancies}, which are defined by $d(C) := \ord_C(\divis(\pi^* \omega))$. Indeed, written in terms of discrepancies, this relation becomes $d(E_P) = 1 + d(F_1) + d(F_2)$. Note that if one chooses local coordinates on $S_\pi$ at a point $P$ of $E_\pi$ and on $(S,O)$, then the discrepancy $d(C)$ of a branch $C \hookrightarrow (S_\pi, P)$ is the order of vanishing along $C$ of the jacobian determinant of $\pi$ relative to those local coordinates.
\end{remark}

 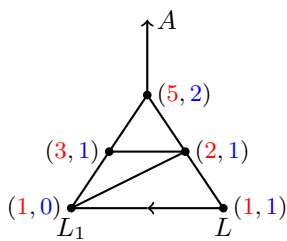
\begin{figure}[h!]
    \begin{center}
\begin{tikzpicture}[scale=1]
   
    \draw [->, color=black, thick](1,0) -- (0,0);
    \draw [-, color=black, thick](0,0) -- (-1,0);
    \draw [-, color=black, thick](0,1.5) -- (-1,0);
      \draw [-, color=black, thick](0,1.5) -- (1,0);
       \draw [-, color=black, thick](-0.5,3/4) -- (0.5,3/4);
       \draw [-, color=black, thick](-1,0) -- (0.5,3/4);
       \draw [->, color=black, thick](0,1.5) -- (0,2.5);
       
  \node[draw,circle, inner sep=1pt,color=black, fill=black] at (-1,0){};
   \node[draw,circle, inner sep=1pt,color=black, fill=black] at (1,0){};
     \node[draw,circle, inner sep=1pt,color=black, fill=black] at (0,1.5){};
      \node[draw,circle, inner sep=1pt,color=black, fill=black] at (-0.5,3/4){};
       \node[draw,circle, inner sep=1pt,color=black, fill=black] at (0.5,3/4){};

\node [right, color=black] at (0,2.5) {$A$};
\node [below, color=black] at (-1,0) {$L_{1}$};
\node [below, color=black] at (1,0) {$L$};

\node [left, color=black] at (-1,0) {\small{$(\red{1},\blue{0})$}};  
\node [right, color=black] at (1,0) {\small{$(\red{1},\blue{1})$}};  
 \node [right, color=black] at (0,1.5) {\small{$(\red{5},\blue{2}) $}};  
 \node [left, color=black] at (-0.5,3/4) {\small{$(\red{3},\blue{1})$}}; 
  \node [right, color=black] at (0.5,3/4) {\small{$(\red{2},\blue{1})$}}; 
  
 \end{tikzpicture}
\end{center}
 \caption{Near each vertex $E \neq A$ of the lotus of Figure \ref{fig:asslotuscusp} is written the pair $(\lambda(E), \ord_E (L))$}
\label{fig:logdisord1}
   \end{figure}

 \begin{figure}[h!]
    \begin{center}
\begin{tikzpicture}[scale=1]

    \draw [->, color=black, thick](1,0) -- (0,0);
    \draw [-, color=black, thick](0,0) -- (-1,0);
    \draw [-, color=black, thick](0,1.5) -- (-1,0);
      \draw [-, color=black, thick](0,1.5) -- (1,0);
       \draw [-, color=black, thick](-0.5,3/4) -- (0.5,3/4);
       \draw [-, color=black, thick](-1,0) -- (0.5,3/4);

        \draw [->, color=black, thick](0,1.5) -- (-1.5,1.5);
   \draw [-, color=black, thick](-1.5,1.5) -- (-3,1.5);
     \draw [-, color=black, thick](0,1.5) -- (-1.5,4);
      \draw [-, color=black, thick](-1.5,4) -- (-3,1.5);
      \draw [->, color=black, thick](-1.5,4) -- (-1.5,5);
     \draw [-, color=black, thick](-3/4,11/4) -- (-3,1.5);
     \draw [-, color=black, thick](-3/4,11/4) -- (-2.5,2.34);
       \draw [-, color=black, thick](-3/4,11/4) -- (-2,3.17);
        \draw [-, color=black, thick](-2,3.17) -- (-1.5,4);
        \draw [-, color=black, thick](-3,1.5) -- (-2.5,2.34);
        \draw [-, color=black, thick](-3,1.5) -- (-2.5,2.34);
     
       \node[draw,circle, inner sep=1pt,color=black, fill=black] at (-1,0){};
   \node[draw,circle, inner sep=1pt,color=black, fill=black] at (1,0){};
     \node[draw,circle, inner sep=1pt,color=black, fill=black] at (0,1.5){};
      \node[draw,circle, inner sep=1pt,color=black, fill=black] at (-0.5,3/4){};
       \node[draw,circle, inner sep=1pt,color=black, fill=black] at (0.5,3/4){};
        \node[draw,circle, inner sep=1pt,color=black, fill=black] at (-3,1.5){};
         \node[draw,circle, inner sep=1pt,color=black, fill=black] at (-1.5,4){};
         \node[draw,circle, inner sep=1pt,color=black, fill=black] at (-3/4,11/4){};
         \node[draw,circle, inner sep=1pt,color=black, fill=black] at (-2.5,2.34){};
         \node[draw,circle, inner sep=1pt,color=black, fill=black] at (-2,3.17){};

   \node [below, color=black] at (-1,0) {$L_{1}$};
\node [below, color=black] at (1,0) {$L$};
\node [below, color=black] at (-3,1.5) {$L_{2}$};
\node [right, color=black] at (-1.5,5) {$A_1$};

\node [left, color=black] at (-1,0) {\small{$(\red{1},\blue{0})$}};  
\node [right, color=black] at (1,0) {\small{$(\red{1},\blue{1})$}};  
 \node [right, color=black] at (0,1.5) {\small{$(\red{5},\blue{2}) $}};  
 \node [left, color=black] at (-0.5,3/4) {\small{$(\red{3},\blue{1})$}}; 
  \node [right, color=black] at (0.5,3/4) {\small{$(\red{2},\blue{1})$}}; 
     \node [left, color=black] at (-3,1.5) {\small{$(\red{1},\blue{0})$}}; 
     \node [left, color=black] at (-1.5,4.2) {\small{$(\red{19},\blue{6})$}}; 
    \node [right, color=black] at (-3/4,11/4) {\small{$(\red{6},\blue{2})$}}; 
   \node [left, color=black] at (-2.5,2.34) {\small{$(\red{7},\blue{2})$}}; 
   \node [left, color=black] at (-1.8,3.4) {\small{$(\red{13},\blue{4})$}}; 
 \end{tikzpicture}
\end{center}
 \caption{Near each vertex $E \neq A_1$ of the lotus of Figure \ref{fig:lotusbranch} is  written the pair $(\lambda(E), \ord_E (L))$}
\label{fig:logdisord2}
   \end{figure}
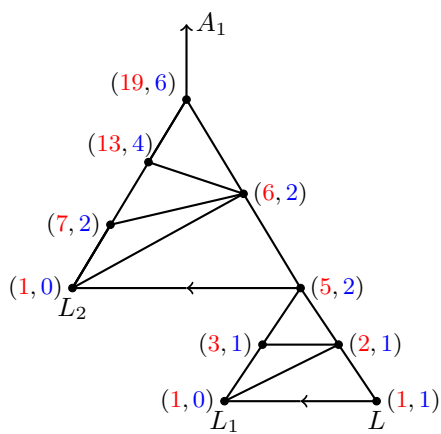

 \begin{figure}[h!]
    \begin{center}
\begin{tikzpicture}[scale=1]
\begin{scope}[shift={(-20,0)}]
   
    \draw [->, color=black, thick](1,0) -- (0,0);
    \draw [-, color=black, thick](0,0) -- (-1,0);
    \draw [-, color=black, thick](0,1.5) -- (-1,0);
      \draw [-, color=black, thick](0,1.5) -- (1,0);
       \draw [-, color=black, thick](-0.5,3/4) -- (0.5,3/4);
       \draw [-, color=black, thick](-1,0) -- (0.5,3/4);
       \draw [-, color=black, thick](-5.2,4.2) -- (-3.7,3.2);

        \draw [->, color=black, thick](0,1.5) -- (-1.5,1.5);
   \draw [-, color=black, thick](-1.5,1.5) -- (-3,1.5);
     \draw [-, color=black, thick](0,1.5) -- (-1.5,4);
      \draw [-, color=black, thick](-1.5,4) -- (-3,1.5);
      \draw [->, color=black, thick](-1.5,4) -- (-1.5,5);
     \draw [-, color=black, thick](-3/4,11/4) -- (-3,1.5);
     \draw [-, color=black, thick](-3/4,11/4) -- (-2.5,2.34);
       \draw [-, color=black, thick](-3/4,11/4) -- (-2,3.17);
       \draw [->, color=black, thick](-3/4,11/4) -- (1/4,15/4);
        \draw [-, color=black, thick](-2,3.17) -- (-1.5,4);
        \draw [-, color=black, thick](-3,1.5) -- (-2.5,2.34);
        \draw [-, color=black, thick](-3,1.5) -- (-2.5,2.34);
        \draw [-, color=black, thick](-3.7,3.2) -- (-2,3.17);
        \draw [-, color=black, thick](-3.7,3.2) -- (-2.5,2.34);
        \draw [->, color=black, thick](-3.7,3.2) -- (-4.7,3.2);
        \draw [-, color=black, thick](-5.7,3.2) -- (-4.7,3.2);
           \draw [->, color=black, thick](-4.7,5.2) -- (-4.7,6.2);
            \draw [-, color=black, thick](-3.7,3.2) -- (-4.7,5.2);
            \draw [-, color=black, thick](-5.7,3.2) -- (-4.7,5.2);
       
  \node[draw,circle, inner sep=1pt,color=black, fill=black] at (-1,0){};
   \node[draw,circle, inner sep=1pt,color=black, fill=black] at (1,0){};
     \node[draw,circle, inner sep=1pt,color=black, fill=black] at (0,1.5){};
      \node[draw,circle, inner sep=1pt,color=black, fill=black] at (-0.5,3/4){};
       \node[draw,circle, inner sep=1pt,color=black, fill=black] at (0.5,3/4){};
        \node[draw,circle, inner sep=1pt,color=black, fill=black] at (-3,1.5){};
         \node[draw,circle, inner sep=1pt,color=black, fill=black] at (-1.5,4){};
         \node[draw,circle, inner sep=1pt,color=black, fill=black] at (-3/4,11/4){};
         \node[draw,circle, inner sep=1pt,color=black, fill=black] at (-2.5,2.34){};
         \node[draw,circle, inner sep=1pt,color=black, fill=black] at (-2,3.17){};
          \node[draw,circle, inner sep=1pt,color=black, fill=black] at (-3.7,3.2){};
          \node[draw,circle, inner sep=1pt,color=black, fill=black] at (-5.7,3.2){};
           \node[draw,circle, inner sep=1pt,color=black, fill=black] at (-4.7,5.2){};
           \node[draw,circle, inner sep=1pt,color=black, fill=black] at (-5.2,4.2){};
       
   \node [below, color=black] at (-1,0) {$L_{1}$};
\node [below, color=black] at (1,0) {$L$};
\node [below, color=black] at (-3,1.5) {$L_{2}$};
\node [right, color=black] at (1/4,15/4) {$A_3$};
\node [right, color=black] at (-1.5,5) {$A_1$};
 \node [below, color=black] at (-5.7,3.2) {$L_{3}$};
 \node [above, color=black] at (-4.7,6.2) {$A_2$};

\node [left, color=black] at (-1,0) {\small{$(\red{1},\blue{0})$}};  
\node [right, color=black] at (1,0) {\small{$(\red{1},\blue{1})$}};  
 \node [right, color=black] at (0,1.5) {\small{$(\red{5},\blue{2}) $}};  
 \node [left, color=black] at (-0.5,3/4) {\small{$(\red{3},\blue{1})$}}; 
  \node [right, color=black] at (0.5,3/4) {\small{$(\red{2},\blue{1})$}}; 
     \node [left, color=black] at (-3,1.5) {\small{$(\red{1},\blue{0})$}}; 
     \node [left, color=black] at (-1.5,4.2) {\small{$(\red{19},\blue{6})$}}; 
    \node [right, color=black] at (-3/4,11/4) {\small{$(\red{6},\blue{2})$}}; 
    \node [below, color=black] at (-4,3.2) {\small{$(\red{20},\blue{6})$}}; 
   \node [left, color=black] at (-2.5,2.34) {\small{$(\red{7},\blue{2})$}}; 
   \node [left, color=black] at (-1.8,3.4) {\small{$(\red{13},\blue{4})$}}; 
   \node [left, color=black] at (-5.7,3.2) {\small{$(\red{1},\blue{0})$}};  
   \node [left, color=black] at (-4.7,5.2) {\small{$(\red{41},\blue{12})$}}; 
   \node [left, color=black] at (-5.2,4.2) {\small{$(\red{21},\blue{6})$}};  
   
   \end{scope}
   
 \end{tikzpicture}
\end{center}
 \caption{Near each vertex $E \notin \{A_1, A_2, A_3\}$ of the lotus of Figure \ref{fig:lotus3branches}
is  written the pair $(\lambda(E), \ord_E (L))$}
\label{fig:logdisord3}
   \end{figure}
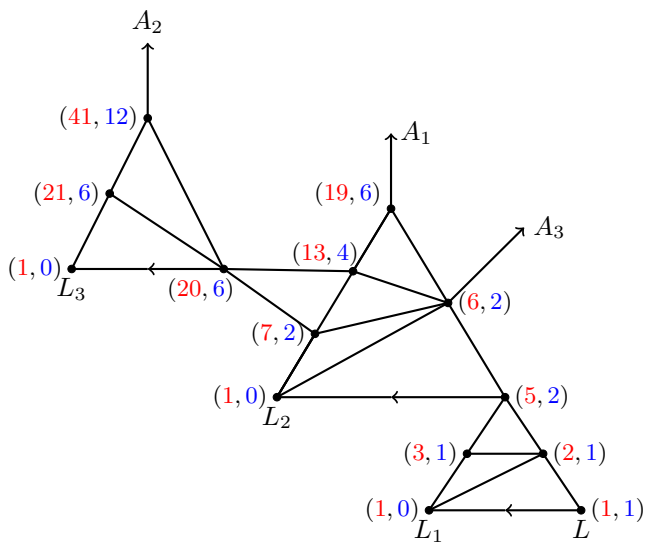

   Before transforming Proposition \ref{prop:addlotus} into an algorithm of computation of the pairs $(\lambda(E), \ord_E(L))$ using lotuses, let us make an important remark about three sets which are naturally in one-to-one correspondences:

\begin{remark}
    \label{rem:oneone}
      Let $\Lambda(\hat{\calc})$ be the lotus of a finite active constellation of crosses $\hat{\calc}$. Then the following sets are naturally in one-to-one correspondences:
        \begin{itemize}
            \item 
              the set of edges of $\Lambda(\hat{\calc})$ which are bases of petals;
            
            \item 
               the set of lateral vertices of $\Lambda(\hat{\calc})$;
               
            \item the set of active points of $\hat{\calc}$.
        \end{itemize}
     Each base of a petal of $\Lambda(\hat{\calc})$ corresponds to its opposite vertex, and such vertices are exactly the lateral vertices of the lotus. Moreover, the bases of petals are exactly the dual segments of the crosses based at active points of $\hat{\calc}$. Therefore, the partial order relation $\preceq$ on $\calc$ induces partial orders on the two other sets. 
\end{remark}

Using the partial order on the set of lateral vertices of $\Lambda(\hat{\calc})$ described in Remark \ref{rem:oneone}, we may express in the following way the algorithm of computation deduced from Proposition \ref{prop:addlotus}:

\medskip
\begin{center}
\fbox{
\begin{minipage}{0.75\textwidth}
   \begin{corollary}
     \label{cor:reclord}
       Let $\Lambda$ be the lotus of a finite active constellation of crosses, with initial vertex $L$. Then the pairs $(\lambda(E), \ord_E(L))$ associated to its vertices $E$ may be computed recursively by the following algorithm: 
          \begin{itemize}
              \item 
                 The pair associated to the initial vertex $L$ is equal to $(1,1)$. 
              \item
                 The pair associated to every other basic vertex $L_i$ is equal to $(1,0)$. 
               \item 
                  Following increasingly the partial order on the set of lateral vertices of $\Lambda(\hat{\calc})$ defined in Remark \ref{rem:oneone}, associate to each lateral vertex the vector sum of the pairs associated to the vertices of the unique petal base opposite to it. 
          \end{itemize}
    \end{corollary}
\end{minipage}
}
\end{center}
\medskip

\begin{example}
    \label{ex:logdiscorders}
    In Figures \ref{fig:logdisord1}, \ref{fig:logdisord2} and \ref{fig:logdisord3} (whose lotuses are those of Figures \ref{fig:asslotuscusp}, \ref{fig:lotusbranch} and \ref{fig:lotus3branches}) we represent near each vertex $E$ of the lotus different from the arrowheads  $A_l$ the pair $(\lambda(E), \ord_E (L))$ formed by the log-discrepancy of $E$ and the order of vanishing of $L$ along $E$. The computations are performed using Corollary \ref{cor:reclord}, starting from the values $(\lambda(L), \ord_L (L)) = (1,1)$ and $(\lambda(L_i), \ord_{L_i} (L)) = (1,0)$ for every $i \geq 1$. For instance, in Figure \ref{fig:logdisord2}, the pairs of the two rupture vertices in the sense of Definition \ref{def:vocablotus} are computed as:
      \[(5,2) = (3,1) + (2, 1), \ \  (19,6) = (13,4) + (6,2). \]
\end{example}

\medskip
\section{Multiplicities of strict transforms} 
\label{sec:multstr}

In this section we explain how to compute on an associated lotus $\Lambda(\hat{\calc}_A)$ the multiplicities of a reduced plane curve singularity $A \hookrightarrow (S,O)$ at all infinitely near points lying on it (see Corollary \ref{cor:multedges}). This computation is based on the {\em proximity equalities} of Theorem \ref{thm:proxeq}.

\medskip
Let us recall the definition of the {\em multiplicity} of a plane curve singularity: 

\begin{definition}  
    \label{def:mult}
    Let $A \hookrightarrow (S,O)$ be a plane curve singularity, with defining function 
    $f_A \in \maxid_{S, O}$. By choosing local coordinates $(x,y)$ on $(S,O)$, $f_A$ 
    becomes a series in $\Cc \{x,y\}$. The {\bf multiplicity} $\boxed{e_O(A)} \in \Zz_{> 0}$ 
    of $A$ at $O$ is the smallest degree of the monomials appearing in $f_A$. 
\end{definition}

 For instance, one has $e_O(A) = 2$ for the cusp singularity $A$ of Example \ref{ex:cuspembres}, defined by $f_A = y^2 - x^3$.

\begin{remark}  
   \label{rem:multind}
       Definition \ref{def:mult} is independent of the choices of defining function $f_A$ and of local coordinates. This notion may be thought as a local analog of the degree of an effective divisor in a complex projective plane. Indeed, both have similar geometric interpretations: the degree is the intersection number of the projective curve with a line not contained in it and the multiplicity of a plane curve singularity $A$ in $(\Cc^2,0)$ is the intersection number of $A$ at $O$ with a line not tangent to $A$. 
\end{remark}

Generalizing equation \eqref{eq:totransfsm}, multiplicities at $O$ may also be computed as orders of vanishing, as may be readily seen by working in local coordinates:

\begin{proposition} 
    \label{prop:ordvanish}
      Let $A \hookrightarrow (S,O)$ be a plane curve singularity. Consider the blowup morphism $\pi_O: S_O \to S$ of $S$ at $O$, with exceptional divisor $E_O$. Then:
         \[e_O(A) = \ord_{E_O}(A).\]
\end{proposition}

The notion of multiplicity may be extended to infinitely near points of $O$:

\begin{definition} 
   \label{def:multinpoint}
    Let $A \hookrightarrow (S,O)$ be a plane curve singularity and $P$ be an infinitely near point of $O$. The {\bf multiplicity} $\boxed{e_P(A)}$ of $A$ at $P$ is the multiplicity at $P$ of the strict transform of $A$ on the minimal model $S_P$ containing $P$. 
\end{definition}

The {\bf proximity equalities} of the following theorem (see \cite[Theorem 3.5.3]{C 00}) relate the multiplicities of a plane curve singularity at their various infinitely near points: 

\begin{theorem}
  \label{thm:proxeq}
     Let $A \hookrightarrow (S,O)$ be a plane curve singularity and $P$ be an infinitely near point of $O$ lying on $A$ in the sense of Definition \ref{def:infnear}(\ref{passthroughin}). Then, with the notations of Definition \ref{def:infnear} (\ref{proxnot}):
       \[   e_P(A) = \sum_{Q : Q \to P} e_Q(A).  \]   
\end{theorem}

Theorem \ref{thm:proxeq} shows that the multiplicities of the strict transforms of a plane curve singularity $A$ may be computed recursively on the lotus $\Lambda(\hat{\calc}_A)$ of any active constellation of crosses $\hat{\calc}_A$ adapted to it in the sense of Definition \ref{def:adaptedactconst}, as weights
on the set of edges appearing in Remark \ref{rem:oneone}:

\medskip
\begin{center}
\fbox{
\begin{minipage}{0.75\textwidth}
   \begin{corollary}
    \label{cor:multedges}
    Let $\Lambda(\hat{\calc}_A)$ be the lotus of a finite active constellation of crosses adapted to $A$, in the sense of Definition \ref{def:adaptedactconst}. 
      Denote by $\calp$ the set of edges of $\Lambda(\hat{\calc}_A)$ which are bases of petals, partially ordered as explained in Remark \ref{rem:oneone}. Perform the following algorithm of computation of weights of the edges in $\calp$ in a descending way relative to this partial order:
  \begin{enumerate}
    \item 
       One starts by attaching the weight $1$ to each base edge which corresponds to an inactive point of the underlying constellation. 
    \item 
       One computes the weight of a lower edge $e$ as the sum of weights of the higher edges which are incident to the vertex opposite to $e$ in the petal with base $e$. 
  \end{enumerate}
  Then, each edge gets decorated by the multiplicity of $A$ at the infinitely near point which corresponds to it as explained in Remark \ref{rem:oneone}.
\end{corollary}
\end{minipage}
}
\end{center}
\medskip

 \begin{figure}[h!]
    \begin{center}
\begin{tikzpicture}[scale=1]
   
    \draw [->, color=orange, thick](1,0) -- (0,0);
    \draw [-, color=orange, thick](0,0) -- (-1,0);
    \draw [-, color=black, thick](0,1.5) -- (-1,0);
      \draw [-, color=black, thick](0,1.5) -- (1,0);
       \draw [-, color=orange, thick](-0.5,3/4) -- (0.5,3/4);
       \draw [-, color=orange, thick](-1,0) -- (0.5,3/4);
        \draw [->, color=orange, thick](0,1.5) -- (0,2.5);
           
  \node[draw,circle, inner sep=1pt,color=black, fill=black] at (-1,0){};
   \node[draw,circle, inner sep=1pt,color=black, fill=black] at (1,0){};
     \node[draw,circle, inner sep=1pt,color=black, fill=black] at (0,1.5){};
      \node[draw,circle, inner sep=1pt,color=black, fill=black] at (-0.5,3/4){};
       \node[draw,circle, inner sep=1pt,color=black, fill=black] at (0.5,3/4){};

   \node [below, color=black] at (-1,0) {$L_{1}$};
   \node [below, color=black] at (1,0) {$L$};
   \node [above, color=black] at (0,2.5) {$A$};

   \node [right, color=black] at (0.5,3/4) {$E_{1}$};
   \node [left, color=black] at (-0.5,3/4) {$E_2$};
   \node [right, color=black] at (0,1.5) {$E_3$};

\node [below, color=black] at (0,0) {\small{$\orange{2}$}};  
\node [below, color=black] at (0,0.55) {\small{$\orange{1}$}};
\node [above, color=black] at (0,0.7) {\small{$\orange{1}$}};
\node [right, color=black] at (0,2) {\small{$\orange{1}$}};
   \end{tikzpicture}
\end{center}
 \caption{Computation of multiplicities at infinitely near points of the branch $A$ of Figure \ref{fig:asslotuscusp}}
\label{fig:multbranch1}
   \end{figure}
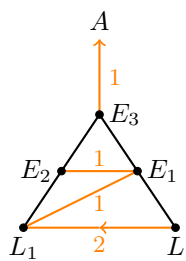


 \begin{figure}[h!]
    \begin{center}
\begin{tikzpicture}[scale=1]
   
    \draw [->, color=orange, thick](1,0) -- (0,0);
    \draw [-, color=orange, thick](0,0) -- (-1,0);
    \draw [-, color=black, thick](0,1.5) -- (-1,0);
      \draw [-, color=black, thick](0,1.5) -- (1,0);
       \draw [-, color=orange, thick](-0.5,3/4) -- (0.5,3/4);
       \draw [-, color=orange, thick](-1,0) -- (0.5,3/4);

        \draw [->, color=orange, thick](0,1.5) -- (-1.5,1.5);
   \draw [-, color=orange, thick](-1.5,1.5) -- (-3,1.5);
     \draw [-, color=black, thick](0,1.5) -- (-1.5,4);
      \draw [-, color=black, thick](-1.5,4) -- (-3,1.5);
      \draw [->, color=orange, thick](-1.5,4) -- (-1.5,5);
     \draw [-, color=orange, thick](-3/4,11/4) -- (-3,1.5);
     \draw [-, color=orange, thick](-3/4,11/4) -- (-2.5,2.34);
       \draw [-, color=orange, thick](-3/4,11/4) -- (-2,3.17);
           
  \node[draw,circle, inner sep=1pt,color=black, fill=black] at (-1,0){};
   \node[draw,circle, inner sep=1pt,color=black, fill=black] at (1,0){};
     \node[draw,circle, inner sep=1pt,color=black, fill=black] at (0,1.5){};
      \node[draw,circle, inner sep=1pt,color=black, fill=black] at (-0.5,3/4){};
       \node[draw,circle, inner sep=1pt,color=black, fill=black] at (0.5,3/4){};
        \node[draw,circle, inner sep=1pt,color=black, fill=black] at (-3,1.5){};
         \node[draw,circle, inner sep=1pt,color=black, fill=black] at (-1.5,4){};
         \node[draw,circle, inner sep=1pt,color=black, fill=black] at (-3/4,11/4){};
         \node[draw,circle, inner sep=1pt,color=black, fill=black] at (-2.5,2.34){};
         \node[draw,circle, inner sep=1pt,color=black, fill=black] at (-2,3.17){};

   \node [below, color=black] at (-1,0) {$L_{1}$};
\node [below, color=black] at (1,0) {$L$};
\node [below, color=black] at (-3,1.5) {$L_{2}$};
\node [above, color=black] at (-1.5,5) {$A_1$};

\node [below, color=black] at (0,0) {\small{$\orange{6}$}};  
\node [below, color=black] at (0,0.55) {\small{$\orange{3}$}};
\node [above, color=black] at (0,0.7) {\small{$\orange{3}$}};
\node [below, color=black] at (-1.5,1.5) {\small{$\orange{3}$}}; 
\node [below, color=black] at (-1.7,2.3) {\small{$\orange{1}$}};  
\node [above, color=black] at (-1.5,2.5) {\small{$\orange{1}$}}; 
\node [above, color=black] at (-1.5,3) {\small{$\orange{1}$}}; 
\node [right, color=black] at (-1.5,4.5) {\small{$\orange{1}$}};

   \end{tikzpicture}
\end{center}
 \caption{Computation of multiplicities at infinitely near points of the branch $A_1$ of Figure \ref{fig:lotusbranch}}
\label{fig:multbranch2}
   \end{figure}
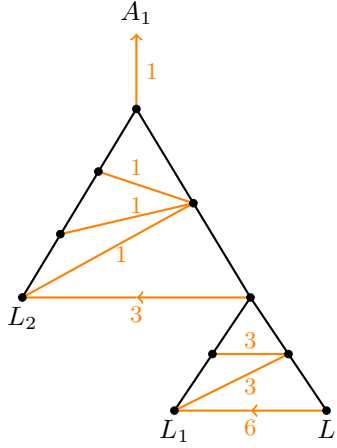

 \begin{figure}[h!]
    \begin{center}
\begin{tikzpicture}[scale=1]
   
    \draw [->, color=orange, thick](1,0) -- (0,0);
    \draw [-, color=orange, thick](0,0) -- (-1,0);
    \draw [-, color=black, thick](0,1.5) -- (-1,0);
      \draw [-, color=black, thick](0,1.5) -- (1,0);
       \draw [-, color=orange, thick](-0.5,3/4) -- (0.5,3/4);
       \draw [-, color=orange, thick](-1,0) -- (0.5,3/4);
       \draw [-, color=orange, thick](-5.2,4.2) -- (-3.7,3.2);

        \draw [->, color=orange, thick](0,1.5) -- (-1.5,1.5);
   \draw [-, color=orange, thick](-1.5,1.5) -- (-3,1.5);
     \draw [-, color=black, thick](0,1.5) -- (-1.5,4);
      \draw [-, color=black, thick](-1.5,4) -- (-3,1.5);
      \draw [->, color=orange, thick](-1.5,4) -- (-1.5,5);
     \draw [-, color=orange, thick](-3/4,11/4) -- (-3,1.5);
     \draw [-, color=orange, thick](-3/4,11/4) -- (-2.5,2.34);
     \draw [-, color=orange, thick](-2,3.17) -- (-2.5,2.34);
       \draw [-, color=orange, thick](-3/4,11/4) -- (-2,3.17);
       \draw [->, color=orange, thick](-3/4,11/4) -- (1/4,15/4);
        \draw [-, color=black, thick](-2,3.17) -- (-1.5,4);
        \draw [-, color=black, thick](-3,1.5) -- (-2.5,2.34);
        \draw [-, color=black, thick](-3,1.5) -- (-2.5,2.34);
        \draw [-, color=black, thick](-3.7,3.2) -- (-2,3.17);
        \draw [-, color=black, thick](-3.7,3.2) -- (-2.5,2.34);
        \draw [->, color=orange, thick](-3.7,3.2) -- (-4.7,3.2);
        \draw [-, color=orange, thick](-5.7,3.2) -- (-4.7,3.2);
        \draw [->, color=orange, thick](-4.7,5.2) -- (-4.7,6.2);
        \draw [-, color=black, thick](-3.7,3.2) -- (-4.7,5.2);
        \draw [-, color=black, thick](-5.7,3.2) -- (-4.7,5.2);
       
  \node[draw,circle, inner sep=1pt,color=black, fill=black] at (-1,0){};
   \node[draw,circle, inner sep=1pt,color=black, fill=black] at (1,0){};
     \node[draw,circle, inner sep=1pt,color=black, fill=black] at (0,1.5){};
      \node[draw,circle, inner sep=1pt,color=black, fill=black] at (-0.5,3/4){};
       \node[draw,circle, inner sep=1pt,color=black, fill=black] at (0.5,3/4){};
        \node[draw,circle, inner sep=1pt,color=black, fill=black] at (-3,1.5){};
         \node[draw,circle, inner sep=1pt,color=black, fill=black] at (-1.5,4){};
         \node[draw,circle, inner sep=1pt,color=black, fill=black] at (-3/4,11/4){};
         \node[draw,circle, inner sep=1pt,color=black, fill=black] at (-2.5,2.34){};
         \node[draw,circle, inner sep=1pt,color=black, fill=black] at (-2,3.17){};
          \node[draw,circle, inner sep=1pt,color=black, fill=black] at (-3.7,3.2){};
          \node[draw,circle, inner sep=1pt,color=black, fill=black] at (-5.7,3.2){};
           \node[draw,circle, inner sep=1pt,color=black, fill=black] at (-4.7,5.2){};
           \node[draw,circle, inner sep=1pt,color=black, fill=black] at (-5.2,4.2){};
       
   \node [below, color=black] at (-1,0) {$L_{1}$};
\node [below, color=black] at (1,0) {$L$};
\node [below, color=black] at (-3,1.5) {$L_{2}$};
\node [right, color=black] at (1/4,15/4) {$A_3$};
\node [above, color=black] at (-1.5,5) {$A_1$};
 \node [below, color=black] at (-5.7,3.2) {$L_{3}$};
 \node [above, color=black] at (-4.7,6.2) {$A_2$};

\node [below, color=black] at (0,0) {\small{$\orange{20}$}};  
\node [below, color=black] at (0,0.55) {\small{$\orange{10}$}};
\node [above, color=black] at (0,0.7) {\small{$\orange{10}$}};
\node [below, color=black] at (-1.5,1.5) {\small{$\orange{10}$}}; 
\node [below, color=black] at (-1.7,2.3) {\small{$\orange{5}$}};  
\node [above, color=black] at (-1.5,2.5) {\small{$\orange{3}$}}; 
\node [left, color=black] at (-1.85,2.8) {\small{$\orange{2}$}}; 
\node [above, color=black] at (-1.5,3) {\small{$\orange{1}$}}; 
\node [below, color=black] at (0,14/4) {\small{$\orange{1}$}}; 
\node [below, color=black] at (-1.2,4.7) {\small{$\orange{1}$}}; 
\node [below, color=black] at (-4.7,3.2) {\small{$\orange{1}$}};
\node [below, color=black] at (-4.7,4.3) {\small{$\orange{1}$}};
\node [below, color=black] at (-4.4,6) {\small{$\orange{1}$}};
   \end{tikzpicture}
\end{center}
 \caption{Computation of multiplicities at infinitely near points of the plane curve singularity $A_1+ A_2 + A_3$ of Figure \ref{fig:lotus3branches}}
\label{fig:multbranch3}
   \end{figure}
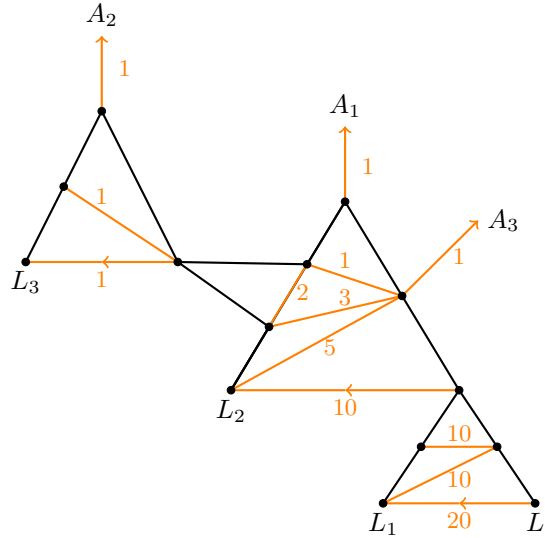

\begin{example}  \label{excompmult}
    The method of computation of multiplicities explained in Corollary \ref{cor:multedges} is illustrated in Figures \ref{fig:multbranch1}, \ref{fig:multbranch2} and \ref{fig:multbranch3}. Let us give details in the case of Figure \ref{fig:multbranch1}: 
     \begin{itemize}
         \item 
           One starts by attaching the weight $1$ to the edge $[E_3 A]$.
         \item 
           Then one computes successively the weights of the edges $[E_1 E_2]$, $[L_1 E_1]$ and $[L L_1]$, as explained now.
         \item 
           The triangle $E_1 E_2 E_3$ is the petal having  base $[E_1 E_2]$, therefore the opposite vertex of this base is $E_3$. There is only one higher edge incident to it, namely $[E_3 A]$, whose weight is $1$. Therefore the weight of $[E_1 E_2]$ is also $1$. 
         \item 
            Similarly, the weight of $[L_1 E_1]$ is the sum of weights of the edges incident to $E_2$, which is the vertex opposite $[L_1 E_1]$ in the petal $L_1 E_1 E_2$. There is only one edge of this type, namely $[E_1 E_2]$, whose weight we just computed to be $1$. Therefore the weight of $[L_1 E_1]$ is also $1$.
         \item 
             The vertex opposite $[L L_1]$ is $E_1$. There are two higher edges incident to it, namely $[L_1 E_1]$ and $[E_1 E_2]$. The weight of $[L L_1]$ is the sum of their weights, that is $2$. 
     \end{itemize}

     Similarly, in Figure \ref{fig:multbranch3}, the weight $10$ on the base of the membrane with vertex $L_2$ is computed as the sum $1 + 1 + 3 + 5$ of weights of the edges which are incident to its opposite vertex.
\end{example}

\medskip
\section{Orders of vanishing and intersection numbers}  
\label{sec:ovint}

In this section we explain how to compute the orders of vanishing $\ord_E(A)$ of a reduced  plane curve singularity $A$ at the vertices $E$ of an associated lotus, as well as the intersection numbers $(A_l \cdot A_m)_O$ of pairs of distinct branches of $A$ either from the multiplicities of those branches at infinitely near points or as particular orders of vanishing (see Corollaries \ref{cor:firstmethint} and \ref{cor:secmethint}).

\medskip

We start recalling the definition of the intersection number of two plane curve singularities at the same point:

\begin{definition}   
        \label{def:intnumber}
      Let $A, B \hookrightarrow (S,O)$ be two plane curve singularities. Let $f_A, f_B \in \calo_{S,O}$ be defining functions for them and let $\boxed{(f_A, f_B)}$ be the ideal of $\calo_{S,O}$ generated by $f_A$ and $f_B$. The {\bf intersection number of $A$ and $B$ at $O$}, denoted $\boxed{(A \cdot B)_O} \in \Zz_{\geq 0} \cup \{ \infty\}$  is the dimension of the quotient vector space
        \[ \calo_{S,O}/(f_A, f_B).  \]
       
\end{definition}

Note that $(A \cdot B)_O = \infty$ if and only if $A$ and $B$ have at least one common branch.

  As the ideal $(f_A, f_B)$ does not depend on the choice of defining germs of $A$ and $B$, the intersection number $(A \cdot B)_O$ is well-defined. It is the local contribution to the computation of the global intersection number of two effective divisors without common irreducible components in a smooth surface, at least one of which is non-compact (see Section \ref{sec:properaction}). Indeed, if $A$ and $B$ are two such divisors in a smooth surface $\Sigma$, then: 
    \[ A \cdot B = \sum_{P \in \Sigma} (A \cdot B)_P.\]

 B\'ezout's theorem for curves $A, B$ in $\Cc\Pp^2$, which states that:
    \[ A \cdot B = \deg(A) \cdot \deg(B),\]
  has the following local analog for $A, B \hookrightarrow (S,O)$:
    \[ (A \cdot B)_O \geq e_O(A) \cdot e_O(B),\]
  with equality if and only if $A$ and $B$ are transversal, in the sense that they have no common tangent line. This disymmetry between B\'ezout's theorem in $\Cc\Pp^2$ (which is an equality) and its local analog (which is an inequality) may be repaired by looking also at the infinitely near points common to the two curve singularities (see \cite[Section 8.4, Theorem 13]{BK 86}):

\begin{theorem} \label{thm:intmult}
    Let $A, B \hookrightarrow (S,O)$ be two plane curve singularities. Then their intersection number $(A \cdot B)_O$ may be computed as follows:
    \[ (A \cdot B)_O = \sum e_P(A) e_P(B),  \]
    the sum being taken over the infinitely near points $P$ common to (the strict transforms of) $A$ and $B$.
\end{theorem}

This theorem gives a first method of computation on a lotus of the intersection number of two branches $A_l$ and $A_m$ of $A$: 

\medskip
\begin{center}
\fbox{
\begin{minipage}{0.75\textwidth}
    \begin{corollary}
       \label{cor:firstmethint}
          Let $\Lambda(\hat{\calc}_A)$ be the lotus of a finite active constellation of crosses adapted to $A$, in the sense of Definition \ref{def:adaptedactconst}. Once the multiplicities of the strict transforms of two distinct branches $A_l$ and $A_m$ of $A$ are computed as edge decorations (see Corollary \ref{cor:multedges}), the intersection number $(A_l \cdot A_m)_O$ is the sum of the products of multiplicities associated to the same edge. 
    \end{corollary}
\end{minipage}
}
\end{center}
\medskip

We consider that this method is rather tedious on examples. For this reason, we do not illustrate it with figures. We find a second method more convenient. Let us explain it.

 \begin{figure}[h!]
    \begin{center}
\begin{tikzpicture}[scale=1]
   
    \draw [->, color=orange, thick](1,0) -- (0,0);
    \draw [-, color=orange, thick](0,0) -- (-1,0);
    \draw [-, color=black, thick](0,1.5) -- (-1,0);
      \draw [-, color=black, thick](0,1.5) -- (1,0);
       \draw [-, color=orange, thick](-0.5,3/4) -- (0.5,3/4);
       \draw [-, color=orange, thick](-1,0) -- (0.5,3/4);
        \draw [->, color=orange, thick](0,1.5) -- (0,2.5);
           
  \node[draw,circle, inner sep=1pt,color=black, fill=black] at (-1,0){};
   \node[draw,circle, inner sep=1pt,color=black, fill=black] at (1,0){};
     \node[draw,circle, inner sep=1pt,color=black, fill=black] at (0,1.5){};
      \node[draw,circle, inner sep=1pt,color=black, fill=black] at (-0.5,3/4){};
       \node[draw,circle, inner sep=1pt,color=black, fill=black] at (0.5,3/4){};

   \node [below, color=black] at (-1,0) {$L_{1}$};
\node [below, color=black] at (1,0) {$L$};
\node [above, color=black] at (0,2.5) {$A$};

  \node [left, color=blue] at (-1,0) {$\mathbf{0}$};
\node [right, color=blue] at (1,0) {$\mathbf{0}$};
\node [right, color=blue] at (0.5,3/4) {$\mathbf{2}$};
\node [left, color=blue] at (-0.5,3/4) {$\mathbf{3}$};
\node [left, color=blue] at (0,1.5) {$\mathbf{6}$};

\node [below, color=black] at (0,0) {\small{$\orange{2}$}};  
\node [below, color=black] at (0,0.55) {\small{$\orange{1}$}};
\node [above, color=black] at (0,0.7) {\small{$\orange{1}$}};
\node [right, color=black] at (0,2) {\small{$\orange{1}$}};
   \end{tikzpicture}
\end{center}
 \caption{The orders of vanishing $\ord_E(A)$, when $E$  is a curve associated to a vertex of the lotus of Figure \ref{fig:asslotuscusp}}
\label{fig:ordvanish1}
   \end{figure}
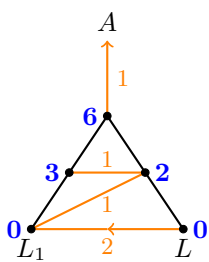


 \begin{figure}[h!]
    \begin{center}
\begin{tikzpicture}[scale=1]
   
    \draw [->, color=orange, thick](1,0) -- (0,0);
    \draw [-, color=orange, thick](0,0) -- (-1,0);
    \draw [-, color=black, thick](0,1.5) -- (-1,0);
      \draw [-, color=black, thick](0,1.5) -- (1,0);
       \draw [-, color=orange, thick](-0.5,3/4) -- (0.5,3/4);
       \draw [-, color=orange, thick](-1,0) -- (0.5,3/4);

        \draw [->, color=orange, thick](0,1.5) -- (-1.5,1.5);
   \draw [-, color=orange, thick](-1.5,1.5) -- (-3,1.5);
     \draw [-, color=black, thick](0,1.5) -- (-1.5,4);
      \draw [-, color=black, thick](-1.5,4) -- (-3,1.5);
      \draw [->, color=orange, thick](-1.5,4) -- (-1.5,5);
     \draw [-, color=orange, thick](-3/4,11/4) -- (-3,1.5);
     \draw [-, color=orange, thick](-3/4,11/4) -- (-2.5,2.34);
       \draw [-, color=orange, thick](-3/4,11/4) -- (-2,3.17);
           
  \node[draw,circle, inner sep=1pt,color=black, fill=black] at (-1,0){};
   \node[draw,circle, inner sep=1pt,color=black, fill=black] at (1,0){};
     \node[draw,circle, inner sep=1pt,color=black, fill=black] at (0,1.5){};
      \node[draw,circle, inner sep=1pt,color=black, fill=black] at (-0.5,3/4){};
       \node[draw,circle, inner sep=1pt,color=black, fill=black] at (0.5,3/4){};
        \node[draw,circle, inner sep=1pt,color=black, fill=black] at (-3,1.5){};
         \node[draw,circle, inner sep=1pt,color=black, fill=black] at (-1.5,4){};
         \node[draw,circle, inner sep=1pt,color=black, fill=black] at (-3/4,11/4){};
         \node[draw,circle, inner sep=1pt,color=black, fill=black] at (-2.5,2.34){};
         \node[draw,circle, inner sep=1pt,color=black, fill=black] at (-2,3.17){};

   \node [below, color=black] at (-1,0) {$L_{1}$};
\node [below, color=black] at (1,0) {$L$};
\node [below, color=black] at (-3,1.5) {$L_{2}$};
\node [above, color=black] at (-1.5,5) {$A_1$};

  \node [left, color=blue] at (-1,0) {$\mathbf{0}$};
\node [right, color=blue] at (1,0) {$\mathbf{0}$};
\node [right, color=blue] at (0.5,3/4) {$\mathbf{6}$};
\node [left, color=blue] at (-0.5,3/4) {$\mathbf{9}$};
\node [right, color=blue] at (0,1.5) {$\mathbf{18}$};
\node [left, color=blue] at (-3,1.5) {$\mathbf{0}$};
\node [right, color=blue] at (-3/4,11/4) {$\mathbf{21}$};
\node [right, color=blue] at (-1.5,4) {$\mathbf{66}$};
\node [left, color=blue] at (-2.5,2.34) {$\mathbf{22}$};
\node [left, color=blue] at (-2,3.17){$\mathbf{44}$};

\node [below, color=black] at (0,0) {\small{$\orange{6}$}};  
\node [below, color=black] at (0,0.55) {\small{$\orange{3}$}};
\node [above, color=black] at (0,0.7) {\small{$\orange{3}$}};
\node [below, color=black] at (-1.5,1.5) {\small{$\orange{3}$}}; 
\node [below, color=black] at (-1.7,2.3) {\small{$\orange{1}$}};  
\node [above, color=black] at (-1.5,2.5) {\small{$\orange{1}$}}; 
\node [above, color=black] at (-1.5,3) {\small{$\orange{1}$}}; 
\node [right, color=black] at (-1.5,4.5) {\small{$\orange{1}$}};

   \end{tikzpicture}
\end{center}
 \caption{The orders of vanishing $\ord_E(A_1)$, when $E$  is a curve associated to a vertex of the lotus of Figure \ref{fig:lotusbranch}}
\label{fig:ordvanish2}
   \end{figure}
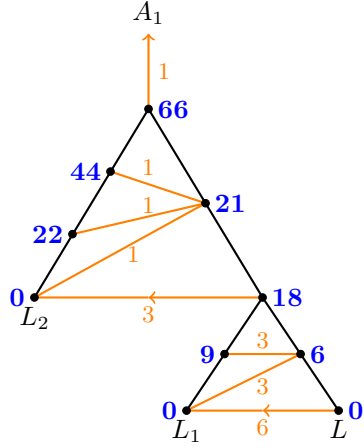


 \begin{figure}[h!]
    \begin{center}
\begin{tikzpicture}[scale=1]
   
    \draw [->, color=orange, thick](1,0) -- (0,0);
    \draw [-, color=orange, thick](0,0) -- (-1,0);
    \draw [-, color=black, thick](0,1.5) -- (-1,0);
      \draw [-, color=black, thick](0,1.5) -- (1,0);
       \draw [-, color=orange, thick](-0.5,3/4) -- (0.5,3/4);
       \draw [-, color=orange, thick](-1,0) -- (0.5,3/4);
       \draw [-, color=orange, thick](-5.2,4.2) -- (-3.7,3.2);

        \draw [->, color=orange, thick](0,1.5) -- (-1.5,1.5);
   \draw [-, color=orange, thick](-1.5,1.5) -- (-3,1.5);
     \draw [-, color=black, thick](0,1.5) -- (-1.5,4);
      \draw [-, color=black, thick](-1.5,4) -- (-3,1.5);
      \draw [->, color=orange, thick](-1.5,4) -- (-1.5,5);
     \draw [-, color=orange, thick](-3/4,11/4) -- (-3,1.5);
     \draw [-, color=orange, thick](-3/4,11/4) -- (-2.5,2.34);
     \draw [-, color=orange, thick](-2,3.17) -- (-2.5,2.34);
       \draw [-, color=orange, thick](-3/4,11/4) -- (-2,3.17);
       \draw [->, color=orange, thick](-3/4,11/4) -- (1/4,15/4);
        \draw [-, color=black, thick](-2,3.17) -- (-1.5,4);
        \draw [-, color=black, thick](-3,1.5) -- (-2.5,2.34);
        \draw [-, color=black, thick](-3,1.5) -- (-2.5,2.34);
        \draw [-, color=black, thick](-3.7,3.2) -- (-2,3.17);
        \draw [-, color=black, thick](-3.7,3.2) -- (-2.5,2.34);
        \draw [->, color=orange, thick](-3.7,3.2) -- (-4.7,3.2);
        \draw [-, color=orange, thick](-5.7,3.2) -- (-4.7,3.2);
        \draw [->, color=orange, thick](-4.7,5.2) -- (-4.7,6.2);
        \draw [-, color=black, thick](-3.7,3.2) -- (-4.7,5.2);
        \draw [-, color=black, thick](-5.7,3.2) -- (-4.7,5.2);
       
  \node[draw,circle, inner sep=1pt,color=black, fill=black] at (-1,0){};
   \node[draw,circle, inner sep=1pt,color=black, fill=black] at (1,0){};
     \node[draw,circle, inner sep=1pt,color=black, fill=black] at (0,1.5){};
      \node[draw,circle, inner sep=1pt,color=black, fill=black] at (-0.5,3/4){};
       \node[draw,circle, inner sep=1pt,color=black, fill=black] at (0.5,3/4){};
        \node[draw,circle, inner sep=1pt,color=black, fill=black] at (-3,1.5){};
         \node[draw,circle, inner sep=1pt,color=black, fill=black] at (-1.5,4){};
         \node[draw,circle, inner sep=1pt,color=black, fill=black] at (-3/4,11/4){};
         \node[draw,circle, inner sep=1pt,color=black, fill=black] at (-2.5,2.34){};
         \node[draw,circle, inner sep=1pt,color=black, fill=black] at (-2,3.17){};
          \node[draw,circle, inner sep=1pt,color=black, fill=black] at (-3.7,3.2){};
          \node[draw,circle, inner sep=1pt,color=black, fill=black] at (-5.7,3.2){};
           \node[draw,circle, inner sep=1pt,color=black, fill=black] at (-4.7,5.2){};
           \node[draw,circle, inner sep=1pt,color=black, fill=black] at (-5.2,4.2){};
       
   \node [below, color=black] at (-1,0) {$L_{1}$};
\node [below, color=black] at (1,0) {$L$};
\node [below, color=black] at (-3,1.5) {$L_{2}$};
\node [right, color=black] at (1/4,15/4) {$A_3$};
\node [above, color=black] at (-1.5,5) {$A_1$};
 \node [below, color=black] at (-5.7,3.2) {$L_{3}$};
 \node [above, color=black] at (-4.7,6.2) {$A_2$};

   \node [left, color=blue] at (-1,0) {$\mathbf{0}$};
\node [right, color=blue] at (1,0) {$\mathbf{0}$};
\node [right, color=blue] at (0.5,3/4) {$\mathbf{20}$};
\node [left, color=blue] at (-0.5,3/4) {$\mathbf{30}$};
\node [right, color=blue] at (0,1.5) {$\mathbf{60}$};
\node [left, color=blue] at (-3,1.5) {$\mathbf{0}$};
\node [right, color=blue] at (-3/4,11/4) {$\mathbf{70}$};
\node [right, color=blue] at (-1.5,4) {$\mathbf{219}$};
\node [left, color=blue] at (-2.5,2.22) {$\mathbf{75}$};
\node [left, color=blue] at (-1.9,3.4){$\mathbf{148}$};
\node [left, color=blue] at  (-5.7,3.2) {$\mathbf{0}$};
\node [right, color=blue] at (-4.7,5.2) {$\mathbf{452}$};
\node [right, color=blue] at (-3.74,3.4) {$\mathbf{225}$};
\node [left, color=blue] at (-5.2,4.3){$\mathbf{226}$};

\node [below, color=black] at (0,0) {\small{$\orange{20}$}};  
\node [below, color=black] at (0,0.55) {\small{$\orange{10}$}};
\node [above, color=black] at (0,0.7) {\small{$\orange{10}$}};
\node [below, color=black] at (-1.5,1.5) {\small{$\orange{10}$}}; 
\node [below, color=black] at (-1.7,2.3) {\small{$\orange{5}$}};  
\node [above, color=black] at (-1.5,2.5) {\small{$\orange{3}$}}; 
\node [left, color=black] at (-1.85,2.8) {\small{$\orange{2}$}}; 
\node [above, color=black] at (-1.5,3) {\small{$\orange{1}$}}; 
\node [below, color=black] at (0,14/4) {\small{$\orange{1}$}}; 
\node [below, color=black] at (-1.2,4.7) {\small{$\orange{1}$}}; 
\node [below, color=black] at (-4.7,3.2) {\small{$\orange{1}$}};
\node [below, color=black] at (-4.7,4.3) {\small{$\orange{1}$}};
\node [below, color=black] at (-4.4,6) {\small{$\orange{1}$}};
   \end{tikzpicture}
\end{center}
 \caption{The orders of vanishing $\ord_E(A_1+A_2+A_3)$, when $E$  is a curve associated to a vertex of the lotus of Figure \ref{fig:lotus3branches}}
\label{fig:ordvanish3}
   \end{figure}
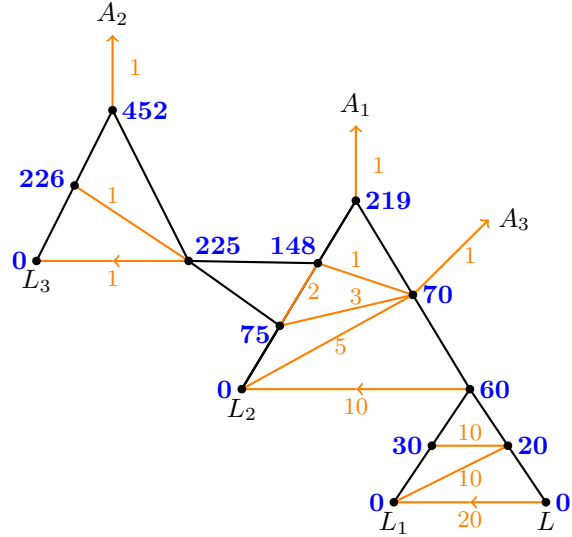


 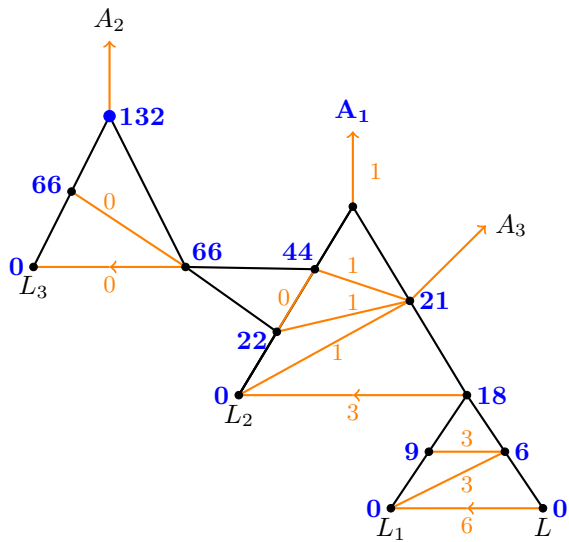
\begin{figure}[h!]
    \begin{center}
\begin{tikzpicture}[scale=1]
   
    \draw [->, color=orange, thick](1,0) -- (0,0);
    \draw [-, color=orange, thick](0,0) -- (-1,0);
    \draw [-, color=black, thick](0,1.5) -- (-1,0);
      \draw [-, color=black, thick](0,1.5) -- (1,0);
       \draw [-, color=orange, thick](-0.5,3/4) -- (0.5,3/4);
       \draw [-, color=orange, thick](-1,0) -- (0.5,3/4);
       \draw [-, color=orange, thick](-5.2,4.2) -- (-3.7,3.2);

        \draw [->, color=orange, thick](0,1.5) -- (-1.5,1.5);
   \draw [-, color=orange, thick](-1.5,1.5) -- (-3,1.5);
     \draw [-, color=black, thick](0,1.5) -- (-1.5,4);
      \draw [-, color=black, thick](-1.5,4) -- (-3,1.5);
      \draw [->, color=orange, thick](-1.5,4) -- (-1.5,5);
     \draw [-, color=orange, thick](-3/4,11/4) -- (-3,1.5);
     \draw [-, color=orange, thick](-3/4,11/4) -- (-2.5,2.34);
     \draw [-, color=orange, thick](-2,3.17) -- (-2.5,2.34);
       \draw [-, color=orange, thick](-3/4,11/4) -- (-2,3.17);
       \draw [->, color=orange, thick](-3/4,11/4) -- (1/4,15/4);
        \draw [-, color=black, thick](-2,3.17) -- (-1.5,4);
        \draw [-, color=black, thick](-3,1.5) -- (-2.5,2.34);
        \draw [-, color=black, thick](-3,1.5) -- (-2.5,2.34);
        \draw [-, color=black, thick](-3.7,3.2) -- (-2,3.17);
        \draw [-, color=black, thick](-3.7,3.2) -- (-2.5,2.34);
        \draw [->, color=orange, thick](-3.7,3.2) -- (-4.7,3.2);
        \draw [-, color=orange, thick](-5.7,3.2) -- (-4.7,3.2);
        \draw [->, color=orange, thick](-4.7,5.2) -- (-4.7,6.2);
        \draw [-, color=black, thick](-3.7,3.2) -- (-4.7,5.2);
        \draw [-, color=black, thick](-5.7,3.2) -- (-4.7,5.2);
       
  \node[draw,circle, inner sep=1pt,color=black, fill=black] at (-1,0){};
   \node[draw,circle, inner sep=1pt,color=black, fill=black] at (1,0){};
     \node[draw,circle, inner sep=1pt,color=black, fill=black] at (0,1.5){};
      \node[draw,circle, inner sep=1pt,color=black, fill=black] at (-0.5,3/4){};
       \node[draw,circle, inner sep=1pt,color=black, fill=black] at (0.5,3/4){};
        \node[draw,circle, inner sep=1pt,color=black, fill=black] at (-3,1.5){};
         \node[draw,circle, inner sep=1pt,color=black, fill=black] at (-1.5,4){};
         \node[draw,circle, inner sep=1pt,color=black, fill=black] at (-3/4,11/4){};
         \node[draw,circle, inner sep=1pt,color=black, fill=black] at (-2.5,2.34){};
         \node[draw,circle, inner sep=1pt,color=black, fill=black] at (-2,3.17){};
          \node[draw,circle, inner sep=1pt,color=black, fill=black] at (-3.7,3.2){};
          \node[draw,circle, inner sep=1pt,color=black, fill=black] at (-5.7,3.2){};
           \node[draw,circle, inner sep=1.5pt,color=blue, fill=blue] at (-4.7,5.2){};
           \node[draw,circle, inner sep=1pt,color=black, fill=black] at (-5.2,4.2){};
       
   \node [below, color=black] at (-1,0) {$L_{1}$};
\node [below, color=black] at (1,0) {$L$};
\node [below, color=black] at (-3,1.5) {$L_{2}$};
\node [right, color=black] at (1/4,15/4) {$A_3$};
\node [above, color=blue] at (-1.5,5) {$\mathbf{A_1}$};
 \node [below, color=black] at (-5.7,3.2) {$L_{3}$};
 \node [above, color=black] at (-4.7,6.2) {$A_2$};

   \node [left, color=blue] at (-1,0) {$\mathbf{0}$};
\node [right, color=blue] at (1,0) {$\mathbf{0}$};
\node [right, color=blue] at (0.5,3/4) {$\mathbf{6}$};
\node [left, color=blue] at (-0.5,3/4) {$\mathbf{9}$};
\node [right, color=blue] at (0,1.5) {$\mathbf{18}$};
\node [left, color=blue] at (-3,1.5) {$\mathbf{0}$};
\node [right, color=blue] at (-3/4,11/4) {$\mathbf{21}$};
\node [left, color=blue] at (-2.5,2.22) {$\mathbf{22}$};
\node [left, color=blue] at (-1.9,3.4){$\mathbf{44}$};
\node [left, color=blue] at  (-5.7,3.2) {$\mathbf{0}$};
\node [right, color=blue] at (-4.7,5.2) {$\mathbf{132}$};
\node [right, color=blue] at (-3.74,3.4) {$\mathbf{66}$};
\node [left, color=blue] at (-5.2,4.3){$\mathbf{66}$};

\node [below, color=black] at (0,0) {\small{$\orange{6}$}};  
\node [below, color=black] at (0,0.55) {\small{$\orange{3}$}};
\node [above, color=black] at (0,0.7) {\small{$\orange{3}$}};
\node [below, color=black] at (-1.5,1.5) {\small{$\orange{3}$}}; 
\node [below, color=black] at (-1.7,2.3) {\small{$\orange{1}$}};  
\node [above, color=black] at (-1.5,2.5) {\small{$\orange{1}$}}; 
\node [left, color=black] at (-2.2,2.8) {\small{$\orange{0}$}}; 
\node [above, color=black] at (-1.5,3) {\small{$\orange{1}$}}; 
\node [below, color=black] at (-1.2,4.7) {\small{$\orange{1}$}}; 
\node [below, color=black] at (-4.7,3.2) {\small{$\orange{0}$}};
\node [below, color=black] at (-4.7,4.3) {\small{$\orange{0}$}};
   \end{tikzpicture}
\end{center}
 \caption{The orders of vanishing $\ord_E(A_1)$, when $E$  is a curve associated to a vertex of the lotus of Figure \ref{fig:lotus3branches}  which allows to compute $(A_1 \cdot A_2)_O$ as $\ord_{E_{A_2}}(A_1)$}
\label{fig:ordvanish4}
   \end{figure}

Similarly to the case of multiplicities considered in Proposition \ref{prop:ordvanish}, local intersection numbers may be also computed as orders of vanishing on resolutions:

\begin{proposition}  
   \label{prop:intbr}
    Let $A$ be a branch and let $D$ be a curve singularity on $(S,O)$,  which does not contain $A$ in its support. Consider an embedded resolution $\pi : S_{\pi} \to S$ of $A$ such that the strict transforms $A_\pi$ and $D_\pi$ of $A$ and $D$ are disjoint. Denote by $\boxed{E_A}$ the unique irreducible component of the exceptional divisor of $\pi$ which intersects the strict transform $A_{\pi}$ of $A$ by $\pi$. Then:
      \[ (A \cdot D)_O = \ord_{E_A} (D).\]
\end{proposition}

\begin{proof}
    Using the projection formula (see Theorem \ref{thm:projform}), we get:
      \begin{eqnarray*}
          (A \cdot D)_O & = & \pi_*(A_{\pi}) \cdot D  =  A_{\pi} \cdot \pi^*(D) = 
               A_{\pi} \cdot \left( \sum_{E \subseteq E_{\pi}} \ord_E(D) \ E + D_{\pi} \right)  \\
         & = & A_{\pi} \cdot (\ord_{E_A} (D) \ E_A) = \ord_{E_A} (D). 
      \end{eqnarray*}
\end{proof}

One has the following immediate consequence of Proposition \ref{prop:ordvanish}:

\begin{proposition}  \label{prop:ordvanishspec}
      Let $L + L_1$ be a cross on $(S,O)$. Let $D = \ord_L(D) L + \ord_{L_1}(D) L_1 + D'$ be a curve singularity on $(S,O)$, where $L$ and $L_1$ are not components of $D'$.  Consider the blowup morphism $\pi_O: S_O \to S$ of $S$ at $O$, with exceptional divisor $E_O$. Then:
         \[\ord_{E_O}(D) = \ord_L(D) + \ord_{L_1}(D) + e_O(D').\]
\end{proposition}

Therefore: 

\medskip
\begin{center}
\fbox{
\begin{minipage}{0.75\textwidth}
    \begin{corollary}
       \label{cor:secmethint}
    One may compute the orders of vanishing $\ord_E(A)$ of a plane curve singularity $A \hookrightarrow (S,O)$ using an associated lotus $\Lambda(\hat{\calc}_A)$ by the following recursive algorithm:
  \begin{enumerate}
      \item 
        Decorate the base and interior edges with the multiplicities of strict transforms of $A$, as explained in Corollary \ref{cor:multedges}. 
      \item 
        Then, starting from the petal corresponding to the active point $O$ of the active constellation of crosses $\hat{\calc}_A$, decorate the vertex opposite to the base $e$ of each petal with the sum of decorations of $e$ and of its two vertices. 
  \end{enumerate}
\end{corollary}
\end{minipage}
}
\end{center}
\medskip

  \begin{example}
      \label{ex:ordercomp}
    Figures \ref{fig:ordvanish1}, \ref{fig:ordvanish2} and \ref{fig:ordvanish3} illustrate 
    the result of applying the algorithm of Corollary \ref{cor:secmethint} to Figures \ref{fig:multbranch1}, \ref{fig:multbranch2} and \ref{fig:multbranch3} respectively. Figure \ref{fig:ordvanish4} illustrates the result of applying the same algorithm on the lotus of Figure \ref{fig:ordvanish3} in order to compute the intersection number $(A_1 \cdot A_2)_O$ as $\ord_{E_{A_2}}(A_1)$. Note that we do not need to compute the order of vanishing $\ord_{E_{A_1}}(A_1)$ (equal, by the way, to $44 + 21+1 =66$) in order to reach by this algorithm the value of $\ord_{E_{A_2}}(A_1)$.
  \end{example}

\begin{remark}
    \label{rem:charpolzeta}
    Once we know the orders of vanishing of $A$ and the log-discrepancies as decorations of the vertices of the lotus, one may compute the {\em topological Zeta function}, the {\em characteristic polynomial of the monodromy} and the {\em multivariable Alexander polynomial} of $A$ by looking at those values at particular vertices, as explained by Wall in \cite[Sections 8.4 and 10.3]{W 04}. 
\end{remark}

\medskip
\section{The Eggers-Wall tree associated to a  lotus}  
\label{sec:EWfromlot}

In this section we explain following \cite[Section 1.6.1]{GBGPPP 20} how to associate an  \emph{Eggers-Wall} tree 
$\Theta_{L}(A)$ to a  reduced  plane curve singularity $A \hookrightarrow (S,O)$, relative to a smooth branch $L$ (see Definition \ref{def:EW}).  Then, we describe the main properties of the natural embedding built in \cite[Section 8]{GBGPPP 19} of $\Theta_{L}(A)$ in the space of semivaluations of $(S,O)$  
(see Proposition \ref{prop:ew-emb}). Finally, we show how to construct $\Theta_{L}(A)$ from an associated lotus $\hat{\calc}_A$, decorated by log-discrepancies $\lambda(E)$ and orders of vanishing $\ord_E(L)$, whose computation was presented in Section \ref{sec:ldov} (see Proposition \ref{prop:fromlotustoEW}).

\medskip

We denote by 
  \[ 
  \boxed{\Cc\{ x^{1/ \Nn} \}} := \bigcup_{n \geq 1 } \Cc\{ x^{1/ n} \} 
  \] 
  the ring of  convergent {\bf Newton-Puiseux series} and by  $\boxed{\nu_x}$ its {\bf $x$-adic valuation},  defined on a nonzero series $\xi  = \sum c_{\alpha} x^{\alpha} \in  \Cc\{x^{1/ n}\}$ by: 
  \[ \boxed{\nu_x(\xi)} := \min\{ \alpha \mid c_\alpha \ne 0 \} \in \Zz_{\geq 0}. \]

\medskip
 Consider a {\em reduced} curve singularity $A \hookrightarrow (S,O)$. Let $L$ be a smooth branch on $(S,O)$ and $(x,y)$ a local coordinate system on $(S,O)$ such that $L = Z(x)$. We assume that $L$ is not a component of $A$.  Let $f \in \Cc\{x,y\}$ be a defining function of $A$ in this coordinate system.  As a consequence of the {\bf Newton-Puiseux Theorem} (see \cite[Theorem 1.2.20]{GBGPPP 20} and the references preceding it), there exists a finite subset $\boxed{\cZ_x(A)} \subset \Cc\{ x^{1/ \Nn} \}$ of {\bf zeros of $f$ relative to $x$}
   and a unit $u(x,y)$ of the local ring $\Cc\{x,y\}$, such that:
    \begin{equation}
      \label{fmla:NewtPuiseux}
                 f(x,y) = u(x,y) \prod_{\xi \in \cZ_x(A)} (y - \xi). 
      \end{equation}

The set $\cZ_x(A)$  is equal to the disjoint union of the 
sets $\cZ_x(A_l)$, when $A_l$ varies among the branches of $A$. In turn, each set $\cZ_x(A_l)$ is an orbit of the action on $\Cc\{ x^{1/ \Nn} \}$ of the Galois group of the field extension induced by the ring extension $\Cc\{ x \} \hookrightarrow \Cc\{ x^{1/ \Nn} \}$. Namely, if $\xi \in \Cc\{ x^{1/ \Nn} \}$ belongs to $\cZ_x(A_l)$ then, after writing $\xi = \eta(x^{\frac{1}{n}})$ with $\eta \in \Cc\{t\}$ and $n \geq 1$, we get:
\begin{equation}
\label{eq:Zer}
   \cZ_x(A_l) = \{ \eta(\rho \cdot x^{\frac{1}{n}}); \ \  \rho \in \Cc, \rho^n =1  \}. 
   \end{equation}
  That is, the Galois conjugates of $\eta(x^{\frac{1}{n}})$ are simply obtained by replacing $x^{\frac{1}{n}}$ by all possible $n$-th roots of $x$. Note that the elements $\eta(\rho \cdot x^{\frac{1}{n}})$ of the set $\cZ_x(A_l)$ are pairwise distinct if and only if $n$ is the lowest common denominator of the exponents of $\xi$.

The elements of $\cZ_x(A)$ allow to define the following important {\em invariants} (see Remark \ref{rem:genericEW}) of $A$ relative to the smooth branch $L$:

\begin{definition} 
   \label{def:charexpcoinc}
    The set of {\bf characteristic exponents} of a branch $A_l$ of $A$ relative to $L$  is 
    \[\boxed{{\mathrm{Ch}}_L (A_l)}  := \{ \nu_x(\xi - \xi') , \:  \xi, \xi'  \in \cZ_x(A_l), \: \xi \neq \xi'  \}   \subseteq \Qq_{> 0}. \]
    
  The {\bf order of coincidence}  of two distinct branches $A_l$ and $A_m$ of $A$ relative to $L$  is  
  \[  \boxed{k_L (A_l, A_m)} := \max\{ \nu_x(\xi - \xi') , \:  \xi \in \cZ_x(A_l), \: \xi' \in \cZ_x(A_m) \} \in \Qq_{> 0}. \]
\end{definition}

\begin{example}
   \label{ex:charexpordcoinc}
    Assume that $A = A_1 + A_2$, where $x^{\frac{3}{2}} + x^{\frac{5}{2}} + x^{\frac{11}{4}} \in \cZ_x(A_1)$ and $ - x^{\frac{3}{2}} +  x^{\frac{5}{2}} +  x^{\frac{11}{4}} \in \cZ_x(A_2)$. As both series belong to $\Cc\{x^{\frac{1}{4}}\}$, we get their Galois conjugates replacing successively $x^{\frac{1}{4}}$ by $x^{\frac{1}{4}}$, $ i x^{\frac{1}{4}}$, $- x^{\frac{1}{4}}$, $- i x^{\frac{1}{4}}$. Therefore: 
      \begin{eqnarray*}
          \cZ_x(A_1) = \{ x^{\frac{3}{2}} + x^{\frac{5}{2}} + x^{\frac{11}{4}}, \ - x^{\frac{3}{2}} - x^{\frac{5}{2}} -i  x^{\frac{11}{4}}, \ x^{\frac{3}{2}} + x^{\frac{5}{2}} - x^{\frac{11}{4}},  \ - x^{\frac{3}{2}} - x^{\frac{5}{2}} + ix^{\frac{11}{4}}  \},  \\
            \cZ_x(A_2) = \{ - x^{\frac{3}{2}} - x^{\frac{5}{2}} + x^{\frac{11}{4}}, \ x^{\frac{3}{2}} + x^{\frac{5}{2}} -i  x^{\frac{11}{4}}, \ - x^{\frac{3}{2}} -  x^{\frac{5}{2}}- x^{\frac{11}{4}},  \  x^{\frac{3}{2}}+ x^{\frac{5}{2}} + ix^{\frac{11}{4}}  \}. 
      \end{eqnarray*}
    We get:
      \[ \begin{array}{l}
          {\mathrm{Ch}}_L (A_1)  =  {\mathrm{Ch}}_L (A_2) = \{ \frac{3}{2}, \frac{11}{4} \},\\
          k_L (A_1, A_2)  =  \frac{11}{4}.
      \end{array} \]
\end{example}

The joint information contained in the sets ${\mathrm{Ch}}_L (A_l)$ of characteristic exponents and in the orders of coincidence $k_L (A_l, A_m)$ may be encoded geometrically in the {\em Eggers-Wall tree of $A$ relative to $L$} (see \cite[Section 1.6.6]{GBGPPP 20} for historical information about the evolution of this notion):

\medskip
\begin{definition} 
   \label{def:EW}
      Let $A \hookrightarrow (S,O)$ be a reduced plane curve singularity and let $L \hookrightarrow (S,O)$ be a smooth branch. Let $(x,y)$ be a local coordinate system on $(S,O)$ such that $L = Z(x)$.          

              \noindent 
              $\bullet$
		The {\bf Eggers-Wall tree} $\boxed{\Theta_L(A_l)}$ of a branch $A_l \neq L$ of $A$ relative to $L$ is a compact segment endowed with a homeomorphism $\boxed{\ex_L}: \Theta_L(A_l) \to [0, \infty]$ called the {\bf exponent function},  and with {\bf marked points}, which are the preimages by the exponent function of the characteristic exponents of $A_l$ relative to $L$. The point $(\ex_L)^{-1}(0)$ is {\bf labeled by $L$} and $(\ex_L)^{-1}(\infty)$ is {\bf labeled by $A_l$}. The tree $\Theta_L (A_l)$ is also endowed with the  {\bf index function}   $\boxed{\de_L}: \Theta_L(A_l) \to \Zz_{>0}$, whose value $\de_L(P)$ on a point $P \in \Theta_L(A_l)$ is equal to the lowest common multiple of the denominators of the exponents of the marked points belonging to the half-open segment $[LP)$.

              \noindent 
        $\bullet$ 
          The {\bf Eggers-Wall tree} $\boxed{\Theta_L(L)}$ is reduced to a point labeled by $L$, at which $\ex_L(L) = 0$ and $\de_L(L) = 1$.

              \noindent 
        $\bullet$
          The {\bf Eggers-Wall tree} $\boxed{\Theta_L (A)}$ of $A$ relative to $L$ is obtained from the disjoint union of the Eggers-Wall trees $\Theta_L(A_l)$ of its branches $A_l$ by identifying, for each pair of distinct branches $A_l$ and $A_m$ of $A$, their points with equal exponents not greater than the order of coincidence $k_L(A_l, A_m)$. Its {\bf marked points} are its ramification points and the images of the marked points of the trees $\Theta_L(A_l)$ by the identification map. Its {\bf labeled points} are analogously the images of the labeled points of the trees $\Theta_L(A_l)$, the identification map being label-preserving. It is endowed with an {\bf exponent function} $\boxed{\ex_L}: \Theta_L(A) \to [0, \infty]$ and an {\bf index function}  $\boxed{\de_L}: \Theta_L(A) \to \Zz_{>0}$ obtained by gluing the exponent functions  and index functions on the trees $\Theta_L(A_l)$ respectively.
  
    We say that a point $P \in \Theta_L (A)$ is {\bf rational} if $\ex_L (P) \in \Qq_{\geq 0} \cup \{\infty\}$.

    The tree $\Theta_L (A)$ is {\bf rooted at the point labeled by $L$}. Its root is the unique minimum for the partial order on $\Theta_L (A)$ defined by:
       \[ \boxed{P \preceq_L Q} \  \Longleftrightarrow \  P \in [L Q] \]
    whenever  $P, Q  \in \Theta_L (A)$.
    The {\bf leaves} of $\Theta_L (A)$ are its ends points different from $L$, that is, the maximal points of the tree with respect to $\preceq_L$.
\end{definition}

Note that inside the Eggers-Wall tree $\Theta_L(A)$ of a reduced plane curve singularity with several branches, the Eggers-Wall tree $\Theta_L(A_l)$ of each of the branches $A_l$ of $A$ is equal as a topological space to the unique segment $[LA_l]$ joining the points $L$ and $A_l$. It is in order to have the same property for the branch $L$ that we define $\Theta_L(L) := [LL] = \{L\}$.


 \begin{figure}[h!]
    \begin{center}
\begin{tikzpicture}[scale=1]
\begin{scope}[shift={(0,0)}]

     \draw [->, color=blue,  line width=2.5pt](0,0) -- (0,4.5);
     \draw [->, color=blue,  line width=2.5pt](0,1) -- (0,0);

       \node [below, color=black] at (0,0) {$L$};
        \node [above, color=black] at (0,4.5) {$A_1$};
        \node [below, color=black] at (0,-1) {$\Theta_L(A_1)$};

    \node[draw,circle, inner sep=1.5pt,color=violet, fill=violet] at (0,1.5){};
    \node[draw,circle, inner sep=1.5pt,color=violet, fill=violet] at (0,3){};

        \node [left, color=blue] at (0,0.75) {\small{$1$}};
        \node [left, color=blue] at (0,2.25) {\small{$2$}};
        \node [left, color=blue] at (0,3.75) {\small{$4$}};
        \node [right, color=violet] at (0,1.5) {\small{$\frac{3}{2}$}}; 
        \node [right, color=violet] at (0,3) {\small{$\frac{11}{4}=k_L (A_1, A_2)=$}}; 
                            
      \end{scope}
      
      \begin{scope}[shift={(3.3,0)}]

     \draw [->, color=blue,  line width=2.5pt](0,0) -- (0,4.5);
     \draw [->, color=blue,  line width=2.5pt](0,1) -- (0,0);

       \node [below, color=black] at (0,0) {$L$};
       \node [above, color=black] at (0,4.5) {$A_2$};
       \node [below, color=black] at (0,-1) {$\Theta_L(A_2)$};

    \node[draw,circle, inner sep=1.5pt,color=violet, fill=violet] at (0,1.5){};
    \node[draw,circle, inner sep=1.5pt,color=violet, fill=violet] at (0,3){};

                \node [right, color=blue] at (0,0.75) {\small{$1$}};
                \node [right, color=blue] at (0,2.25) {\small{$2$}};
                \node [right, color=blue] at (0,3.75) {\small{$4$}};
           \node [left, color=violet] at (0,1.5) {\small{$\frac{3}{2}$}}; 
            \node [left, color=violet] at (0,3) {\small{$\frac{11}{4}$}}; 
                            
      \end{scope}

  \draw[->][thick, color=black](4,3) .. controls (5,2.5) ..(6,3);

      \begin{scope}[shift={(8,0)}]

     \draw [-, color=blue,  line width=2.5pt](0,0) -- (0,3);
     \draw [->, color=blue,  line width=2.5pt](0,3) -- (1,4.3);
      \draw [->, color=blue,  line width=2.5pt](0,3) -- (-1,4.3);
     \draw [->, color=blue,  line width=2.5pt](0,1) -- (0,0);

       \node [below, color=black] at (0,0) {$L$};
       \node [above, color=black] at (-1,4.3) {$A_1$};
           \node [above, color=black] at (1,4.3) {$A_2$};
           \node [below, color=black] at (0,-1) {$\Theta_L(A_1+A_2)$};

    \node[draw,circle, inner sep=1.5pt,color=violet, fill=violet] at (0,1.5){};
    \node[draw,circle, inner sep=1.5pt,color=violet, fill=violet] at (0,3){};

                \node [right, color=blue] at (0,0.75) {\small{$1$}};
                 \node [right, color=blue] at (0,2.25) {\small{$2$}};
                  \node [left, color=blue] at (-0.5,3.55) {\small{$4$}};
                    \node [right, color=blue] at (0.5,3.55) {\small{$4$}};
           \node [left, color=violet] at (0,1.5) {\small{$\frac{3}{2}$}}; 
            \node [left, color=violet] at (0,2.9) {\small{$\frac{11}{4}$}}; 
                            
      \end{scope}
     
   \end{tikzpicture}
\end{center}
 \caption{Construction of the Eggers-Wall tree of the plane curve singularity of Example \ref{ex:charexpordcoinc}}
\label{fig:constrEW}
   \end{figure}
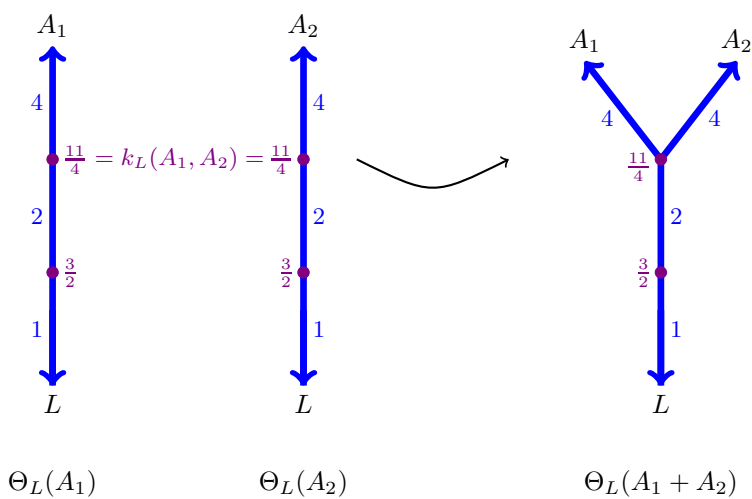


\begin{example}
  \label{ex:constrEW}
     In Figure \ref{fig:constrEW} we illustrate the construction of the Eggers-Wall tree $\Theta_L(A_1 + A_2)$ of the plane curve singularity of Example \ref{ex:charexpordcoinc} by gluing the Eggers-Wall trees $\Theta_L(A_1)$ and $\Theta_L(A_2)$.
\end{example}

\begin{figure}[h!]
    \begin{center}
\begin{tikzpicture}[scale=1.2]
\begin{scope}[shift={(-20,0)}]

 \draw [->, color=black, thick](1,0) -- (0,0);
    \draw [-, color=black, thick](0,0) -- (-1,0);
    \draw [-, color=black, thick](0,1.5) -- (-1,0);
      \draw [-, color=black, thick](0,1.5) -- (1,0);
       \draw [-, color=black, thick](-0.5,3/4) -- (0.5,3/4);
       \draw [-, color=black, thick](-1,0) -- (0.5,3/4);
       \draw [->, color=black, thick](0,1.5) -- (0,2.5);
       
  \node[draw,circle, inner sep=1pt,color=black, fill=black] at (-1,0){};
   \node[draw,circle, inner sep=1pt,color=black, fill=black] at (1,0){};
     \node[draw,circle, inner sep=1pt,color=black, fill=black] at (0,1.5){};
      \node[draw,circle, inner sep=1pt,color=black, fill=black] at (-0.5,3/4){};
       \node[draw,circle, inner sep=1pt,color=black, fill=black] at (0.5,3/4){};

\node [right, color=black] at (0,2.5) {$A$};

\node [left, color=black] at (-1,0) {\small{$(\red{1},\blue{0})$}};  
\node [right, color=black] at (1,0) {\small{$(\red{1},\blue{1})$}};  
 \node [right, color=black] at (0,1.5) {\small{$(\red{5},\blue{2}) $}};  
 \node [left, color=black] at (-0.5,3/4) {\small{$(\red{3},\blue{1})$}}; 
  \node [right, color=black] at (0.5,3/4) {\small{$(\red{2},\blue{1})$}}; 
  
    \draw[->][thick, color=black](-1,-0.75) .. controls (-1.5,-2) ..(-1,-2.75); 
   
   \end{scope}

   \begin{scope}[shift={(-20,-4)}]
  
 \draw [->, color=black, thick](1,0) -- (0,0);
    \draw [-, color=black, thick](0,0) -- (-1,0);
    \draw [-, color=blue, line width=2.5pt](0,1.5) -- (-1,0);
      \draw [-, color=blue, line width=2.5pt](0,1.5) -- (1,0);
       \draw [-, color=black, thick](-0.5,3/4) -- (0.5,3/4);
       \draw [-, color=black, thick](-1,0) -- (0.5,3/4);
        \draw [->, color=blue, line width=2.5pt](0,1.5) -- (0,2.5);

  \node[draw,circle, inner sep=1.5pt,color=violet, fill=violet] at (-1,0){};
   \node[draw,circle, inner sep=1.5pt,color=violet, fill=violet] at (1,0){};
     \node[draw,circle, inner sep=1.5pt,color=violet, fill=violet] at (0,1.5){};
      \node[draw,circle, inner sep=1.5pt,color=black, fill=black] at (-0.5,3/4){};
       \node[draw,circle, inner sep=1.5pt,color=black, fill=black] at (0.5,3/4){};

\node [right, color=black] at (0.1,2.5) {$A$};
 \node [right, color=black] at (0,1.5) {\small{$(\red{5},\blue{2}) \to \violet{\frac{3}{2}}$}};  
 \node [right, color=black] at (1,0) {\small{$(\red{1},\blue{1}) \to \violet{0}$}};  
 \node [left, color=black] at (-1,0) {\small{$(\red{1},\blue{0})\to \violet{\infty}$}};  

\end{scope}

\begin{scope}[shift={(-14,-3)}]
   
  \draw[->][thick, color=black](-4,0) .. controls (-3,-0.5) ..(-2,0); 
   
     \draw [->, color=blue,  line width=2.5pt](0,0) -- (0,3);
     \draw [->, color=blue,  line width=2.5pt](0,1) -- (0,0);
      \node[draw,circle, inner sep=1.5pt,color=violet, fill=violet] at (0,1.5){};
      \draw [->, color=blue,  line width=2.5pt](0,1.5) -- (-1.4,1.5);

       \node [below, color=black] at (0,0) {$L$};
         \node [left, color=black] at (-1.4,1.5) {$L_{1}$};
           \node [above, color=black] at (0,3) {$A$};

           \node [right, color=violet] at (0,0) {\small{$0$}};  
           \node [right, color=violet] at (0,1.5) {\small{$\frac{3}{2}$}}; 
                \node [below, color=violet] at (-1.3,1.4) {\small{$\infty$}};

                 \node [right, color=blue] at (0,0.75) {\small{$1$}};
                 \node [right, color=blue] at (0,2.25) {\small{$2$}};
                   \node [above, color=blue] at (-0.7,1.5) {\small{$1$}};

      \end{scope}
         
   \end{tikzpicture}
\end{center}
 \caption{Going from $\lambda(E)$ and $\ord_E(L)$ to  $\Theta_L(\hat{A})$ when $\hat{A} = L + L_1 + A$}
\label{fig:logdistoEW1}
   \end{figure}
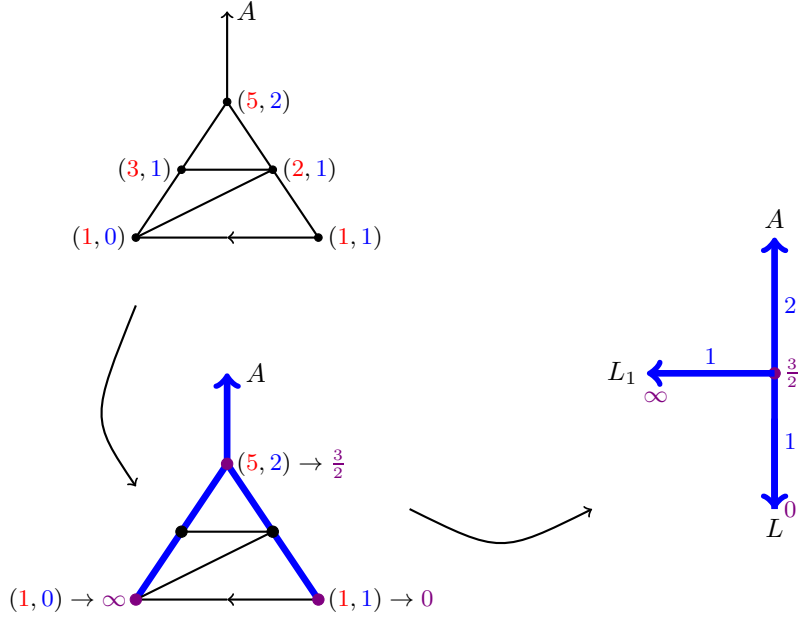

    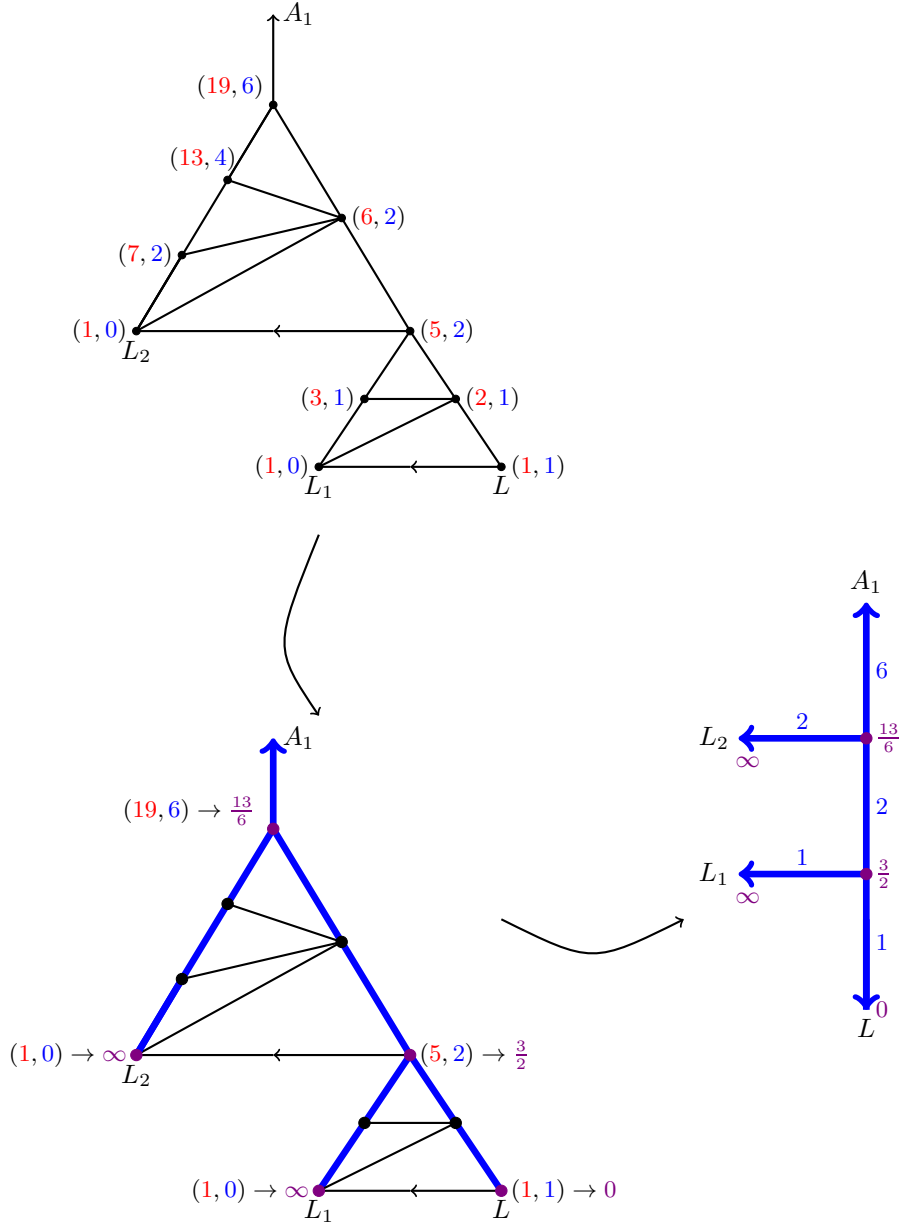
\begin{figure}[h!]
    \begin{center}
\begin{tikzpicture}[scale=1.2]
\begin{scope}[shift={(-20,0)}]
   
    \draw [->, color=black, thick](1,0) -- (0,0);
    \draw [-, color=black, thick](0,0) -- (-1,0);
    \draw [-, color=black, thick](0,1.5) -- (-1,0);
      \draw [-, color=black, thick](0,1.5) -- (1,0);
       \draw [-, color=black, thick](-0.5,3/4) -- (0.5,3/4);
       \draw [-, color=black, thick](-1,0) -- (0.5,3/4);

        \draw [->, color=black, thick](0,1.5) -- (-1.5,1.5);
   \draw [-, color=black, thick](-1.5,1.5) -- (-3,1.5);
     \draw [-, color=black, thick](0,1.5) -- (-1.5,4);
      \draw [-, color=black, thick](-1.5,4) -- (-3,1.5);
      \draw [->, color=black, thick](-1.5,4) -- (-1.5,5);
     \draw [-, color=black, thick](-3/4,11/4) -- (-3,1.5);
     \draw [-, color=black, thick](-3/4,11/4) -- (-2.5,2.34);
       \draw [-, color=black, thick](-3/4,11/4) -- (-2,3.17);
        \draw [-, color=black, thick](-2,3.17) -- (-1.5,4);
        \draw [-, color=black, thick](-3,1.5) -- (-2.5,2.34);
        \draw [-, color=black, thick](-3,1.5) -- (-2.5,2.34);

  \node[draw,circle, inner sep=1pt,color=black, fill=black] at (-1,0){};
   \node[draw,circle, inner sep=1pt,color=black, fill=black] at (1,0){};
     \node[draw,circle, inner sep=1pt,color=black, fill=black] at (0,1.5){};
      \node[draw,circle, inner sep=1pt,color=black, fill=black] at (-0.5,3/4){};
       \node[draw,circle, inner sep=1pt,color=black, fill=black] at (0.5,3/4){};
        \node[draw,circle, inner sep=1pt,color=black, fill=black] at (-3,1.5){};
         \node[draw,circle, inner sep=1pt,color=black, fill=black] at (-1.5,4){};
         \node[draw,circle, inner sep=1pt,color=black, fill=black] at (-3/4,11/4){};
         \node[draw,circle, inner sep=1pt,color=black, fill=black] at (-2.5,2.34){};
         \node[draw,circle, inner sep=1pt,color=black, fill=black] at (-2,3.17){};
                 
   \node [below, color=black] at (-1,0) {$L_{1}$};
\node [below, color=black] at (1,0) {$L$};
\node [below, color=black] at (-3,1.5) {$L_{2}$};
\node [right, color=black] at (-1.5,5) {$A_1$};

 \node [left, color=black] at (-1,0) {\small{$(\red{1},\blue{0})$}};  
\node [right, color=black] at (1,0) {\small{$(\red{1},\blue{1})$}};  
 \node [right, color=black] at (0,1.5) {\small{$(\red{5},\blue{2}) $}};  
 \node [left, color=black] at (-0.5,3/4) {\small{$(\red{3},\blue{1})$}}; 
  \node [right, color=black] at (0.5,3/4) {\small{$(\red{2},\blue{1})$}}; 
     \node [left, color=black] at (-3,1.5) {\small{$(\red{1},\blue{0})$}}; 
     \node [left, color=black] at (-1.5,4.2) {\small{$(\red{19},\blue{6})$}}; 
    \node [right, color=black] at (-3/4,11/4) {\small{$(\red{6},\blue{2})$}};  
   \node [left, color=black] at (-2.5,2.34) {\small{$(\red{7},\blue{2})$}}; 
   \node [left, color=black] at (-1.8,3.4) {\small{$(\red{13},\blue{4})$}}; 
 
    \draw[->][thick, color=black](-1,-0.75) .. controls (-1.5,-2) ..(-1,-2.75); 
   
   \end{scope}

   \begin{scope}[shift={(-20,-8)}]
   
    \draw [->, color=black, thick](1,0) -- (0,0);
    \draw [-, color=black, thick](0,0) -- (-1,0);
    \draw [-, color=blue, line width=2.5pt](0,1.5) -- (-1,0);
      \draw [-, color=blue, line width=2.5pt](0,1.5) -- (1,0);
       \draw [-, color=black, thick](-0.5,3/4) -- (0.5,3/4);
       \draw [-, color=black, thick](-1,0) -- (0.5,3/4);

        \draw [->, color=black, thick](0,1.5) -- (-1.5,1.5);
   \draw [-, color=black, thick](-1.5,1.5) -- (-3,1.5);
     \draw [-, color=blue, line width=2.5pt](0,1.5) -- (-1.5,4);
      \draw [-, color=black, thick](-1.5,4) -- (-3,1.5);
      \draw [->, color=blue, line width=2.5pt](-1.5,4) -- (-1.5,5);
     \draw [-, color=black, thick](-3/4,11/4) -- (-3,1.5);
     \draw [-, color=black, thick](-3/4,11/4) -- (-2.5,2.34);
       \draw [-, color=black, thick](-3/4,11/4) -- (-2,3.17);
   
        \draw [-, color=blue, line width=2.5pt](-2,3.17) -- (-1.5,4);
        \draw [-, color=blue, line width=2.5pt](-3,1.5) -- (-2.5,2.34);
        \draw [-, color=blue, line width=2.5pt](-3,1.5) -- (-2.5,2.34);
           \draw [-, color=blue, line width=2.5pt](-2,3.17) -- (-2.5,2.34);

  \node[draw,circle, inner sep=1.5pt,color=violet, fill=violet] at (-1,0){};
   \node[draw,circle, inner sep=1.5pt,color=violet, fill=violet] at (1,0){};
     \node[draw,circle, inner sep=1.5pt,color=violet, fill=violet] at (0,1.5){};
      \node[draw,circle, inner sep=1.5pt,color=black, fill=black] at (-0.5,3/4){};
       \node[draw,circle, inner sep=1.5pt,color=black, fill=black] at (0.5,3/4){};
        \node[draw,circle, inner sep=1.5pt,color=violet, fill=violet] at (-3,1.5){};
         \node[draw,circle, inner sep=1.5pt,color=violet, fill=violet] at (-1.5,4){};
         \node[draw,circle, inner sep=1.5pt,color=black, fill=black] at (-3/4,11/4){};
         \node[draw,circle, inner sep=1.5pt,color=black, fill=black] at (-2.5,2.34){};
         \node[draw,circle, inner sep=1.5pt,color=black, fill=black] at (-2,3.17){};

   \node [below, color=black] at (-1,0) {$L_{1}$};
\node [below, color=black] at (1,0) {$L$};
\node [below, color=black] at (-3,1.5) {$L_{2}$};
\node [right, color=black] at (-1.5,5) {$A_1$};

\node [left, color=black] at (-1,0) {\small{$(\red{1},\blue{0})\to \violet{\infty}$}};  
\node [right, color=black] at (1,0) {\small{$(\red{1},\blue{1}) \to \violet{0}$}};  
 \node [right, color=black] at (0,1.5) {\small{$(\red{5},\blue{2}) \to \violet{\frac{3}{2}}$}};   
     \node [left, color=black] at (-3,1.5) {\small{$(\red{1},\blue{0}) \to \violet{\infty}$}}; 
     \node [left, color=black] at (-1.6,4.2) {\small{$(\red{19},\blue{6})\to \violet{\frac{13}{6}}$}}; 
      
   \end{scope}

   \begin{scope}[shift={(-15,-6)}]
   
  \draw[->][thick, color=black](-4,1) .. controls (-3,0.5) ..(-2,1); 
   
     \draw [->, color=blue,  line width=2.5pt](0,0) -- (0,4.5);
     \draw [->, color=blue,  line width=2.5pt](0,1) -- (0,0);

      \draw [->, color=blue,  line width=2.5pt](0,1.5) -- (-1.4,1.5);
       \draw [->, color=blue,  line width=2.5pt](0,3) -- (-1.4,3);

       \node [below, color=black] at (0,0) {$L$};
         \node [left, color=black] at (-1.4,1.5) {$L_{1}$};
           \node [above, color=black] at (0,4.5) {$A_1$};
           \node [left, color=black] at (-1.4,3) {$L_{2}$};
       
            \node[draw,circle, inner sep=1.5pt,color=violet, fill=violet] at (0,1.5){};
      \node[draw,circle, inner sep=1.5pt,color=violet, fill=violet] at (0,3){};

                           \node [right, color=blue] at (0,0.75) {\small{$1$}};
                 \node [right, color=blue] at (0,2.25) {\small{$2$}};
                  \node [right, color=blue] at (0,3.75) {\small{$6$}};
                   \node [above, color=blue] at (-0.7,1.5) {\small{$1$}};
                   \node [above, color=blue] at (-0.7,3) {\small{$2$}};
                    \node [right, color=violet] at (0,0) {\small{$0$}};  
           \node [right, color=violet] at (0,1.5) {\small{$\frac{3}{2}$}}; 
            \node [right, color=violet] at (0,3) {\small{$\frac{13}{6}$}}; 
             \node [below, color=violet] at (-1.3,2.9) {\small{$\infty$}};
                \node [below, color=violet] at (-1.3,1.4) {\small{$\infty$}};

      \end{scope}
         
   \end{tikzpicture}
\end{center}
 \caption{Going from $\lambda(E)$ and $\ord_E(L)$ to  $\Theta_L(\hat{A})$ when $\hat{A} =L + L_1 + L_2 + A_1$}
\label{fig:logdistoEW2}
   \end{figure}

     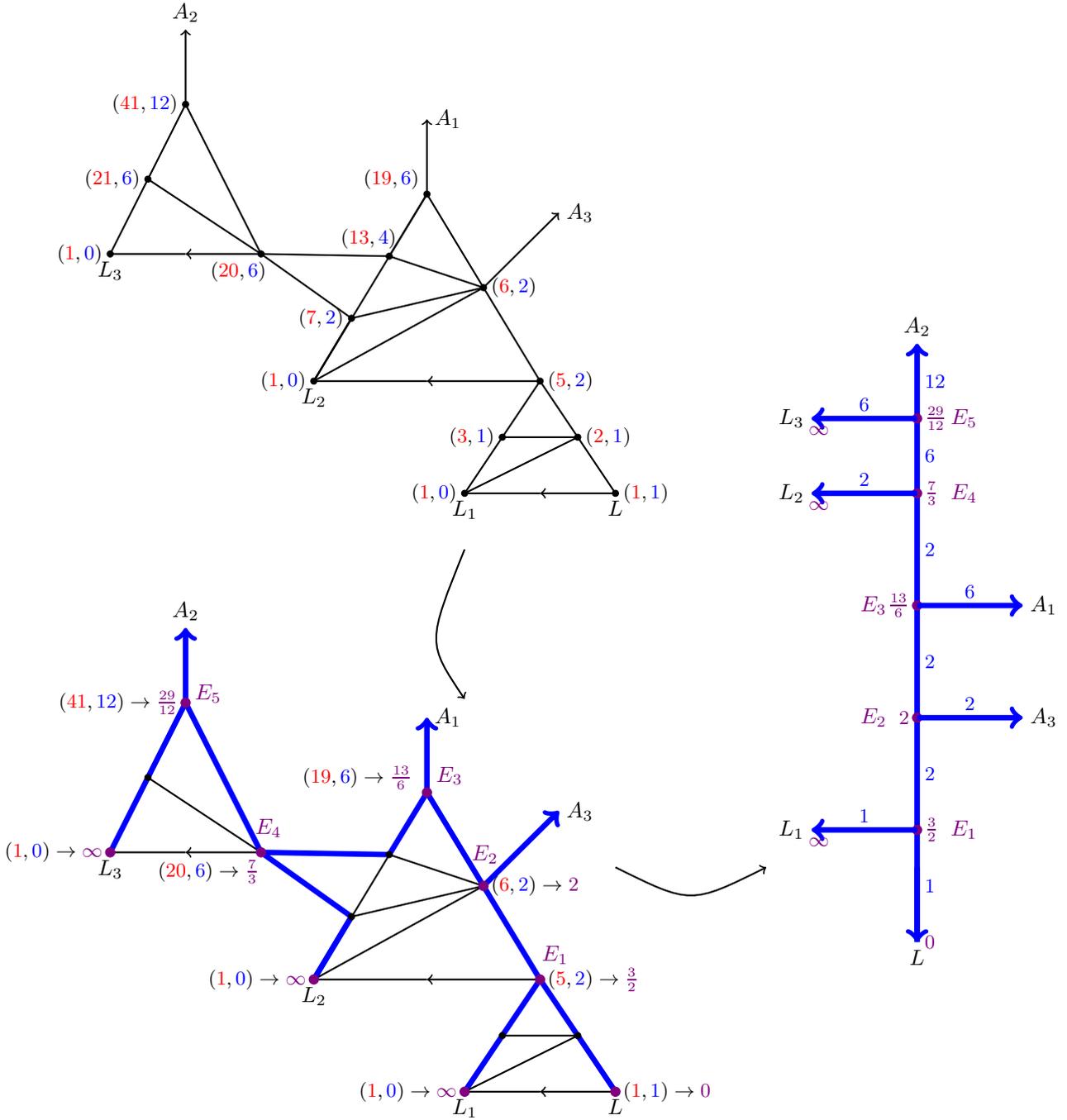
\begin{figure}[h!]
    \begin{center}
\begin{tikzpicture}[scale=1.2]
\begin{scope}[shift={(-20,0)}]
   
    \draw [->, color=black, thick](1,0) -- (0,0);
    \draw [-, color=black, thick](0,0) -- (-1,0);
    \draw [-, color=black, thick](0,1.5) -- (-1,0);
      \draw [-, color=black, thick](0,1.5) -- (1,0);
       \draw [-, color=black, thick](-0.5,3/4) -- (0.5,3/4);
       \draw [-, color=black, thick](-1,0) -- (0.5,3/4);
       \draw [-, color=black, thick](-5.2,4.2) -- (-3.7,3.2);

        \draw [->, color=black, thick](0,1.5) -- (-1.5,1.5);
   \draw [-, color=black, thick](-1.5,1.5) -- (-3,1.5);
     \draw [-, color=black, thick](0,1.5) -- (-1.5,4);
      \draw [-, color=black, thick](-1.5,4) -- (-3,1.5);
      \draw [->, color=black, thick](-1.5,4) -- (-1.5,5);
     \draw [-, color=black, thick](-3/4,11/4) -- (-3,1.5);
     \draw [-, color=black, thick](-3/4,11/4) -- (-2.5,2.34);
       \draw [-, color=black, thick](-3/4,11/4) -- (-2,3.17);
       \draw [->, color=black, thick](-3/4,11/4) -- (1/4,15/4);
        \draw [-, color=black, thick](-2,3.17) -- (-1.5,4);
        \draw [-, color=black, thick](-3,1.5) -- (-2.5,2.34);
        \draw [-, color=black, thick](-3,1.5) -- (-2.5,2.34);
        \draw [-, color=black, thick](-3.7,3.2) -- (-2,3.17);
        \draw [-, color=black, thick](-3.7,3.2) -- (-2.5,2.34);
        \draw [->, color=black, thick](-3.7,3.2) -- (-4.7,3.2);
        \draw [-, color=black, thick](-5.7,3.2) -- (-4.7,3.2);
           \draw [->, color=black, thick](-4.7,5.2) -- (-4.7,6.2);
            \draw [-, color=black, thick](-3.7,3.2) -- (-4.7,5.2);
            \draw [-, color=black, thick](-5.7,3.2) -- (-4.7,5.2);
       
  \node[draw,circle, inner sep=1pt,color=black, fill=black] at (-1,0){};
   \node[draw,circle, inner sep=1pt,color=black, fill=black] at (1,0){};
     \node[draw,circle, inner sep=1pt,color=black, fill=black] at (0,1.5){};
      \node[draw,circle, inner sep=1pt,color=black, fill=black] at (-0.5,3/4){};
       \node[draw,circle, inner sep=1pt,color=black, fill=black] at (0.5,3/4){};
        \node[draw,circle, inner sep=1pt,color=black, fill=black] at (-3,1.5){};
         \node[draw,circle, inner sep=1pt,color=black, fill=black] at (-1.5,4){};
         \node[draw,circle, inner sep=1pt,color=black, fill=black] at (-3/4,11/4){};
         \node[draw,circle, inner sep=1pt,color=black, fill=black] at (-2.5,2.34){};
         \node[draw,circle, inner sep=1pt,color=black, fill=black] at (-2,3.17){};
          \node[draw,circle, inner sep=1pt,color=black, fill=black] at (-3.7,3.2){};
          \node[draw,circle, inner sep=1pt,color=black, fill=black] at (-5.7,3.2){};
           \node[draw,circle, inner sep=1pt,color=black, fill=black] at (-4.7,5.2){};
           \node[draw,circle, inner sep=1pt,color=black, fill=black] at (-5.2,4.2){};
       
   \node [below, color=black] at (-1,0) {$L_{1}$};
\node [below, color=black] at (1,0) {$L$};
\node [below, color=black] at (-3,1.5) {$L_{2}$};
\node [right, color=black] at (1/4,15/4) {$A_3$};
\node [right, color=black] at (-1.5,5) {$A_1$};
 \node [below, color=black] at (-5.7,3.2) {$L_{3}$};
 \node [above, color=black] at (-4.7,6.2) {$A_2$};

\node [left, color=black] at (-1,0) {\small{$(\red{1},\blue{0})$}};  
\node [right, color=black] at (1,0) {\small{$(\red{1},\blue{1})$}};  
 \node [right, color=black] at (0,1.5) {\small{$(\red{5},\blue{2}) $}};  
 \node [left, color=black] at (-0.5,3/4) {\small{$(\red{3},\blue{1})$}}; 
  \node [right, color=black] at (0.5,3/4) {\small{$(\red{2},\blue{1})$}}; 
     \node [left, color=black] at (-3,1.5) {\small{$(\red{1},\blue{0})$}}; 
     \node [left, color=black] at (-1.5,4.2) {\small{$(\red{19},\blue{6})$}}; 
    \node [right, color=black] at (-3/4,11/4) {\small{$(\red{6},\blue{2})$}}; 
    \node [below, color=black] at (-4,3.2) {\small{$(\red{20},\blue{6})$}}; 
   \node [left, color=black] at (-2.5,2.34) {\small{$(\red{7},\blue{2})$}}; 
   \node [left, color=black] at (-1.8,3.4) {\small{$(\red{13},\blue{4})$}}; 
   \node [left, color=black] at (-5.7,3.2) {\small{$(\red{1},\blue{0})$}};  
   \node [left, color=black] at (-4.7,5.2) {\small{$(\red{41},\blue{12})$}}; 
   \node [left, color=black] at (-5.2,4.2) {\small{$(\red{21},\blue{6})$}};  
   
    \draw[->][thick, color=black](-1,-0.75) .. controls (-1.5,-2) ..(-1,-2.75); 
   
   \end{scope}

   \begin{scope}[shift={(-20,-8)}]
   
    \draw [->, color=black, thick](1,0) -- (0,0);
    \draw [-, color=black, thick](0,0) -- (-1,0);
    \draw [-, color=blue, line width=2.5pt](0,1.5) -- (-1,0);
      \draw [-, color=blue, line width=2.5pt](0,1.5) -- (1,0);
       \draw [-, color=black, thick](-0.5,3/4) -- (0.5,3/4);
       \draw [-, color=black, thick](-1,0) -- (0.5,3/4);
       \draw [-, color=black, thick](-5.2,4.2) -- (-3.7,3.2);

        \draw [->, color=black, thick](0,1.5) -- (-1.5,1.5);
   \draw [-, color=black, thick](-1.5,1.5) -- (-3,1.5);
     \draw [-, color=blue, line width=2.5pt](0,1.5) -- (-1.5,4);
      \draw [-, color=black, thick](-1.5,4) -- (-3,1.5);
      \draw [->, color=blue, line width=2.5pt](-1.5,4) -- (-1.5,5);
     \draw [-, color=black, thick](-3/4,11/4) -- (-3,1.5);
     \draw [-, color=black, thick](-3/4,11/4) -- (-2.5,2.34);
       \draw [-, color=black, thick](-3/4,11/4) -- (-2,3.17);
       \draw [->, color=blue, line width=2.5pt](-3/4,11/4) -- (1/4,15/4);
        \draw [-, color=blue, line width=2.5pt](-2,3.17) -- (-1.5,4);
        \draw [-, color=blue, line width=2.5pt](-3,1.5) -- (-2.5,2.34);
        \draw [-, color=blue, line width=2.5pt](-3,1.5) -- (-2.5,2.34);
        \draw [-, color=blue, line width=2.5pt](-3.7,3.2) -- (-2,3.17);
        \draw [-, color=blue, line width=2.5pt](-3.7,3.2) -- (-2.5,2.34);
        \draw [->, color=black, thick](-3.7,3.2) -- (-4.7,3.2);
        \draw [-, color=black, thick](-5.7,3.2) -- (-4.7,3.2);
           \draw [->, color=blue, line width=2.5pt](-4.7,5.2) -- (-4.7,6.2);
            \draw [-, color=blue, line width=2.5pt](-3.7,3.2) -- (-4.7,5.2);
            \draw [-, color=blue, line width=2.5pt](-5.7,3.2) -- (-4.7,5.2);
       
  \node[draw,circle, inner sep=1.5pt,color=violet, fill=violet] at (-1,0){};
   \node[draw,circle, inner sep=1.5pt,color=violet, fill=violet] at (1,0){};
     \node[draw,circle, inner sep=1.5pt,color=violet, fill=violet] at (0,1.5){};
      \node[draw,circle, inner sep=1pt,color=black, fill=black] at (-0.5,3/4){};
       \node[draw,circle, inner sep=1pt,color=black, fill=black] at (0.5,3/4){};
        \node[draw,circle, inner sep=1.5pt,color=violet, fill=violet] at (-3,1.5){};
         \node[draw,circle, inner sep=1.5pt,color=violet, fill=violet] at (-1.5,4){};
         \node[draw,circle, inner sep=1.5pt,color=violet, fill=violet] at (-3/4,11/4){};
         \node[draw,circle, inner sep=1pt,color=black, fill=black] at (-2.5,2.34){};
         \node[draw,circle, inner sep=1pt,color=black, fill=black] at (-2,3.17){};
          \node[draw,circle, inner sep=1.5pt,color=violet, fill=violet] at (-3.7,3.2){};
          \node[draw,circle, inner sep=1.5pt,color=violet, fill=violet] at (-5.7,3.2){};
           \node[draw,circle, inner sep=1.5pt,color=violet, fill=violet] at (-4.7,5.2){};
           \node[draw,circle, inner sep=1pt,color=black, fill=black] at (-5.2,4.2){};
       
   \node [below, color=black] at (-1,0) {$L_{1}$};
\node [below, color=black] at (1,0) {$L$};
\node [below, color=black] at (-3,1.5) {$L_{2}$};
\node [right, color=black] at (1/4,15/4) {$A_3$};
\node [right, color=black] at (-1.5,5) {$A_1$};
 \node [below, color=black] at (-5.7,3.2) {$L_{3}$};
 \node [above, color=black] at (-4.7,6.2) {$A_2$};

\node [left, color=black] at (-1,0) {\small{$(\red{1},\blue{0})\to \violet{\infty}$}};  
\node [right, color=black] at (1,0) {\small{$(\red{1},\blue{1}) \to \violet{0}$}};  
 \node [right, color=black] at (0,1.5) {\small{$(\red{5},\blue{2}) \to \violet{\frac{3}{2}}$}};  
 
 \node [above, color=violet] at (0.2,1.6) {$E_1$};

     \node [left, color=black] at (-3,1.5) {\small{$(\red{1},\blue{0}) \to \violet{\infty}$}}; 
     \node [left, color=black] at (-1.6,4.2) {\small{$(\red{19},\blue{6})\to \violet{\frac{13}{6}}$}}; 
     
      \node [above, color=violet] at (-1.2,4)  {$E_3$};
     
    \node [right, color=black] at (-3/4,11/4) {\small{$(\red{6},\blue{2})\to\violet{2}$}}; 
    
 \node [above, color=violet] at (-0.73,2.95)  {$E_2$};

    \node [below, color=black] at (-4.4,3.2) {\small{$(\red{20},\blue{6})\to \violet{\frac{7}{3}}$}}; 
    
     \node [above, color=violet] at (-3.6,3.3)  {$E_4$};

   \node [left, color=black] at (-5.7,3.2) {\small{$(\red{1},\blue{0}) \to \violet{\infty}$}};  
   \node [left, color=black] at (-4.7,5.2) {\small{$(\red{41},\blue{12})\to \violet{\frac{29}{12}}$}}; 
   
   \node [above, color=violet] at (-4.4,5.1)  {$E_5$};

   \end{scope}

   \begin{scope}[shift={(-15,-6)}]
   
  \draw[->][thick, color=black](-4,1) .. controls (-3,0.5) ..(-2,1); 
   
     \draw [->, color=blue,  line width=2.5pt](0,0) -- (0,8);
     \draw [->, color=blue,  line width=2.5pt](0,1) -- (0,0);
      \node[draw,circle, inner sep=1.5pt,color=violet, fill=violet] at (0,1.5){};
      \node[draw,circle, inner sep=1.5pt,color=violet, fill=violet] at (0,3){};
      \node[draw,circle, inner sep=1.5pt,color=violet, fill=violet] at (0,4.5){};
      \node[draw,circle, inner sep=1.5pt,color=violet, fill=violet] at (0,6){};
      \node[draw,circle, inner sep=1.5pt,color=violet, fill=violet] at (0,7){};
      
      \draw [->, color=blue,  line width=2.5pt](0,1.5) -- (-1.4,1.5);
       \draw [->, color=blue,  line width=2.5pt](0,3) -- (1.4,3);
       \draw [->, color=blue,  line width=2.5pt](0,4.5) -- (1.4,4.5);
       \draw [->, color=blue,  line width=2.5pt](0,6) -- (-1.4,6);
       \draw [->, color=blue,  line width=2.5pt](0,7) -- (-1.4,7);
      
       \node [below, color=black] at (0,0) {$L$};
         \node [left, color=black] at (-1.4,1.5) {$L_{1}$};
          \node [right, color=black] at (1.4,3) {$A_3$};
           \node [right, color=black] at (1.4,4.5) {$A_1$};
           \node [left, color=black] at (-1.4,6) {$L_{2}$};
           \node [left, color=black] at (-1.4,7) {$L_{3}$};
           \node [above, color=black] at (0,8) {$A_2$};
             \node [left, color=violet] at (0.9,1.5) {$E_{1}$};
             \node [left, color=violet] at (-0.3,3) {$E_{2}$};
              \node [left, color=violet] at (-0.3,4.5) {$E_{3}$};
             \node [left, color=violet] at (0.9,6) {$E_{4}$};
               \node [left, color=violet] at (0.9,7) {$E_{5}$};

           \node [right, color=violet] at (0,0) {\small{$0$}};  
           \node [right, color=violet] at (0,1.5) {\small{$\frac{3}{2}$}}; 
            \node [left, color=violet] at (0,3) {\small{$2$}};     
             \node [left, color=violet] at (0,4.5) {\small{$\frac{13}{6}$}};        
              \node [right, color=violet] at (0,6) {\small{$\frac{7}{3}$}};
               \node [right, color=violet] at (0,7) {\small{$\frac{29}{12}$}};

                \node [below, color=violet] at (-1.3,1.5) {\small{$\infty$}};
                \node [below, color=violet] at (-1.3,6) {\small{$\infty$}};
                \node [below, color=violet] at (-1.3,7) {\small{$\infty$}};
                
                 \node [right, color=blue] at (0,0.75) {\small{$1$}};
                 \node [right, color=blue] at (0,2.25) {\small{$2$}};
                  \node [right, color=blue] at (0,3.75) {\small{$2$}};
                  \node [right, color=blue] at (0,5.25) {\small{$2$}};
                  \node [right, color=blue] at (0,6.5) {\small{$6$}};
                   \node [right, color=blue] at (0,7.5) {\small{$12$}};
                   \node [above, color=blue] at (-0.7,1.5) {\small{$1$}};
                \node [above, color=blue] at (-0.7,6) {\small{$2$}};
                \node [above, color=blue] at (-0.7,7) {\small{$6$}};
                \node [above, color=blue] at (0.7,3) {\small{$2$}};
                \node [above, color=blue] at (0.7,4.5) {\small{$6$}};    
                
      \end{scope}
         
   \end{tikzpicture}
\end{center}
 \caption{Going from $\lambda(E)$ and $\ord_E(L)$ to  $\Theta_L(\hat{A})$ when $\hat{A} =L + L_1 + L_2 + L_3 + A_1 + A_2 + A_3$}
\label{fig:logdistoEW3}
   \end{figure}

The Eggers-Wall tree of the plane curve singularity whose branches are defined by the Newton-Puiseux series of Example \ref{ex:EWtoseries} below is as shown on the right of Figure \ref{fig:logdistoEW3}. One may find in \cite[Section 1.6.1]{GBGPPP 20} other examples of Eggers-Wall trees constructed from Newton-Puiseux series associated to the branches of plane curve singularities.

\begin{remark} 
   \label{rem:genericEW}
  The set  of characteristic exponents ${\mathrm{Ch}}_L (A_l)$, the order of coincidence  $k_L (A_l, A_m)$, the Eggers-Wall tree $\Theta_{L}(A)$ and the functions $\ex_L$ and $\de_L$, do not depend on the choices of local coordinate $y$ and of the defining function $x$ of the branch $L$  (see  \cite[Proposition 26]{GBGPPP 18}).  When $L$ is transversal to $A$, $\Theta_{L}(A)$ becomes independent of $L$. We call it the {\bf generic Eggers-Wall tree of $A$}. We explained in \cite[Section 4]{GBGPPP 18} how to get from it all other possible Eggers-Wall trees, relative to branches $L$ which are tangent to $A$. 
\end{remark}

\begin{remark} \label{rem-conjugates} 
   If $A$ is a branch different from $L$, then the intersection number $(L\cdot A)_O$ (see Definition \ref{def:intnumber}) is equal to the maximum achieved by the index function $\de_{L}$ on the segment $[L A] \subset \Theta_L(A)$. In particular, $\de_{L} (A) = (L\cdot A)_O$ (see \cite[Remark 3.25]{GBGPPP 19}). More generally, the value of the index function $\de_L (P)$, for $P \in (LA]$, is characterized in the following way. Take a Newton-Puiseux series $ \eta = \sum c_\alpha x^\alpha \in \cZ_x(A)$. If $\ex_L (P) = \gamma$, denote by $A_{\gamma}$ the unique branch which has the truncated Newton-Puiseux series $\xi_{\gamma}  :=   \sum_{\alpha < \gamma} c_\alpha x^\alpha$ as a root relative to the local coordinates $(x,y)$, that is, such that $\xi_{\gamma}  \in \cZ_x(A_{\gamma})$. Then, $\de_L (P)$ is equal both to the number of Galois conjugates of $\xi_{\gamma}$ and to the intersection number $(L\cdot A_\gamma)_O$. 
\end{remark}

We explain now the relations between Eggers-Wall trees and lotuses by using the notion of {\em semivaluation}, in the sense of Definition \ref{def:semival}. 

The  {\bf semivaluation space} $\boxed{\calv}$ of $\calo_{S, O}$ is the set of semivaluations of the local ring $\calo_{S, O}$, endowed with the topology of pointwise convergence. We denote by  $\boxed{\calv_L}$ the subset of $\calv$ consisting of those semivaluations which take the value $1$ on $L$. Favre and Jonsson have studied the properties of these semivaluation  spaces in their book \cite{FJ 04}. They proved that the semivaluation space $\calv_L$ is a compact $\Rr$-tree rooted at $\ord_L$ (see \cite[Prop. 3.61]{FJ 04}). 
In \cite{GBGPPP 19} we defined a natural embedding 
    \[ \boxed{V_L} : \Theta_{L} (A) \hookrightarrow \calv \]
of the Eggers-Wall tree in the semivaluation space $\calv_L$ and we proved the following result (see \cite[Theorem 8.11 and the proof of Proposition 8.16]{GBGPPP 19}, recalling  that the intersection semivaluations $I_{A_l}$ were introduced in Definition \ref{def:intsv}):

\begin{proposition} 
   \label{prop:ew-emb} 
  The embedding $V_L : \Theta_{L} (A) \to \calv$ is such that $V_L (L) = \ord_L$ and $V_L(A_l) = \frac{1}{(L\cdot A_l)_O} I_{A_l}$, for every branch $A_l$ of $A$. If $P \in \Theta_L (A)$ is a rational point, then there exists a unique prime exceptional divisor $E(P)$ on a model of $(S,O)$ such that $V_L (P) = \frac{1}{\ord_{E(P)} (L)} \ord_{E(P)}$ and 
    \begin{equation}
       \label{eq: rat-point}
       \frac{\lambda( E(P) )}{\ord_{E(P)}(L)} = 1 + \ex_L (P).
      \end{equation}
   In addition, assume that  $A_0$ is a branch of $A$ and $P \in \Theta_{L} (A_0)$ is a point such that $\de_L (P) = \de_L (A_0)$ and that $A_l$ is a branch of $A$ whose strict transform is a curvetta at a point of $E(P)$ on a model of $(S,O)$, in the sense of Definition \ref{def:curvetta}. Then, the restriction of $\de_L$ to the segment $(P A_l]$ of $\Theta_{L} (A)$ has constant value $(L\cdot A_l)_O = \ord_{E(P)} (L)$. 
\end{proposition}

Next, we consider a reduced plane curve singularity $A \hookrightarrow (S,O)$ together with a finite active constellation of crosses $\hat{\calc}_A$ adapted to $A$ in the sense of Definition \ref{def:adaptedactconst}. The following result explains how the lotus $\Lambda(\hat{\calc}_A)$ determines the Eggers-Wall tree of the associated completion $\hat{A}$ of $A$, in the sense of Definition \ref{def:adaptedactconst}:

\begin{proposition}  
    \label{prop:fromlotustoEW}
Let $\hat{\calc}_A$ be a finite active constellation of crosses adapted to a reduced plane curve singularity $A$. Let $X_O = L + L_1$ be the cross of $\hat{\calc}_A$ at $(S,O)$. Then: 

\begin{enumerate}[(a)]
\item \label{a-ewlotus}
   The lateral boundary  $\partial_+ \Lambda(\hat{\calc}_A)$ of the lotus $\Lambda(\hat{\calc}_A)$ is isomorphic to the Eggers-Wall tree $\Theta_{L} (\hat{A})$ by an isomorphism which preserves the labels of vertices by the branches of $\hat{A}$. 

\item \label{b-ewlotus}
   This isomorphism induces a bijection between the  marked points in the interior of the tree
   $\Theta_{L} (\hat{A})$ and the rupture vertices of the lotus $\Lambda (\hat{\calc}_A)$, which sends a point $P$ to the vertex labeled by $E(P)$ in the lotus $\Lambda (\hat{\calc}_A)$ in such a way that:
    \[
      \ex_{L} (P) =  \frac{\lambda (E(P))}{\ord_{E(P)} (L)} -1.
     \]
\item  \label{c-ewlotus} 
   If $A_l$ is a branch of $\hat{A}$ different from $L$ and $L_1$, then there is a unique ramification point $P$ of $\Theta_{L} (\hat{A})$ such that $E(P)+ A_l$ is a cross of $\hat{\calc}_A$ and 
   the index function $\de_{L}$ is constantly equal to $(L \cdot A_l)_O = \ord_{E(P)} (L)$ on the segment $(P A_l]$. 
\end{enumerate}
\end{proposition}

\begin{proof}
    The statement \ref{a-ewlotus} follows by combining Theorems 1.5.29 and Theorems 1.6.27 of \cite{GBGPPP 18}. Then, \ref{b-ewlotus} and \ref{c-ewlotus} are direct consequences of Proposition \ref{prop:ew-emb}. We illustrate this fact in Example \ref{ex:lotusEW3branches} below.
\end{proof}

\begin{remark}
    The Eggers-Wall tree $\Theta_{L} (\hat{A})$ has the special property that all its marked points are either ramification points or end points. Moreover, combining Propositions \ref{prop:dualgraph} and \ref{prop:fromlotustoEW}, we see that the dual graph  $G(X_{\hat{\calc}_A})$ is naturally homeomorphic to the Eggers-Wall tree $\Theta_{L} (\hat{A})$. Indeed, they are both identified with the lateral boundary $\partial_+ \Lambda(\hat{\calc}_A)$ of the lotus $\Lambda(\hat{\calc}_A)$.
\end{remark}

In the following Corollary \ref{cor:lotustoEW} of 
Proposition \ref{prop:fromlotustoEW}, we explain how to build an Eggers-Wall tree from a lotus:

\medskip
\begin{center}
\fbox{
\begin{minipage}{0.75\textwidth}
    \begin{corollary}
       \label{cor:lotustoEW}
    Let $\Lambda$ be the lotus of a finite active constellation of crosses $\hat{\calc}_A$, with initial vertex $L$. The Eggers-Wall tree $\Theta_L(\hat{A})$ of the associated completion $\hat{A}$ of $A$ may be computed as follows:
  \begin{enumerate}
      \item 
         Compute the pairs $(\lambda(E), \ord_E(L))$ associated to its vertices $E$ as explained in Corollary \ref{cor:reclord}. We will be only interested in those values for rupture vertices $E$.       
      \item 
         Define $\Theta_L(\Lambda) := \partial_+(\Lambda)$. Let $\preceq_L$ be the partial order on it defined by the root $L$.  
       \item 
          For each membrane $\Lambda_i$ of $\Lambda$ in the sense of Definition \ref{def:vocablotus}, with initial vertex $E_i$, denote by $[E_i L_i]$ the intersection $\Theta_L(\Lambda) \cap \Lambda_i$. It verifies $E_i \prec_L L_i$.
      \item 
          To each rupture vertex $E$ of $\Lambda$, associate the exponent:
            \[
               \ex_{L} (E) :=  \frac{\lambda (E)}{\ord_{E} (L)} -1.
             \]      
      \item 
          Define $\de_{L}$ on the interval $(E_i L_i]$ to be constantly equal to $\ord_{E_i} (L)$.       
  \end{enumerate}
\end{corollary}
\end{minipage}
}
\end{center}
\medskip

We see that $\Theta_L(\hat{A})$ depends only on the lotus $\Lambda$. This allows us to define:

\begin{definition}
   \label{def:assEWtolotus}
    Let $\Lambda$ be the lotus of a finite active constellation of crosses $\hat{\calc}_A$, with initial vertex $L$. Its {\bf associated Eggers-Wall tree} $(\boxed{\Theta_L(\Lambda)}, \ex_L, \de_L)$ is the subtree of the $1$-skeleton of $\Lambda$ constructed and decorated as explained in Corollary \ref{cor:lotustoEW}.
\end{definition}

\begin{example} 
    \label{ex:lotusEW3branches}
  In Figures \ref{fig:logdistoEW1}, \ref{fig:logdistoEW2} and \ref{fig:logdistoEW3} we illustrate how to recover the Eggers-Wall trees of the completions $\hat{A}$ from the knowledge of the pairs of invariants $(\lambda(E), \ord_E (L))$, shown in Figures \ref{fig:logdisord1}, \ref{fig:logdisord2} and \ref{fig:logdisord3}. 

  Let us explain in detail the computations illustrated in Figure \ref{fig:logdistoEW3}, which allow to recover the Eggers-Wall tree $\Theta_L (\hat{A})$ 
in the example of the lotus $\Lambda(\hat{\calc}_A)$, its vertices $E$ being decorated with the pairs $(\lambda(E), \ord_{E} (L))$, as shown in Figure \ref{fig:logdisord3}. This lotus is associated to a finite active constellation $\hat{\calc}_A$ of crosses compatible 
with a curve $A = A_1 + A_2 + A_3$ 
with three branches. 
Notice that in this case the completion $\hat{A}$ is equal to $L + L_1+ L_2 + L_3 + A_1+ A_2 + A_3$. 
    
By Proposition \ref{prop:fromlotustoEW} \ref{a-ewlotus}, the 
Eggers-Wall tree $\Theta_L (\hat{A})$ is homeomorphic  as a topological space with the lateral boundary $\partial_+ \Lambda (\hat{\calc}_A)$, represented in blue bold line in Figure  \ref{fig:dualgr3branches}, the end points being labeled by the branches of $\hat{A}$. 
For simplicity,   we label the ramification points of $\Theta_L( \hat{A})$
with the names of the exceptional prime divisors of the corresponding rupture vertices in the lotus. 
We  reconstruct the values of the index function and the exponent function 
in a recursive way. 

Since  $L+L_1$ is a cross  at $(S,O)$,  
the index function $\de_{L}$ has constant value $1$ on the segment $[L  L_1]$ of $\Theta_{L} (\hat{A})$. 
There is only one rupture vertex $E_1$ on the blue bold path of the lotus going from $L$ to $L_1$, and  $E_1+L_2$ is a cross of $\hat{\calc}_A$. By Proposition \ref{prop:fromlotustoEW} \ref{c-ewlotus}, $E_1$ is the ramification  point of the tree $\Theta_L( \hat{A})$ which lies in the segment $[L L_1]$. 
By Proposition \ref{prop:fromlotustoEW} \ref{b-ewlotus}, the value of the exponent function at this point  is determined by the log-discrepancy of $E_1$ and the vanishing order $\ord_{E_1}(L)$, namely we get $\ex_L (E_1) = \frac{\lambda(E_1)}{\ord_{E_1} (L)} -1= \frac{5}{2} -1 =\frac{3}{2}$, and  the index function $\de_L$ 
has constant value $\de_L (L_2) = \ord_{E_1} (L)=2$ on $(E_1 L_2]$. 

There are three rupture vertices $E_2$, $E_3$, and $E_4$ 
of the lotus,  which lie on the blue bold path of the lotus going from $E_1$ to $L_2$,  where $(\lambda(E_2), \ord_{E_2} (L)) = (6,2)$, $(\lambda(E_3), \ord_{E_3} (L)) = (19,6)$, and $(\lambda(E_4), \ord_{E_4} (L)) = (20,6)$. These three vertices are the ramification points of $\Theta_L ( \hat{A})$ on the segment $[E_1 L_2]$. By Proposition \ref{prop:fromlotustoEW} \ref{b-ewlotus}, we get 
$\ex_L (E_2) =  \frac{6}{2} -1 =2$, $\ex_L (E_3) = \frac{19}{6} -1  = \frac{13}{6}$ and $\ex_L (E_4) = \frac{20}{6} -1 = \frac{7}{3}$.
These exceptional divisors are components of the crosses  
$E_2 + A_3$, $E_3 + A_1$ and $E_4+ L_3$ of $\hat{\calc}_A$. 
The  index function $\de_{L}$ has constant value  $\ord_{E_2}(L) = 2$ in restriction to $(E_2 A_3]$, while in restriction to $(E_3 A_1)$ or to $(E_4 L_3]$ it has constant value $\ord_{E_3} (L)  = \ord_{E_4} (L) = 6$. 
There is only one rupture point $E_5$ on the blue bold path of the lotus going from $E_4$ to $L_3$, where $(\lambda(E_5), \ord_{E_5} (L)) = (41,12)$, and $E_5 + A_2$ is a cross of $\hat{\calc}_A$. Hence  $E_5$ is  the ramification point  of $\Theta_L (\hat{A})$ in the segment $[E_4  L_3]$,   $\ex_L (E_5) = \frac{41}{12} -1 = \frac{29}{12}$ and the index function is equal to $\ord_{E_5} (L) = 12$ on $(E_5 A_2]$.

The Eggers-Wall tree $\Theta_L (A)$ is the union of the segments $[L A_l]$, for $l =1, 2, 3$, endowed with the restrictions of the exponent and index functions of $\Theta_L (\hat{A})$.

We recover the characteristic exponents of a branch $A_l$ of $\hat{A}$ from the Eggers-Wall tree, by using the fact that they are the values of the exponent function at the points of discontinuity of the restriction of the index function to the segment $[L A_l]$. In the case of the branch $A_1$, the points of discontinuity of the index function on $[L A_1]$ are $E_1 $  with $\ex_L (E_1) = \frac{3}{2}$ and $E_3$ with $\ex_L (E_3) = \frac{13}{6}$.  Hence 
${\mathrm{Ch}}_L (A_1) = \{ \frac{3}{2} , \frac{13}{6} \}$. Similarly, we get that ${\mathrm{Ch}}_L (A_2) = \{ \frac{3}{2} , \frac{7}{3}, \frac{29}{12} \}$ and ${\mathrm{Ch}}_L (A_3) = \{ \frac{3}{2} \}$. 
\end{example}

\medskip
  Once one has an Eggers-Wall tree, it is an easy task to write down Newton-Puiseux series defining branches whose sum has the given tree as Eggers-Wall tree relative to $L := Z(x)$:

\begin{example}
   \label{ex:EWtoseries}
    Let us consider the Eggers-Wall tree of Figure \ref{fig:logdistoEW3}. One obtains it by choosing, for instance, branches with the following Newton-Puiseux roots: 
    \[ \begin{array}{lll}
         L_1 : 0, &  
         L_2 :  x^{\frac{3}{2}},  &  
         L_3  : x^{\frac{3}{2}} + x^{\frac{7}{3}},  \\
         A_1 : x^{\frac{3}{2}} + x^{\frac{13}{6}},  &
         A_2 : x^{\frac{3}{2}} + x^{\frac{7}{3}}  + x^{\frac{29}{12}},  & 
         A_3 : x^{\frac{3}{2}} +x^2.
      \end{array}  \]
    This allows to obtain a concrete series which admits an active constellation of crosses whose lotus is shown in Figure \ref{fig:lotus3branches} (and again in the upper part of Figure \ref{fig:logdistoEW3}). Indeed, it is enough to multiply the minimal polynomials of the above Newton-Puiseux series associated to $A_1, A_2$ and $A_3$. In particular, keeping only the minimal polynomial of the Newton-Puiseux series $x^{\frac{3}{2}} + x^{\frac{13}{6}}$, which is equal to 
      \[ (y^2 - x^3)^3 - 6 x^8 (y^2 - x^3) - 8 x^{11} - x^{13},  \] 
    one gets a branch $A_1$ which admits an active constellation of crosses whose lotus is shown in Figure \ref{fig:lotusbranch} (and again in the upper part of Figure \ref{fig:logdistoEW2}). 
\end{example}

\begin{remark}
    We can apply Proposition \ref{prop:fromlotustoEW} interchanging the roles of $L$ and $L_1$ in order to obtain the Eggers-Wall tree $\Theta_{L_1}(\hat{A})$ from $\hat{\calc}_A$ (see Remark \ref{rem:genericEW}).
\end{remark}

\begin{remark}
  \label{rem:splicediagfromEW}
     As explained in \cite[Section 5]{GBGPPP 18}, the knowledge of an Eggers-Wall tree of a plane curve singularity allows also to compute an associated {\em splice diagram}.
\end{remark}

\medskip
\section{Newton lotuses and continued fractions}  \label{sec:Nlot}

In this section we explain following \cite[Section 1.5.2]{GBGPPP 20}  how to associate a {\em Newton lotus} to every finite set of positive rational numbers (see Definition \ref{def:deflot}). It is a sublotus of a {\em universal lotus}, canonically associated to any basis of a two-dimensional lattice (see Figure \ref{fig:Unilotus}). Newton lotuses are used as building blocks of the lotuses associated to Eggers-Wall trees (see Section \ref{sec:lotfromEW}).

\medskip

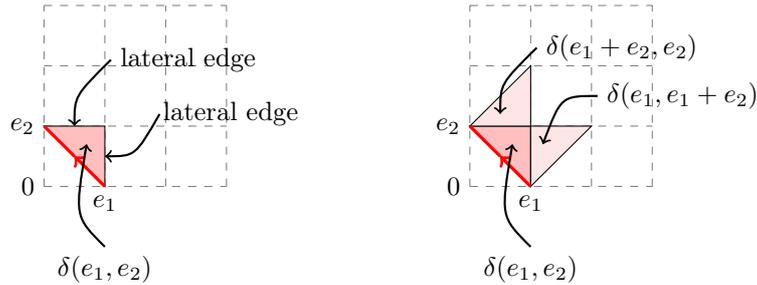
\begin{figure}[h!]
    \begin{center}
\begin{tikzpicture}[scale=0.8]
   \draw [dashed, gray] (0,0) grid (3,3);
\node [below] at (1,0) {$e_{1}$};
\node [left] at (0,1) {$e_{2}$};
\node [left] at (0,0) {$0$};

\draw [fill=pink](1,0) -- (0,1)--(1,1)--cycle;
\draw [->, very thick, red] (1,0)--(0.5, 0.5);
\draw [-, very thick, red] (0.5, 0.5)--(0,1);

\draw[->][thick, color=black](1,-1) .. controls (0.5,-0.5) ..(0.7,0.7);
\node [below] at (1,-1) {$\delta(e_{1},e_{2})$};

\draw[->][thick, color=black](1.9,1.2) .. controls (1.5,0.5) ..(1,0.5);
\node [right] at (1.8,1.2) {$\hbox{\rm lateral edge}$};
\draw[->][thick, color=black](1.1,2.1) .. controls (0.5,1.5) ..(0.5,1);
\node [right] at (1.1,2.1) {$\hbox{\rm lateral edge}$};


\begin{scope}[shift={(7,0)},scale=1]
   \draw [dashed, gray] (0,0) grid (3,3);
\node [below] at (1,0) {$e_{1}$};
\node [left] at (0,1) {$e_{2}$};
\node [left] at (0,0) {$0$};

\draw [fill=pink](1,0) -- (0,1)--(1,1)--cycle;
\draw [fill=pink!40](1,0) -- (1,1)--(2,1)--cycle;
\draw [fill=pink!40](0,1) -- (1,1)--(1,2)--cycle;
\draw [->, very thick, red] (1,0)--(0.5, 0.5);
\draw [-, very thick, red] (0.5, 0.5)--(0,1);

\draw[->][thick, color=black](1,-1) .. controls (0.5,-0.5) ..(0.7,0.7);
\node [below] at (1,-1) {$\delta(e_{1},e_{2})$};

\draw[->][thick, color=black](2.1,1.5) .. controls (1.5,1.5) ..(1.2,0.7);
\node [right] at (2.1,1.5) {$\delta(e_{1},e_{1}+e_{2})$};

\draw[->][thick, color=black](1.1,2.3) .. controls (0.5,1.7) ..(0.5,1.2);
\node [right] at (1.1,2.3){$\delta(e_{1}+e_{2}, e_2)$};
\end{scope}
  \end{tikzpicture}
\end{center}
 \caption{Vocabulary and notations about petals. The petal $\delta(e_{1},e_{2})$ is the parent of both petals $\delta(e_{1} + e_{2},e_{2})$ and $\delta(e_{1}, e_{1} + e_{2})$}
 \label{fig:Vocapetals}
   \end{figure}

 By a {\bf lattice} we mean a free abelian group of finite rank. Let $\boxed{N}$ be a lattice of rank $2$ endowed with a fixed basis  $\boxed{\mathcal{B} := (e_1, e_2)}$. We denote $\boxed{N_\Rr} := N \otimes_{\Zz} \Rr$. It is a real vector space with basis $(e_1, e_2)$.  
 
 \begin{definition} \label{def:petalsunivlotus}
    The {\bf petal} $\boxed{\delta (e_1, e_2)}$ associated with the basis $(e_1, e_2)$ of the lattice $N$ is the triangle of $N_\Rr$ with vertices $e_1, e_2$ and $e_1 + e_2$. The edge $[e_1, e_2]$ is called the {\bf base} of the petal. It has two {\bf lateral edges} $[e_1, e_1+ e_2]$ and $[e_1 + e_2, e_2]$. 
    As $(e_1, e_1 +e_2)$ and $(e_1 + e_2, e_2)$ are also bases of the lattice $N$, they define new petals $\delta(e_1, e_1 +e_2)$ and $\delta (e_1 + e_2, e_2)$ inside the plane $N_\Rr$ (see Figure \ref{fig:Vocapetals}). By iterating this process, at the $n$-th step of construction of petals to $\delta (e_1, e_2)$, one adds $2^n$ petals to those already constructed. The base of a new petal is a lateral edge of a petal constructed at the previous step, called its {\bf parent}.  The result is an infinite simplicial complex  $\boxed{\Lambda{(\mathcal{B})}} = \boxed{\Lambda(e_1, e_2)}$ embedded in the cone $\Rr_{\geq 0} e_1 + \Rr_{\geq 0} e_2$, called the {\bf universal lotus} associated to $(N, \mathcal{B})$ (see Figure \ref{fig:Unilotus}, which is identical to \cite[Figure 1.26]{GBGPPP 20} and almost identical to \cite[Figure 10]{PP 11}). 
\end{definition}

 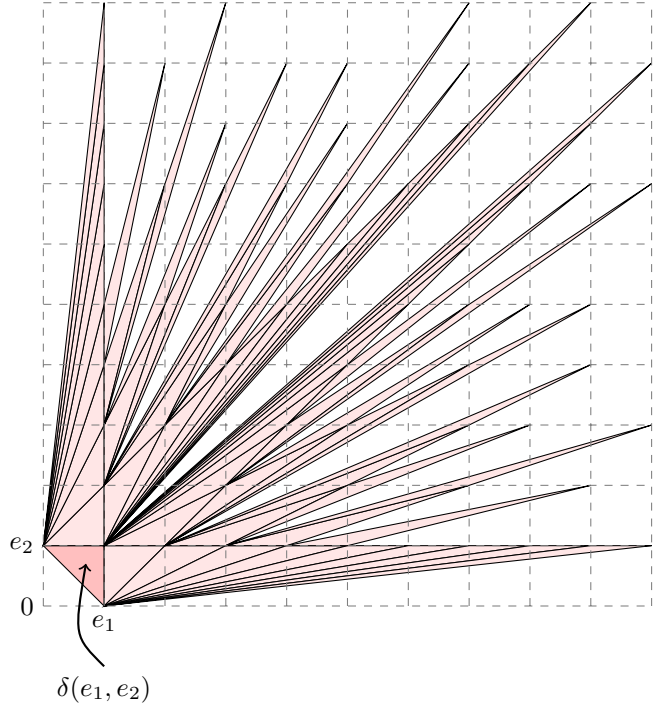
\begin{figure}[h!]
     \begin{center}
\begin{tikzpicture}[scale=0.8]

\draw [fill=pink](1,0) -- (0,1)--(1,1)--cycle;
\draw [fill=pink!40](1,0) -- (1,1)--(2,1)--cycle;
\draw [fill=pink!40](1,0) -- (2,1)--(3,1)--cycle;
\draw [fill=pink!40](1,0) -- (3,1)--(4,1)--cycle;
\draw [fill=pink!40](1,0) -- (4,1)--(5,1)--cycle;
\draw [fill=pink!40](1,0) -- (5,1)--(6,1)--cycle;
\draw [fill=pink!40](1,0) -- (6,1)--(7,1)--cycle;
\draw [fill=pink!40](1,0) -- (7,1)--(8,1)--cycle;
\draw [fill=pink!40](1,0) -- (8,1)--(9,1)--cycle;
\draw [fill=pink!40](1,0) -- (9,1)--(10,1)--cycle;

\draw [fill=pink!40](1,1) -- (2,1)--(3,2)--cycle;
\draw [fill=pink!40](1,1) -- (3,2)--(4,3)--cycle;
\draw [fill=pink!40](1,1) -- (4,3)--(5,4)--cycle;
\draw [fill=pink!40](1,1) -- (5,4)--(6,5)--cycle;
\draw [fill=pink!40](1,1) -- (6,5)--(7,6)--cycle;
\draw [fill=pink!40](1,1) -- (7,6)--(8,7)--cycle;
\draw [fill=pink!40](1,1) -- (8,7)--(9,8)--cycle;
\draw [fill=pink!40](1,1) -- (9,8)--(10,9)--cycle;

\draw [fill=pink!40](2,1) -- (3,2)--(5,3)--cycle;
\draw [fill=pink!40](2,1) -- (5,3)--(7,4)--cycle;
\draw [fill=pink!40](2,1) -- (7,4)--(9,5)--cycle;

\draw [fill=pink!40](2,1) -- (3,1)--(5,2)--cycle;
\draw [fill=pink!40](2,1) -- (5,2)--(7,3)--cycle;
\draw [fill=pink!40](2,1) -- (7,3)--(9,4)--cycle;

\draw [fill=pink!40](3,2) -- (5,3)--(8,5)--cycle;

\draw [fill=pink!40](3,2) -- (4,3)--(7,5)--cycle;
\draw [fill=pink!40](3,2) -- (7,5)--(10,7)--cycle;

\draw [fill=pink!40](4,3) -- (5,4)--(9,7)--cycle;

\draw [fill=pink!40](3,1) -- (5,2)--(8,3)--cycle;
\draw [fill=pink!40](3,1) -- (4,1)--(7,2)--cycle;
\draw [fill=pink!40](3,1) -- (7,2)--(10,3)--cycle;

\draw [fill=pink!40](4,1) -- (5,1)--(9,2)--cycle;

\draw [fill=pink!40](0,1) -- (1,1)--(1,2)--cycle;
\draw [fill=pink!40](0,1) -- (1,2)--(1,3)--cycle;
\draw [fill=pink!40](0,1) -- (1,3)--(1,4)--cycle;
\draw [fill=pink!40](0,1) -- (1,4)--(1,5)--cycle;
\draw [fill=pink!40](0,1) -- (1,5)--(1,6)--cycle;
\draw [fill=pink!40](0,1) -- (1,6)--(1,7)--cycle;
\draw [fill=pink!40](0,1) -- (1,7)--(1,8)--cycle;
\draw [fill=pink!40](0,1) -- (1,8)--(1,9)--cycle;
\draw [fill=pink!40](0,1) -- (1,9)--(1,10)--cycle;

\draw [fill=pink!40](1,1) -- (1,2)--(2,3)--cycle;
\draw [fill=pink!40](1,1) -- (2,3)--(3,4)--cycle;
\draw [fill=pink!40](1,1) -- (3,4)--(4,5)--cycle;
\draw [fill=pink!40](1,1) -- (4,5)--(5,6)--cycle;
\draw [fill=pink!40](1,1) -- (5,6)--(6,7)--cycle;
\draw [fill=pink!40](1,1) -- (6,7)--(7,8)--cycle;
\draw [fill=pink!40](1,1) -- (7,8)--(8,9)--cycle;
\draw [fill=pink!40](1,1) -- (8,9)--(9,10)--cycle;

\draw [fill=pink!40](1,2) -- (2,3)--(3,5)--cycle;
\draw [fill=pink!40](1,2) -- (3,5)--(4,7)--cycle;
\draw [fill=pink!40](1,2) -- (4,7)--(5,9)--cycle;

\draw [fill=pink!40](1,2) -- (1,3)--(2,5)--cycle;
\draw [fill=pink!40](1,2) -- (2,5)--(3,7)--cycle;
\draw [fill=pink!40](1,2) -- (3,7)--(4,9)--cycle;

\draw [fill=pink!40](2,3) -- (3,5)--(5,8)--cycle;

\draw [fill=pink!40](2,3) -- (3,4)--(5,7)--cycle;
\draw [fill=pink!40](2,3) -- (5,7)--(7,10)--cycle;

\draw [fill=pink!40](3,4) -- (4,5)--(7,9)--cycle;

\draw [fill=pink!40](1,3) -- (2,5)--(3,8)--cycle;
\draw [fill=pink!40](1,3) -- (1,4)--(2,7)--cycle;
\draw [fill=pink!40](1,3) -- (2,7)--(3,10)--cycle;

\draw [fill=pink!40](1,4) -- (1,5)--(2,9)--cycle;

\draw [dashed, gray] (0,0) grid (10,10);
\node [below] at (1,0) {$e_{1}$}; 
\node [left] at (0,1) {$e_{2}$}; 
\node [left] at (0,0) {$0$}; 

\draw[->][thick, color=black](1,-1) .. controls (0.5,-0.5) ..(0.7,0.7);  
\node [below] at (1,-1) {$\delta(e_{1},e_{2})$}; 
   \end{tikzpicture}
\end{center}
  \caption{Partial view of the universal lotus $\Lambda(e_{1},e_{2})$ of Definition \ref{def:petalsunivlotus}}  
  \label{fig:Unilotus} 
    \end{figure}

\medskip 

We define now  special kinds of sublotuses of the universal lotus $\Lambda(e_{1},e_{2})$.

\begin{definition}  
      \label{def:deflot} $\,$ 
   \begin{enumerate}
\item 
       A {\bf Newton lotus relative to} $\mathcal{B}$ is either the segment $[e_1, e_2]$ or a union of a finite set of petals of the universal lotus  $\Lambda(\mathcal{B})$,  which contains the parent of each of its petals different from $\delta (e_1, e_2)$. 

\item 
       If $\gamma \in \Qq_{\geq 0} \cup \{ \infty \}$, its  {\bf lotus relative to} $\mathcal{B}$, denoted $\boxed{\Lambda(\gamma)_\mathcal{B}} \subset N_\Rr$, is the simplicial complex obtained as the union of petals of  the universal lotus  $\Lambda(\mathcal{B})$ whose interiors intersect the ray $\Rr_{\geq 0}(e_1 + \gamma e_2)$ of slope $\gamma$ relative to $\mathcal{B}$.  If $\gamma \in \{ 0, \infty \}$, then its lotus $\Lambda(\gamma)_\mathcal{B}$ is just the segment $[e_1, e_2]$.      
\item
       If $\cale \subseteq  \Qq_{\geq 0} \cup \{ \infty \}$, then its {\bf lotus relative to} $\mathcal{B}$ is  the union $\bigcup_{\gamma \in \cale} \Lambda( \gamma )_\mathcal{B}$ of the lotuses of its elements. We denote it by $ \boxed{\Lambda (\cale)_\mathcal{B}}$. In addition, the lotus $\Lambda (\cale)_\mathcal{B}$ is equipped with a finite set of {\bf marked points} $\boxed{p(\gamma)}$ for $\gamma \in \cale$. If $\gamma \in \cale$, the marked point $p(\gamma) \in \Lambda (\cale)_\mathcal{B}$ is the unique primitive element of the lattice $N$ which lies in the cone $\Rr_{\geq 0} e_1 + \Rr_{\geq 0} e_2$ and which has slope $\gamma$ relative to $\mathcal{B}$.              
\end{enumerate}
\end{definition}

  Note that every Newton lotus relative to $(e_1, e_2)$ is of the form $\Lambda (\cale)_\mathcal{B}$ for some {\em finite} subset $\cale$ of $\Qq_{\geq 0} \cup \{ \infty \}$. One may choose for instance the set of slopes of the rays determined by all the vertices of the given Newton lotus.

\medskip 

Let us explain how to construct an {\em abstract lotus} which is combinatorially equivalent to any given Newton lotus by using {\em continued fraction expansions} of positive rational numbers.

\begin{definition}  
     \label{def:contfrac}
  Let $ k \in \Zz_{>0}$ and let $a_1, \dots , a_k$ be natural numbers such that $a_1 \geq 0$ and $a_j >0$ whenever $j \in  \{ 2, \dots, k \}$. The {\bf continued fraction} with {\bf terms} $a_1, \dots , a_k$ is the non-negative rational number: 
   \[ \boxed{[a_1, a_2, \dots, a_k]} : = 
          a_1 + \cfrac{1}{a_2 + \cfrac{1}{ \cdots + \cfrac{1}{a_k}}}. \]
\end{definition}

Any $\gamma \in \Qq_+^*$ may be written uniquely as a continued fraction 
$[a_1, a_2, \dots, a_k]$, if one imposes the constraint that $a_k>1$ whenever $\gamma \neq 1$. One speaks then of the {\bf continued fraction expansion} of $\gamma$. Note that its  first term  $a_1$ is the integer part of $\gamma$, therefore it vanishes if and only if $\gamma \in (0, 1)$.

\begin{definition}  
    \label{def:decomponenumber}
   Let $\gamma \in \Qq_+^*$. Consider its continued fraction expansion 
   $\gamma =  [a_1, a_2, \dots, a_k]$. Its {\bf abstract lotus}  \index{lotus!abstract}
   $\boxed{\Lambda(\gamma)}$ is the simplicial complex constructed from the continued fraction expansion
   $\gamma =  [a_1, a_2, \dots, a_k]$ by the following procedure:
   \medskip
   
   \noindent
   $\bullet$
   Start from an affine triangle $[A_1 A_2 V]$, with vertices $A_1, A_2, V$. 
       
       \noindent
   $\bullet$
        Draw a polygonal line $P_0 P_1 P_2 \dots P_{k-1}$ whose vertices belong alternatively to the sides $[A_1 V]$, $[A_2 V]$, and such that  $P_0 := A_2$ and    
          \[ \left\{ \begin{array}{l} P_1 \in [A_1 V), \mbox{ with } P_1 =A_1 \mbox{ if and only if }  a_1 =0,  \\
            P_i \in (P_{i-2} V) \mbox{ for any }  i \in \{2, \dots, k-1\}.
                \end{array}  \right. \]
        By convention, we set also $P_{-1} := A_1, P_k := V$. The resulting subdivision  
        of the triangle $[A_1 A_2 V]$ into $k$ triangles 
        is the {\bf zigzag decomposition associated with $\lambda$}.

       \noindent
   $\bullet$
        Decompose then each segment $[P_{i -1} P_{i + 1}]$ (for $i \in \{0, \dots, k-1 \}$) 
        into $a_{i+1}$ segments, and join the interior points of $[P_{i -1} P_{i + 1}]$ 
        created in this way to $P_i$. One obtains then a new triangulation of the 
        initial triangle $[A_1 A_2 V]$, which is by definition the abstract lotus $\Lambda(\gamma)$. 
               
    \medskip 
    The {\bf base} of the abstract lotus $\Lambda(\gamma)$ is the segment $[A_1 A_2]$, oriented from $A_1$ to $A_2$. 
\end{definition}


    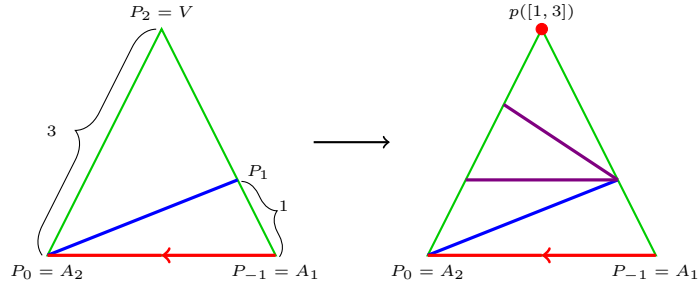
\begin{figure}[h!]
    \begin{center}
\begin{tikzpicture}[scale=0.5]

\draw [decorate,decoration={brace,amplitude=10pt},xshift=-4pt,yshift=0pt]
(0,0) --(3,6) node [black,midway,xshift=-0.6cm]
{$\;$};
\node [right] at (-0.5,3.3) {{ \tiny{$3$}}};

\draw [decorate,decoration={brace,amplitude=7pt},yshift=-4pt,xshift=0pt]
(5.1,2.1) -- (6.1,0.1) node  [black,midway,yshift=0.6cm]
{$\;$};
\node [right] at (5.6,1.3) {{ \tiny{$1$}}};

\draw [color=black!20!green, thick] (0,0) -- (6,0)--(3,6)--cycle;
\draw [-,color=blue, very thick] (0,0) -- (5,2);
\draw [->, color=red, very thick] (6,0)--(3,0);
\draw [-, color=red, very thick] (3,0)--(0,0);
\node [below] at (0,0) {{\tiny $P_{0}=A_{2}$}};
\node [below] at (6,0) {{\tiny $P_{-1}=A_1$}};
\node [right] at (5,2.2) {{\tiny $P_{1}$}};
\node [above] at (3,6) {{\tiny $P_{2}=V$}};

 \draw [->, thick](7,3) -- (9, 3) ;

\begin{scope}[shift={(10,0)}]
\draw [color=black!20!green, thick] (0,0) -- (6,0)--(3,6)--cycle;
\draw [-,color=blue, very thick] (0,0) -- (5,2);
\draw [->, color=red, very thick] (6,0)--(3,0);
\draw [-, color=red, very thick] (3,0)--(0,0);
\draw [-, color=violet, very thick] (5,2)--(1,2);
\draw [-, color=violet, very thick] (5,2)--(2,4);
\node[draw,circle, inner sep=1.5pt,color=red, fill=red] at (3,6){};
\node [above] at (3,6) {{\tiny $p([1,3])$}};
\node [below] at (0,0) {{\tiny $P_0=A_{2}$}};
\node [below] at (6,0) {{\tiny $P_{-1}=A_1$}};
\end{scope}
\end{tikzpicture}
\end{center}
 \caption{Construction of the Newton lotus $\Lambda(4/3=[1,3])$}
\label{fig:4/3}
   \end{figure}


    \begin{figure}[h!]
    \begin{center}
\begin{tikzpicture}[scale=0.5]

\draw [decorate,decoration={brace,amplitude=10pt},xshift=-4pt,yshift=0pt]
(0,0) --(1.5,2.9) node [black,midway,xshift=-0.6cm]
{$\;$};
\node [right] at (-0.9,2) {{ \tiny{$1$}}};

\draw [decorate,decoration={brace,amplitude=7pt},yshift=-4pt,xshift=0pt]
(5.2,2) -- (6.1,0.2) node  [black,midway,yshift=0.6cm]
{$\;$};
\node [right] at (5.7,1.3) {{ \tiny{$1$}}};

\draw [decorate,decoration={brace,amplitude=10pt},yshift=-4pt,xshift=0pt]
(3.2,6) -- (5.2,2.2) node  [black,midway,yshift=0.6cm]
{$\;$};
\node [right] at (4.5,4.3) {{ \tiny{$2$}}};

\draw [color=black!20!green, thick] (0,0) -- (6,0)--(3,6)--cycle;
\draw [-,color=blue, very thick] (0,0) -- (5,2)--(1.5,3);
\draw [->, color=red, very thick] (6,0)--(3,0);
\draw [-, color=red, very thick] (3,0)--(0,0);
\node [below] at (0,0) {{\tiny $P_{0}$}};
\node [below] at (6,0) {{\tiny $P_{-1}$}};
\node [right] at (5,2.2) {{\tiny $P_{1}$}};
\node [left] at (1.4,3.1) {{\tiny $P_{2}$}};
\node [above] at (3,6) {{\tiny $P_{3}=V$}};

 \draw [->, thick](7,3) -- (9, 3) ;

\begin{scope}[shift={(10,0)}]
\draw [color=black!20!green, thick] (0,0) -- (6,0)--(3,6)--cycle;
\draw [-,color=blue, very thick] (0,0) -- (5,2)--(1.5,3);
\draw [->, color=red, very thick] (6,0)--(3,0);
\draw [-, color=red, very thick] (3,0)--(0,0);
\draw [-, color=violet, very thick] (1.5,3)--(3.9,4.2);
\node[draw,circle, inner sep=1.5pt,color=red, fill=red] at (3,6){};
\node [above] at (3,6) {{\tiny $p([1,1,2])$}};
\node [below] at (0,0) {{\tiny $P_{0}$}};
\node [below] at (6,0) {{\tiny $P_{-1}$}};
\end{scope}
\end{tikzpicture}
\end{center}
 \caption{Construction of the Newton lotus  $\Lambda(5/3=[1,1,2])$}
\label{fig:5/3}
   \end{figure}
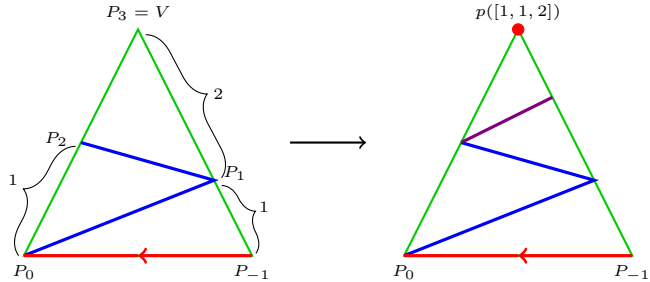

\begin{example}
   \label{ex:constrNlotus}
     In Figures \ref{fig:4/3} and \ref{fig:5/3} are represented the abstract lotuses associated to $4/3 = [1, 3]$ and $5/3 = [1, 1, 2]$, constructed following Definition \ref{def:decomponenumber}. 
\end{example}

\medskip 
   Klein's geometric interpretation of continued fraction expansions (see \cite[Section 3]{PP 07}) allows to prove by induction on the sum $a_1 + \cdots + a_k$ the following relation between both kinds of lotuses (see \cite[Proposition 1.5.20]{GBGPPP 20}):

\begin{proposition}
   Let $\gamma \in \Qq_+^*$. The Newton lotus  $\Lambda(\gamma)_\mathcal{B}$ relative to the basis $\mathcal{B}$ and the abstract lotus $\Lambda(\gamma)$ are combinatorially isomorphic, seen as triangulated polygons, where the isomorphism sends $e_1$ to $A_1$, $e_2$ to $A_2$ and $p(\gamma)$ to $V$. 
\end{proposition}

\medskip

  Suppose now that one has two numbers $\gamma, \mu \in \Qq_+^*$. If $\gamma =  [a_1,\dots, a_k]$ and $\mu = [b_1,\dots, b_l]$, let $j\in \{0, \dots,  \min\{k,l\}\}$ be maximal such that $a_i = b_i$ for all $i \in \{1,\dots, j\}$.  We may assume, up to permutation of $\gamma$ and $\mu$, that $k =j$ or $a_{j+1} < b_{j+1}$. Define then:
 \begin{equation} \label{eq:defwedge}  
     \boxed{\gamma \wedge \mu}  := \left\{ 
       \begin{array}{l}
             [a_1, \dots , a_j],  \mbox{ if }    k = j,    \\
                \left[a_1, \dots, a_j, a_{j+1} \right],    \mbox{ if }  k =  j + 1,  \\
               \left[a_1, \dots, a_j, a_{j+1}  + 1 \right],   \mbox{ if }  k > j + 1.
       \end{array}  \right.
  \end{equation}
    
    The next proposition allow us to describe the intersection of two lotuses 
    of the form $\Lambda(\gamma)$
    in terms of the operation  $\wedge$ on $\Qq_+^*$ (see \cite[Proposition 1.5.23]{GBGPPP 20}). 
           
\begin{proposition} \label{prop:lotus-wedge} 
   For any $\gamma, \mu \in \Qq_+^*$, one has:
       \[ \Lambda(\gamma)_\mathcal{B} \cap \Lambda(\mu)_\mathcal{B} = 
            \Lambda( \gamma \wedge \mu)_\mathcal{B}.  \]
   Therefore, the lotus $\Lambda(\{\gamma, \mu \})_\mathcal{B}$ is isomorphic, 
   as a simplicial complex with an oriented base, to the triangulated polygon 
   $\boxed{\Lambda(\gamma) \uplus \Lambda(\mu)}$ obtained by gluing the abstract lotuses $\Lambda(\gamma)$ and $\Lambda(\mu)$ along $\Lambda(\gamma \wedge \mu)$.
\end{proposition}

Iterating the gluing operation, one defines the {\bf abstract lotus} 
associated with $\gamma_1, \dots, \gamma_k \in \Qq^*_+$, as
   \[
     \boxed{\Lambda(\gamma_1,\dots, \gamma_k)} := 
      \Lambda (\gamma_1, \dots, \gamma_{k-1}) \uplus \Lambda(\gamma_k).\] 
It is combinatorially equivalent to the Newton lotus 
$\Lambda(\gamma_1,\dots, \gamma_k)_\mathcal{B}$, seen as a triangulated polygon with marked points $p(\gamma_1), \dots, p(\gamma_k)$ and an oriented base.

\begin{example} 
  In Figure \ref{fig:4/3+5/3} we represent the abstract lotus $\Lambda(\{\frac{4}{3}, \frac{5}{3}\})$, obtained by gluing the lotuses $\Lambda(\frac{4}{3})$ and $\Lambda(\frac{5}{3})$ of Figures \ref{fig:4/3} and \ref{fig:5/3} as explained  in Proposition \ref{prop:lotus-wedge}.
\end{example}

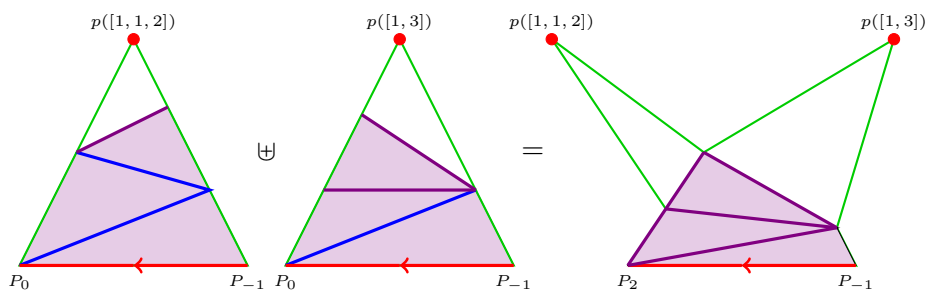
\begin{figure}[h!]
    \begin{center}
\begin{tikzpicture}[scale=0.5]

\begin{scope}[shift={(0,0)}]
  
  \draw[fill=violet!20](0,0) -- (6,0) -- (5,2)  --(1.5,3)--cycle;
   \draw[fill=violet!20](1.5,3)--(5,2)  --(3.9,4.2)--cycle;
   \draw [color=black!20!green, thick] (0,0) -- (6,0)--(3,6)--cycle;
   \draw [-,color=blue, very thick] (0,0) -- (5,2)--(1.5,3);
   \draw [->, color=red, very thick] (6,0)--(3,0);
   \draw [-, color=red, very thick] (3,0)--(0,0);
   \draw [-, color=violet, very thick] (1.5,3)--(3.9,4.2);
     \node[draw,circle, inner sep=1.5pt,color=red, fill=red] at (3,6){};
     \node [above] at (3,6) {{\tiny $p([1,1,2])$}};
     \node [below] at (0,0) {{\tiny $P_{0}$}};
     \node [below] at (6,0) {{\tiny $P_{-1}$}};
\end{scope}
\node at (6.5,3) {{\large $\uplus$}};

\begin{scope}[shift={(7,0)}]

    \draw[fill=violet!20](0,0) -- (6,0) -- (5,2)  --(1,2)--cycle;
    \draw[fill=violet!20](5,2)--(1,2)  --(2,4)--cycle;
    \draw [color=black!20!green, thick] (0,0) -- (6,0)--(3,6)--cycle;       
    \draw [-,color=blue, very thick] (0,0) -- (5,2);
    \draw [->, color=red, very thick] (6,0)--(3,0);
    \draw [-, color=red, very thick] (3,0)--(0,0);
    \draw [-, color=violet, very thick] (5,2)--(1,2);
    \draw [-, color=violet, very thick] (5,2)--(2,4);
      \node[draw,circle, inner sep=1.5pt,color=red, fill=red] at (3,6){};
      \node [above] at (3,6) {{\tiny $p([1,3])$}};
      \node [below] at (0,0) {{\tiny $P_{0}$}};
      \node [below] at (6,0) {{\tiny $P_{-1}$}};

\end{scope}

\node at (13.5,3) {{\large $=$}};

\begin{scope}[shift={(16,0)}]
\draw [color=black!20!green, thick] (1,1.5) -- (-2,6) --(2,3)--(7,6)--(5.5,1)--(6,0)--cycle;
\draw[fill=violet!20](0,0) -- (6,0)--(5.5,1) --(0,0)--cycle;
\draw[fill=violet!20](0,0) -- (1,1.5)--(5.5,1) --cycle;
\draw[fill=violet!20](1,1.5)--(5.5,1) --(2,3)--cycle;
\draw [->, color=red, very thick] (6,0)--(3,0);
\draw [-, color=red, very thick] (3,0)--(0,0);
\draw [-, color=violet, very thick] (0,0)--(2,3);
\draw [-, color=violet, very thick] (0,0)--(5.5,1);
\draw [-, color=violet, very thick] (5.5,1)--(1,1.5);
\draw [-, color=violet, very thick] (5.5,1)--(2,3);
\node[draw,circle, inner sep=1.5pt,color=red, fill=red] at (-2,6){};
\node [above] at (-2,6) {{\tiny $p([1,1,2])$}};
\node[draw,circle, inner sep=1.5pt,color=red, fill=red] at (7,6){};
\node [above] at (7,6) {{\tiny $p([1,3])$}};
\node [below] at (0,0) {{\tiny $P_{2}$}};
\node [below] at (6,0) {{\tiny $P_{-1}$}};
\end{scope}
\end{tikzpicture}
\end{center}
  \caption{The abstract lotus $\Lambda(\{[1,3], [1,1,2]\})$ is obtained by gluing the abstract lotuses $\Lambda([1,3])$ and $\Lambda([1,1,2])$} 
 \label{fig:4/3+5/3}
 \end{figure}

We can see a Newton lotus as the lotus associated to a finite active constellation of crosses by taking a cross at $O$ and by fixing an orientation of its dual segment:

\begin{proposition} 
   \label{prop: lotus-const}
    Let $\cale$ be a nonempty finite subset of $\Qq_+$. Let $X_O := L + L_1$ be a cross at $O$. Then, there exists a unique minimal finite active constellation of crosses $\boxed{\hat{\calc} (\cale)}$ above $O$ all of whose points are active, such that $X_O$ is its cross at $O$, and the lotuses $\Lambda(\hat{\calc} (\cale)) $ and  $\Lambda(\cale)$ are isomorphic as triangulated polygons with oriented bases.
\end{proposition}

\begin{proof}
      Choose local coordinates $(x,y)$ on $(S, O)$ such that $L = Z(x)$ and $L_1 = Z(y)$. Consider the lotus $\Lambda(\cale)_{\mathcal{B}}$ as an instruction to perform a sequence of blowups of the closed orbit $O$ of $\Cc^2$ seen as a toric surface with weight lattice $N$. The underlying constellation of $\hat{\calc} (\cale)$ consists of all the points which are blown up. Its crosses are the germs of the total transforms at those points of the sum of coordinate axes of $\Cc^2$, and all the points are declared active. It is simple to show by induction on the number of petals of $\Lambda(\cale)$ that this active constellation of crosses depends only on the cross $X_O$ and not on the local coordinate system $(x,y)$ defining it. 
\end{proof}

The following remark points out the connection of lotuses with toroidal geometry which was essential in the approach to lotuses developed in \cite{GBGPPP 20}.

\begin{remark}
   The rays spanned by the integral points in the lateral boundary of the lotus $\Lambda(\hat{\calc} (\cale))_{\mathcal{B}}$ define a {\em regular fan} $\mathcal{F}(\cale)$ subdividing the cone $\Rr_{\geq 0} e_1+ \Rr_{\geq 0} e_2$. The model of the active constellation $\hat{\calc} (\cale)$ can be seen as the {\em toroidal map} $\pi: (\Sigma, \partial \Sigma) \to (S, X_O)$ defined by  the fan $\mathcal{F}(\cale)$,  with respect to the {\em boundaries} $\partial \Sigma = \pi^{-1} (X_O)$ and $\partial S = X_O$ (see \cite[Section 1.3.3]{GBGPPP 20} for more details). 
\end{remark}

To conclude this section, let us mention that Faber and Schober started in \cite{FS 24} the study of the relation of Newton lotuses with Conway and Coxeter's {\em frieze patterns}.

\medskip
\section{Construction of a lotus from an Eggers-Wall tree}  \label{sec:lotfromEW}

In this section we explain how to associate lotuses  to a given {\em complete} Eggers-Wall tree, in the sense of Definition \ref{def:EW-complete}. This association depends on the choice of a {\em trunk decomposition} of the Eggers-Wall tree (see Definition \ref{def:trunkdec}). This may lead to different lotuses associated to the same Eggers-Wall tree. The Eggers-Wall tree associated to any of those lotuses as explained in Corollary \ref{cor:lotustoEW} is isomorphic to the starting Eggers-Wall tree (see Proposition \ref{prop:reconstrEW}). The basic building block of our construction is the association of a {\em Newton lotus} to every finite set of positive rational numbers (see Definition \ref{def:deflot}). 

\medskip 

In order to pass from an Eggers-Wall tree to a lotus whose associated tree in the sense of Definition \ref{def:assEWtolotus} is isomorphic to the given one, we need to work with the notion of {\em complete} Eggers-Wall tree:

\begin{definition} 
    \label{def:EW-complete}
    Let $A$ be a reduced plane curve singularity and $L$ be a smooth branch on $(S,O)$.  The Eggers-Wall tree $\Theta_{L} (A)$ is {\bf complete} if the ends of the level sets of the index function are leaves of $\Theta_{L} (A)$. 
\end{definition}

\begin{remark}
    Recall that given a reduced plane curve singularity $A$ together with a finite active constellation of crosses $\hat{\calc}_A$ compatible with $A$, the lotus $\Lambda (\hat{\calc}_A)$ determines the  Eggers-Wall tree  $\Theta_{L} (\hat{A})$ of the  completion $\hat{A}$  of $A$ relative to $\hat{\calc}_A$ (see Proposition  \ref{prop:fromlotustoEW}). The Eggers-Wall tree $\Theta_{L} (\hat{A})$ is complete according to Definition \ref{def:EW-complete}, which is the reason why we choose this terminology. For instance, the Eggers Wall tree of Figure \ref{fig:logdistoEW3} is complete, but the Eggers-Wall subtree $\Theta_L (A_1+ A_2+ A_3)$ is not, because the end $E_1$ of the level set ${\de}_L^{-1}(1)$ is not one of its leaves. 
\end{remark}

Notice that if $\Theta_{L} (A)$ is not complete, we may add some additional branches to $A$ in order to obtain a reduced plane curve singularity $A'$ such that $\Theta_L (A')$ is complete. It is enough to define $A' := \hat{A}$ to be the completion of $A$ relative to an active constellation of crosses adapted to it in the sense of Definition \ref{def:adaptedactconst}.

\medskip


 \begin{figure}[h!]
    \begin{center}
\begin{tikzpicture}[scale=1]
  
  \begin{scope}[shift={(-10,0)}]
  
     \draw [->, color=black,  line width=1pt](0,0) -- (0,3);
     \draw [->, color=black,  line width=1pt](0,1) -- (0,0);
      \node[draw,circle, inner sep=1pt,color=black, fill=black] at (0,1.5){};
      \draw [->, color=black,  line width=1pt](0,1.5) -- (-1.4,1.5);

       \node [below, color=black] at (0,0) {$L$};
         \node [left, color=black] at (-1.4,1.5) {$L_{1}$};
           \node [above, color=black] at (0,3) {$A$};

           \node [right, color=black] at (0,0) {\small{$0$}};  
           \node [right, color=black] at (0,1.5) {\small{$\mathbf{\frac{3}{2}}$}}; 
                \node [below, color=black] at (-1.3,1.4) {\small{$\infty$}};

                 \node [right, color=black] at (0,0.75) {\small{$1$}};
                 \node [right, color=black] at (0,2.25) {\small{$2$}};
                   \node [above, color=black] at (-0.7,1.5) {\small{$1$}};

   \draw[->][thick, color=black](0.5,1.7) .. controls (1,1.9) ..(1.5,1.7);
  \end{scope}

\begin{scope}[shift={(-7,0)}]

     \draw [->, color=black,  line width=1pt](0,2) -- (0,3.5);
     \draw [->, color=black,  line width=1pt](0,1.5) -- (0,0);
      \node[draw,circle, inner sep=1pt,color=black, fill=black] at (0,1.5){};
       \node[draw,circle, inner sep=1pt,color=black, fill=black] at (0,2){};
      \draw [->, color=black,  line width=1pt](0,1.5) -- (-1.4,1.5);

       \node [below, color=black] at (0,0) {$L$};
         \node [left, color=black] at (-1.4,1.5) {$L_{1}$};
           \node [above, color=black] at (0,3.5) {$A$};

           \node [right, color=blue] at (0,1.5) {\small{$\mathbf{\frac{3}{2}}$}}; 
              \draw[->][thick, color=blue](-0.2,2) .. controls (-0.4,2) ..(-0.2,1.6); 
              
                  \draw[->][thick, color=black](0.5,1.7) .. controls (1,1.9) ..(1.5,1.7); 
    \end{scope}
    
    \begin{scope}[shift={(-4.5,0)}]
  
 \draw [->, color=black,  line width=1pt](0,2) -- (0,3.5);
  \node[draw,circle, inner sep=1pt,color=black, fill=black] at (0,2){};
  \node [above, color=black] at (0,3.5) {$A$};

    \draw [->, color=black, thick](1,0) -- (0,0);
    \draw [-, color=black, thick](0,0) -- (-1,0);
    \draw [-, color=black, thick](0,1.5) -- (-1,0);
      \draw [-, color=black, thick](0,1.5) -- (1,0);
       \draw [-, color=black, thick](-0.5,3/4) -- (0.5,3/4);
       \draw [-, color=black, thick](-1,0) -- (0.5,3/4);

  \node[draw,circle, inner sep=1pt,color=black, fill=black] at (-1,0){};
   \node[draw,circle, inner sep=1pt,color=black, fill=black] at (1,0){};
     \node[draw,circle, inner sep=1pt,color=black, fill=black] at (0,1.5){};
      \node[draw,circle, inner sep=1pt,color=black, fill=black] at (-0.5,3/4){};
       \node[draw,circle, inner sep=1pt,color=black, fill=black] at (0.5,3/4){};

   \node [below, color=black] at (-1,0) {$L_{1}$};
\node [below, color=black] at (1,0) {$L$};

  \draw[->][thick, color=blue](-0.2,2) .. controls (-0.4,2) ..(-0.2,1.6);
    \draw[->][thick, color=black](1,1.7) .. controls (1.5,1.9) ..(2,1.7); 
                
    \end{scope}
    
     \begin{scope}[shift={(-1.5,0)}]
      \draw [->, color=black, thick](1,0) -- (0,0);
    \draw [-, color=black, thick](0,0) -- (-1,0);
    \draw [-, color=black, thick](0,1.5) -- (-1,0);
      \draw [-, color=black, thick](0,1.5) -- (1,0);
       \draw [-, color=black, thick](-0.5,3/4) -- (0.5,3/4);
       \draw [-, color=black, thick](-1,0) -- (0.5,3/4);
        \draw [->, color=black, thick](0,1.5) -- (0,2.5);
           
  \node[draw,circle, inner sep=1pt,color=black, fill=black] at (-1,0){};
   \node[draw,circle, inner sep=1pt,color=black, fill=black] at (1,0){};
     \node[draw,circle, inner sep=1pt,color=black, fill=black] at (0,1.5){};
      \node[draw,circle, inner sep=1pt,color=black, fill=black] at (-0.5,3/4){};
       \node[draw,circle, inner sep=1pt,color=black, fill=black] at (0.5,3/4){};

   \node [below, color=black] at (-1,0) {$L_{1}$};
\node [below, color=black] at (1,0) {$L$};
\node [above, color=black] at (0,2.5) {$A$};
     
       \end{scope}

  \end{tikzpicture}
\end{center}
 \caption{Constructing a lotus from the Eggers-Wall tree on the right of Figure \ref{fig:logdistoEW1}}
\label{fig:EWtolotus1}
   \end{figure}
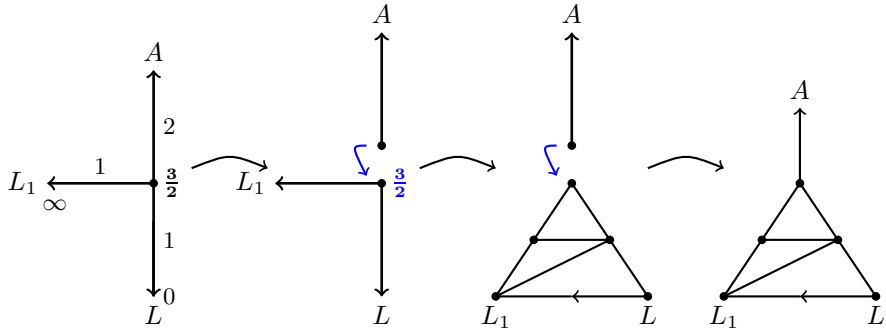


 \begin{figure}[h!]
    \begin{center}
\begin{tikzpicture}[scale=1]
  
  \begin{scope}[shift={(-10,0)}]

     \draw [->, color=black,  line width=1pt](0,0) -- (0,4.5);
     \draw [->, color=black,  line width=1pt](0,1) -- (0,0);

      \draw [->, color=black,  line width=1pt](0,1.5) -- (-1.4,1.5);
       \draw [->, color=black,  line width=1pt](0,3) -- (-1.4,3);

       \node [below, color=black] at (0,0) {$L$};
         \node [left, color=black] at (-1.4,1.5) {$L_{1}$};
           \node [above, color=black] at (0,4.5) {$A_1$};
           \node [left, color=black] at (-1.4,3) {$L_{2}$};
       
            \node[draw,circle, inner sep=1pt,color=black, fill=black] at (0,1.5){};
      \node[draw,circle, inner sep=1pt,color=black, fill=black] at (0,3){};

                           \node [right, color=black] at (0,0.75) {\small{$1$}};
                 \node [right, color=black] at (0,2.25) {\small{$2$}};
                  \node [right, color=black] at (0,3.75) {\small{$6$}};
                   \node [above, color=black] at (-0.7,1.5) {\small{$1$}};
                   \node [above, color=black] at (-0.7,3) {\small{$2$}};
                    \node [right, color=black] at (0,0) {\small{$0$}};  
           \node [right, color=black] at (0,1.5) {\small{$\frac{3}{2}$}}; 
            \node [right, color=black] at (0,3) {\small{$\frac{13}{6}$}}; 
             \node [below, color=black] at (-1.3,2.9) {\small{$\infty$}};
                \node [below, color=black] at (-1.3,1.4) {\small{$\infty$}};

  \draw[->][thick, color=black](0.5,2) .. controls (1,2.2) ..(1.5,2);
 \end{scope}

\begin{scope}[shift={(-7,0)}]

     \draw [-, color=black,  line width=1pt](0,0) -- (0,1.5);
     \draw [-, color=black,  line width=1pt](0,2) -- (0,3.5);
     \draw [->, color=black,  line width=1pt](0,0.5) -- (0,0);
      \draw [->, color=black,  line width=1pt](0,4) -- (0,5.5);
      
      \draw [->, color=black,  line width=1pt](0,1.5) -- (-1.4,1.5);
       \draw [->, color=black,  line width=1pt](0,3.5) -- (-1.4,3.5);

       \node [below, color=black] at (0,0) {$L$};
         \node [left, color=black] at (-1.4,1.5) {$L_{1}$};
           \node [above, color=black] at (0,5.5) {$A_1$};
           \node [left, color=black] at (-1.4,3.5) {$L_{2}$};
       
            \node[draw,circle, inner sep=1pt,color=black, fill=black] at (0,1.5){};
      \node[draw,circle, inner sep=1pt,color=black, fill=black] at (0,2){};
          \node[draw,circle, inner sep=1pt,color=black, fill=black] at (0,3.5){};
            \node[draw,circle, inner sep=1pt,color=black, fill=black] at (0,4){};
                   
           \node [right, color=blue] at (0,1.5) {\small{$\mathbf{\frac{3}{2}}$}}; 
            \node [right, color=blue] at (0,3.5) {\small{$\mathbf{\frac{4}{3}}$}}; 
          
           \draw[->][thick, color=blue](-0.2,2) .. controls (-0.4,2) ..(-0.2,1.6); 
            \draw[->][thick, color=blue](-0.2,4) .. controls (-0.4,4) ..(-0.2,3.6); 

 \draw[->][thick, color=black](0.5,2) .. controls (1,2.2) ..(1.5,2);
    \end{scope}
    
    \begin{scope}[shift={(-2,0)}]
  
    \draw [->, color=black, thick](1,0) -- (0,0);
    \draw [-, color=black, thick](0,0) -- (-1,0);
    \draw [-, color=black, thick](0,1.5) -- (-1,0);
      \draw [-, color=black, thick](0,1.5) -- (1,0);
       \draw [-, color=black, thick](-0.5,3/4) -- (0.5,3/4);
       \draw [-, color=black, thick](-1,0) -- (0.5,3/4);

  \node[draw,circle, inner sep=1pt,color=black, fill=black] at (-1,0){};
   \node[draw,circle, inner sep=1pt,color=black, fill=black] at (1,0){};
     \node[draw,circle, inner sep=1pt,color=black, fill=black] at (0,1.5){};
      \node[draw,circle, inner sep=1pt,color=black, fill=black] at (-0.5,3/4){};
       \node[draw,circle, inner sep=1pt,color=black, fill=black] at (0.5,3/4){};

   \node [below, color=black] at (-1,0) {$L_{1}$};
\node [below, color=black] at (1,0) {$L$};

    \draw [->, color=black, thick](0,2) -- (-1.5,2);
   \draw [-, color=black, thick](-1.5,2) -- (-3,2);
     \draw [-, color=black, thick](0,2) -- (-1.5,4);
      \draw [-, color=black, thick](-1.5,4) -- (-3,2);
  
     \draw [-, color=black, thick](-3/4,3) -- (-3,2);
     \draw [-, color=black, thick](-3/4,3) -- (-2,10/3);
       \draw [-, color=black, thick](-3/4,3) -- (-2.5, 8/3);
       
          \node[draw,circle, inner sep=1pt,color=black, fill=black] at (0,2){};
          \node[draw,circle, inner sep=1pt,color=black, fill=black] at (-3,2){};
          
         \node[draw,circle, inner sep=1pt,color=black, fill=black] at (-1.5,4){};
         \node[draw,circle, inner sep=1pt,color=black, fill=black] at (-3/4,3){};
         \node[draw,circle, inner sep=1pt,color=black, fill=black] at (-2,10/3){};
        \node[draw,circle, inner sep=1pt,color=black, fill=black] at (-2.5,8/3){};

 \node [above, color=black] at (-1.5,5.5) {$A_{1}$};
  \node[draw,circle, inner sep=1pt,color=black, fill=black] at (-1.5,4.5){};
  \draw [->, color=black, thick](-1.5,4.5) -- (-1.5,5.5);

 \draw[->][thick, color=blue](-1.5,4.4) .. controls (-1.7,4.4) ..(-1.5,4.1);
  \draw[->][thick, color=blue](-0.2,1.9) .. controls (-0.4,1.9) ..(-0.2,1.5);
    \draw[->][thick, color=black](0.5,2) .. controls (1,2.2) ..(1.5,2); 
                
\end{scope}

 \begin{scope}[shift={(2.5,0)}]

 \draw [->, color=black, thick](1,0) -- (0,0);
    \draw [-, color=black, thick](0,0) -- (-1,0);
    \draw [-, color=black, thick](0,1.5) -- (-1,0);
      \draw [-, color=black, thick](0,1.5) -- (1,0);
       \draw [-, color=black, thick](-0.5,3/4) -- (0.5,3/4);
       \draw [-, color=black, thick](-1,0) -- (0.5,3/4);

        \draw [->, color=black, thick](0,1.5) -- (-1.5,1.5);
   \draw [-, color=black, thick](-1.5,1.5) -- (-3,1.5);
     \draw [-, color=black, thick](0,1.5) -- (-1.5,4);
      \draw [-, color=black, thick](-1.5,4) -- (-3,1.5);
      \draw [->, color=black, thick](-1.5,4) -- (-1.5,5);
     \draw [-, color=black, thick](-3/4,11/4) -- (-3,1.5);
     \draw [-, color=black, thick](-3/4,11/4) -- (-2.5,2.34);
       \draw [-, color=black, thick](-3/4,11/4) -- (-2,3.17);
           
  \node[draw,circle, inner sep=1pt,color=black, fill=black] at (-1,0){};
   \node[draw,circle, inner sep=1pt,color=black, fill=black] at (1,0){};
     \node[draw,circle, inner sep=1pt,color=black, fill=black] at (0,1.5){};
      \node[draw,circle, inner sep=1pt,color=black, fill=black] at (-0.5,3/4){};
       \node[draw,circle, inner sep=1pt,color=black, fill=black] at (0.5,3/4){};
        \node[draw,circle, inner sep=1pt,color=black, fill=black] at (-3,1.5){};
         \node[draw,circle, inner sep=1pt,color=black, fill=black] at (-1.5,4){};
         \node[draw,circle, inner sep=1pt,color=black, fill=black] at (-3/4,11/4){};
         \node[draw,circle, inner sep=1pt,color=black, fill=black] at (-2.5,2.34){};
         \node[draw,circle, inner sep=1pt,color=black, fill=black] at (-2,3.17){};

   \node [below, color=black] at (-1,0) {$L_{1}$};
\node [below, color=black] at (1,0) {$L$};
\node [below, color=black] at (-3,1.5) {$L_{2}$};
\node [above, color=black] at (-1.5,5) {$A_1$};

        \end{scope}
  \end{tikzpicture}
\end{center}
 \caption{Constructing a lotus from the Eggers-Wall tree on the right of Figure \ref{fig:logdistoEW2}}
\label{fig:EWtolotus2}
   \end{figure}
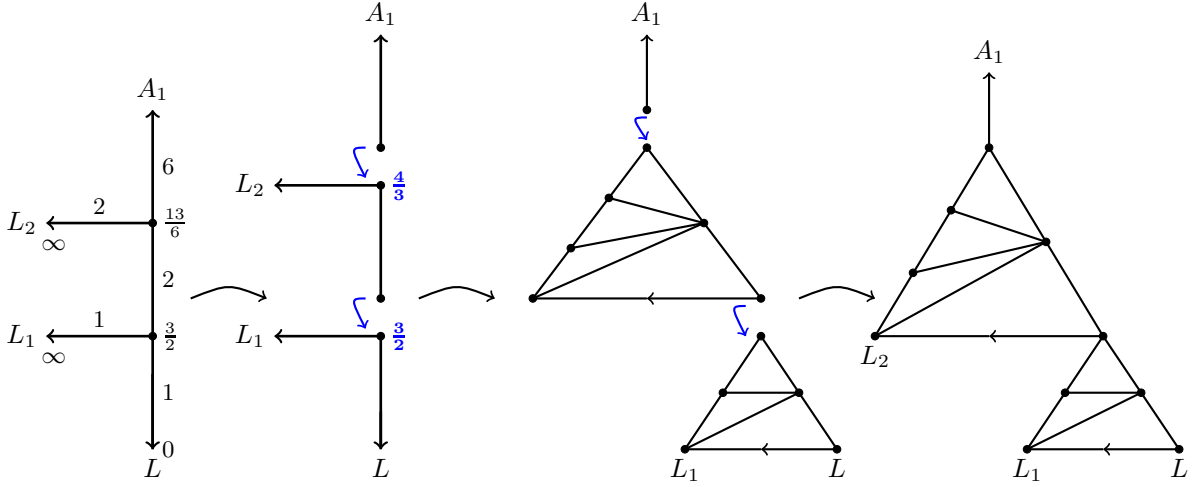


 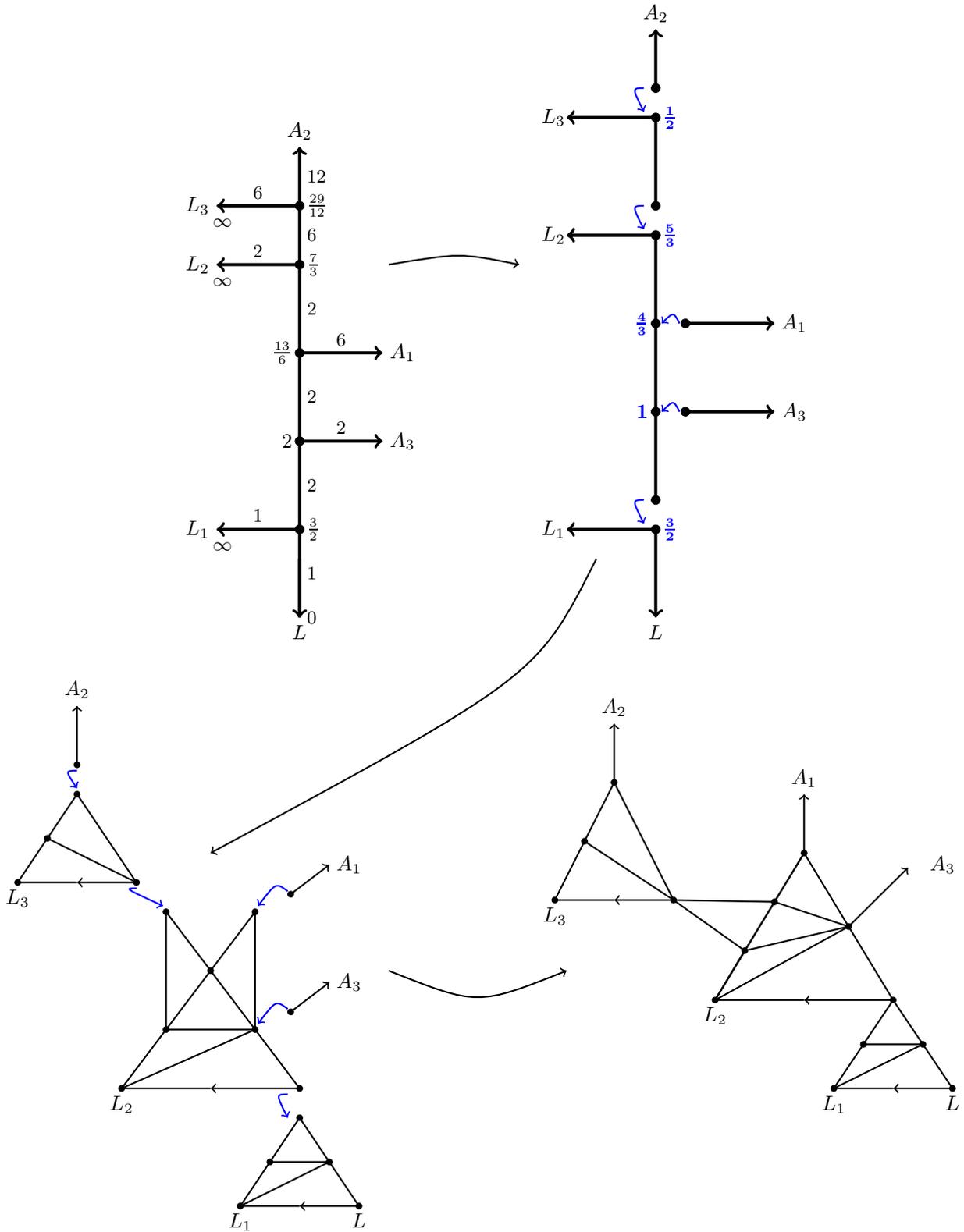
\begin{figure}[h!]
    \begin{center}
\begin{tikzpicture}[scale=1]
  
  \begin{scope}[shift={(-10,0)}]
 \draw [->, color=black,  line width=1.5pt](0,0) -- (0,8);
     \draw [->, color=black,  line width=1.5pt](0,1) -- (0,0);
      \node[draw,circle, inner sep=1.5pt,color=black, fill=black] at (0,1.5){};
      \node[draw,circle, inner sep=1.5pt,color=black, fill=black] at (0,3){};
      \node[draw,circle, inner sep=1.5pt,color=black, fill=black] at (0,4.5){};
      \node[draw,circle, inner sep=1.5pt,color=black, fill=black] at (0,6){};
      \node[draw,circle, inner sep=1.5pt,color=black, fill=black] at (0,7){};
      
      \draw [->, color=black,  line width=1.5pt](0,1.5) -- (-1.4,1.5);
       \draw [->, color=black,  line width=1.5pt](0,3) -- (1.4,3);
       \draw [->, color=black,  line width=1.5pt](0,4.5) -- (1.4,4.5);
       \draw [->, color=black,  line width=1.5pt](0,6) -- (-1.4,6);
       \draw [->, color=black,  line width=1.5pt](0,7) -- (-1.4,7);
      
       \node [below, color=black] at (0,0) {$L$};
         \node [left, color=black] at (-1.4,1.5) {$L_{1}$};
          \node [right, color=black] at (1.4,3) {$A_3$};
           \node [right, color=black] at (1.4,4.5) {$A_1$};
           \node [left, color=black] at (-1.4,6) {$L_{2}$};
           \node [left, color=black] at (-1.4,7) {$L_{3}$};
           \node [above, color=black] at (0,8) {$A_2$};

           \node [right, color=black] at (0,0) {\small{$0$}};  
           \node [right, color=black] at (0,1.5) {\small{$\frac{3}{2}$}}; 
            \node [left, color=black] at (0,3) {$2$}; 
             \node [left, color=black] at (0,4.5) {\small{$\frac{13}{6}$}}; 
              \node [right, color=black] at (0,6) {\small{$\frac{7}{3}$}};
               \node [right, color=black] at (0,7) {\small{$\frac{29}{12}$}};
                \node [below, color=black] at (-1.3,1.4) {\small{$\infty$}};
                \node [below, color=black] at (-1.3,5.9) {\small{$\infty$}};
                \node [below, color=black] at (-1.3,6.9) {\small{$\infty$}};
                
                 \node [right, color=black] at (0,0.75) {\small{$1$}};
                 \node [right, color=black] at (0,2.25) {\small{$2$}};
                  \node [right, color=black] at (0,3.75) {\small{$2$}};
                  \node [right, color=black] at (0,5.25) {\small{$2$}};
                  \node [right, color=black] at (0,6.5) {\small{$6$}};
                   \node [right, color=black] at (0,7.5) {\small{$12$}};
                   \node [above, color=black] at (-0.7,1.5) {\small{$1$}};
                \node [above, color=black] at (-0.7,6) {\small{$2$}};
                \node [above, color=black] at (-0.7,7) {\small{$6$}};
                \node [above, color=black] at (0.7,3) {\small{$2$}};
                \node [above, color=black] at (0.7,4.5) {\small{$6$}}; 

 \draw[->][thick, color=black](1.5,6) .. controls (2.7,6.2) ..(3.7,6);
 
\end{scope}
  \begin{scope}[shift={(-4,0)}]

     \draw [->, color=black,  line width=1.5pt](0,1.5) -- (0,0);
     \draw [->, color=black,  line width=1.5pt](0,1.5) -- (-1.5,1.5);
      \node[draw,circle, inner sep=1.5pt,color=black, fill=black] at (0,1.5){};
      
       \draw [-, color=black,  line width=1.5pt](0,2) -- (0,6.5);
        \draw [->, color=black,  line width=1.5pt](0.5,3.5) -- (2,3.5);
         \node[draw,circle, inner sep=1.5pt,color=black, fill=black] at (0,6.5){};
               \node[draw,circle, inner sep=1.5pt,color=black, fill=black] at (0,5){};
         \node[draw,circle, inner sep=1.5pt,color=black, fill=black] at (0,3.5){};
         \node[draw,circle, inner sep=1.5pt,color=black, fill=black] at (0.5,3.5){};
         \node[draw,circle, inner sep=1.5pt,color=black, fill=black] at (0,2){};
         \draw [->, color=black,  line width=1.5pt](0.5,5) -- (2,5);
           \node[draw,circle, inner sep=1.5pt,color=black, fill=black] at (0.5,5){};
             \node[draw,circle, inner sep=1.5pt,color=black, fill=black] at (0,6.5){};
               \node[draw,circle, inner sep=1.5pt,color=black, fill=black] at (0,8.5){};
               \node[draw,circle, inner sep=1.5pt,color=black, fill=black] at (0,7){};
                  \draw [->, color=black,  line width=1.5pt](0,6.5) -- (-1.5,6.5);
                   \draw [-, color=black,  line width=1.5pt](0,7) -- (0,8.5);
                   \draw [->, color=black,  line width=1.5pt](0,8.5) -- (-1.5,8.5);     
                   \draw [->, color=black,  line width=1.5pt](0,6.5) -- (-1.5,6.5);
                    \draw [->, color=black,  line width=1.5pt](0,9) -- (0,10);
                    \node[draw,circle, inner sep=1.5pt,color=black, fill=black] at (0,9){};

       \node [below, color=black] at (0,0) {$L$};
         \node [left, color=black] at (-1.4,1.5) {$L_{1}$};
          \node [right, color=black] at (2,3.5) {$A_3$};
           \node [right, color=black] at (2,5) {$A_1$};
           \node [left, color=black] at (-1.4,6.5) {$L_{2}$};
           \node [left, color=black] at (-1.4,8.5) {$L_{3}$};
           \node [above, color=black] at (0,10) {$A_2$};

           \node [right, color=blue] at (0,1.5) {\small{$\mathbf{\frac{3}{2}}$}}; 
            \node [left, color=blue] at (0,3.5) {$\mathbf{1}$}; 
             \node [left, color=blue] at (0,5) {\small{$\mathbf{\frac{4}{3}}$}}; 
              \node [right, color=blue] at (0,6.5) {\small{$\mathbf{\frac{5}{3}}$}};
               \node [right, color=blue] at (0,8.5) {\small{$\mathbf{\frac{1}{2}}$}};

 \draw[->][thick, color=blue](-0.2,2) .. controls (-0.4,2) ..(-0.2,1.6); 
 \draw[->][thick, color=blue](-0.2,7) .. controls (-0.4,7) ..(-0.2,6.6); 
 \draw[->][thick, color=blue](-0.2,9) .. controls (-0.4,9) ..(-0.2,8.6);
 
  \draw[->][thick, color=blue](0.4,3.5) .. controls (0.3,3.7) ..(0.1,3.5);
    \draw[->][thick, color=blue](0.4,5) .. controls (0.3,5.2) ..(0.1,5);

\draw[->][thick, color=black](-1,1) .. controls (-2,-1) ..(-7.5,-4);

  \end{scope}
  
 \begin{scope}[shift={(-10,-10)}] 
    \draw [->, color=black, thick](1,0) -- (0,0);
    \draw [-, color=black, thick](0,0) -- (-1,0);
    \draw [-, color=black, thick](0,1.5) -- (-1,0);
      \draw [-, color=black, thick](0,1.5) -- (1,0);
       \draw [-, color=black, thick](-0.5,3/4) -- (0.5,3/4);
       \draw [-, color=black, thick](-1,0) -- (0.5,3/4);

  \node[draw,circle, inner sep=1pt,color=black, fill=black] at (-1,0){};
   \node[draw,circle, inner sep=1pt,color=black, fill=black] at (1,0){};
     \node[draw,circle, inner sep=1pt,color=black, fill=black] at (0,1.5){};
      \node[draw,circle, inner sep=1pt,color=black, fill=black] at (-0.5,3/4){};
       \node[draw,circle, inner sep=1pt,color=black, fill=black] at (0.5,3/4){};

   \node [below, color=black] at (-1,0) {$L_{1}$};
\node [below, color=black] at (1,0) {$L$};

    \draw [->, color=black, thick](0,2) -- (-1.5,2);
   \draw [-, color=black, thick](-1.5,2) -- (-3,2);
     \draw [-, color=black, thick](0,2) -- (-1.5,4);
      \draw [-, color=black, thick](-1.5,4) -- (-3,2);
  
     \draw [-, color=black, thick](-3/4,3) -- (-3,2);
     \draw [-, color=black, thick](-3/4,3) -- (-2.25,3);
     \draw [-, color=black, thick](-2.25,5) -- (-2.25, 3);
      \draw [-, color=black, thick](-3/4,5) -- (-3/4, 3);
      \draw [-, color=black, thick](-3/4,5) -- (-1.5, 4);
      \draw [-, color=black, thick](-2.25,5) -- (-1.5, 4);
       
          \node[draw,circle, inner sep=1pt,color=black, fill=black] at (0,2){};
          \node[draw,circle, inner sep=1pt,color=black, fill=black] at (-3,2){};
          
         \node[draw,circle, inner sep=1pt,color=black, fill=black] at (-1.5,4){};
         \node[draw,circle, inner sep=1pt,color=black, fill=black] at (-3/4,3){};
         \node[draw,circle, inner sep=1pt,color=black, fill=black] at (-3/4,5){};
        \node[draw,circle, inner sep=1pt,color=black, fill=black] at (-2.25,3){};
        \node[draw,circle, inner sep=1pt,color=black, fill=black] at (-2.25,5){};
        
       \draw [->, color=black, thick](-0.15,3.3) -- (0.5,3.8);
       \node [right, color=black] at (0.5,3.8) {$A_3$};
       \node[draw,circle, inner sep=1pt,color=black, fill=black] at (-0.15,3.3){};
       
       \draw [->, color=black, thick](-0.15,5.3) -- (0.5,5.8);
       \node [right, color=black] at (0.5,5.8) {$A_1$};
       \node[draw,circle, inner sep=1pt,color=black, fill=black] at (-0.15,5.3){};

 \node [above, color=black] at (-3.75,8.5) {$A_{2}$};
  \node[draw,circle, inner sep=1pt,color=black, fill=black] at (-3.75,7.5){};
  \draw [->, color=black, thick](-3.75,7.5) -- (-3.75,8.5);
  
  \node [below, color=black] at (-3,2) {$L_2$};
   \node [below, color=black] at (-4.75,5.5) {$L_3$};

\draw [->, color=black, thick](-2.75,5.5)--(-3.75,5.5);
 \draw [-, color=black, thick](-3.75,5.5) -- (-4.75,5.5);
 \draw [-, color=black, thick](-3.75,7) -- (-4.75,5.5);
 \draw [-, color=black, thick](-3.75,7) -- (-2.75,5.5);
  \draw [-, color=black, thick](-2.75,5.5)-- (-4.25,6.25);
  \node[draw,circle, inner sep=1pt,color=black, fill=black] at (-4.75,5.5){};
   \node[draw,circle, inner sep=1pt,color=black, fill=black] at (-2.75,5.5){};
    \node[draw,circle, inner sep=1pt,color=black, fill=black] at (-3.75,7){};
   \node[draw,circle, inner sep=1pt,color=black, fill=black] at (-4.25,6.25){};

 \draw[->][thick, color=blue](-3.75,7.4) .. controls (-3.95,7.4) ..(-3.75,7.1);
\draw[->][thick, color=blue](-0.2,5.35) .. controls (-0.4,5.5) ..(-0.7,5.1);
\draw[->][thick, color=blue](-0.2,3.35) .. controls (-0.4,3.5) ..(-0.7,3.1);
\draw[->][thick, color=blue](-2.75,5.4) .. controls (-2.95,5.4) ..(-2.3,5.1);
 
  \draw[->][thick, color=blue](-0.2,1.9) .. controls (-0.4,1.9) ..(-0.2,1.5);
   \draw[->][thick, color=black](1.5,4) .. controls (3,3.4) ..(4.5,4);

   \end{scope}

\begin{scope}[shift={(0,-8)}]
  \draw [->, color=black, thick](1,0) -- (0,0);
    \draw [-, color=black, thick](0,0) -- (-1,0);
    \draw [-, color=black, thick](0,1.5) -- (-1,0);
      \draw [-, color=black, thick](0,1.5) -- (1,0);
       \draw [-, color=black, thick](-0.5,3/4) -- (0.5,3/4);
       \draw [-, color=black, thick](-1,0) -- (0.5,3/4);
       \draw [-, color=black, thick](-5.2,4.2) -- (-3.7,3.2);

        \draw [->, color=black, thick](0,1.5) -- (-1.5,1.5);
   \draw [-, color=black, thick](-1.5,1.5) -- (-3,1.5);
     \draw [-, color=black, thick](0,1.5) -- (-1.5,4);
      \draw [-, color=black, thick](-1.5,4) -- (-3,1.5);
      \draw [->, color=black, thick](-1.5,4) -- (-1.5,5);
     \draw [-, color=black, thick](-3/4,11/4) -- (-3,1.5);
     \draw [-, color=black, thick](-3/4,11/4) -- (-2.5,2.34);
     \draw [-, color=black, thick](-2,3.17) -- (-2.5,2.34);
       \draw [-, color=black, thick](-3/4,11/4) -- (-2,3.17);
       \draw [->, color=black, thick](-3/4,11/4) -- (1/4,15/4);
        \draw [-, color=black, thick](-2,3.17) -- (-1.5,4);
        \draw [-, color=black, thick](-3,1.5) -- (-2.5,2.34);
        \draw [-, color=black, thick](-3,1.5) -- (-2.5,2.34);
        \draw [-, color=black, thick](-3.7,3.2) -- (-2,3.17);
        \draw [-, color=black, thick](-3.7,3.2) -- (-2.5,2.34);
        \draw [->, color=black, thick](-3.7,3.2) -- (-4.7,3.2);
        \draw [-, color=black, thick](-5.7,3.2) -- (-4.7,3.2);
        \draw [->, color=black, thick](-4.7,5.2) -- (-4.7,6.2);
        \draw [-, color=black, thick](-3.7,3.2) -- (-4.7,5.2);
        \draw [-, color=black, thick](-5.7,3.2) -- (-4.7,5.2);
       
  \node[draw,circle, inner sep=1pt,color=black, fill=black] at (-1,0){};
   \node[draw,circle, inner sep=1pt,color=black, fill=black] at (1,0){};
     \node[draw,circle, inner sep=1pt,color=black, fill=black] at (0,1.5){};
      \node[draw,circle, inner sep=1pt,color=black, fill=black] at (-0.5,3/4){};
       \node[draw,circle, inner sep=1pt,color=black, fill=black] at (0.5,3/4){};
        \node[draw,circle, inner sep=1pt,color=black, fill=black] at (-3,1.5){};
         \node[draw,circle, inner sep=1pt,color=black, fill=black] at (-1.5,4){};
         \node[draw,circle, inner sep=1pt,color=black, fill=black] at (-3/4,11/4){};
         \node[draw,circle, inner sep=1pt,color=black, fill=black] at (-2.5,2.34){};
         \node[draw,circle, inner sep=1pt,color=black, fill=black] at (-2,3.17){};
          \node[draw,circle, inner sep=1pt,color=black, fill=black] at (-3.7,3.2){};
          \node[draw,circle, inner sep=1pt,color=black, fill=black] at (-5.7,3.2){};
           \node[draw,circle, inner sep=1pt,color=black, fill=black] at (-4.7,5.2){};
           \node[draw,circle, inner sep=1pt,color=black, fill=black] at (-5.2,4.2){};
       
   \node [below, color=black] at (-1,0) {$L_{1}$};
\node [below, color=black] at (1,0) {$L$};
\node [below, color=black] at (-3,1.5) {$L_{2}$};
\node [right, color=black] at (0.5,3.8) {$A_3$};
\node [above, color=black] at (-1.5,5) {$A_1$};
 \node [below, color=black] at (-5.7,3.2) {$L_{3}$};
 \node [above, color=black] at (-4.7,6.2) {$A_2$};

 \end{scope}

  \end{tikzpicture}
\end{center}
 \caption{Constructing a lotus from the Eggers-Wall tree on the right of Figure \ref{fig:logdistoEW3}}
\label{fig:EWtolotus3}
   \end{figure}


 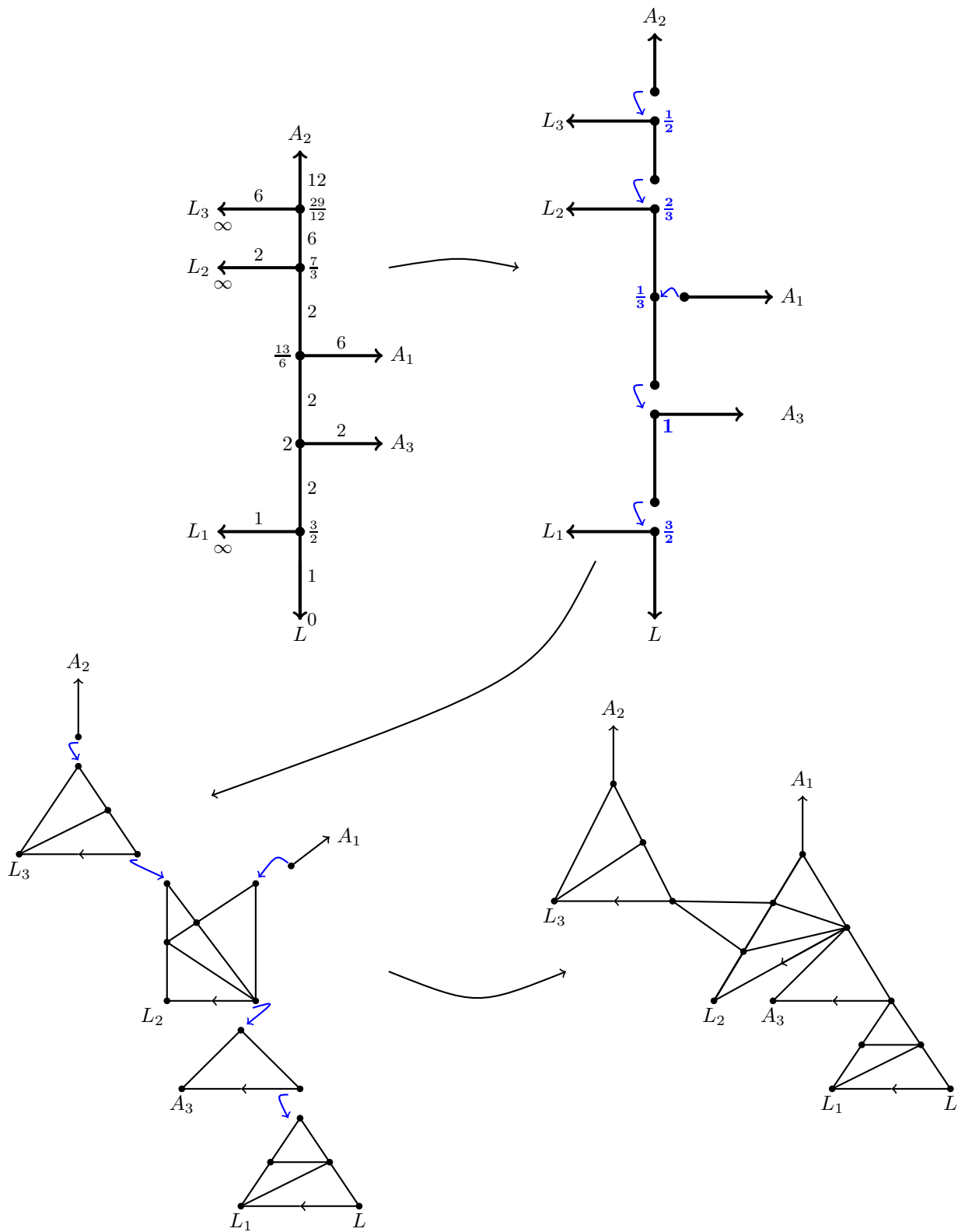
\begin{figure}[h!]
    \begin{center}
\begin{tikzpicture}[scale=1]
  
  \begin{scope}[shift={(-10,0)}]
 \draw [->, color=black,  line width=1.5pt](0,0) -- (0,8);
     \draw [->, color=black,  line width=1.5pt](0,1) -- (0,0);
      \node[draw,circle, inner sep=1.5pt,color=black, fill=black] at (0,1.5){};
      \node[draw,circle, inner sep=1.5pt,color=black, fill=black] at (0,3){};
      \node[draw,circle, inner sep=1.5pt,color=black, fill=black] at (0,4.5){};
      \node[draw,circle, inner sep=1.5pt,color=black, fill=black] at (0,6){};
      \node[draw,circle, inner sep=1.5pt,color=black, fill=black] at (0,7){};
      
      \draw [->, color=black,  line width=1.5pt](0,1.5) -- (-1.4,1.5);
       \draw [->, color=black,  line width=1.5pt](0,3) -- (1.4,3);
       \draw [->, color=black,  line width=1.5pt](0,4.5) -- (1.4,4.5);
       \draw [->, color=black,  line width=1.5pt](0,6) -- (-1.4,6);
       \draw [->, color=black,  line width=1.5pt](0,7) -- (-1.4,7);
      
       \node [below, color=black] at (0,0) {$L$};
         \node [left, color=black] at (-1.4,1.5) {$L_{1}$};
          \node [right, color=black] at (1.4,3) {$A_3$};
           \node [right, color=black] at (1.4,4.5) {$A_1$};
           \node [left, color=black] at (-1.4,6) {$L_{2}$};
           \node [left, color=black] at (-1.4,7) {$L_{3}$};
           \node [above, color=black] at (0,8) {$A_2$};

           \node [right, color=black] at (0,0) {\small{$0$}};  
           \node [right, color=black] at (0,1.5) {\small{$\frac{3}{2}$}}; 
            \node [left, color=black] at (0,3) {$2$}; 
             \node [left, color=black] at (0,4.5) {\small{$\frac{13}{6}$}}; 
              \node [right, color=black] at (0,6) {\small{$\frac{7}{3}$}};
               \node [right, color=black] at (0,7) {\small{$\frac{29}{12}$}};
                \node [below, color=black] at (-1.3,1.4) {\small{$\infty$}};
                \node [below, color=black] at (-1.3,5.9) {\small{$\infty$}};
                \node [below, color=black] at (-1.3,6.9) {\small{$\infty$}};
                
                 \node [right, color=black] at (0,0.75) {\small{$1$}};
                 \node [right, color=black] at (0,2.25) {\small{$2$}};
                  \node [right, color=black] at (0,3.75) {\small{$2$}};
                  \node [right, color=black] at (0,5.25) {\small{$2$}};
                  \node [right, color=black] at (0,6.5) {\small{$6$}};
                   \node [right, color=black] at (0,7.5) {\small{$12$}};
                   \node [above, color=black] at (-0.7,1.5) {\small{$1$}};
                \node [above, color=black] at (-0.7,6) {\small{$2$}};
                \node [above, color=black] at (-0.7,7) {\small{$6$}};
                \node [above, color=black] at (0.7,3) {\small{$2$}};
                \node [above, color=black] at (0.7,4.5) {\small{$6$}}; 

 \draw[->][thick, color=black](1.5,6) .. controls (2.7,6.2) ..(3.7,6);
 
\end{scope}
  \begin{scope}[shift={(-4,0)}]

     \draw [->, color=black,  line width=1.5pt](0,1.5) -- (0,0);
     \draw [->, color=black,  line width=1.5pt](0,1.5) -- (-1.5,1.5);
      \node[draw,circle, inner sep=1.5pt,color=black, fill=black] at (0,1.5){};
      
       \draw [-, color=black,  line width=1.5pt](0,2) -- (0,3.5);
       \draw [-, color=black,  line width=1.5pt](0,4) -- (0,7);
        \draw [->, color=black,  line width=1.5pt](0,3.5) -- (1.5,3.5);
        \node[draw,circle, inner sep=1.5pt,color=black, fill=black] at (0,4){};

               \node[draw,circle, inner sep=1.5pt,color=black, fill=black] at (0,5.5){};
         \node[draw,circle, inner sep=1.5pt,color=black, fill=black] at (0,3.5){};
         \node[draw,circle, inner sep=1.5pt,color=black, fill=black] at (0,2){};
         \draw [->, color=black,  line width=1.5pt](0.5,5.5) -- (2,5.5);
           \node[draw,circle, inner sep=1.5pt,color=black, fill=black] at (0.5,5.5){};
               \node[draw,circle, inner sep=1.5pt,color=black, fill=black] at (0,8.5){};
               \node[draw,circle, inner sep=1.5pt,color=black, fill=black] at (0,7){};
                \node[draw,circle, inner sep=1.5pt,color=black, fill=black] at (0,7.5){};
                  \draw [->, color=black,  line width=1.5pt](0,7) -- (-1.5,7);
                   \draw [-, color=black,  line width=1.5pt](0,7.5) -- (0,8.5);
                   \draw [->, color=black,  line width=1.5pt](0,8.5) -- (-1.5,8.5);     
                    \draw [->, color=black,  line width=1.5pt](0,9) -- (0,10);
                    \node[draw,circle, inner sep=1.5pt,color=black, fill=black] at (0,9){};

       \node [below, color=black] at (0,0) {$L$};
         \node [left, color=black] at (-1.4,1.5) {$L_{1}$};
          \node [right, color=black] at (2,3.5) {$A_3$};
           \node [right, color=black] at (2,5.5) {$A_1$};
           \node [left, color=black] at (-1.4,7) {$L_{2}$};
           \node [left, color=black] at (-1.4,8.5) {$L_{3}$};
           \node [above, color=black] at (0,10) {$A_2$};

           \node [right, color=blue] at (0,1.5) {\small{$\mathbf{\frac{3}{2}}$}}; 
            \node [right, color=blue] at (0,3.3) {$\mathbf{1}$}; 
             \node [left, color=blue] at (0,5.5) {\small{$\mathbf{\frac{1}{3}}$}}; 
              \node [right, color=blue] at (0,7) {\small{$\mathbf{\frac{2}{3}}$}};
               \node [right, color=blue] at (0,8.5) {\small{$\mathbf{\frac{1}{2}}$}};

 \draw[->][thick, color=blue](-0.2,2) .. controls (-0.4,2) ..(-0.2,1.6); 
 \draw[->][thick, color=blue](-0.2,4) .. controls (-0.4,4) ..(-0.2,3.6); 
 \draw[->][thick, color=blue](-0.2,7.5) .. controls (-0.4,7.5) ..(-0.2,7.1); 
 \draw[->][thick, color=blue](-0.2,9) .. controls (-0.4,9) ..(-0.2,8.6);
 
    \draw[->][thick, color=blue](0.4,5.5) .. controls (0.3,5.7) ..(0.1,5.5);

\draw[->][thick, color=black](-1,1) .. controls (-2,-1) ..(-7.5,-3);

  \end{scope}
  
 \begin{scope}[shift={(-10,-10)}] 
    \draw [->, color=black, thick](1,0) -- (0,0);
    \draw [-, color=black, thick](0,0) -- (-1,0);
    \draw [-, color=black, thick](0,1.5) -- (-1,0);
      \draw [-, color=black, thick](0,1.5) -- (1,0);
       \draw [-, color=black, thick](-0.5,3/4) -- (0.5,3/4);
       \draw [-, color=black, thick](-1,0) -- (0.5,3/4);

  \node[draw,circle, inner sep=1pt,color=black, fill=black] at (-1,0){};
   \node[draw,circle, inner sep=1pt,color=black, fill=black] at (1,0){};
     \node[draw,circle, inner sep=1pt,color=black, fill=black] at (0,1.5){};
      \node[draw,circle, inner sep=1pt,color=black, fill=black] at (-0.5,3/4){};
       \node[draw,circle, inner sep=1pt,color=black, fill=black] at (0.5,3/4){};

   \node [below, color=black] at (-1,0) {$L_{1}$};
\node [below, color=black] at (1,0) {$L$};

    \draw [->, color=black, thick](0,2) -- (-1,2);
   \draw [-, color=black, thick](-1,2) -- (-2,2);
 
     \draw [->, color=black, thick](-3/4,3.5) --(-1.5,3.5); 
      \draw [-, color=black, thick](-2.25,3.5)--(-1.5,3.5);
     \draw [-, color=black, thick](-2.25,5.5) -- (-2.25, 3.5);
      \draw [-, color=black, thick](-3/4,5.5) -- (-3/4, 3.5);
      \draw [-, color=black, thick](-3/4,3.5) -- (-2.25, 5.5);
      \draw [-, color=black, thick](-2.25,4.5) -- (-3/4, 3.5);
      \draw [-, color=black, thick](-2.25,4.5) -- (-3/4, 5.5);
      \node[draw,circle, inner sep=1pt,color=black, fill=black] at (-7/4,29/6){};

\draw [-, color=black, thick](-1,3)--(0,2);
          \draw [-, color=black, thick](-1,3)--(-2,2);
          \node[draw,circle, inner sep=1pt,color=black, fill=black] at (0,2){};
         \node[draw,circle, inner sep=1pt,color=black, fill=black] at (-2,2){};
         \node[draw,circle, inner sep=1pt,color=black, fill=black] at (-1,3){};
           \node [below, color=black] at (-2,2) {$A_3$};

         \node[draw,circle, inner sep=1pt,color=black, fill=black] at (-2.25,4.5){};
         \node[draw,circle, inner sep=1pt,color=black, fill=black] at (-3/4,3.5){};
         \node[draw,circle, inner sep=1pt,color=black, fill=black] at (-3/4,5.5){};
        \node[draw,circle, inner sep=1pt,color=black, fill=black] at (-2.25,3.5){};
        \node[draw,circle, inner sep=1pt,color=black, fill=black] at (-2.25,5.5){};

       \draw [->, color=black, thick](-0.15,5.8) -- (0.5,6.3);
       \node [right, color=black] at (0.5,6.3) {$A_1$};
       \node[draw,circle, inner sep=1pt,color=black, fill=black] at (-0.15,5.8){};

 \node [above, color=black] at (-3.75,9) {$A_{2}$};
  \node[draw,circle, inner sep=1pt,color=black, fill=black] at (-3.75,8){};
   \node[draw,circle, inner sep=1pt,color=black, fill=black] at (-3.75,7.5){};
  \draw [->, color=black, thick](-3.75,8) -- (-3.75,9);
  
  \node [below, color=black] at (-2.5,3.5) {$L_2$};
   \node [below, color=black] at (-4.75,6) {$L_3$};

\draw [->, color=black, thick](-2.75,6)--(-3.75,6);
 \draw [-, color=black, thick](-3.75,6) -- (-4.75,6);
 \draw [-, color=black, thick](-3.75,7.5) -- (-4.75,6);
 \draw [-, color=black, thick](-3.75,7.5) -- (-2.75,6);
  \draw [-, color=black, thick](-4.75,6)-- (-3.25,6.75);
  \node[draw,circle, inner sep=1pt,color=black, fill=black] at (-4.75,6){};
   \node[draw,circle, inner sep=1pt,color=black, fill=black] at (-2.75,6){};
   \node[draw,circle, inner sep=1pt,color=black, fill=black] at (-3.25,6.75){};

 \draw[->][thick, color=blue](-3.75,7.9) .. controls (-3.95,7.9) ..(-3.75,7.6);
\draw[->][thick, color=blue](-0.2,5.85) .. controls (-0.4,6) ..(-0.7,5.6);
\draw[->][thick, color=blue](-0.8,3.38) .. controls (-0.4,3.5) ..(-0.9,3.1);
\draw[->][thick, color=blue](-2.75,5.9) .. controls (-2.95,5.9) ..(-2.3,5.6);
 
\draw[->][thick, color=blue](-0.2,1.9) .. controls (-0.4,1.9) ..(-0.2,1.5);
 \draw[->][thick, color=black](1.5,4) .. controls (3,3.4) ..(4.5,4);

   \end{scope}

\begin{scope}[shift={(0,-8)}]
  \draw [->, color=black, thick](1,0) -- (0,0);
    \draw [-, color=black, thick](0,0) -- (-1,0);
    \draw [-, color=black, thick](0,1.5) -- (-1,0);
      \draw [-, color=black, thick](0,1.5) -- (1,0);
       \draw [-, color=black, thick](-0.5,3/4) -- (0.5,3/4);
       \draw [-, color=black, thick](-1,0) -- (0.5,3/4);
       \draw [-, color=black, thick](-4.2,4.2) -- (-5.7,3.2);

        \draw [->, color=black, thick](0,1.5) -- (-1,1.5);
   \draw [-, color=black, thick](-1,1.5) -- (-2,1.5);
   
     \draw [-, color=black, thick](0,1.5) -- (-1.5,4);
      \draw [-, color=black, thick](-1.5,4) -- (-3,1.5);
      \draw [->, color=black, thick](-1.5,4) -- (-1.5,5);
     \draw [-, color=black, thick](-3/4,11/4) -- (-3,1.5);
     \draw [-, color=black, thick](-3/4,11/4) -- (-2.5,2.34);
     \draw [-, color=black, thick](-2,3.17) -- (-2.5,2.34);
       \draw [-, color=black, thick](-3/4,11/4) -- (-2,3.17);
       \draw [-, color=black, thick](-3/4,11/4) -- (-2,1.5);
        \draw [-, color=black, thick](-2,3.17) -- (-1.5,4);
        \draw [-, color=black, thick](-3,1.5) -- (-2.5,2.34);
        \draw [-, color=black, thick](-3,1.5) -- (-2.5,2.34);
        \draw [-, color=black, thick](-3.7,3.2) -- (-2,3.17);
        \draw [-, color=black, thick](-3.7,3.2) -- (-2.5,2.34);
        \draw [->, color=black, thick](-3.7,3.2) -- (-4.7,3.2);
        \draw [-, color=black, thick](-5.7,3.2) -- (-4.7,3.2);
        \draw [->, color=black, thick](-4.7,5.2) -- (-4.7,6.2);
        \draw [-, color=black, thick](-3.7,3.2) -- (-4.7,5.2);
        \draw [-, color=black, thick](-5.7,3.2) -- (-4.7,5.2);
         \draw [->, color=black, thick](-3/4,11/4) -- (-15/8,17/8);
       
  \node[draw,circle, inner sep=1pt,color=black, fill=black] at (-1,0){};
   \node[draw,circle, inner sep=1pt,color=black, fill=black] at (1,0){};
     \node[draw,circle, inner sep=1pt,color=black, fill=black] at (0,1.5){};
      \node[draw,circle, inner sep=1pt,color=black, fill=black] at (-0.5,3/4){};
       \node[draw,circle, inner sep=1pt,color=black, fill=black] at (0.5,3/4){};
        \node[draw,circle, inner sep=1pt,color=black, fill=black] at (-3,1.5){};
         \node[draw,circle, inner sep=1pt,color=black, fill=black] at (-1.5,4){};
         \node[draw,circle, inner sep=1pt,color=black, fill=black] at (-3/4,11/4){};
         \node[draw,circle, inner sep=1pt,color=black, fill=black] at (-2.5,2.34){};
         \node[draw,circle, inner sep=1pt,color=black, fill=black] at (-2,3.17){};
          \node[draw,circle, inner sep=1pt,color=black, fill=black] at (-3.7,3.2){};
          \node[draw,circle, inner sep=1pt,color=black, fill=black] at (-5.7,3.2){};
           \node[draw,circle, inner sep=1pt,color=black, fill=black] at (-4.7,5.2){};
           \node[draw,circle, inner sep=1pt,color=black, fill=black] at (-4.2,4.2){};
             \node[draw,circle, inner sep=1pt,color=black, fill=black] at (-2,1.5){};
       
   \node [below, color=black] at (-1,0) {$L_{1}$};
\node [below, color=black] at (1,0) {$L$};
\node [below, color=black] at (-3,1.5) {$L_{2}$};
\node [below, color=black] at (-2,1.5) {$A_3$};
\node [above, color=black] at (-1.5,5) {$A_1$};
 \node [below, color=black] at (-5.7,3.2) {$L_{3}$};
 \node [above, color=black] at (-4.7,6.2) {$A_2$};

 \end{scope}

  \end{tikzpicture}
\end{center}
 \caption{Constructing a second lotus from the Eggers-Wall tree on the right of Figure \ref{fig:logdistoEW3}}
\label{fig:EWtolotus4}
   \end{figure}

We need also the notion of {\em trunk decomposition} of a complete Eggers-Wall tree:

\begin{definition}
   \label{def:trunkdec}
  Let $A$ be a reduced plane curve singularity and $L$ be a smooth branch on $(S,O)$ such that $\Theta_{L} (A)$ is complete. A {\bf  trunk decomposition} of $\Theta_{L} (A)$ is a finite set $\boxed{\mathfrak{D}}$ of compact subsegments of $\Theta_L (A)$ called {\bf trunks}, such that:
     \begin{enumerate}
      \item 
         The union of the trunks is equal to $\Theta_L (A)$. 
      \item 
          For every  trunk $\tau \in \mathfrak{D}$, one may denote its ends by $R_\tau,   F_\tau$, such that  
          $R_\tau \prec_L F_\tau$. 
       \item
           For every trunk $\tau \in \mathfrak{D}$, the index function $\de_L$ is constant on the half open segment $(R_\tau F_\tau]$ and $F_\tau$ is a leaf of $\Theta_L (A)$.
  \end{enumerate}
 
 \medskip 
 
 Let $\tau = [R_\tau F_\tau]$ be a trunk of the trunk decomposition $\mathcal{D}$ of a complete Eggers-Wall tree $\Theta_L (A)$. Its {\bf marked points} are $R_\tau, F_\tau$ and the marked points of $\Theta_L (A)$ which belong to $\tau$. We consider on $\tau$ the {\bf renormalized exponent function} $\ex_\tau : \tau \to [0, \infty]$ defined by 
\begin{equation}  \label{ex:renormexp}
    \boxed{{\ex}_{\tau} (P)}  :=  {\de}_L(F_\tau) \cdot ({\ex}_L (P)  -{\ex}_L(R_\tau)). 
\end{equation}
We denote by  $\boxed{\cale _{\tau}} \subset \Qq_{\geq 0} \cup \{\infty\}$ the set of values of  the marked points of $\tau$ by the renormalized exponent function $\ex_{\tau}$.

\end{definition}

Notice that $0 = {\ex}_\tau (R_{\tau})$, $\infty = {\ex}_\tau (F_{\tau})$, and $\{0, \infty \} \subset \cale _{\tau}$.

\medskip 
We are ready to introduce the notion of {\em lotus associated to a trunk decomposition} of a complete Eggers-Wall tree: 

\begin{definition}
   \label{def:lotasstotrunkdec}
   Let $A$ be a reduced plane curve singularity and $L$ a smooth branch on $(S,O)$. Assume that the Eggers-Wall tree $\Theta_{L}(A)$ is complete. Let $\mathfrak{D}$ be a trunk decomposition of $\Theta_L(A)$. 
   Then, the {\bf lotus $\boxed{\Lambda_{ \mathfrak{D}}}$ 
   associated with the trunk decomposition
   $\mathfrak{D}$} is constructed as follows: 
  \begin{enumerate}
    \item
      By definition, there is a canonical bijection between the marked points  of a trunk $\tau$  in the decomposition and the  marked vertices of the abstract lotus  $\Lambda(\cale _{\tau})$. We label corresponding marked points and vertices with the same labels. 
    \item
       Identify all the vertices of $\bigsqcup_{\tau \in \mathfrak{D}} \Lambda(\cale _{\tau})$ which have the same label. 
        The result of this identification is a two dimensional simplicial complex $\Lambda_{\mathfrak{D}}$.
        For each trunk $\tau \in \mathfrak{D}$, we keep the same labels when taking the 
        images of the marked vertices of  $\Lambda(\cale _{\tau})$ in the lotus
        $\Lambda_{\mathfrak{D}}$. 
  \end{enumerate}
\end{definition}

Therefore, to each trunk $\tau$ of $\mathfrak{D}$ corresponds a membrane $\Lambda(\cale_{\tau})$ of $\Lambda_{ \mathfrak{D}}$.

By applying the algorithm of Corollary \ref{cor:lotustoEW} to the lotus $\Lambda_{ \mathfrak{D}}$, we get:

  \begin{proposition}
     \label{prop:reconstrEW}
      Let $\mathfrak{D}$ be a trunk decomposition of a complete Eggers-Wall tree $\Theta_L$. Consider its associated lotus $\Lambda_{\mathfrak{D}}$ in the sense of Definition \ref{def:lotasstotrunkdec}. Then the Eggers-Wall tree associated to $\Lambda_{\mathfrak{D}}$ in the sense of Definition \ref{def:assEWtolotus} is isomorphic to the starting Eggers-Wall tree $\Theta_L$. 
  \end{proposition}

  The following proposition explains how to associate a finite active constellation of crosses to a trunk decomposition of a complete Eggers-Wall tree. It is based on Proposition \ref{prop: lotus-const}, which explains how to associate a finite constellation of crosses to a lotus of the form $\Lambda(\cale)$, for a non-empty finite subset $\cale$ of $\Qq_+$.

\begin{proposition} 
   \label{prop: trunk-const}
    Let $A$ be a reduced plane curve singularity and $L$ a smooth branch on $(S,O)$. Assume that the Eggers-Wall tree $\Theta_{L}(A)$ is complete. 
    Let $\mathfrak{D}$ be a trunk decomposition of $\Theta_L(A)$. 
    Then, there exists a unique minimal finite active constellation of crosses $\boxed{\hat{\calc}_A (\mathfrak{D})}$ adapted to $A$, such that the associated  completion $\hat{A}$ of $A$ is equal to $A$ and such that there is an isomorphism of simplicial complexes between the lotuses $\Lambda ( \hat{\calc}_A ( \mathfrak{D}))$ and $\Lambda_{ \mathfrak{D}}$, which preserves the labeling of the vertices by the branches of $A$. 
\end{proposition}

\begin{example}
    \label{ex:fromEWtolotus}
   In Figures \ref{fig:EWtolotus1}, \ref{fig:EWtolotus2} and \ref{fig:EWtolotus3} are represented the constructions of the lotuses associated to the trunk decompositions of the Eggers-Wall trees on the right sides of Figures \ref{fig:logdistoEW1}, \ref{fig:logdistoEW2} and \ref{fig:logdistoEW3} respectively,  performed according to Definition \ref{def:lotasstotrunkdec}. In the first two cases there is only one possible trunk decomposition. 
   
   But in the third case there is a second possible trunk decomposition $\mathfrak{D}'$, leading to a second lotus, shown in Figure \ref{fig:EWtolotus4}. In order to compare them, we look at  the minimal finite active constellations of crosses $\hat{\calc}$ and $\hat{\calc}'$ corresponding to the lotuses of Figures \ref{fig:logdistoEW3} and \ref{fig:EWtolotus4} respectively, as explained in Proposition \ref{prop: trunk-const}. Both constellations of crosses  have the same underlying constellation. Look at the infinitely near point $O_3$ which belongs to the exceptional divisor of the third blowup in the resolution process. In $\hat{\calc}$ we consider the cross $X_{O_3} = E_3 + L_2$ dual to the base segment $[E_3 L_2]$, while in $\hat{C}'$ we take a different cross,  $X_{O_3}' = E_3 +A_3$, which is dual to the base segment $[E_3 A_3]$. 

   The blue numbers of Figures \ref{fig:EWtolotus1}--\ref{fig:EWtolotus4} are the renormalized exponents $\ex_{\tau}(P)$, computed using their defining formula \eqref{ex:renormexp}. For instance, the renormalized exponents $\frac{4}{3}$,  $\frac{1}{2}$ and $\frac{2}{3}$ in Figures \ref{fig:EWtolotus2}, \ref{fig:EWtolotus3} and \ref{fig:EWtolotus4} respectively are computed as:
     \begin{eqnarray*}
        2\cdot \left(\frac{13}{6} - \frac{3}{2} \right) &= &   \dfrac{4}{3}, \\
        6\cdot \left(\frac{29}{12} - \frac{7}{3} \right) &=  &  \dfrac{1}{2}, \\
        2\cdot \left(\frac{7}{3} - 2 \right) &= &   \dfrac{2}{3}. 
      \end{eqnarray*}
\end{example}

\begin{remark}
  \label{rem:alternEWtolot}
    Figures \ref{fig:EWtolotus1}, \ref{fig:EWtolotus2} concern completions of branches $A$, obtained by adding one branch $L_i$ for each discontinuity point of the index function  ${\de}_L$ on the Eggers-Wall segment $\Theta_L(A)$. In this case each trunk of the corresponding trunk decomposition is decorated by one rational number. An alternative way to compute this sequence of rational numbers was explained in \cite[Examples 1.6.32--33]{GBGPPP 20}. In the example of Eggers-Wall tree shown on the left of Figure \ref{fig:EWtolotus2}, it  consists in rewriting as follows the finite Newton-Puiseux series whose exponents are the characteristic exponents of $A$: 

\begin{eqnarray*}
   x^{\tfrac{3}{2}} + x^{\tfrac{13}{6}} = x^{\tfrac{3}{2}}\left( 1 + x^{\tfrac{13}{6}- \tfrac{3}{2}}\right) =x^{\tfrac{3}{2}}\left( 1 + x^{\tfrac{4}{6}}\right) =
    x^{\boxed{\tfrac{3}{2}}}\left( 1 + x^{\tfrac{1}{2}\boxed{\tfrac{4}{3}}}\right).
\end{eqnarray*}    
    Note that the surrounded exponents are exactly the rational numbers decorating the trunks of Figure \ref{fig:EWtolotus2}. As a second example of this method, let us perform the analogous computation starting from the Newton-Puiseux series defining the branch $A_2$ of Example \ref{ex:EWtoseries}:
    \[  \begin{array}{ll}
   x^{\tfrac{3}{2}} + x^{\tfrac{7}{3}} + x^{\tfrac{29}{12}} & = x^{\tfrac{3}{2}}\left( 1 + x^{\tfrac{7}{3}- \tfrac{3}{2}} + x^{\tfrac{29}{12}- \tfrac{3}{2}}\right) =  x^{\tfrac{3}{2}}\left( 1 + x^{\tfrac{5}{6}} + x^{\tfrac{11}{12}}\right)   = 
   x^{\tfrac{3}{2}}\left( 1 + x^{\tfrac{5}{6}}\left(1 + x^{\tfrac{11}{12} - \tfrac{5}{6}}\right)\right) \\
    & = x^{\tfrac{3}{2}}\left( 1 + x^{\tfrac{5}{6}}\left(1 + x^{\tfrac{1}{12}} \right) \right)
    = x^{\boxed{\tfrac{3}{2}}}\left( 1 + x^{\tfrac{1}{2}\boxed{\tfrac{5}{3}}} \left(1 + x^{\tfrac{1}{6}\boxed{\tfrac{1}{2}}}\right) \right).
  \end{array} \] 
\end{remark}

\medskip
\section{Computation of intersection numbers from Eggers-Wall trees}
\label{sec:intfromEW}

In Section \ref{sec:ovint} we explained one method of computation of the intersection numbers of pairs of branches of a plane curve singularity $A$, using a corresponding lotus. In this section, we explain how to compute them using the Eggers-Wall tree associated to the lotus  in Section \ref{sec:EWfromlot}. Besides the exponent function $\ex_L$ and the index function $\de_L$, the Eggers-Wall tree $\Theta_{L} (A)$ of a reduced plane curve singularity $A$ is endowed with a third function, the {\em contact complexity function} $\ic_L$, which is determined by the first two functions. It is $\ic_L$ which allows to determine easily the intersection numbers of the branches of $A$ (see Theorem \ref{thm:intfromEW}). 
\medskip

\begin{definition} 
   \label{def:contact complexity} 
    Let $A \hookrightarrow (S,O)$ be a reduced plane curve singularity and $L$ be a smooth branch  on $S$. 
    The {\bf contact complexity}  function  
          \[ {\ic}_L : \Theta_L (A) \to [0, \infty]  \] 
     is defined at a point $P \in \Theta_L (A)$ by the formula 
        \begin{equation}  
          \label{eq:intcoefint}
       \boxed{{\ic}_L(P)} := \int_L^{P} \frac{d \: {\ex}_L}{{\de}_L}. 
   \end{equation}
\end{definition}

  One may express  the value of the integral \eqref{eq:intcoefint} in terms of the characteristic exponents of the branches of $A$ and of the exponent ${\ex}_L(P)$. Take a branch $A_l$ of $A$ such that $P \in [L  A_l]$. Denote by $\boxed{\alpha_1} <  \dots < \boxed{\alpha_g}$ the characteristic exponents of  $A_l$. We set $\boxed{\alpha_0} := 0$ and $\boxed{\alpha_{g+1}} := \infty$. Define  $P_j :=  \ex_L^{-1}(\alpha_j)$, for each $j \in \{ 0, \dots, g+1 \}$. The index function ${\de}_L$ has value  $\boxed{{\de}_{j}} $ on the half-open segment  $ ( P_j  P_{j+1}]$. Denote by $\boxed{r} \in \{0, \dots,  g\}$ the unique integer such that  $P\in [P_r P_{r+1} ) $. Then the contact complexity ${\ic}_L(P)$ becomes, by formula \eqref{eq:intcoefint}:  
     \begin{equation} 
        \label{eq:intcoefint2}
     {\ic}_L(P) =      
      \left(
    \sum_{j =1}^{r} 
    \frac{\alpha_j - \alpha_{j-1}}{{\de}_{ j-1}}
    \right)  + 
    \frac{{\ex}_L(P)- \alpha_r}{{\de}_{r}}.
     \end{equation}

\begin{example}
   We give the value of the contact complexity function $\ic_L$ on the ramification point $E_5$ of the tree $\Theta_L (L_3 + A_2)$, viewed in the Eggers-Wall tree on the right of Figure \ref{fig:logdistoEW3}. The index function has constant value $1$ on the interval $[L E_1]$, value $2$ on the interval $(E_1 E_4]$ and value $6$ on the interval $(E_4 E_5]$. Therefore, by formula \eqref{eq:intcoefint2} we have: 
 \begin{eqnarray*}
   \ic_L (E_5) & = & \frac{1}{1} (\ex_L (E_1) - \ex_{L} (L)) +  \frac{1}{2} (\ex_L (E_4) - \ex_{L} (E_1)) + \frac{1}{6} (\ex_L (E_5) - \ex_{L} (E_4))  \\ 
    & = & 1\left( \frac{3}{2}-0 \right) + \frac{1}{2} \left( \frac{7}{3} -\frac{3}{2} \right) +\frac{1}{6} \left( \frac{29}{12} - \frac{7}{3}\right)=   \frac{139}{72}.
 \end{eqnarray*}  
\end{example}

In order to compute the intersection number $(A_l \cdot A_m)_O$ of two branches of $A$, we need the notion of {\em center of a tripod}:

\begin{definition}
   \label{def:centertripod}
  Let $A_l$ and $A_m$ be two branches of $A$ different from $L$. The {\bf center $\boxed{\langle L, A_l, A_m \rangle}$ of the tripod}   determined by $A_l, A_m$ and $L$ is the ramification point  of the tree $\Theta_L ( A_l + A_m)$ if $A_l$ and $A_m$ are different. If instead $A_l =A_m$, we set $\langle L, A_l, A_m \rangle:=A_l$. 
\end{definition}

\medskip 
  The center $\langle L, A_l, A_m \rangle$ of the tripod determined by the points $A_l, A_m, L$ of $\Theta_L(A)$ may be also seen as the infimum of the points $A_l, A_m$ with respect to the partial order relation $\preceq_L$ determined by the choice of the root $L$ of $\Theta_L(A)$ (see Definition \ref{def:EW}).  The following {\bf tripod formula}, which in different language goes back at least to works of Smith, Stolz and Max Noether, allows to determine the intersection numbers using Eggers-Wall trees (see \cite[Theorem 3.23 and  Corollary 3.26]{GBGPPP 19}):  
  
\begin{theorem}
    \label{thm:intfromEW} 
      If $A_l$ and $A_m$ are two branches of $A$  different from $L$, then: 
     \begin{equation} 
        \label{eq:tripod}
     (A_l \cdot A_m)_O = \de_L(A_l) \cdot  \de_L(A_m) \cdot  {\ic}_L(\langle L, A_l, A_m\rangle).
     \end{equation}
\end{theorem}

Notice that if $A_i$ is a branch of $A$ different from $L$ then $\de_L(A_i) = (L\cdot A_i)_O$. Using this observation, one can reformulate \eqref{eq:tripod} as: 
\begin{equation} \label{eq: tripod2}
   {\ic}_L(\langle L, A_l, A_m\rangle) = 
   \frac{(A_l \cdot A_m)_O}{(A_l\cdot  L)_O \cdot (A_m\cdot  L)_O}. 
\end{equation}

\begin{remark} \label{rem:ultrametric}
    P\l oski proved in \cite{Pl 85} that for any three pairwise distinct  branches $A_l$, $A_m$ and $A_k$, in the sequence $\frac{(A_l \cdot A_m)_O}{(A_l\cdot  L)_O \cdot (A_m\cdot  L)_O}$, $\frac{(A_l \cdot A_k)_O}{(A_l\cdot  L)_O \cdot (A_k\cdot  L)_O}$ and $\frac{(A_m \cdot A_k)_O}{(A_m\cdot  L)_O \cdot (A_k\cdot  L)_O}$ there are two terms equal and the third one is not less than the two equal terms (see also \cite{GB-P 15}, in which these numbers were called {\em logarithmic distances}). P\l oski's result may be also expressed as the statement that the function
    \begin{equation}\label{eq:UL}
        \boxed{U_{L}} (A_l, A_m) :=\frac{(A_l\cdot  L)_O \cdot (A_m\cdot  L)_O}{(A_l \cdot A_m)_O}
    \end{equation} 
    defines an ultrametric on the set of branches on $(S,O)$ different from $L$. In \cite{GBGPPP 19} and \cite{GBGPPPR 19} analogous maps were studied in the broader context of branches traced on arbitrary normal surface singularities and those surface singularities for which such maps define ultrametrics were identified.
\end{remark}

   \begin{example}  
       \label{ex:fromexptocont}
       In Figures \ref{fig:exptocontact1}, \ref{fig:exptocontact2}, \ref{fig:exptocontact3}    are indicated the passages from the Eggers-Wall trees obtained in Figures \ref{fig:logdistoEW1}, \ref{fig:logdistoEW2} and \ref{fig:logdistoEW3}, which are decorated with indices ${\de}_L$ and exponents $\ex_L$, to the same trees decorated with indices ${\de}_L$  and contact complexities ${\ic}_L$.    
       Let us apply the tripod formula of Theorem \ref{thm:intfromEW} to determine the intersection multiplicities of the pairs of branches $A_i$ shown on Figure \ref{fig:logdistoEW3}. We get $(A_1 \cdot A_2)_O = 6\cdot 12 \cdot \frac{11}{6}=132$, 
       $(A_1 \cdot A_3)_O= 6 \cdot 2 \cdot \frac{7}{4} = 21$ and 
       $(A_2 \cdot A_3)_O = 12 \cdot 2 \cdot \frac{7}{4}= 42$.
   \end{example}


 \begin{figure}[h!]
    \begin{center}
\begin{tikzpicture}[scale=1]
  
  \begin{scope}[shift={(0,0)}]
  
     \draw [->, color=black,  line width=1.5pt](0,0) -- (0,3);
     \draw [->, color=black,  line width=1.5pt](0,1) -- (0,0);
      \node[draw,circle, inner sep=1.5pt,color=black, fill=black] at (0,1.5){};
      \draw [->, color=black,  line width=1.5pt](0,1.5) -- (-1.4,1.5);

       \node [below, color=black] at (0,0) {$L$};
         \node [left, color=black] at (-1.4,1.5) {$L_{1}$};
           \node [above, color=black] at (0,3) {$A$};

           \node [right, color=black] at (0,0) {\small{$0$}};  
           \node [right, color=black] at (0,1.5) {\small{$\frac{3}{2}$}}; 
                \node [below, color=black] at (-1.3,1.4) {\small{$\infty$}};

                 \node [right, color=black] at (0,0.75) {\small{$1$}};
                 \node [right, color=black] at (0,2.25) {\small{$2$}};
                   \node [above, color=black] at (-0.7,1.5) {\small{$1$}};
  
  \end{scope}

\begin{scope}[shift={(7,0)}]
  
  \draw[->][thick, color=black](-5,1.7) .. controls (-4,1.3) ..(-3,1.7); 
   
     \draw [->, color=black,  line width=1.5pt](0,0) -- (0,3);
     \draw [->, color=black,  line width=1.5pt](0,1) -- (0,0);
      \node[draw,circle, inner sep=1.5pt,color=black, fill=black] at (0,1.5){};
      \draw [->, color=black,  line width=1.5pt](0,1.5) -- (-1.4,1.5);

       \node [below, color=black] at (0,0) {$L$};
         \node [left, color=black] at (-1.4,1.5) {$L_{1}$};
           \node [above, color=black] at (0,3) {$A$};

           \node [right, color=black] at (0,0) {\small{$0$}};  
           \node [right, color=blue] at (0,1.5) {\small{$\mathbf{\frac{3}{2}}$}}; 
                \node [below, color=black] at (-1.3,1.4) {\small{$\infty$}};

                 \node [right, color=black] at (0,0.75) {\small{$1$}};
                 \node [right, color=black] at (0,2.25) {\small{$2$}};
                   \node [above, color=black] at (-0.7,1.5) {\small{$1$}};
    \end{scope}

  \end{tikzpicture}
\end{center}
 \caption{
 Passage from the Eggers-Wall tree decorated with the indices ${\de}_L$ and exponents $\ex_L$ of its marked points to the indices ${\de}_L$ and contact complexities ${\ic}_L$, in the case of the tree of Figure \ref{fig:logdistoEW1}; here the contact complexity is equal to the exponent at the unique ramification point}
\label{fig:exptocontact1}
   \end{figure}
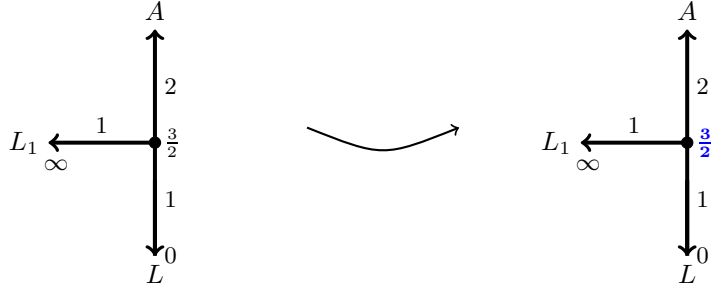


 \begin{figure}[h!]
    \begin{center}
\begin{tikzpicture}[scale=1]
  
  \begin{scope}[shift={(0,0)}] 
   
     \draw [->, color=black,  line width=1.5pt](0,0) -- (0,4.5);
     \draw [->, color=black,  line width=1.5pt](0,1) -- (0,0);

      \draw [->, color=black,  line width=1.5pt](0,1.5) -- (-1.4,1.5);
       \draw [->, color=black,  line width=1.5pt](0,3) -- (-1.4,3);

       \node [below, color=black] at (0,0) {$L$};
       \node [left, color=black] at (-1.4,1.5) {$L_{1}$};
        \node [above, color=black] at (0,4.5) {$A_1$};
        \node [left, color=black] at (-1.4,3) {$L_{2}$};
        \node[draw,circle, inner sep=1.5pt,color=black, fill=black] at (0,1.5){};
      \node[draw,circle, inner sep=1.5pt,color=black, fill=black] at (0,3){};

                \node [right, color=black] at (0,0.75) {\small{$1$}};
                 \node [right, color=black] at (0,2.25) {\small{$2$}};
                  \node [right, color=black] at (0,3.75) {\small{$6$}};
                   \node [above, color=black] at (-0.7,1.5) {\small{$1$}};
                   \node [above, color=black] at (-0.7,3) {\small{$2$}};
                    \node [right, color=black] at (0,0) {\small{$0$}};  
           \node [right, color=black] at (0,1.5) {\small{$\frac{3}{2}$}}; 
            \node [right, color=black] at (0,3) {\small{$\frac{13}{6}$}}; 
             \node [below, color=black] at (-1.3,2.9) {\small{$\infty$}};
                \node [below, color=black] at (-1.3,1.4) {\small{$\infty$}};

 \end{scope}

\begin{scope}[shift={(7,0)}]

 \draw[->][thick, color=black](-5,3) .. controls (-4,2.5) ..(-3,3); 
   
     \draw [->, color=black,  line width=1.5pt](0,0) -- (0,4.5);
     \draw [->, color=black,  line width=1.5pt](0,1) -- (0,0);

      \draw [->, color=black,  line width=1.5pt](0,1.5) -- (-1.4,1.5);
       \draw [->, color=black,  line width=1.5pt](0,3) -- (-1.4,3);

       \node [below, color=black] at (0,0) {$L$};
         \node [left, color=black] at (-1.4,1.5) {$L_{1}$};
           \node [above, color=black] at (0,4.5) {$A_1$};
           \node [left, color=black] at (-1.4,3) {$L_{2}$};
       
            \node[draw,circle, inner sep=1.5pt,color=violet, fill=violet] at (0,1.5){};
      \node[draw,circle, inner sep=1.5pt,color=violet, fill=violet] at (0,3){};

                \node [right, color=black] at (0,0.75) {\small{$1$}};
                 \node [right, color=black] at (0,2.25) {\small{$2$}};
                  \node [right, color=black] at (0,3.75) {\small{$6$}};
                   \node [above, color=black] at (-0.7,1.5) {\small{$1$}};
                   \node [above, color=black] at (-0.7,3) {\small{$2$}};
                    \node [right, color=black] at (0,0) {\small{$0$}};  
           \node [right, color=blue] at (0,1.5) {\small{$\mathbf{\frac{3}{2}}$}}; 
            \node [right, color=blue] at (0,3) {\small{$\mathbf{\frac{11}{6}}$}}; 
             \node [below, color=black] at (-1.3,2.9) {\small{$\infty$}};
                \node [below, color=black] at (-1.3,1.4) {\small{$\infty$}};

    \end{scope}
  \end{tikzpicture}
\end{center}
 \caption{
 Passage from the Eggers-Wall tree decorated with the indices ${\de}_L$  and the exponents $\ex_L$ of its marked points to the indices ${\de}_L$ and the contact complexities ${\ic}_L$, in the case of the tree of Figure
 \ref{fig:logdistoEW2}}
\label{fig:exptocontact2}
   \end{figure}
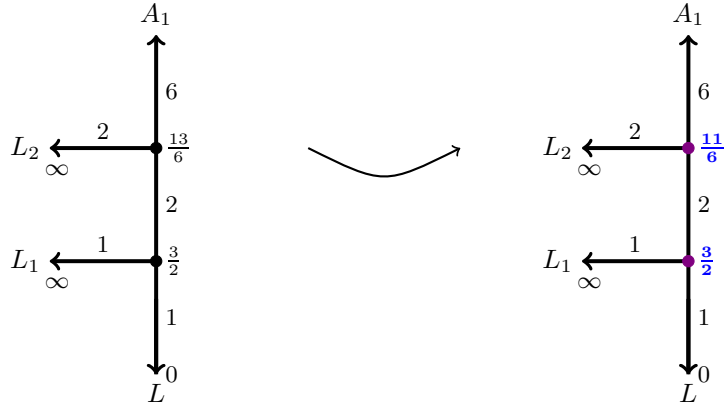


 \begin{figure}[h!]
    \begin{center}
\begin{tikzpicture}[scale=1]
  
  \begin{scope}[shift={(0,0)}]
 \draw [->, color=black,  line width=1.5pt](0,0) -- (0,8);
     \draw [->, color=black,  line width=1.5pt](0,1) -- (0,0);
      \node[draw,circle, inner sep=1.5pt,color=black, fill=black] at (0,1.5){};
      \node[draw,circle, inner sep=1.5pt,color=black, fill=black] at (0,3){};
      \node[draw,circle, inner sep=1.5pt,color=black, fill=black] at (0,4.5){};
      \node[draw,circle, inner sep=1.5pt,color=black, fill=black] at (0,6){};
      \node[draw,circle, inner sep=1.5pt,color=black, fill=black] at (0,7){};
      
      \draw [->, color=black,  line width=1.5pt](0,1.5) -- (-1.4,1.5);
       \draw [->, color=black,  line width=1.5pt](0,3) -- (1.4,3);
       \draw [->, color=black,  line width=1.5pt](0,4.5) -- (1.4,4.5);
       \draw [->, color=black,  line width=1.5pt](0,6) -- (-1.4,6);
       \draw [->, color=black,  line width=1.5pt](0,7) -- (-1.4,7);
      
       \node [below, color=black] at (0,0) {$L$};
         \node [left, color=black] at (-1.4,1.5) {$L_{1}$};
          \node [right, color=black] at (1.4,3) {$A_3$};
           \node [right, color=black] at (1.4,4.5) {$A_1$};
           \node [left, color=black] at (-1.4,6) {$L_{2}$};
           \node [left, color=black] at (-1.4,7) {$L_{3}$};
           \node [above, color=black] at (0,8) {$A_2$};

           \node [right, color=black] at (0,0) {\small{$0$}};  
           \node [right, color=black] at (0,1.5) {\small{$\frac{3}{2}$}}; 
            \node [left, color=black] at (0,3) {$2$}; 
             \node [left, color=black] at (0,4.5) {\small{$\frac{13}{6}$}}; 
              \node [right, color=black] at (0,6) {\small{$\frac{7}{3}$}};
               \node [right, color=black] at (0,7) {\small{$\frac{29}{12}$}};
                \node [below, color=black] at (-1.3,1.4) {\small{$\infty$}};
                \node [below, color=black] at (-1.3,5.9) {\small{$\infty$}};
                \node [below, color=black] at (-1.3,6.9) {\small{$\infty$}};
                
                 \node [right, color=black] at (0,0.75) {\small{$1$}};
                 \node [right, color=black] at (0,2.25) {\small{$2$}};
                  \node [right, color=black] at (0,3.75) {\small{$2$}};
                  \node [right, color=black] at (0,5.25) {\small{$2$}};
                  \node [right, color=black] at (0,6.5) {\small{$6$}};
                   \node [right, color=black] at (0,7.5) {\small{$12$}};
                   \node [above, color=black] at (-0.7,1.5) {\small{$1$}};
                \node [above, color=black] at (-0.7,6) {\small{$2$}};
                \node [above, color=black] at (-0.7,7) {\small{$6$}};
                \node [above, color=black] at (0.7,3) {\small{$2$}};
                \node [above, color=black] at (0.7,4.5) {\small{$6$}};

\end{scope}
  \begin{scope}[shift={(7,0)}]
  \draw[->][thick, color=black](-4,4) .. controls (-3,3.5) ..(-2,4); 
   
     \draw [->, color=black,  line width=1.5pt](0,0) -- (0,8);
     \draw [->, color=black,  line width=1.5pt](0,1) -- (0,0);
      \node[draw,circle, inner sep=1.5pt,color=black, fill=black] at (0,1.5){};
      \node[draw,circle, inner sep=1.5pt,color=black, fill=black] at (0,3){};
      \node[draw,circle, inner sep=1.5pt,color=black, fill=black] at (0,4.5){};
      \node[draw,circle, inner sep=1.5pt,color=black, fill=black] at (0,6){};
      \node[draw,circle, inner sep=1.5pt,color=black, fill=black] at (0,7){};
      
      \draw [->, color=black,  line width=1.5pt](0,1.5) -- (-1.4,1.5);
       \draw [->, color=black,  line width=1.5pt](0,3) -- (1.4,3);
       \draw [->, color=black,  line width=1.5pt](0,4.5) -- (1.4,4.5);
       \draw [->, color=black,  line width=1.5pt](0,6) -- (-1.4,6);
       \draw [->, color=black,  line width=1.5pt](0,7) -- (-1.4,7);
      
       \node [below, color=black] at (0,0) {$L$};
         \node [left, color=black] at (-1.4,1.5) {$L_{1}$};
          \node [right, color=black] at (1.4,3) {$A_3$};
           \node [right, color=black] at (1.4,4.5) {$A_1$};
           \node [left, color=black] at (-1.4,6) {$L_{2}$};
           \node [left, color=black] at (-1.4,7) {$L_{3}$};
           \node [above, color=black] at (0,8) {$A_2$};

    \node [right, color=black] at (0,0) {\small{$0$}};  
    \node [right, color=blue] at (0,1.5) {\small{$\mathbf{\frac{3}{2}}$}}; 
    \node [left, color=blue] at (0,3) {$\mathbf{\frac{7}{4}}$}; 
    \node [left, color=blue] at (0,4.5) {\small{$\mathbf{\frac{11}{6}}$}}; 
    \node [right, color=blue] at (0,6) {\small{$\mathbf{\frac{23}{12}}$}};
    \node [right, color=blue] at (0,7) {\small{$\mathbf{\frac{139}{72}}$}};
    \node [below, color=black] at (-1.3,1.4) {\small{$\infty$}};
    \node [below, color=black] at (-1.3,5.9) {\small{$\infty$}};
    \node [below, color=black] at (-1.3,6.9) {\small{$\infty$}};
                
                 \node [right, color=black] at (0,0.75) {\small{$1$}};
                 \node [right, color=black] at (0,2.25) {\small{$2$}};
                  \node [right, color=black] at (0,3.75) {\small{$2$}};
                  \node [right, color=black] at (0,5.25) {\small{$2$}};
                  \node [right, color=black] at (0,6.5) {\small{$6$}};
                   \node [right, color=black] at (0,7.5) {\small{$12$}};
                   \node [above, color=black] at (-0.7,1.5) {\small{$1$}};
                \node [above, color=black] at (-0.7,6) {\small{$2$}};
                \node [above, color=black] at (-0.7,7) {\small{$6$}};
                \node [above, color=black] at (0.7,3) {\small{$2$}};
                \node [above, color=black] at (0.7,4.5) {\small{$6$}}; 

  \end{scope}

\begin{scope}[shift={(7,0)}]

 \end{scope}

  \end{tikzpicture}
\end{center}
 \caption{
 Passage from the Eggers-Wall tree decorated with the indices ${\de}_L$ and the exponents $\ex_L$ of its marked points to the indices ${\de}_L$ and the contact complexities ${\ic}_L$, in the case of the tree of Figure
 \ref{fig:logdistoEW3}}
\label{fig:exptocontact3}
   \end{figure}
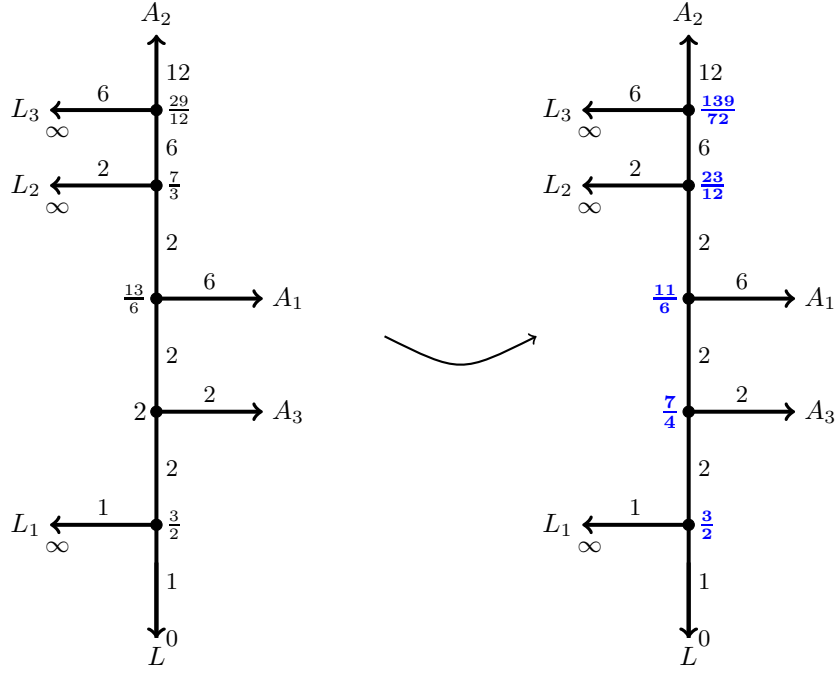

\medskip
\section{The minimal generating set of the semigroup of a branch}
\label{sec:computgensg}

In this section we recall the definition of the {\em semigroup} of a branch (see Definition \ref{def:sgbranch}) and we explain how to compute its minimal generating sequence using a lotus associated to the branch (see Corollary \ref{cor:mingenseq}). 
\medskip

The following notion goes back at least to Arf's paper \cite{A 48} (see also \cite{DV 48} and \cite{I 98}):

\begin{definition}
    \label{def:sgbranch}
    Let $A$ be an abstract branch. Denote by $\tilde{A} \to A$ a normalization morphism and by $\boxed{\nu_A}$ the restriction to the local ring $\calo_{A}$ of the canonical valuation of $\calo_{\tilde{A}}$ (which computes the order of vanishing at the marked point of $\tilde{A}$). The {\bf semigroup $\Gamma(A)$} of $A$ is defined by:
      \[\boxed{\Gamma(A)} := \{ \nu_A(f), \  f \in \calo_A \smallsetminus \{0\} \} \subset \Zz_{\geq 0}. \]
\end{definition}

The set $\Gamma(A)$ is a subsemigroup of $(\Zz_{\geq 0}, +)$, because for every $f, g \in \calo_A \smallsetminus \{0\}$:
   \[\nu_A(fg)  = \nu_A(f) + \nu_A(g),\]
by the definition of valuations (see Definition \ref{def:semival} \eqref{itemprod}). In fact $\Gamma(A)$ is a submonoid of $(\Zz_{\geq 0}, +)$, because the neutral element $0 \in \Zz_{\geq 0}$ is achieved as $\nu_A(u)$, for every unit $u$ of $\calo_A$. It is nevertheless traditional in singularity theory to speak of the {\em semigroup} of a branch instead of its {\em monoid}.

\begin{remark}
    \label{rem:alterninterpr}
    When $A \hookrightarrow (S, O)$ is a {\em plane} branch, there is an alternative interpretation of the semigroup $\Gamma(A)$ as the set of intersection numbers of $A$ with the plane curve singularities which do not contain $A$:
      \[\Gamma(A)= \{ (A \cdot Z(f))_O, \  f \in \calo_{S,O}, \ f|_A \neq 0 \}. \]
\end{remark}

Every subsemigroup $\Gamma$ of $(\Zz_{\geq 0}, +)$ is finitely generated and has a unique minimal generating set with respect to inclusion. It consists of the indecomposable elements of $\Gamma$. By ordering them increasingly, one gets the {\bf minimal generating sequence} of $\Gamma$. 

\medskip

When $\Gamma$ is the semigroup of a plane branch $A$, this sequence may be expressed in terms of the characteristic exponents of $A$. In order to explain this, we need more notations. Denote by $\boxed{\alpha_1}<  \dots < \boxed{\alpha_g}$ the {\em generic} characteristic exponents of $A$, that is, the elements of the set ${\mathrm{Ch}}_L (A_l)$ of Definition \ref{def:charexpcoinc}, where $L$ is a smooth branch {\em transversal} to $A$. This means that $1 < \alpha_1$. Using Zariski's notations in \cite[Chapitre II.3]{Z 86}, define successively:
 
 \begin{eqnarray}
    \label{eq:beta0}
 \boxed{\beta_0} &:=  & \mbox{the minimal common denominator of } \ \alpha_1,  \dots,  \alpha_g \\ 
     \boxed{\beta_i}&  := & \beta_0 \alpha_i, \ \mbox{ for all } i \in  \{1, \dots, g\}, \\
     \boxed{e_i}& : = &\gcd(\beta_0,\ldots, \beta_i), \  \mbox{ for all }  i \in \{0,\ldots,g\}, \\
    \boxed{n_i} & := &\frac{e_{i-1}}{e_i}, \  \mbox{ for all }  i \in \{1,\ldots,g\}, \\
   \boxed{\overline{\beta}_0} 
       & :=&  \beta_0, \\  \boxed{\overline{\beta}_1} & := &\beta_1, \\ 
      \label{eq:recbetabar}
     \boxed{\overline{\beta}_i}&  :=  &n_{i-1} \overline{\beta}_{i-1} + \beta_i - \beta_{i-1}, \  \mbox{ for all }  i \in \{2,\ldots,g\}. 
 \end{eqnarray}

One has the following result (see Zariski's \cite[Theorem 3.9]{Z 86}):

\begin{proposition} 
    \label{prop:mingenfrombeta}
    Assume that $A \hookrightarrow (S, O)$ is a plane branch. In terms of the previous notations, 
    the minimal generating sequence of its semigroup $\Gamma(A)$ is:
      \[(\overline{\beta}_0, \dots, \overline{\beta}_g). \]
\end{proposition}

This proposition indicates the following procedure to compute the minimal generating sequence of $\Gamma(A)$:

\medskip
\begin{center}
\fbox{
\begin{minipage}{0.75\textwidth}
    \begin{corollary}
       \label{cor:mingenseq}
          Let $A \hookrightarrow (S, O)$ be a plane branch.
          Start from the lotus $\Lambda(\hat{\calc}_A)$ of a finite active constellation of crosses such that the branch $L$ corresponding to the initial vertex is transversal to $A$. 
          \begin{itemize}
      \item 
         Construct the Eggers-Wall tree $\Theta_L(A)$ from $\Lambda(\hat{\calc}_A)$, as explained in Corollary \ref{cor:lotustoEW}.
      \item   
         If $P_1, \dots, P_g$ are the points of discontinuity of the index function $\de_L$, define $\alpha_i := \ex_L (P_i)$, for $i = 1, \dots, g$.
      \item 
         Transform the sequence $(\alpha_1, \dots, \alpha_g)$ into the sequence $(\overline{\beta}_0, \dots, \overline{\beta}_g)$ using formulae \eqref{eq:beta0}--\eqref{eq:recbetabar}.   
 \end{itemize}
    \end{corollary}
\end{minipage}
}
\end{center}
\medskip

\medskip

On the other hand, by combining all the formulae \eqref{eq:recbetabar}, we get for all $i \in \{1, \dots,  g\}$ (see also \cite[Chapitre II.3, Th\'eor\`eme 3.9 (a)]{Z 86}): 
\begin{equation}
  \label{eq:gen-exp}
     \overline{\beta}_i=(n_1-1)n_2\cdots n_{i-1}\beta_1+(n_2-1)n_3\cdots n_{i-1}\beta_2+\cdots+(n_{i-1}-1)\beta_{i-1}+\beta_i.
\end{equation}

If $P_0=L, P_1,\ldots, P_g, P_{g+1}=A$ are the marked points of $\Theta_L(A)$, with $P_i \preceq_L P_{i+1}$ for $0\leq i \leq g$, then:
\begin{equation}
\label{eq:betai}
\beta_0={\de}_L(A),\;\;\;\beta_i={\de}_L(A)\ex_L(P_i)
,\;\; \hbox{\rm for all $i \in \{1, \dots, g\}$}
\end{equation}
\noindent and the relation \eqref{eq:gen-exp} becomes
\begin{equation}
   \label{eq:gen-Egg}
     \overline{\beta}_{i}={\de}_L(A)\left(\ex_L(P_i)+\sum_{j=1}^{i-1} \left(\frac{{\de}_L(P_{i})}{{\de}_L(P_{j})}-\frac{{\de}_L(P_{i})}{{\de}_L(P_{j+1})}\right)\ex_L(P_j)\right),\;\; \hbox{\rm for all $i \in \{1, \dots, g\}$}.
\end{equation}

This formula allows to compute the minimal generating sequence of $\Gamma(A)$ directly from $(\Theta_L(A), \de_L, \ex_L)$.

 \begin{example}  
    \label{ex:mingencusp1}
     Consider the Eggers-Wall tree obtained in Figure \ref{fig:logdistoEW1}. Then $g =1$ and $\alpha_1 = \frac{3}{2}$. As $\alpha_1 > 1$, it is a generic Eggers-Wall tree.  Therefore, $\overline{\beta}_0 = \beta_0 = 2$ and:
\[\overline{\beta}_1 = \beta_1 = \beta_0 \alpha_1 = 3.\]
  Using  \eqref{eq:betai}, we get $\overline{\beta}_0=\beta_0={\de}_L(A)=2$ and $\overline{\beta}_{1}=\beta_1={\de}_L(A)\ex_L(P_1)=2\cdot \frac{3}{2}=3$.
The minimal generating sequence of the semigroup $\Gamma(A)$ is therefore $(2,3)$. 
 \end{example}

\begin{example}  
    \label{ex:mingenirr1}
     Consider the Eggers-Wall tree obtained in Figure \ref{fig:logdistoEW2}. Then $g =2$, $\alpha_1 = \frac{3}{2}$ and $\alpha_2 = \frac{13}{6}$. Therefore: 
       \[\overline{\beta}_0 = e_0 = \beta_0 = 6,
       \quad
        \] 
        \[
       \overline{\beta}_1 = \beta_1 = \beta_0 \alpha_1 = 9,
       \quad
        \] 
        \[
    \beta_2 = \beta_0 \alpha_2 = 13, \]
       \[e_1 = \gcd(\beta_0, \beta_1) = 3,  \quad e_2 = \gcd(\beta_0, \beta_1, \beta_2) = 1, \]
       \[ n_1 = \frac{e_0}{e_1} = 2, \quad
        n_1 = \frac{e_1}{e_2} = 3, \]
       \[\overline{\beta}_2 = n_1 \beta_1 + \beta_2 - \beta_1 = 2 \cdot 9 + 13 - 9 = 22. \]

 Alternatively, using  \eqref{eq:betai} and \eqref{eq:gen-Egg} we get: 
 \begin{itemize}
     \item $\overline{\beta}_0={\de}_L(A_1)=6$,
     \item $\overline{\beta}_{1}={\de}_L(A_1)\ex_L(P_1)=6\cdot \frac{3}{2}=9$ and 
     \item $\displaystyle{\overline{\beta}_{2}={\de}_L(A_1)\left(\ex_L(P_2)+\left(\frac{{\de}_L(P_{2})}{{\de}_L(P_{1})}-\frac{{\de}_L(P_{2})}{{\de}_L(P_{2})}\right){\ex}_L(P_1)\right)= 6\left(\frac{13}{6}+(2-1)\cdot \frac{3}{2} \right)=22}$.
     \end{itemize}    
     The minimal generating sequence of the semigroup $\Gamma(A_1)$ is, therefore, $(6,9, 22)$. 
 \end{example}

Let us explain now a second method of computation of a generating sequence 
of the semigroup $\Gamma(A)$ from an associated lotus, without passing through the corresponding Eggers-Wall tree. 

Assume that $A \hookrightarrow (S, O)$ is a plane branch. Fix a smooth branch $L$ on $(S,O)$, which is not necessarily transversal to $A$. Complete it into a cross $X_O := L + L_1$, then consider a finite active constellation of crosses $\hat{\calc}_{A}$ adapted to $A$ in the sense of Definition \ref{def:adaptedactconst}. Denote:
   \[\hat{A} = A + L + L_1 + \cdots + L_g.\]

The following  proposition is a consequence of 
\cite[Theorem 8.6]{S 91} (see also \cite[Corollaries 5.6 and 5.9]{PP 03}): 

\begin{proposition} 
    \label{prop:gensg}
    The sequence 
      \[((A\cdot L)_O, (A\cdot L_1)_O, \dots, (A\cdot L_g)_O)\] 
    is a  generating sequence of the semigroup $\Gamma(A)$. 
\end{proposition}

Let $\boxed{L_0} := L$. For each $j \in \{0, \dots, g\}$, denote by $\boxed{E_j}$ the component of the exceptional divisor of the model of $\hat{\calc}_{A}$ on which $L_j$ is a curvetta, in the sense of Definition \ref{def:curvetta}. That is, the vertex of the lotus $\Lambda(\hat{\calc}_A)$ denoted by $E_j$ is the unique lateral vertex in the sense of Definition \ref{def:vocablotus} such that $[L_j E_j]$ is a lateral edge. Using Proposition \ref{prop:intbr}, we get the following immediate consequence of Proposition \ref{prop:gensg}:

\begin{corollary}
    \label{cor:altcompmingen}
    With the notations above, 
    the sequence 
      \begin{equation} \label{eq:genseq} (\ord_{E_0}(A), \dots, \ord_{E_g}(A))
      \end{equation}
    is a generating sequence of the semigroup $\Gamma(A)$. 
\end{corollary}

Therefore, a generating sequence of $\Gamma(A)$ is immediate to read once one has computed the orders of vanishing of $A$ along all the components of the exceptional divisor of $\hat{\calc}_A$, as explained in Section \ref{sec:ovint}.

\begin{remark}
   If the  lotus of the finite active constellation of crosses $\hat{\calc}_{A}$  has the minimal 
   number of bases then the sequence \eqref{eq:genseq} is the minimal one. 
   In this case the smooth branches $L, L_1, \ldots, L_g$ correspond  to the strict transforms of the branches defined by a sequence of {\em approximate semiroots} in the sense of Abhyankar \cite{A 89} (see also \cite{PP 03}). Those semiroots are generalizations both of Mac Lane's {\em key polynomials} introduced in \cite{MacL 36} (see also \cite{FJ 04}) and of Abhyankar and Moh's {\em approximate roots} of \cite{AM 73}. 
\end{remark}

\begin{example}  
    \label{ex:mingencusp2}
  Consider Figure \ref{fig:ordvanish1}. One sees on it that $\ord_{E_0}(A)=2$ and $\ord_{E_1}(A)=3$. We deduce that the corresponding minimal generating sequence is $(2,3)$, as already found in Example \ref{ex:mingencusp1}. 
 \end{example}

\begin{example}  
    \label{ex:mingenirr2}
  Consider Figure \ref{fig:ordvanish2}. One sees on it that $\ord_{E_0}(A)=6$, $\ord_{E_1}(A)=9$ and $\ord_{E_2}(A)=22$. The minimal generating sequence is therefore $(6,9, 22)$, as already found in Example \ref{ex:mingenirr1}. 
\end{example}

\begin{example}  
    Consider the lotus of a finite constellation of crosses $\hat{\calc}_{A}$ of the form indicated in Figure \ref{fig:nonminimal bases}. Denote by $E_i$, $i = 1, 2,3$ the exceptional divisors labeled in the order of their creation. It follows that $\ord_{E_1} (A)  = 2$, hence $(A\cdot L)_O = (A \cdot L_1)_O = 2$ and $\ord_{E_3} (A)  = 3$, hence $(A \cdot L_3)_O = 3$. This implies that the sequence $((A\cdot L), (A\cdot L_1)_O, (A\cdot L_2))_O = (2,2, 3) $ is not the minimal generating sequence of $\Gamma(A)$, as $2$ appears twice.
\end{example}

 \begin{figure}[h!] 
    \begin{center}
\begin{tikzpicture}[scale=1]

        \draw [->, color=black, thick](0,1.5) -- (-1.5,1.5);
    \draw [-, color=black, thick](-1.5,1.5) -- (-3,1.5);
     \draw [-, color=black, thick](0,1.5) -- (-1.5,4);

      \draw [-, color=black, thick](-1.5,4) -- (-2.5,2.34);

      \draw [->, color=black, thick](-1.5,4) -- (-1.5,5);
     \draw [-, color=black, thick](-3/4,11/4) -- (-3,1.5);

        \draw [->, color=black, thick](-3/4,11/4) -- (-13/8, 509/200);
         \draw [-, color=black, thick](-2.5,2.34) -- (-13/8, 509/200);

       \draw [-, color=black, thick](-3/4,11/4) -- (-2,3.17);
        \draw [-, color=black, thick](-2,3.17) -- (-1.5,4);

     \node[draw,circle, inner sep=1pt,color=black, fill=black] at (0,1.5){};
        \node[draw,circle, inner sep=1pt,color=black, fill=black] at (-3,1.5){};
         \node[draw,circle, inner sep=1pt,color=black, fill=black] at (-1.5,4){} ;
         \node[draw,circle, inner sep=1pt,color=black, fill=black] at (-3/4,11/4){};
         \node[draw,circle, inner sep=1pt,color=black, fill=black] at (-2.5,2.34){};
         \node[draw,circle, inner sep=1pt,color=black, fill=black] at (-2,3.17){};
 
 \node [below, color=black] at (-3,1.5) {$L_{1}$};
\node [below, color=black] at (0,1.5) {$L$};
\node [right, color=black] at (-1.5,5) {$A$};
\node [left, color=black] at (-2.5,2.34) {$L_2$};
\node [right, color=black] at (-3/4,11/4) {\blue{$E_1$}};
\node [left, color=black] at (-2,3.17) {\blue{$E_2$}};
\node [right, color=black] at (-1.5,4) {\blue{$E_3$}};

 \end{tikzpicture}
\end{center}
 \caption{A lotus of a finite constellation of crosses $\hat{\calc}_{A}$ which does not have the minimal number of bases among lotuses of constellations of crosses adapted to $A$ in the sense of Definition \ref{def:adaptedactconst}}
\label{fig:nonminimal bases}
   \end{figure}
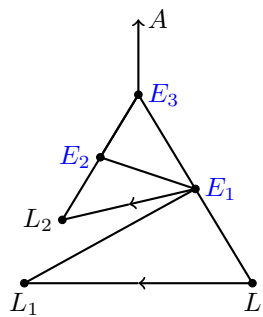

\begin{remark}
   \label{rem:variousbranches}
    One may define a semigroup $\Gamma(A)$ also for  reduced curve singularities $A$ with an arbitrary number of branches $(A_i)_{1 \leq i \leq r}$. It is the subsemigroup of $(\Zz_{\geq 0}^r, +)$ formed by the vectors 
       \[ (\nu_{A_1}(f) , \dots, \nu_{A_r}(f)), \]  
    for $f  \in \calo_{A}$ which are not divisors of zero. Here $\nu_{A_i}$ denotes the valuation of $\calo_{A}$ which gives the order of vanishing on the normalization of $A_i$.  In \cite{CDG 99}, Campillo, Delgado and Gusein-Zade described a minimal system of generators of the semigroup of a plane curve singularity in terms of an embedded resolution of $A$. The coordinates of those generators are vanishing orders along particular components of the exceptional divisor, therefore they may be computed from the lotus, as explained in Section \ref{sec:ovint}. Note that in \cite{HC 20} Hernandes and de Carvalho studied a natural semiring which contains $\Gamma(A)$.
\end{remark}

\medskip
\section{The  delta invariant and the Milnor number}  \label{sec:deltaMiln}

In this section, we explain how to compute the {\em delta invariant} (see Definition \ref{def:deltainv}) and the {\em Milnor number} (see Definition \ref{def:Milnumber}) of a plane curve singularity using an associated lotus (see Corollary \ref{cor:deltamincomp}) and we recall a formula expressing the Milnor number of a plane branch in terms of its characteristic exponents (see Proposition \ref{prop:mufromcharexp}).

\medskip
The following definition applies to any {\em abstract} reduced curve singularity (that is, a germ of a reduced complex analytic space of pure dimension one), not necessarily embeddable in a smooth germ of surface: 

\begin{definition}
   \label{def:deltainv}
    Let $A$ be a reduced curve singularity. Consider a normalization morphism $\tilde{A} \to A$, where $\tilde{A}$ is a finite disjoint union of smooth branches. It corresponds to an embedding $\calo_{A} \hookrightarrow \calo_{\tilde{A}}$ of the local ring of $A$ into the multi-local ring of $\tilde{A}$.  The {\bf delta invariant} $\boxed{\delta(A)}$ of $A$ is the dimension of the quotient complex vector space $\calo_{\tilde{A}}/\calo_{A}$. 
\end{definition}

When the reduced curve singularity is a plane one, there is a formula expressing its delta invariant in terms of its multiplicities at its infinitely near points in the sense of Definition \ref{def:multinpoint} (see \cite[Theorem 3.11.12]{C 00} or \cite[Exercise 5.4.16]{DP 00}): 

\begin{theorem}  
    \label{thm:deltaformula}
    Let $(S,O)$ be a smooth germ of surface and let $A \hookrightarrow (S,O)$ be a reduced plane curve singularity. Then:
       \[\delta(A) = \sum_P \frac{e_P(A)\cdot (e_P(A)-1)}{2},\]
    the sum being taken over the infinitely near points $P$ of $O$ belonging to the strict transforms of $A$. 
\end{theorem}

\begin{remark}
    In \cite{S 91}, Stevens generalized this formula to arbitrary (not necessarily plane) curve singularities. In particular, in the Corollary on page 98 of his paper, he characterized plane curve singularities through the previous formula: if $A$ is not planar, then $\delta(A)$ is strictly less than  the second member of the equality. 
\end{remark}

As may be deduced from Theorem \ref{thm:deltaformula}, the delta invariant of a reduced plane curve singularity with several branches is determined by the delta invariants of its branches and their mutual intersection numbers (see \cite[Corollary 3.11.13]{C 00}):

\begin{theorem} 
    \label{thm:deltabranches}
    Let $A\hookrightarrow (S,O)$ be a reduced plane curve singularity, whose branches are denoted $(A_l)_{1 \leq l \leq r}$. Then:
      \[ \delta(A) = \sum_{l = 1}^r \delta(A_l) +  \sum_{l < m} (A_l \cdot A_m)_O.     \]
\end{theorem}

\begin{remark}
    In \cite[Proposition 4]{H 57}, Hironaka generalized Theorem \ref{thm:deltabranches} to reduced curve singularities {\em of arbitrary embedding dimension}. Namely, if $A = A_1 \cup \dots \cup A_r$ is a decomposition of such a singularity into not necessarily irreducible singularities, pairwise without common branches, he proved that:
      \[\delta(A) = \sum_{l = 1}^r \delta(A_l) +  \sum_{l=1}^{r-1} A_l \cdot (A_{l+1} \cup \dots \cup A_r).  \]
    In this formula appears the following notion of {\em intersection number} of two germs $B$ and $C$ of reduced curves without common branches, contained in an ambient smooth germ $(M,O)$, in terms of their defining ideals $I_B$ and $I_C$ inside the local ring $\calo_{M,O}$ of $(M,O)$ (see \cite[Definition 1]{H 57}):
      \[ B \cdot C := \dim (\calo_{M,O}/(I_B + I_C)). \]
\end{remark}

The {\em Milnor number} of a plane curve singularity is a particular case of the general notion defined by Milnor \cite{M 68} for germs of complex algebraic hypersurfaces with isolated singularities. As for the multiplicity (see Definition \ref{def:mult}) one defines it first for elements of the local ring $\calo_{S,O}$: 

\begin{definition} 
\label{def:Milnumber}
    Let $(S,O)$ be a smooth germ of surface and $f \in \calo_{S,O}$ be a reduced germ. Consider a local coordinate system $(x,y)$ of $(S,O)$. Let  $(\partial_x f, \partial_y f)$ be the ideal generated by the partial derivatives $\partial_x f, \partial_y f$ in the local ring $\calo_{S, O}$. The {\bf Milnor number} $\boxed{\mu(f)}$ of $f$ is the dimension of the quotient
          \[ \calo_{S,O}/(\partial_x f, \partial_y f) \]
    seen as a vector space over the base field.   
\end{definition}

It is a direct consequence of the change of variables formulae that the {\bf Jacobian ideal} $\boxed{J(f)} := (\partial_x f, \partial_y f)$ of the ring $\calo_{S,O}$ is independent of the choice of local coordinates $(x,y)$. Therefore $\mu(f)$ depends indeed only on $f$ and not on those local coordinates. But one has a stronger invariance property:

\begin{proposition}  \label{prop:invMiln}
     Let $f \in \calo_{S,O}$ be reduced. The Milnor number $\mu(f)$ depends only on the plane curve singularity defined by $f$. 
\end{proposition}

\begin{proof}    
    We explain an ingenious algebraic proof given by P\l oski \cite[Property 2.2]{P 95}. His main idea was to use the identity 
       \begin{equation} 
           \label{eq:Teissierlemma}
            Z(f) \cdot Z(\partial_y f) = \mu(f) + Z(f) \cdot Z(x) - 1
      \end{equation}
    which is a special case of Teissier's \cite[Proposition II.1.2]{T 73}. Here and in the rest of this proof, we write simply $A \cdot B$ instead of $(A \cdot B)_O$.

    Before proving equation (\ref{eq:Teissierlemma}), let us see how it implies the assertion of the proposition. Equation (\ref{eq:Teissierlemma}) shows that:
    \begin{equation} 
       \label{eq:conseqTL}       
      \mu(f) = Z(f) \cdot Z(\partial_y f) -  Z(f) \cdot Z(x) +  1 = 
      \dim \frac{\calo_{S,O}}{(f, \partial_y f)}- \dim \frac{\calo_{S,O}}{(f, x)}+ 1.
    \end{equation}

    Consider a second defining function of $Z(f)$. It is of the form $uf$, where $u$ is a unit of the local ring $\calo_{S,O}$, that is, a germ which does not vanish at $O$. We have the following equalities of ideals of $\calo_{S,O}$: 
     \[\begin{array}{c}
         (uf, \partial_y (uf)) = (uf, (\partial_y u) f + u (\partial_y f) ) =  (f, (\partial_y u) f + u (\partial_y f) ) = (f, u (\partial_y f) ) =(f,  (\partial_y f) ), \\
     (uf, x) = ( f, x).
     \end{array}\]
    Combining the resulting equalities of ideals with equation (\ref{eq:conseqTL}), we get the desired equality $\mu(uf) = \mu(f)$. 

    Let us prove now equation (\ref{eq:Teissierlemma}). The main idea is to write the divisor $Z(\partial_y f)$ as a sum $\sum_{i \in I} A_i$ of branches $A_i$ and to parametrize each branch as $t \mapsto (x_i(t), y_i(t))$ in order to get a normalization of it. If $h$ is a series in the variable $t$, denote by $\nu_t(h) \in \Zz$ its order.  Then:

       \begin{eqnarray*}
       Z(f) \cdot Z(\partial_y f) &  = &\sum_{i \in I} Z(f) \cdot A_i = \sum_{i \in I} \nu_t (f(x_i(t), y_i(t)) )=\sum_{i \in I} \left(1 + \nu_t(\frac{d}{dt}f(x_i(t), y_i(t))) \right)  \\
       &  = &\sum_{i \in I} \left(1 + \nu_t((\partial_x f) (x_i(t), y_i(t)) \frac{d}{dt} x_i(t)) \right) = 
       \sum_{i \in I} \nu_t((\partial_x f) (x_i(t), y_i(t)) x_i(t))   \\
       & = &\sum_{i \in I}\nu_t((\partial_x f) (x_i(t), y_i(t)) + \sum_{i \in I}\nu_t(x_i(t)) = \sum_{i \in I} Z(\partial_x f) \cdot A_i + 
       \sum_{i \in I} Z(x) \cdot A_i  \\
       & =  &Z(\partial_x f) \cdot Z(\partial_y f) + 
       Z(x) \cdot Z(\partial_y f) = 
       \mu(f) + \nu_y ((\partial_y f)(0, y))  = \mu(f) + \nu_y (f(0, y)) -1  \\
       & = &\mu(f) + Z(f) \cdot Z(x) - 1. 
       \end{eqnarray*}
\end{proof}

\begin{remark}
    The proof of Proposition \ref{prop:invMiln} cannot be as easy as that of the independence of $\mu(f)$ from the choice of local coordinate system for fixed $f$, because the Jacobian ideal $J(f) := (\partial_x f, \partial_y f)$ may depend on the choice of defining function $f$. Assume for instance that $g := (1 + x) \cdot f$ is a second such defining function. Therefore:
        \[\left\{  \begin{array}{l}
                      \partial_x g = f + (1+x) \cdot \partial_x f \\
                      \partial_y g =  (1+x) \partial_y f
                   \end{array}  \right. .\]
    This shows that one has the equality $J(f) = J(g)$ if and only if $f \in J(f)$. Zariski showed in \cite[Theorem 4]{Z 66} that whenever $f$ is irreducible, this happens if and only if the branch $Z(f)$ is analytically equivalent to a branch of the form $Z(y^n - x^m)$, for coprime positive integers $m$ and $n$ (a result which was generalized to arbitrary dimensions by Saito \cite{S 71}). In particular, this never happens when $Z(f)$ has more than one characteristic exponent. 
\end{remark}

The higher dimensional analog of Proposition 
\ref{prop:invMiln} is also true and may be proved similarly. An alternative proof may be found in \cite[Corollary 6.4.10]{DP 00}.

Proposition \ref{prop:invMiln} allows to speak about the {\bf Milnor number} $\boxed{\mu(A)}$ of a reduced plane curve singularity $A \hookrightarrow (S,O)$. 
 The following theorem relates the delta-invariant and the Milnor number of a plane curve singularity. It was proved topologically by Milnor \cite[Theorem 10.5]{M 68} and extended in \cite[Proposition 1.2.1]{BG 80} by Buchweitz and Greuel to reduced curve singularities of arbitrary embedding dimension: 

\begin{theorem}  
   \label{thm:deltamur}
    Let  $A\hookrightarrow (S,O)$ be a reduced plane curve singularity. Then:
      \[ \mu(A) = 2 \delta(A) - r(A) + 1, \]
    where $\boxed{r(A)}$ denotes the number of branches of $A$. 
\end{theorem}

\medskip
\begin{center}
\fbox{
\begin{minipage}{0.75\textwidth}
    \begin{corollary}
       \label{cor:deltamincomp}
           Let  $A\hookrightarrow (S,O)$ be a reduced plane curve singularity and let $\Lambda$ be the lotus of an active constellation of crosses $\hat{\calc}_A$ adapted to it. Decorate its edges by the multiplicities of $A$ at the infinitely near points of $\hat{\calc}_A$, as explained in Corollary \ref{cor:multedges}. Then $\delta(A)$ and $\mu(A)$ may be computed from those multiplicities using the formulae of Theorems \ref{thm:deltaformula} and \ref{thm:deltamur}.      
    \end{corollary}
\end{minipage}
}
\end{center}
\medskip

\begin{example}
   \label{ex:branchdeltamiln}
  Consider the branch $A_1$ which has a lotus decorated with the multiplicities of $A_1$ at its infinitely near points as shown in Figure \ref{fig:multbranch2}. Using Theorem \ref{thm:deltaformula}, we get:
  \[\delta(A_1) = \frac{1}{2} \cdot \left(  6 \cdot 5 + 3 \cdot 2 +  3 \cdot 2 + 3 \cdot 2 \right) = 24.\]
  Therefore, by Theorem \ref{thm:deltamur} we get:
    \[  \mu(A_1) = 2\delta(A_1) - r(A_1) + 1 = 2 \cdot 24 - 1 + 1 = 48. \]
\end{example}

\begin{example}
   \label{ex:3brdeltamiln}
  Consider the plane curve singularity $A := A_1+ A_2 + A_3$ which has a lotus decorated with the multiplicities of $A$ at its infinitely near points as shown in Figure \ref{fig:multbranch3}. Using Theorem \ref{thm:deltaformula}, we get:
  \[\delta(A) = \frac{1}{2} \cdot \left(  20 \cdot 19 + 10 \cdot 9 +  10 \cdot 9 + 10 \cdot 9 + 5 \cdot 4 + 3 \cdot 2 + 2 \cdot 1  \right) = 339.\]
  Therefore, by Theorem \ref{thm:deltamur} we get:
    \[  \mu(A) = 2\delta(A) - r(A) + 1 = 2 \cdot 339 - 3 + 1 = 676. \]
\end{example}

\medskip

One may compute as follows the Milnor number of a {\em branch} from its characteristic exponents (see \cite[Chapitre II.3]{Z 86}):  

\begin{proposition}
    \label{prop:mufromcharexp}
    Let $A  \hookrightarrow (S,O)$ be a plane branch. Then, with the notations \eqref{eq:beta0}--\eqref{eq:recbetabar},  we have:
      \[  \mu(A)=\sum_{i=1}^g(n_i-1)\overline{\beta}_i-\beta_0 +1 = 
                  n_g \overline{\beta}_g - \beta_g - \beta_0 +1.
       \]
\end{proposition}

Using \eqref{eq:betai} and \eqref{eq:gen-Egg}, the formula of Proposition \ref{prop:mufromcharexp} may be rewritten as follows in terms of the Eggers-Wall tree $(\Theta_L(A), \ex_L, \de_L)$:
\begin{equation}
   \label{eq:MilnbranchEW}
     \mu(A)={\de}_L(A)^2\left(\frac{\ex_L(P_g)}{{\de}_L(P_g)}+\sum_{j=1}^{g-1}\left(\frac{1}{{\de}_L(P_j)}-\frac{1}{{\de}_L(P_{j+1})}\right)\ex_L(P_j)\right)-{\de}_L(A)(\ex_L(P_g)+1) +1. 
\end{equation}

\begin{example}
     Consider the Eggers-Wall tree obtained in Figure \ref{fig:logdistoEW2}.  Using Proposition \ref{prop:mufromcharexp},   we get $\mu(A_1)=3\cdot 22-13-6+1=48$. 

   We may proceed alternatively using formula \eqref{eq:MilnbranchEW}: 
   \[\mu(A_1)=6^2\left(\frac{13}{6} \cdot\frac{1}{2}+\left(\frac{1}{1}-\frac{1}{2}\right)\cdot \frac{3}{2}\right)-6\left(\frac{13}{6}+1\right)+1  = 48. \]
   Happily, we get the same result as in Example \ref{ex:branchdeltamiln}.
\end{example}

\medskip
\section{Extensions to algebraically closed fields of positive characteristic}  \label{sec:charp}

In this section, we discuss extensions of the previous uses of lotuses as computational architectures to the case of a formal germ $A$ defined by reduced non-zero element 
of $\Kk[[x, y ]]$, where $\Kk$ is an algebraically closed field of characteristic $p >0$. The essential point is that in this new setting we can still consider the lotus $\Lambda(\hat{\calc})$ associated with an active constellation of crosses $\hat{\calc}$. This fact allows us to introduce in Definition \ref{Ew-p-lotus} a notion of Eggers-Wall tree of $A$ even if some of the branches of $A$ do not have Newton-Puiseux roots.

\medskip

\medskip 

Let $A$ be the formal plane curve singularity defined by a reduced and non-zero element of the maximal ideal $(x,y)\Kk[[x,y]]$ of the local ring $\Kk[[x, y ]]$. The notion of active constellation of crosses $\hat{\calc}_A$ adapted to $A$ from Definition \ref{def:adaptedactconst} may be defined in positive characteristic without any change. The determination of the weighted dual graph from the lateral boundary of the lotus explained in Proposition \ref{prop:dualgraph} remains unchanged. 
We can still compute from the lotus $\Lambda( \hat{\calc}_A )$ the log-discrepancies and the orders of vanishing of the first axis as explained in Corollary \ref{cor:reclord} and the multiplicities of the strict transforms of the branches of $A$ as explained in Corollary \ref{cor:multedges}.  The intersection multiplicities of pairs of branches of $A$ and the orders of vanishing of the branches of $A$ along the exceptional divisors corresponding to active points of the constellation may still be determined from the lotus $\Lambda( \hat{\calc}_A )$ as described in Corollaries \ref{cor:firstmethint} and \ref{cor:secmethint}. 
\medskip

The notion of Eggers-Wall tree is subtler in positive characteristic, due to the absence or the different behaviour of associated Newton-Puiseux series when they exist. 

\begin{example}  \label{ex:noNProot}
   Let $\Kk$ be an algebraically closed  field of characteristic $\boxed{p} > 0$. As explained by Kedlaya in \cite[Section 1]{Kedlaya 01}, Chevalley \cite[Page 64]{C 51} showed that the Artin-Schreier polynomial 
     \[\boxed{g_p (y)} := y^p - y - x^{-1} \in \Kk((x))[y]\] 
   has no root in the fraction field $\Kk((x^{1/\Nn}))$ of the ring $\Kk[[x^{1/\Nn}]]$ of formal Newton-Puiseux series in $x$. Later, Abhyankar \cite[Section 2]{A 56} noticed that the roots of $g_p(y)$ are of the form 
      \[\xi_j := j \ +\sum_{r =1}^\infty x^{-\frac{1}{p^r}},\] 
   for $j$ varying in the prime field $\mathbb{F}_p \hookrightarrow \Kk$. Notice that the series $\xi_j$ are not Newton-Puiseux series, since their exponents have not bounded denominators. For an example in $\Kk[[x]][y]$, set 
   \[\boxed{f_p(y)} := x^p g_p \left(\frac{y}{x}\right) = y^p - x^{p-1} y - x^{p-1}.\] 
   Using the relation between $f_p(y)$ and $g_p(y)$ we see that the roots of $f_p (y)$ are 
   \[\eta_j := x \cdot \xi_j = jx \ + \sum_{r =1}^\infty x^{1 -\frac{1}{p^r}}.\] 
   Therefore, $f_p(y)$ has no Newton-Puiseux root, similarly to $g_p(y)$. Another change of variable was performed by Abhyankar \cite[Section 2]{A 56}, in order to get the example \[\boxed{h_p(y)} := y^p - x^{p-1}y -1\] without Newton-Puiseux roots. Let us mention that Kedlaya described in \cite{Kedlaya 01} and  \cite{Kedlaya 17} the algebraic closure of the fraction field of the ring $\Kk[[x]]$ in terms of generalized power series, when $\Kk$ is a field of positive characteristic. 
\end{example}

\medskip

We will characterize in Corollary 
\ref{cor:NP-p} below those monic irreducible polynomials in the ring $\Kk [[x]][y]$ which admit Newton-Puiseux roots in $x$.

\medskip 

Because, as shown in Example \ref{ex:noNProot}, Newton-Puiseux series do not necessarily exist expressing the roots of the elements of $\Kk[[x]][y]$, it is more convenient to define the Eggers-Wall tree directly from the lotus, without passing through the consideration of Newton-Puiseux series: 

\begin{definition} \label{Ew-p-lotus}
    Consider a reduced formal germ $A$ defined by a power series $f \in (x,y)\Kk[[x,y]] \smallsetminus \{0\}$. Take any finite constellation of crosses $\hat{\calc}_A$ adapted to $A$. Denote as before $\boxed{X_O} := L + L_1$ and by $\boxed{\hat{A}}$ the {\bf completion of $A$}  determined by $\hat{\calc}_A$, in the sense of Definition \ref{def:adaptedactconst}. Consider the associated lotus $\Lambda(\hat{\calc}_A)$, in the sense of Definition \ref{def:lotusfacc}. The {\bf Eggers-Wall tree} $\boxed{\Theta_{L} (\hat{A})}$, endowed with the {\bf index function} $\boxed{\de_L}$ and the {\bf exponent function} $\boxed{\ex_L}$, is the tree rooted at $L$ obtained by the construction performed on the lotus $\Lambda(\hat{\calc}_A)$ using Definition \ref{def:assEWtolotus}. Therefore, $\Theta_{L}(A)$ is the union of segments $[L A_l]$ joining $L$ with the branches $A_l$ of $A$.  The {\bf characteristic exponents} of a branch $A_l$ of $A$ are the values of the exponent function $\ex_L$ on the points of discontinuity of the index function $\de_L$ in restriction to the segment $[L A_l]$. The {\bf contact complexity function} $\boxed{\ic_L}$ is defined from $\de_L$ and $\ex_L$ using Definition \ref{def:contact complexity}.
\end{definition}

Since the intersection numbers of pairs of branches of $A$, and also of any branch of $A$ and $L$ are  determined by the lotus $\Lambda$ (see Corollaries \ref{cor:reclord} and \ref{cor:firstmethint}), it follows that the tripod formula  \eqref{eq:tripod} holds. Taking into account the definition of $U_L$ from formula \eqref{eq:UL}, we deduce that for any pair $(A_l, A_m)$ of distinct branches of $A$ different from $L$ we have: 
\begin{equation} \label{eq-ultra}
      U_L (A_l , A_m) 
       = \frac{1}{\ic_L (\langle L, A_l, A_m \rangle )}. 
\end{equation}
This implies that $\Theta_L (A)$ is equal to the unique end-rooted tree associated with the ultrametric distance $U_L$ defined on the set of 
branches of $A$ different from $L$ (see \cite[Th. 111 and Remark 114]{GBGPPP 18}). 

\begin{example} 
    Let us consider the plane curve singularity  $A$ defined by the polynomial $f_3 (y) := y^3 - x^2 y - x^2 $. As we explained in Example \ref{ex:noNProot}, this polynomial does not have roots in the ring of Newton-Puiseux series in $x$ when the characteristic of $\Kk$ is $3$. The germ $A$ has the same minimal embedded resolution process by blowups of points which is described in Figure \ref{fig:embrescusp} as the cusp of Example \ref{ex:assconstex}. Consider the active constellation of crosses associated to this process when the initial cross is defined by $L := Z(x)$ and $L_1 := Z(y)$. Its associated lotus is indicated in Figure \ref{fig:asslotuscusp}. In Figure \ref{fig:dualgrcusp} is  represented the weighted dual graph of the minimal embedded resolution of $A$ as a subset of the lotus. The difference with Example \ref{ex:assconstex}, is that  here we have to swap the labels $L$ and $L_1$ in the Figures \ref{fig:embrescusp}, \ref{fig:asslotuscusp} and \ref{fig:dualgrcusp}. In Figure \ref{fig:logdisord1-carp} we represent the lotus when the vertices are decorated with the pairs $(\lambda(E), \ord_E (L))$. In Figure \ref{fig:logdistoEW1-carp} is explained how to pass from this decorated lotus to the Eggers-Wall tree $\Theta_L (A+L+L_1)$. Figure \ref{fig:logdistoEW1} represents also the Eggers-Wall tree $\Theta_{L_1} (A + L + L_1)$. 
\end{example}

 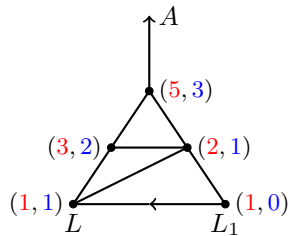
\begin{figure}[h!]
    \begin{center}
\begin{tikzpicture}[scale=1]
   
    \draw [->, color=black, thick](1,0) -- (0,0);
    \draw [-, color=black, thick](0,0) -- (-1,0);
    \draw [-, color=black, thick](0,1.5) -- (-1,0);
      \draw [-, color=black, thick](0,1.5) -- (1,0);
       \draw [-, color=black, thick](-0.5,3/4) -- (0.5,3/4);
       \draw [-, color=black, thick](-1,0) -- (0.5,3/4);
       \draw [->, color=black, thick](0,1.5) -- (0,2.5);
       
  \node[draw,circle, inner sep=1pt,color=black, fill=black] at (-1,0){};
   \node[draw,circle, inner sep=1pt,color=black, fill=black] at (1,0){};
     \node[draw,circle, inner sep=1pt,color=black, fill=black] at (0,1.5){};
      \node[draw,circle, inner sep=1pt,color=black, fill=black] at (-0.5,3/4){};
       \node[draw,circle, inner sep=1pt,color=black, fill=black] at (0.5,3/4){};

\node [right, color=black] at (0,2.5) {$A$};
\node [below, color=black] at (-1,0) {$L$};
\node [below, color=black] at (1,0) {$L_1$};

\node [left, color=black] at (-1,0) {\small{$(\red{1},\blue{1})$}};  
\node [right, color=black] at (1,0) {\small{$(\red{1},\blue{0})$}};  
 \node [right, color=black] at (0,1.5) {\small{$(\red{5},\blue{3}) $}};  
 \node [left, color=black] at (-0.5,3/4) {\small{$(\red{3},\blue{2})$}}; 
  \node [right, color=black] at (0.5,3/4) {\small{$(\red{2},\blue{1})$}}; 
  
 \end{tikzpicture}
\end{center}
 \caption{At each vertex $E \neq A$ of the lotus of Figure \ref{fig:asslotuscusp} is written the pair $(\lambda(E), \ord_E (L))$}
\label{fig:logdisord1-carp}
   \end{figure}

\begin{figure}[h!]
    \begin{center}
\begin{tikzpicture}[scale=1.2]
\begin{scope}[shift={(-20,0)}]

 \draw [->, color=black, thick](1,0) -- (0,0);
    \draw [-, color=black, thick](0,0) -- (-1,0);
    \draw [-, color=black, thick](0,1.5) -- (-1,0);
      \draw [-, color=black, thick](0,1.5) -- (1,0);
       \draw [-, color=black, thick](-0.5,3/4) -- (0.5,3/4);
       \draw [-, color=black, thick](-1,0) -- (0.5,3/4);
       \draw [->, color=black, thick](0,1.5) -- (0,2.5);
       
  \node[draw,circle, inner sep=1pt,color=black, fill=black] at (-1,0){};
   \node[draw,circle, inner sep=1pt,color=black, fill=black] at (1,0){};
     \node[draw,circle, inner sep=1pt,color=black, fill=black] at (0,1.5){};
      \node[draw,circle, inner sep=1pt,color=black, fill=black] at (-0.5,3/4){};
       \node[draw,circle, inner sep=1pt,color=black, fill=black] at (0.5,3/4){};

\node [right, color=black] at (0,2.5) {$A$};

\node [left, color=black] at (-1,0) {\small{$(\red{1},\blue{1})$}};  
\node [right, color=black] at (1,0) {\small{$(\red{1},\blue{0})$}};  
 \node [right, color=black] at (0,1.5) {\small{$(\red{5},\blue{3}) $}};  
 \node [left, color=black] at (-0.5,3/4) {\small{$(\red{3},\blue{2})$}}; 
  \node [right, color=black] at (0.5,3/4) {\small{$(\red{2},\blue{1})$}}; 
  
    \draw[->][thick, color=black](-1,-0.75) .. controls (-1.5,-2) ..(-1,-2.75); 
   
   \end{scope}

   \begin{scope}[shift={(-20,-4)}]
  
 \draw [->, color=black, thick](1,0) -- (0,0);
    \draw [-, color=black, thick](0,0) -- (-1,0);
    \draw [-, color=blue, line width=2.5pt](0,1.5) -- (-1,0);
      \draw [-, color=blue, line width=2.5pt](0,1.5) -- (1,0);
       \draw [-, color=black, thick](-0.5,3/4) -- (0.5,3/4);
       \draw [-, color=black, thick](-1,0) -- (0.5,3/4);
        \draw [->, color=blue, line width=2.5pt](0,1.5) -- (0,2.5);

  \node[draw,circle, inner sep=1.5pt,color=violet, fill=violet] at (-1,0){};
   \node[draw,circle, inner sep=1.5pt,color=violet, fill=violet] at (1,0){};
     \node[draw,circle, inner sep=1.5pt,color=violet, fill=violet] at (0,1.5){};
      \node[draw,circle, inner sep=1.5pt,color=black, fill=black] at (-0.5,3/4){};
       \node[draw,circle, inner sep=1.5pt,color=black, fill=black] at (0.5,3/4){};

\node [right, color=black] at (0.1,2.5) {$A$};
 \node [right, color=black] at (0,1.5) {\small{$(\red{5},\blue{3}) \to \violet{\frac{2}{3}}$}};  
 \node [right, color=black] at (1,0) {\small{$(\red{1},\blue{0}) \to \violet{\infty}$}};  
 \node [left, color=black] at (-1,0) {\small{$(\red{1},\blue{1})\to \violet{0}$}};  

\end{scope}

\begin{scope}[shift={(-14,-3)}]
   
  \draw[->][thick, color=black](-4,0) .. controls (-3,-0.5) ..(-2,0); 
   
     \draw [->, color=blue,  line width=2.5pt](0,0) -- (0,3);
     \draw [->, color=blue,  line width=2.5pt](0,1) -- (0,0);
      \node[draw,circle, inner sep=1.5pt,color=violet, fill=violet] at (0,1.5){};
      \draw [->, color=blue,  line width=2.5pt](0,1.5) -- (-1.4,1.5);

       \node [below, color=black] at (0,0) {$L$};
         \node [left, color=black] at (-1.4,1.5) {$L_{1}$};
           \node [above, color=black] at (0,3) {$A$};

           \node [right, color=violet] at (0,0) {\small{$0$}};  
           \node [right, color=violet] at (0,1.5) {\small{$\frac{2}{3}$}}; 
                \node [below, color=violet] at (-1.3,1.4) {\small{$\infty$}};

                 \node [right, color=blue] at (0,0.75) {\small{$1$}};
                 \node [right, color=blue] at (0,2.25) {\small{$3$}};
                   \node [above, color=blue] at (-0.7,1.5) {\small{$1$}};

      \end{scope}
         
   \end{tikzpicture}
\end{center}
 \caption{Going from $\lambda(E)$ and $\ord_E(L)$ to  $\Theta_L(\hat{A})$ when $\hat{A} =L + L_1 + A$}
\label{fig:logdistoEW1-carp}
   \end{figure}
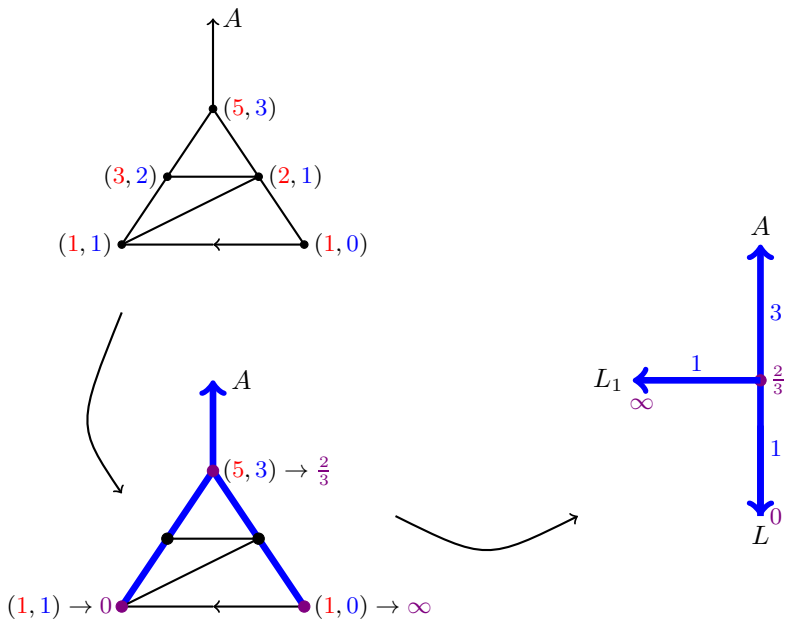

 \begin{figure}[h!]
    \begin{center}
\begin{tikzpicture}[scale=1]

    \draw [->, color=black, thick](1,0) -- (0,0);
    \draw [-, color=black, thick](0,0) -- (-1,0);
    \draw [-, color=black, thick](0,1.5) -- (-1,0);
      \draw [-, color=black, thick](0,1.5) -- (1,0);
       \draw [-, color=black, thick](-0.5,3/4) -- (0.5,3/4);
       \draw [-, color=black, thick](-1,0) -- (0.5,3/4);

        \draw [->, color=black, thick](0,1.5) -- (-1.5,1.5);
   \draw [-, color=black, thick](-1.5,1.5) -- (-3,1.5);
     \draw [-, color=black, thick](0,1.5) -- (-1.5,4);
      \draw [-, color=black, thick](-1.5,4) -- (-3,1.5);
      \draw [->, color=black, thick](-1.5,4) -- (-1.5,5);
     \draw [-, color=black, thick](-3/4,11/4) -- (-3,1.5);
        \draw [-, color=black, thick](-2,3.17) -- (-1.5,4);
        \draw [-, color=black, thick](-3,1.5) -- (-2.5,2.34);
        \draw [-, color=black, thick](-3,1.5) -- (-2.5,2.34);
     
       \node[draw,circle, inner sep=1pt,color=black, fill=black] at (-1,0){};
   \node[draw,circle, inner sep=1pt,color=black, fill=black] at (1,0){};
     \node[draw,circle, inner sep=1pt,color=black, fill=black] at (0,1.5){};
      \node[draw,circle, inner sep=1pt,color=black, fill=black] at (-0.5,3/4){};
       \node[draw,circle, inner sep=1pt,color=black, fill=black] at (0.5,3/4){};
        \node[draw,circle, inner sep=1pt,color=black, fill=black] at (-3,1.5){};
         \node[draw,circle, inner sep=1pt,color=black, fill=black] at (-1.5,4){};
         \node[draw,circle, inner sep=1pt,color=black, fill=black] at (-3/4,11/4){};

   \node [below, color=black] at (-1,0) {$L_{1}$};
\node [below, color=black] at (1,0) {$L$};
\node [below, color=black] at (-3,1.5) {$A_{1}$};
\node [right, color=black] at (-1.5,5) {$A_2$};

\node [left, color=black] at (-1,0) {\small{$(\red{1},\blue{0})$}};  
\node [right, color=black] at (1,0) {\small{$(\red{1},\blue{1})$}};  
 \node [right, color=black] at (0,1.5) {\small{$(\red{5},\blue{2}) $}};  
 \node [left, color=black] at (-0.5,3/4) {\small{$(\red{3},\blue{1})$}}; 
  \node [right, color=black] at (0.5,3/4) {\small{$(\red{2},\blue{1})$}}; 
     \node [left, color=black] at (-3,1.5) {\small{$(\red{1},\blue{0})$}}; 
     \node [left, color=black] at (-1.5,4.2) {\small{$(\red{7},\blue{2})$}}; 
    \node [right, color=black] at (-3/4,11/4) {\small{$(\red{6},\blue{2})$}}; 
 \end{tikzpicture}
\end{center}
 \caption{The lotus of Example \ref{ex:diff-trees}.    
    Near each vertex  is  written the pair $(\lambda(E), \ord_E (L))$}
\label{fig:E}
   \end{figure}
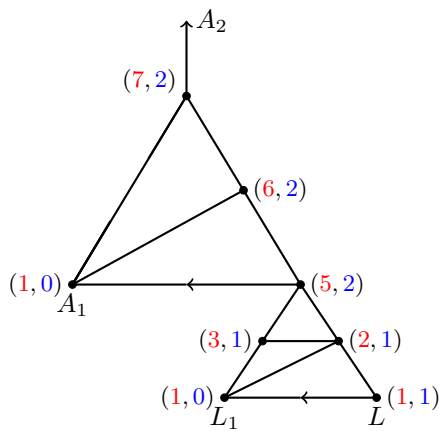

\begin{remark}  \label{rem:fromsgtocharexp}
     If $f \in \Kk[[x,y]]$ is irreducible and $L := Z(x)$ is transversal to the branch $A$ defined by $f$, then we can also determine the Eggers-Wall tree $\Theta_L (A)$ using the minimal generating sequence of the semigroup of $A$. Notice that we can always find a  {\em good parametrization} of $A$  thanks to the {\em Normalisation theorem} (see \cite[Theorem 2.1]{Ploski 13}). Now, from this good parametrization we can construct key polynomials for $A$ and obtain the minimal generating sequence of the semigroup $\Gamma (A)$ (see for example \cite{GB-P 15}).  One may define the {\em characteristic exponents} of $A$ to be equal to the characteristic exponents of a complex branch $A'$ such that $\Gamma(A) = \Gamma(A')$. The multiplicity sequence of the infinitely near points of the strict transforms of a branch $A$ by the successive blowups in its resolution process is known to be equivalent to the sequence of characteristic exponents of a Newton-Puiseux series associated to $A$, whenever the field $\Kk$ is of characteristic zero. When $\Kk$ is of positive characteristic, the characteristic exponents of $A$, or equivalently the characteristic pairs of $A$, were defined by Moh \cite{M 73} and Angerm\"uller \cite{A 77} as the result obtained from the multiplicity sequence by the numerical algorithm used in the case of characteristic zero. A different approach valid in arbitrary characteristic was given in terms of {\em Newton coefficients}, defined in terms of Newton polygons and maximal contact curves by Lejeune-Jalabert in her PhD Thesis and was published in \cite{L 88}. A systematic presentation of the equisingularity theory of algebroid branches defined over an algebraically closed field of arbitrary characteristic was given in Campillo's book \cite{C 80}, by using  {\em Hamburguer-Noether expansions}.  It included a comparison with the previous approaches in the literature. The Eggers-Wall tree $\Theta_L (A)$ of a branch $A$ relative to a transversal smooth branch $L$ in the sense of Definition \ref{Ew-p-lotus} gives rise to the same characteristic exponents as in the previous approaches,  because the computation of the multiplicity sequence from the lotus explained in Section \ref{sec:multstr} is independent of the characteristic. Since the smooth branch $L$ is transversal to $A$, the Eggers-Wall tree $\Theta_L (A)$ is the same as the generic Eggers-Wall tree of a complex branch $A'$ with the same semigroup (see Remark \ref{rem:genericEW}).  
\end{remark}
\medskip 

If $A$ has several branches and $L = Z(x)$ is transversal to all of them, 
then we can use Remark \ref{rem:fromsgtocharexp} to determine the Eggers-Wall tree of each branch $A_l$ of $A$ by computing generators of its semigroup. If $A_l$ and $A_m$ are two branches of $A$, then we determine the value of $\ic_L (\langle L, A_l, A_m \rangle)$ from the intersection multiplicities $(L\cdot A_l)_O$, $(L\cdot A_m)_O$  and $(A_l \cdot A_m)_O$ by the  tripod formula \eqref{eq: tripod2}:
    \[
        \ic_L (\langle L, A_l, A_m \rangle ) = \frac{(A_l\cdot A_m)_O}{(A_l \cdot L)_O (A_m \cdot L)_O}.
    \]
This determines the bifurcation point $\langle L, A_l, A_m \rangle$ on the tree $\Theta_L (A_l+ A_m)$. Then, we can use formula  
\eqref{eq:intcoefint2} to find the value $\ex_L ( \langle L, A_l, A_m \rangle )$.  The Eggers-Wall tree $\Theta_L (A)$ is determined by applying this process for any pair of branches of $A$ different from $L$. 

\begin{example}  \label{ex:2charp}
      The following example is taken from Example 4.11 of the preprint version of \cite{GB-P 15}. Denote by $p \geq 2$ a prime number. 
      Let us consider the branches $A_{1}$ defined by $f_{1}(x,y)=(y- x^{p})^{p} + y^{p^{2}-1}$ and $A_{2}$ defined by $f_{2}(x,y)=(y^{p-1}- x^{p})^{p} + y^{p^{2}-1}$ over an algebraically closed base field of characteristic $p$ and take $L =Z(x)$. By Corollary \ref{cor:NP-p}, the polynomials $f_1$ and $f_2$ do not have roots in the ring of formal Newton-Puiseux series in $x$. A good parametrization of $A_{1}$ is $(\phi_1(t), \psi_1(t)) =(t^{p} + t^{p^{2}-1}, t^{p^{2}})$ and one of $A_{2}$ is $(\phi_2(t), \psi_2(t)) =(t^{p^{2}-p} + t^{p^{2}-1}, t^{p^{2}})$. The minimal generating sequence of the semigroup
      $\Gamma (A_{1})$ is $(p, p^{3} + p^{2}-2p-1)$ while the minimal generating sequence of  $\Gamma (A_{2})$  is $(p^{2}-p, p^{2}, p^{3} -1)$. We have  $\nu_t (f_2 (\phi_1(t), \psi_1(t)) = p^3$, thus $(A_1\cdot A_2)_O = p^3$. 
      Therefore: 
      \[
        \ic_L (\langle L, A_1, A_2 \rangle) = \frac{ p^3}{ p \cdot (p^2 - p) } = \frac{p}{p-1}.
       \] 
      We get in this case $\ex_L (\langle L, A_1, A_2 \rangle) = \frac{p}{p-1}$. 
      In Figure \ref{fig:ExnP+} we represent the Eggers-Wall trees  $\Theta_{L}(A_{1})$ and $\Theta_{L} (A_{2})$ for $p=3$ and in Figure \ref{fig:ExnP++} we represent the corresponding Eggers-Wall tree $\Theta_{L}(A_{1}+A_{2})$. 
\end{example}

When the field $\Kk$ has positive characteristic, the  semivaluation space of the local ring $\Kk[[x, y]]$ was considered in \cite{J 15}. 
The lotus $\Lambda( \hat{\calc}_A )$ associated to an active constellation of crosses adapted to a formal germ $A$ may be also embedded in this semivaluation space, 
by the same arguments as those given in \cite[\S 7]{PP 07} and in  \cite[Remark 1.5.34]{GBGPPP 18}. 
 
Using the methods of \cite[Section 1.6]{GBGPPP 20}, we may prove that Proposition \ref{prop:fromlotustoEW}  is also true in positive characteristic. In \cite[Section 1.5.3]{GBGPPP 20} we associated a lotus to a {\em toroidal pseudo-resolution} of plane curve singularity. In \cite[Section 1.6.6, pages 130-131]{GBGPPP 20}, we gave another definition of the Eggers-Wall tree in terms of the {\em fan tree} of a toroidal pseudo resolution (see \cite[Section 1.4]{GBGPPP 20}). One may check that both definitions are equivalent by using the fact that the contact complexity function  is determined by the tripod formula at any marked point of the tree.

\begin{figure}[h!]
    \begin{center}
\begin{tikzpicture}[scale=1]
\begin{scope}[shift={(0,0)}]
   
     \draw [->, color=blue,  line width=2.5pt](0,0) -- (0,3);
     \draw [->, color=blue,  line width=2.5pt](0,1) -- (0,0);
      \node[draw,circle, inner sep=1.5pt,color=violet, fill=violet] at (0,1.5){};
    
       \node [below, color=black] at (0,0) {$L$};
        \node [above, color=black] at (0,3) {$A_{1}$};    
           
        \node [right, color=violet] at (0,0) {\small{$0$}};  
        \node [right, color=violet] at (0,1.5) {\small{$\frac{29}{3}$}}; 
                
         \node [right, color=blue] at (0,0.75) {\small{$1$}};
        \node [right, color=blue] at (0,2.25) {\small{$3$}};

\end{scope}

\begin{scope}[shift={(4,0)}]
   
     \draw [->, color=blue,  line width=2.5pt](0,0) -- (0,4.5);
     \draw [->, color=blue,  line width=2.5pt](0,1) -- (0,0);
       
    \node [below, color=black] at (0,0) {$L$};
    \node [above, color=black] at (0,4.5) {$A_2$};

    \node[draw,circle, inner sep=1.5pt,color=violet, fill=violet] at (0,1.5){};
    \node[draw,circle, inner sep=1.5pt,color=violet, fill=violet] at (0,3){};

    \node [right, color=blue] at (0,0.75) {\small{$1$}};
    \node [right, color=blue] at (0,2.25) {\small{$2$}};
    \node [right, color=blue] at (0,3.75) {\small{$6$}};                   
    \node [right, color=violet] at (0,0) {\small{$0$}};  
    \node [right, color=violet] at (0,1.5) {\small{$\frac{3}{2}$}}; 
    \node [right, color=violet] at (0,3) {\small{$\frac{17}{6}$}};              
  \end{scope}
   \end{tikzpicture}
\end{center}
 \caption{The Eggers-Wall trees  of the branches $A_1$ and $A_2$ of Example \ref{ex:2charp} in terms of Definition \ref{Ew-p-lotus}, when $p=3$}
\label{fig:ExnP+}
   \end{figure}
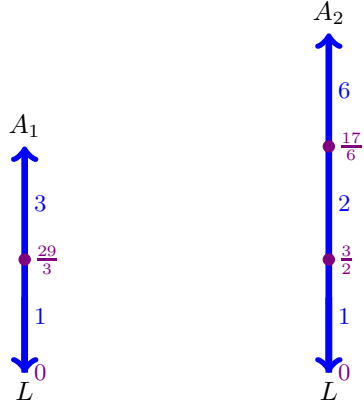
 
  \begin{figure}[h!]
     \begin{center}
\begin{tikzpicture}[scale=1]

      \draw [->, color=blue,  line width=2.5pt](0,0) -- (0,4.5);
      \draw [->, color=blue,  line width=2.5pt](0,1) -- (0,0);
      \draw [->, color=blue,  line width=2.5pt](0,1.5) -- (-1.5,4);

    \node [below, color=black] at (0,0) {$L$};
    \node [above, color=black] at (0,4.5) {$A_2$};
    \node [above, color=black] at (-1.5,4) {$A_1$};

    \node[draw,circle, inner sep=1.5pt,color=violet, fill=violet] at (0,1.5){};
    \node[draw,circle, inner sep=1.5pt,color=violet, fill=violet] at (0,3){};
    \node[draw,circle, inner sep=1.5pt,color=violet, fill=violet] at (-1.5/2,5.5/2) {};

    \node [right, color=blue] at (0,0.75) {\small{$1$}};
    \node [right, color=blue] at (0,2.25) {\small{$2$}};
    \node [right, color=blue] at (0,3.75) {\small{$6$}};
    \node [right, color=violet] at (0,0) {\small{$0$}};
    \node [right, color=violet] at (0,1.5) {\small{$\frac{3}{2}$}};
    \node [right, color=violet] at (0,3) {\small{$\frac{17}{6}$}};
    \node [left, color=violet] at (-1.5/2,5.5/2) {\small{$\frac{29}{3}$}};

    \node [left, color=blue] at (-0.4,2.2) {\small{$1$}};
    \node [left, color=blue] at (-1.1,3.3) {\small{$3$}};

    \end{tikzpicture}
\end{center}
  \caption{The Eggers-Wall tree $\Theta_L (A_1 + A_2)$ of Example \ref{ex:2charp}, when $p =3$}
\label{fig:ExnP++}
    \end{figure}
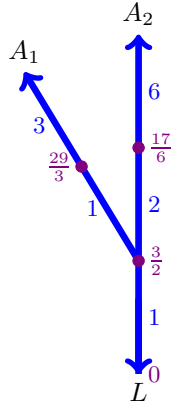

\medskip 

Definition \ref{def:deltainv} of the delta invariant $\delta(A)$ extends to formal germs defined over 
algebraically closed base fields $\Kk$ of characteristic $p >0$. If Definition \ref{def:Milnumber} of $\mu(f)$ extends without problems to series $f \in \Kk[[x,y]]$ when $\Kk$ is of positive characteristic, the definition of the Milnor number  $\mu(A)$ is instead subtler. Indeed, if $u$ is a unit of the local ring $\Kk[[x,y]]$, then, in contrast with the characteristic zero situation described in Proposition \ref{prop:invMiln},  $\mu(f)$ is not necessarily equal to $\mu(uf)$, and in addition $\mu(f)$ may even take an infinite value.  For instance, if $f = y^2 + x^3$, $g = ( 1+ x) \cdot f $ and  $p = 3$, then $\mu(f)=+\infty$ and $\mu(g)=3$. However, in any characteristic,  $\delta(uf) = \delta(f )$ and
$r(uf ) = r(f )$ for any unit $u$ in the formal power series ring $\Kk[[x, y]]$. Theorem \ref{thm:deltamur} does not remain true in positive characteristic, but one has instead the inequality $\mu(f) \geq  2 \delta(f) - r(f) + 1$, and equality means that $f$ has no {\em wild vanishing cycles} (see \cite{D 73} and \cite{M-W 01}). In recent years, several papers were dedicated to Milnor numbers in arbitrary characteristic (see for example \cite{Bou 10}, \cite{GB-P 16}, \cite{Ng 16}, \cite{GB-P 18} \cite{Hefez 18}, \cite{Hefez 19}).

\medskip
\section{Positive characteristic singularities with Newton-Puiseux roots}
\label{sec:existNP}

As in the previous section, we assume that $\Kk$ is an algebraically closed field of characteristic $\boxed{p} > 0$. By analogy with the case of characteristic zero considered in Section \ref{sec:EWfromlot}, we may introduce the local ring 
  \[ \boxed{\Kk[[x^{1/\mathbb N}]]} := \bigcup_{n\geq 1} \Kk[[ x^{1/n}]] \] 
of {\em formal Newton-Puiseux series with coefficients in the field $\Kk$}.
In this section we characterize the monic irreducible polynomials in the ring $\Kk[[x]][y]$ which admit Newton-Puiseux roots in $x$ (see  Corollary \ref{cor:NP-p}).
Then, we introduce a version of Eggers-Wall tree for a plane curve singularity $A$ whose branches have associated Newton-Puiseux series (see Definition \ref{def:EWp}) and we compare it with the Eggers-Wall tree of Definition \ref{Ew-p-lotus} (see Corollary \ref{cor:compartrees}).
\medskip

Let us define first the {\em characteristic exponents} of a series  $\xi \in  \Kk[[x^{1/\mathbb N}]]$. In the case of positive characteristic, we cannot proceed as in the characteristic zero Definition \ref{def:charexpcoinc}.
Instead, we define them 
by looking at the jumps of the minimal denominators of fractional exponents in the expansion of $\xi$ (compare with formulae \eqref{eq:beta0}--\eqref{eq:recbetabar}).

\begin{definition} \label{def:ch-p}
   Let $\xi = \sum_j c_{j} x^{\frac{j}{n}} \in \Kk [[ x^{1/n}]] \smallsetminus \{0\}$, where $n$ is the minimal common denominator of the exponents appearing in the expansion of $\xi$ (that is, of the elements of the support of $\xi$). We set $\boxed{e_0} :=n$, $\boxed{b_1} := \min \{ j : c_j \ne 0, n \nmid  j \}$ and $\boxed{e_1} := \mathrm{gcd} \{ e_0 , b_1 \}$. Assume that $b_r$ and $e_r$ are defined  and that $e_r\neq 1$. Define then  $\boxed{b_{r+1}} := \min \{ j : c_j \ne 0, e_r \nmid  j \}$ and $\boxed{e_{r+1}} := \mathrm{gcd} \{ e_r, b_{r+1} \}$. Since $n$ is the minimal common denominator of the fractional exponents of $\xi$, there is an integer $h \in \Zz_{\geq 0}$ such that $e_h =1$ and the process stops. The {\bf set of characteristic exponents} of the series $\xi$ is then: 
   \begin{equation} \label{eq: ch-exp-p}
      \boxed{\mathrm{Ch} (\xi)} := \left\{ \frac{b_i}{n} : i = 1, \dots, h \right\}.
   \end{equation}
   Set $\boxed{n_j} := e_{j-1}/e_j$ for $j = 1, \dots, h$. We denote by convenience $\boxed{b_0} := 0$ and $\boxed{b_{h+1}} := \infty$. The positive integer $\boxed{m_j}$ is then defined by the equality $\frac{b_j}{n} = \frac{m_j}{n_1 \cdots n_j}$, for every $j \in \{1, \dots, h\}$.
\end{definition}

It is a consequence of the previous definitions that $m_j$ and $n_j$ are coprime, for every $j \in \{1, \dots, h\}$.

\begin{remark}  
If $\xi \in \Kk [[ x^{\frac{1}{n}}]] $ is a Newton-Puiseux root of a formal branch $A$ and if $L$ is the smooth germ defined by $x$, then the characteristic exponents of $\xi$  according to Definition
\ref{eq: ch-exp-p} do not always coincide with the characteristic exponents of $A$ with respect to $L$ from Definition \ref{Ew-p-lotus}. For instance, Campillo  considered in \cite[Example 3.5.4]{C 80} the branch $A$  with Newton-Puiseux root 
     \[ \xi= x^{\frac{p+1}{p}}  + x^{\frac{p^3+p^2+p+1}{p^3}} \in \Kk [[ x^{\frac{1}{p^3}}]] \]
and computed data equivalent to the minimal set of generators of the semigroup $\Gamma(A)$:
\[ 
\overline{\beta}_0 = p^3, \quad
\overline{\beta}_1 =  p^3 + p^2, \quad
\overline{\beta}_2 = p^4 + p^3 + p^2 + p, \quad
\overline{\beta}_3 = p^5 + p^4 + p^3 + p^2 + p + 1.
\]
It follows that the series  $\xi$ has two characteristic exponents according to Definition \ref{eq: ch-exp-p}, while the branch $A$ has three characteristic exponents according to Definition \ref{Ew-p-lotus}.  Teissier mentioned this example in \cite[Section 1]{T 18} and studied this phenomenon in terms of an  {\em overweight deformation} of the monomial curve defined by the semigroup $\Gamma(A)$. Similarly, in the case of a formal germ with branches $A_1$ and $A_2$ admitting Newton-Puiseux roots, the exponent of the ramification point of the tree $\Theta_L (A_1+ A_2)$ of Definition \ref{Ew-p-lotus} may be different from the order of coincidence $k_L (A_1, A_2)$ of the Newton-Puiseux roots 
of $A_1$ and $A_2$ of Definition \ref{def:EWp} (see also Example \ref{ex:diff-trees}). 
\end{remark}

We need the following elementary number-theoretical definition:

\begin{definition} \label{def:partcoprime}
  We denote by $\boxed{\nu_p}$ the {\bf $p$-adic valuation} on the ring $\Zz$ of integers. That is, if $n$ is a nonzero integer, then $\nu_p (n)$ is the largest power of $p$ which divides $n$. We denote by
   \[
       \boxed{{n[:p]}} := n\cdot p^{-\nu_p (n)} \in \Zz
   \]
  the {\bf part of $n$ coprime to $p$}.
\end{definition}

By construction, $n[:p]\not\equiv 0$ (mod $p$). Notice that if $n$ and $m$ are nonzero integers, then: 
\begin{equation} \label{eq: mult-p}
      (nm)[:p] =n[:p] \cdot m[:p].
\end{equation}

\medskip

For every $n \in \Zz_{>0}$, denote by $\boxed{\mathbb U_n} \subset \Kk^*$ the multiplicative group of $n$th roots of unity of $\Kk$. Its order depends on the characteristic $p$ of $\Kk$:

\begin{lemma} \label{lem: Um} 
    Consider $n \in \Zz_{>0}$.   
\begin{enumerate} 
    \item The group $\mathbb{U}_n$ has order $n[:p]$ and  $\mathbb{U}_n = \mathbb{U}_{n[:p]}$.
        \item The extension $\Kk((x)) \subset \Kk((x))[x^{\frac{1}{n}}]$ is Galois if and only if $n=n[:p]$, that is, $p$ does not divide $n$.
\end{enumerate}
\end{lemma}

\begin{proof}
  Notice that $x^{n} - 1 = (x^{n[:p]} -1)^{p^{\nu_p(n)}}$ and that the polynomial $x^{n[:p]} -1$ has $n[:p]$ different roots in $\Kk^*$, since $p$ does not divide $n[:p]$. This proves the first assertion.

  Any element of the automorphism group of the field extension 
    \[ \Kk((x)) \subset \Kk((x))[x^{\frac{1}{n}}]  \]  
  is determined by the image of $x^{\frac{1}{n}}$, which is of the form $\rho \cdot x^{\frac{1}{n}}$, for some $\rho \in \mathbb U_n$. This automorphism group is therefore of order $n[:p]$. The field extension is Galois if and only if this order is equal to the degree $n$ of the extension, that is, if and only if $n[:p] =n$.
\end{proof}

\begin{proposition} \label{prop: degree-p}
    Assume that $\xi \in  \Kk[[x^{\frac{1}{n}}]] \smallsetminus \{0\}$ and that $n$ is the minimal common denominator of the exponents of $\xi$. Then:
   \begin{enumerate}
      \item \label{pr:1} There are exactly $n[:p]$ different conjugates of $\xi$ under the action of the group $\mathbb{U}_n$.

      \item \label{pr:2}The minimal polynomial $f$ of $\xi$ over $\Kk((x))$ belongs to $\Kk[[x]][y]$ and has degree $n$. 

      \item \label{pr:3} $\Kk((x))[\xi]= \Kk((x))[ x^{\frac{1}{n}}]$.

      \item \label{pr:4} The series $\xi$ is a root of multiplicity $p^{\nu_p (n)}$ of the polynomial $f$.
   \end{enumerate}
\end{proposition}

\begin{proof}
\eqref{pr:1}
  In order to prove the first assertion, it is enough to verify  that if $\rho \in \mathbb{U}_n$ is such that $\xi(x^{\frac{1}{n}}) = \xi (\rho \cdot x^{\frac{1}{n}})$, then $\rho = 1$. Consider the set ${\mathrm{Ch}} (\xi) := \left\{ \frac{b_i}{n} : i = 1, \dots, h \right\}$ of characteristic exponents of $\xi$ (see Definition \ref{def:ch-p}). We must have $x^{\frac{b_i}{n}} = (\rho \cdot x^{\frac{1}{n}})^{b_i}$, therefore   $\rho^{b_i} = 1$ for all  $i \in \{1, \dots, h\}$. Since $n$ and $b_1, \dots, b_h$ are coprime, there exist $a_0, \dots, a_h \in \Zz$ such that one has the B\'ezout identity:
    \[ a_0 \cdot n + a_1 \cdot b_1 + \cdots + a_h \cdot b_h = 1. \]
Therefore:
  \[ \rho = \rho^{a_0 \cdot n + a_1 \cdot b_1 + \cdots + a_h \cdot b_h} = 
  (\rho^n)^{a_0} \cdot (\rho^{b_1})^{a_1} \cdots (\rho^{b_h})^{a_h} = 1.\] 

  \medskip
  \eqref{pr:2}
  Let us prove the second assertion. Denote $\xi = \sum_j c_j x^{\frac{j}{n}}$. Set $q := p^{\nu_p (n)}$. As $q$ is a power of the characteristic $p$ of $\Kk$, we get: 
   \begin{equation} \label{eq:powerell}
     \xi^q = 
      \sum_j c_j^q \ x^{\frac{j q}{n}} = \sum_j c_j^q \ x^{\frac{j}{n[:p]}} = \tau(x^{\frac{1}{n[:p]}}), 
   \end{equation}
   where
   \[  \tau(t) := \sum_j c_j^{q}\  t^j \in \Kk[[t]].\]
   The hypothesis that $n$ is the minimal common denominator of the exponents of $\xi$ implies that $n[:p]$ is the minimal common denominator of the exponents of $\xi^q$. By Lemma \ref{lem: Um}, the extension 
  \begin{equation} \label{eq:np}
     \Kk((x)) \subset \Kk((x))[x^{\frac{1}{n[:p]}}]
   \end{equation}
  is Galois of degree $n[:p]$. By the first assertion of the present proposition applied to $\xi^q = \tau(x^{\frac{1}{n[:p]}})$, there are exactly $(n[:p])[:p] = n[:p]$ different conjugates of $\xi^q$ under the action of $\mathbb{U}_{n[:p]}$. The polynomial
\begin{equation} \label{eq: np6}
    g(y) := \prod_{\rho \in \mathbb U_{n[:p]}} (y - \tau (\rho  \cdot x^{\frac{1}{n[:p]}})) \in \Kk[[x^{\frac{1}{n[:p]}}]][y]
\end{equation}
  belongs to $\Kk[[x]][y]$, since its coefficients are invariant by the action of the Galois group $\mathbb{U}_{n[:p]}$ of the extension \eqref{eq:np}. Since $\xi^q$ belongs to $\Kk((x))[x^{\frac{1}{n[:p]}}]$ and the extension \eqref{eq:np} is Galois,  it follows that $g(y)$ is the minimal polynomial of $\xi^q$ over the field $\Kk((x))$. This implies that 
\begin{equation} \label{eq:np4}
    \Kk((x))[\xi^{q}] = \Kk((x))[x^{\frac{1}{n[:p]}}]
\end{equation} and 
\begin{equation} \label{eq: np2}
    [\Kk((x))[\xi^{q}] : \Kk ((x))] = n[:p].
\end{equation}
  Hence, if $p$ does not divide $n$, the statement \eqref{pr:2} is proved. 

\medskip
  {\em Therefore, in what follows we assume that $p$ divides $n$.}

  We claim that the series $\xi^q$ has no $p$-th root $\gamma$ in the ring $\Kk((x))[x^{\frac{1}{n[:p]}}]$. If this were the case, writing $\gamma = \sum_j d_j \ x^{\frac{j}{n[:p]}}$, we obtain that $\gamma^p = \sum  d_j^p  \ x^{\frac{jp}{n[:p]}} = \xi^q$. In particular, the support of the series $\xi^q$ is contained in $\frac{p}{n[:p]} \mathbb{Z}$. It follows that the support of the series $\xi$ is contained in $\frac{p}{q \cdot n[:p]} \mathbb{Z} = \frac{1}{p^{\nu_p (n) -1} \cdot n[:p] }\mathbb{Z} $. This would  contradict the hypothesis that $n = p^{\nu_p (n)} \cdot n[:p]$ is the minimal denominator of the exponents of $\xi$. 

   Combining the previous claim and equality \eqref{eq:np4}, we see that the series $\xi^q $ has no  $p$-th root in the field $\Kk((x))[\xi^{q}]$. Since $q$ is a power of $p$, this implies that the polynomial  $y^{q} - \xi^{q} \in \Kk((x))[\xi^{q}][y] $ is irreducible (see \cite[ Ch.II, \S 5, Th.7]{ZS 75}). This means that $y^{q} - \xi^{q} $ is the minimal polynomial of $\xi$ over $\Kk((x))[\xi^{q}]$. Thus: 
    \begin{equation} \label{eq: np3}
        [\Kk((x))[\xi]: \Kk((x))[\xi^{q}]] = q.
    \end{equation}
   Then, the tower formula for the extensions 
    \[
      \Kk((x))  \subset 
      \Kk((x))[\xi^{q}]
       \subset 
       \Kk((x))[\xi], 
    \]
  and the equations \eqref{eq: np2} and \eqref{eq: np3} imply that
\[
      [\Kk((x))[\xi]: \Kk((x)) ] = 
      [ \Kk((x))[\xi^{q}] :  \Kk((x))] \cdot
      [\Kk((x))[\xi]: \Kk((x))[\xi^{q}]] = n[:p] \cdot q = n.
\]
    As a consequence, the minimal polynomial $f$ of $\xi$ over $\Kk ((x))$ has degree $n$. 
    
    In order to get the whole assertion \eqref{pr:2}, it remains to prove that $f\in K[[x]][y]$. We set $G(y) := g(y^q) \in \Kk[[x]][y]$. Notice that $G(\xi) = g(\xi^q ) = 0$, by equations \eqref{eq:powerell}   and \eqref{eq: np6}.  Since $f$ and $G$ are monic and have the same degree,  this implies that $f= G$, therefore $f \in \Kk[[x]][y]$.

  \medskip
  \eqref{pr:3}
  In order to prove the third assertion, notice that $\Kk((x))[\xi]$ and $ \Kk((x))[x^{\frac{1}{n}}]$ are extensions of degree $n$ of $\Kk((x))$. 
  Since $\xi \in  \Kk((x))[x^{\frac{1}{n}}]$,  
  we deduce that $\Kk((x))[\xi] = \Kk((x))[x^{\frac{1}{n}}]$.

  \medskip
  \eqref{pr:4}
  The last assertion is a consequence of the first and the second one.
\end{proof}

By Definition \ref{def:ch-p}, all the Newton-Puiseux roots of the irreducible polynomial $f$ in Proposition \ref{prop: degree-p} have the same characteristic exponents. This holds because those roots are conjugated under the action of $\mathbb{U}_n$, which implies that they are power series with the same support.

\begin{remark}
      We speak here of the  characteristic exponents of a Newton-Puiseux series $\xi$ using Definition     \ref{def:ch-p}. They do not coincide in general with the characteristic exponents of the branch $A$ admitting $\xi$ as a Newton-Puiseux root, which are those obtained from the Eggers-Wall tree $\Theta_L (A)$ of Definition \ref{Ew-p-lotus} (see Example \ref{ex:diff-trees} below). 
\end{remark}

\medskip 

The following consequence of Proposition \ref{prop: degree-p} provides a characterization of the polynomials in $\Kk[[x]][y]$ which have a root in the ring $\Kk[[x^{\frac{1}{n}}]]$ of Newton-Puiseux series. One may compare it to the characterization given in \cite[Remark 2.1.11]{C 80} in terms of the order of inseparability of a field extension (see \cite[Chapter 2, \S 16]{ZS 75}).

\begin{corollary} \label{cor:NP-p}
   Let $f \in \Kk[[x]][y]$ be a monic irreducible polynomial of degree $n \geq 1$. The following are equivalent: 
   \begin{enumerate} 
      \item \label{item a cor:NP-p} The polynomial $f$ has a root in the ring $\Kk[[x^{\frac{1}{n}}]]$ of  formal Newton-Puiseux series.
        \item \label{item b cor:NP-p}
          The polynomial $f$ belongs to the ring $\Kk[[x]][y^{q}]$, where $q := p^{\nu_p(n)}$. 
    \end{enumerate}
\end{corollary}

\begin{proof}
  The implication  \eqref{item a cor:NP-p} $\Longrightarrow$ \eqref{item b cor:NP-p} is a consequence of the proof of Proposition \ref{prop: degree-p}. 

  Let us prove the implication \eqref{item b cor:NP-p} $\Longrightarrow$ \eqref{item a cor:NP-p}. By hypothesis, there exists a monic polynomial $G \in \Kk[[x]][y]$ such that $f(y) = G( y^q)$. Notice that $G$ is irreducible, as $f$ is irreducible. As the degree $n[:p]$ of $G$ is not divisible by $p$,
  \cite[Corollary 2.4]{Ploski 13} 
  implies that there exists a Newton-Puiseux series $\tau= \sum_j a_j \ x^{\frac{j}{n[:p]}}  \in \Kk[[x^{\frac{1}{n[:p]}}]]$ such that $G(\tau) =0$. Since $\Kk$ is an algebraically closed field, for every $j \geq 1$ there exists $b_j\in \Kk$ such that $b_j^q=a_j$. Consider the series 
    \[ \xi:= \sum_j b_j x^{\frac{j}{q \cdot n[:p]}}\in \Kk [[x^{\frac{1}{n}}]]. \]  
  As $q$ is a power of $p$, one has $\xi^q = \sum_j b_j^{q} \ x^{\frac{j}{n[:p]}} = \tau$. Therefore: 
  \[ f(\xi) = G(\xi^{q}) = G(\tau) = 0. \] 
\end{proof}

Corollary \ref{cor:NP-p}  is  a generalization of \cite[Corollary 2.4]{Ploski 13}, where it is proven that a monic irreducible polynomial $f \in \Kk[[x]][y]$ of degree $n$ coprime to $p$ has Newton-Puiseux roots satisfying \eqref{fmla:NewtPuiseux}, the set \eqref{eq:Zer} has $n$ elements and we have a Galois extension, as in the case of a base field of characteristic zero. This is related to the fact that if $n$ is coprime to the characteristic $p$ of the algebraically closed field $\Kk$, then the polynomial $x^n - 1$ has $n$ pairwise distinct roots in $\Kk$.

\begin{example}
 Using Corollary \ref{cor:NP-p}, we deduce that the polynomials $f_p$ and $h_p$ considered in Example \ref{ex:noNProot} do not have roots in the ring $\Kk[[x^{\frac{1}{n}}]]$ of formal Newton-Puiseux series in $x$.
\end{example}

Next we introduce a version $\Theta_L (A)[:p]$ of Eggers-Wall tree based on Definition \ref{def:EW}, when all the branches of the reduced plane curve singularity $A$  admit Newton-Puiseux roots. Then we compare it to the tree $\Theta_L (A) $ introduced in Definition \ref{Ew-p-lotus} starting from the lotus.

\begin{definition} \label{def:EWp}
  Let $\Kk$ be an algebraically closed field of characteristic $p >0$ and $f \in (x,y)\Kk[[x, y]] \smallsetminus \{0\}$.  Denote by $L$ the branch $(Z(x),O)$ and by $A$ the formal plane curve  singularity defined by $f$. {\em Assume that all the irreducible factors of $f$ admit Newton-Puiseux roots}.

\noindent
     $\bullet$ The {\bf Newton-Puiseux-Eggers-Wall tree} $\boxed{\Theta_L(A_l)[:p]}$ of a branch $A_l \neq L$ of $A$ relative to $L$ is a compact segment endowed with a homeomorphism $\boxed{\ex_L[:p]}: \Theta_L(A_l)[:p] \to [0, \infty]$ called the {\bf exponent function},  and with {\bf marked points}, which are the preimages by the exponent function of the characteristic exponents of any Newton-Puiseux root of $A_l$ relative to $L$. The point $(\ex_L[:p])^{-1}(0)$ is {\bf labeled by $L$} and $(\ex_L[:p])^{-1}(\infty)$ is {\bf labeled by $A_l$}. The tree $\Theta_L (A_l)[:p]$ is also endowed with the  {\bf auxiliary index function}   $\boxed{\de_L^{(a)}}: \Theta_L(A_l)[:p] \to \Zz_{>0}$, which is defined exactly as the index function in Definition \ref{def:EW} by considering the variation of the smallest denominators of the characteristic exponents. The {\bf index function} $\boxed{\de_L[:p]} : \Theta_L(A_l)[:p] \to \Zz_{>0}$ is specified using Definition \ref{def:partcoprime} by: 
     \[
        \de_L[:p] (P) := \left\{ 
                         \begin{array}{lcl}
                            (\de_L^{(a)} (P))[:p] & if & P \ne A_l,
                              \\
                            (L\cdot A_l)_O & if & P =  A_l. 
                         \end{array}
                        \right.
      \]

    \noindent 
       $\bullet$ 
          The {\bf Eggers-Wall tree} $\boxed{\Theta_L(L)[:p]}$ is a point labeled by $L$, at which $\ex_L[:p](L) = 0$ and $\de_L{[:p]}(L) = 1$.

     \noindent
        $\bullet$  The {\bf Newton-Puiseux-Eggers-Wall tree $\boxed{\Theta_L (A)[:p]}$ of $A$ relative to $L$} is obtained from the disjoint union of the trees $\Theta_L (A_l)[:p]$ of its branches $A_l$ by identifying for each pair of distinct branches $A_l$ and $A_m$ of $A$ their points with equal exponent function not greater than the order of coincidence $k(A_l, A_m)$. The {\bf order of coincidence} $\boxed{k(A_l, A_m)}$ is defined as in the case of characteristic zero,  by comparing the Newton-Puiseux roots of the branches (see Definition \ref{def:charexpcoinc}). The {\bf exponent function} $\boxed{\ex_L[:p]} : \Theta_L (A)[:p] \to [0, \infty]$ and the {\bf index function} $\boxed{\de_L[:p]} : \Theta_L (A)[:p] \to \Zz_{>0}$ are obtained by gluing the exponent functions of the trees $\Theta_L (A_l)[:p]$, for $A_l$ running through the branches of $A$. The {\bf marked points} and the {\bf labeled points} of the tree $\Theta_L (A)[:p]$ are specified as in Definition \ref{def:EW}.
\end{definition}

The following proposition is crucial in order to understand the properties of the index function $\de_L[:p]$. It is analogous to the description of the index function in the case of characteristic zero given in Remark \ref{rem-conjugates}. 

\begin{proposition}
    Let $ \xi = \sum_j c_j x^{\frac{j}{n}} \in  \Kk[[x^{\frac{1}{n}}]]$ be a Newton-Puiseux series associated to a branch $A$. Assume that $n$ is the least common denominator of the exponents of $\xi$. Take $P \in (L A)$ and set $\beta: = \ex_{L}[:p] (P) \in (0, \infty)$. Consider the truncation 
    \[ \xi_P := \sum_{\frac{j}{n} < \beta} c_j x^{\frac{j}{n}} \in \Kk[x^{\frac{1}{n}}] \] 
    of the series $\xi$. Then, $\de_L[:p] (P)$ is equal to the number of  different conjugates of $ \xi_P$ under the action of $\mathbb{U}_n$. 
\end{proposition}

\begin{proof}
  Consider the characteristic exponents $\frac{b_1}{n}, \dots, \frac{b_h}{n}$ of the series $\xi$ (see Definition \ref{def:ch-p}). Let $i \in \{0, \dots, h \}$ be the unique integer such that $\frac{b_{i}}{n}  \leq \beta < \frac{b_{i+1}}{n}$, where we denote by convenience $b_0 := 0$ and $b_{h+1} := \infty$. Set $m := \de_L^{(a)} (P)$. Definition \ref{def:EWp} implies that $ \xi_P \in \Kk[[ x^{\frac{1}{m}}]]$ and $m$ is the minimal common denominator of the exponents of $\xi_P$. In particular,  $m$ divides $n$. By Proposition \ref{prop: degree-p}, there are exactly $m[:p]$ different conjugates of $\xi_P$ in $\Kk[[x^{\frac{1}{m}}]]$, which are the different roots of the minimal polynomial of $\xi$ over $\Kk((x))$. The proposition follows, since $m[:p] = \de_L[:p] (P)$ and $m$ divides $n$, which implies that $\mathbb U_m \subset \mathbb U_n$. 
\end{proof}

\begin{remark}
   Let $A$ be  a branch which has a Newton-Puiseux root $\xi \in \Kk[[x^{1/\Nn}]]$. Notice that the index function  $\de_L[:p]$ may be continuous at a marked point $P$ of $\Theta_L (A)[:p]$ and the end $A$ of $\Theta_L (A)[:p]$ is a point of discontinuity of $\de_L[:p]$ if $p$ divides $(L \cdot A)_O$ (see Example \ref{ex:diff-trees}). These phenomena cannot happen for the index function in the case of characteristic zero. 
\end{remark}

Next, we define a contact complexity function on the tree $\Theta_L (A)[:p]$ in the same way as we did in the case of characteristic zero (compare with Definition \ref{def:contact complexity}).

\begin{definition}
  The {\bf contact complexity} function ${\ic}_L[:p]: \Theta_L (A)[:p] \to [0, \infty]$ is defined at every point $P \in \Theta_L (A)[:p]$ by the formula 
  \begin{equation}  
    \label{eq:intcoefint-p}
      \boxed{{\ic}_L[:p](P)} := \int_L^{P} \frac{d \: {\ex}_L[:p]}{{\de}_L[:p]}. 
  \end{equation}
\end{definition}

One may express  the value of the integral \eqref{eq:intcoefint-p} in terms of the characteristic exponents of any branch $A_l$ of $A$ such that $P \in (L A_l]$. Denote by $\xi \in \Kk[[x^{\frac{1}{n}}]]$ a Newton-Puiseux series associated to
$A_l$, where we assume that $n$ is  minimal. We use the notations of Definition \ref{def:ch-p} for the characteristic exponents of $\xi$. Define  $\boxed{P_j} :=  \ex_L^{-1}[:p](b_j/n)$, for each $j \in \{ 0, \dots, h+1 \}$. The index function ${\de}_L[:p]$ has value $(n_1 \cdots n_j)[:p]$ in the half-open segment  $ ( P_j  P_{j+1}]$. Denote by $r \in \{0, \dots,  h\}$ the unique integer such that  $P\in (P_r P_{r+1} ] $ and $\alpha = \ex_L[:p](P)$.

 Formula \eqref{eq:intcoefint-p} implies that:  
\begin{equation} 
        \label{eq:intcoefint2-p}
     {\ic}_L[:p](P) =      
      \left(
    \sum_{j =1}^{r}
    \frac{1}{(n_1 \dots n_{j-1})[:p]} \frac{b_j - b_{j-1}}{n} 
    \right)  + 
    \frac{1}{(n_1 \dots n_{r})[:p]} \left( \alpha - \frac{b_r}{n} \right).
\end{equation}

In the next two lemmas, we use the notations of Definition \ref{def:ch-p} for the characteristic exponents of the Newton-Puiseux series $\xi$ and the notion of order of coincidence from Definition \ref{def:charexpcoinc}.

\begin{lemma} \label{lem:counting} 
   Let  $\xi \in \Kk[[x^{\frac{1}{n}}]]$ be a Newton-Puiseux series. Assume that $n$ is the least common denominator of the exponents of $\xi$. Let $\zeta \in \Kk[[t]]$ be such that $\xi = \zeta(x^{\frac{1}{n}})$. Then, for every $j \in \{1, \dots, h\}$, we have: 
   \begin{enumerate}
       \item \label{item counting a}  
           $\# \{ \rho \in \mathbb U_n : \nu_x( {\zeta} ( \rho \cdot x^{\frac{1}{n}}) - {\zeta}(x^{\frac{1}{n}} )) \geq \frac{b_j}{n} \}  =   e_{j-1}[:p]$,
         \item \label{item counting b}
            $\# \{ \rho \in \mathbb U_n : \nu_x( {\zeta} ( \rho \cdot x^{\frac{1}{n}}) - {\zeta}(x^{\frac{1}{n}}))  =  \frac{b_j}{n} \}  =  e_{j-1}[:p] -e_j[:p]$.
   \end{enumerate}
\end{lemma}

\begin{proof}
  For each $j \in \{ 1, \dots, h\}$, the number of different coefficients of the term $x^{\frac{b_j}{n}}$ in the series ${\zeta} (\rho \cdot x^{\frac{1}{n}})$, when $\rho$ varies among the elements of the group $\mathbb U_n$, is equal to $n_j[:p]$. 

  If $j =1$, then the equality in Formula \eqref{item counting a} is clear, as Proposition \ref{prop: degree-p} implies that $e_0[:p] = n[:p]$ is equal to the number of different conjugates of ${\zeta} (x^{\frac{1}{n}}) $ by the action of $\mathbb{U}_n$. Next, take $1 < j \leq h$. If $\rho \in \mathbb{U}_n$ and if  $\nu_x( {\zeta} ( \rho \cdot x^{\frac{1}{n}}) - {\zeta}(x^{\frac{1}{n}} )) \geq \frac{b_j}{n}$, then the coefficients of the terms $x^{\frac{b_1}{n}}, \dots, x^{\frac{b_{j-1}}{n}}$ in the series ${\zeta}(x^{\frac{1}{n}})$ and ${\zeta}(\rho\cdot x^{\frac{1}{n}})$ coincide. This implies that the number of elements in the set defined by the left-hand side of Formula \eqref{item counting a} is equal to $e_0[:p] /( n_1[:p] \cdots n_{j-1} [:p])=e_{j-1}[:p]$, an equality which follows from equality \eqref{eq: mult-p} and the notations of Definition \ref{def:ch-p}.

  Finally, Formula  \eqref{item counting b} is a  consequence of Formula \eqref{item counting a}. 
\end{proof}

\begin{lemma} \label{lem: comparison}
   Let $f, g \in \Kk[[x]][y]$ be two distinct monic irreducible polynomials. Assume that $f$ and $g$ have Newton-Puiseux roots $\xi $ and $\tau$, respectively. Let $n \in \Zz_{>0}$ be the least common denominator of the exponents of $\xi$. 
   Denote by $\alpha$ the order of coincidence of the branches $Z(f)$ and $Z(g)$ relative to $L=Z(x)$. Let $r \in \{0, \dots, h \}$ be the unique integer such that $\frac{b_r}{n} < \alpha \leq \frac{b_{r+1}}{n}$. Then, we have: 
     \begin{enumerate}
        \item \label{item countg a}
          $\# \{ \eta \in \cZ_x( f)  : \nu_x( \eta - \tau ) =  \frac{b_j}{n} \}  =   e_{j-1}[:p] - e_j[:p]$, for all $j \in \{ 1, \dots, r\}$.  
        \item \label{item countg b}
          $\# \{ \eta \in \cZ_x( f)  : \nu_x( \eta - \tau ) =  \alpha \}  =  e_r[:p]$.
     \end{enumerate}
\end{lemma}

\begin{proof}
    Let $m \in \Zz_{>0}$ be the least common denominator of the exponents of $\tau$.
   Since the roots of $g$ are conjugated under the action of $\mathbb U_m$, we get that 
   \begin{equation} \label{eq-co-fg}
        \alpha = \max \{\nu_x (\eta - \tau) : \eta \in \cZ_x (f)  \}.
   \end{equation}
   
  Fix now $j \in \{1, \dots,r \}$. 
  As in Lemma \ref{lem:counting}, let $\zeta \in \Kk[[t]]$ be such that $\xi = \zeta(x^{\frac{1}{n}})$. The roots of $f$ being conjugated under the action of $\mathbb U_n$, the equality \eqref{eq-co-fg} for the order of coincidence implies that 
   \begin{eqnarray*}
        \# \left\{ \eta \in Z_x(f) : \nu_x (\eta - \tau) = \frac{b_j}{n} \right \} 
        &= &\# \left \{ \rho \in \mathbb U_n : \nu_x ({\zeta}(x^{\frac{1}{n}}) - {\zeta}(\rho \cdot x^{\frac{1}{n}})) = \frac{b_j}{n} \right \}\\
         &= &e_{j-1}[:p] -e_j[:p], 
   \end{eqnarray*}
  where the last equality is a consequence of Lemma \ref{lem:counting}. Equality  \eqref{item countg a} is proved. 

  Equality \eqref{item countg b} results from the fact that the set whose cardinal is to be computed is the complement of the disjoint union of the sets whose cardinals were computed at point \eqref{item countg a}, and from the fact that the total number of roots of $f$ is equal to $n[:p] = e_0[:p]$. 
\end{proof}

The following theorem shows that an analog of the \textbf{tripod formula} of Theorem \ref{thm:intfromEW} holds on the Newton-Puiseux-Eggers-Wall tree $\Theta_L(A)[:p]$ when 
all the branches of $A$ have associated Newton-Puiseux roots in $\Kk[[x^{1/\Nn}]]$: 

\begin{theorem}  \label{th:tripod-p}
    Let $f \in \Kk[[x, y]] \smallsetminus \{0\}$ be a reduced power series defining the plane curve singularity $A$. Denote by $L$ the smooth branch $(Z(x),O)$. Assume that all the branches of $A$ different from $L$ have associated Newton-Puiseux roots in $\Kk[[x^{1/\Nn}]]$. 
    If $A_l$ and $A_m$ are two branches of $A$ different from $L$, we consider the center $P:= \langle L,  A_l, A_m \rangle $ of the tripod determined by $A_l$, $A_m$ and $L$ on the Newton-Puiseux-Eggers-Wall tree $\Theta_L (A_l \cup A_m)[:p]$. Then: 
    \[
       (A_l \cdot A_m)_O =\de_L[:p] (A_l) \cdot \de_L[:p](A_m) \cdot  \ic_L[:p] (P). 
     \]
In particular: 
     \begin{equation} \label{eq:tripod-p}
         \ic_L[:p] (\langle L, A_l, A_m \rangle) =  \frac{ (A_l \cdot A_m)_O }{(A_l \cdot L)_O \cdot (A_m \cdot L)_O}. 
      \end{equation}
\end{theorem}

\begin{proof} 
  Let us fix Newton-Puiseux roots $\xi, \tau \in \Kk[[x^{1/\Nn}]]$ corresponding to $A_l$ and $A_m$ respectively. We denote by $f, g \in \Kk[[x]][y]$ the minimal polynomials of $\xi$ and $\tau$ respectively (see Proposition \ref{prop: degree-p}). Let $\boxed{\mathrm{Res}_y(f,g)}$  be the {\em resultant in $y$} of the polynomials $f$ and $g$. Then: 
  \[
      (A_l \cdot A_m)_O = \nu_x (\mathrm{Res}_y (f, g)) 
        = \nu_x \left( 
       \prod_{\tau_j \in Z_x(g) } \prod_{\xi_s \in Z_x( f) } (\xi_s - \tau_j) 
                \right)
  \]
 where the first equality follows from \cite[Theorem 4.17, Section 4.2]{Hefez 03}. Using the action of the roots of unity, we get 
\begin{equation} \label{eq: im-p}
    (A_l \cdot A_m)_O =  (\deg g) \cdot
      \nu_x \left( \prod_{\xi_s \in Z_x( f) }
         (\xi_s - \tau)  \right).
\end{equation}

  Notice that $  \deg g  = (L\cdot A_m)_O$. We use now the notation introduced in Definition  \ref{def:ch-p} for the characteristic exponents of $\xi$. By Proposition \ref{prop: degree-p}, we have the following formula for the previous intersection multiplicity
  \begin{equation}\label{eq:imlal}
      (L\cdot A_l)_O = n = e_0 [:p] \cdot p^{\nu_p (n)}.
   \end{equation}

  Denote by $\alpha \in \Qq_{>0}$ the order of coincidence of $A_l$ and $A_m$ relative to $L$. By \eqref{eq:intcoefint2-p} and \eqref{eq:imlal}, we get: 
\begin{equation} \label{eq:intcoefint3-p}
    (L \cdot A_l)_O \cdot {\ic}_L[:p](P) = p^{\nu_p (n)}
          \left(
             (e_0[:p] - e_1[:p]) \frac{b_1}{n}  + \cdots + 
             (e_{r-1}[:p] - e_r[:p]) \frac{b_r}{n} +  e_r[:p] \alpha
           \right). 
\end{equation}    

  The result follows by comparing Formula \eqref{eq:intcoefint3-p} with Formula \eqref{eq: im-p} and by applying Lemma \ref{lem: comparison}, which describes the orders of the series $(\xi_k - \tau)$,  for the $n[:p]$ different roots $\xi_k$ of $f$.  Notice that we have to count each root of $f$ with multiplicity $p^{\nu_p (n)}$ (see Proposition \ref{prop: degree-p}).
\end{proof}

Let $A$ be a reduced plane curve singularity such that all its branches different from $L$ admit Newton-Puiseux roots. We compare now the Eggers-Wall trees   $\Theta_L(A)$ of Definition \ref{Ew-p-lotus} with $\Theta_L(A)[:p]$ of Definition \ref{def:EWp}:

\begin{corollary}  \label{cor:compartrees}
    With the hypotheses of Theorem \ref{th:tripod-p}, the trees $\Theta_L (A) $ and $\Theta_L (A) [:p]$ are equal as trees with ends labeled by $L$ and by the branches of $A$. Moreover, $\ic_L = \ic_L[:p]$. 
\end{corollary}

\begin{proof}
   By  \eqref{eq:tripod-p}, we have the following formula
  \begin{equation} \label{eq-ultra-p}
         U_L (A_l , A_m) 
         = \frac{1}{\ic_L[:p] (\langle L, A_l, A_m \rangle )}, 
   \end{equation}
  for the ultrametric distance $U_L$ defined on the set of branches of $A$ different from $L$. This implies that $\Theta_L (A)[:p]$ is equal to the unique end-rooted tree associated with the ultrametric distance $U_L$ defined on the set of branches of $A$ different from $L$ (see \cite[Th. 111 and Remark 114]{GBGPPP 18}). 

  By \eqref{eq-ultra}, the tree $\Theta_L (A)$ has the same property.  By comparing the formulas \eqref{eq-ultra-p} and \eqref{eq-ultra} we get that $\Theta_L (A) = \Theta_L (A)[:p] $ and that on this tree the contact complexity functions are equal, that is, $\ic_L = \ic_L[:p]$.
\end{proof}

 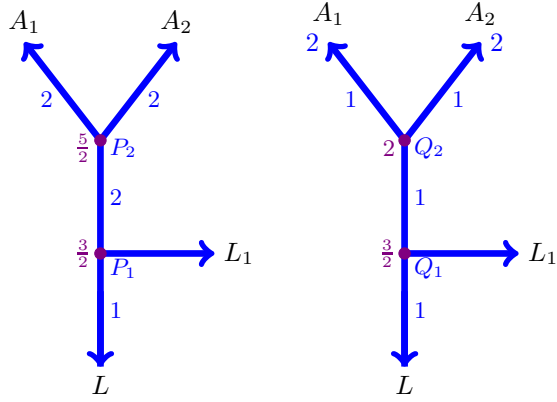
\begin{figure}
    \begin{center}
\begin{tikzpicture}[scale=1]
\begin{scope}[shift={(0,0)}]
   
     \draw [-, color=blue,  line width=2.5pt](0,0) -- (0,3);
     \draw [->, color=blue,  line width=2.5pt](0,3) -- (1,4.3);
      \draw [->, color=blue,  line width=2.5pt](0,3) -- (-1,4.3);
     \draw [->, color=blue,  line width=2.5pt](0,1) -- (0,0);
     \draw [->, color=blue,  line width=2.5pt](0,1.5) -- (1.5,1.5);
         
       \node [below, color=black] at (0,0) {$L$};
       \node [above, color=black] at (-1,4.3) {$A_1$};
        \node [above, color=black] at (1,4.3) {$A_2$};
        \node [right, color=black] at (1.5,1.5) {$L_{1}$};
       
    \node[draw,circle, inner sep=1.5pt,color=violet, fill=violet] at (0,1.5){};
    \node[draw,circle, inner sep=1.5pt,color=violet, fill=violet] at (0,3){};             
    \node [right, color=blue] at (0,0.75) {\small{$1$}};
    \node [right, color=blue] at (0,2.25) {\small{$2$}};
    \node [left, color=blue] at (-0.5,3.55) {\small{$2$}};
    \node [right, color=blue] at (0.5,3.55) {\small{$2$}};
    \node [left, color=violet] at (0,1.5) {\small{$\frac{3}{2}$}}; 
    \node [left, color=violet] at (0,2.9) {\small{$\frac{5}{2}$}}; 
    \node [right, color=blue] at (0,2.9) {\small{$P_{2}$}};
    \node [right, color=blue] at (0,1.3) {\small{$P_{1}$}};
                   
\end{scope}

\begin{scope}[shift={(4,0)}]
     \draw [-, color=blue,  line width=2.5pt](0,0) -- (0,3);
     \draw [->, color=blue,  line width=2.5pt](0,3) -- (1,4.3);
      \draw [->, color=blue,  line width=2.5pt](0,3) -- (-1,4.3);
     \draw [->, color=blue,  line width=2.5pt](0,1) -- (0,0);
     \draw [->, color=blue,  line width=2.5pt](0,1.5) -- (1.5,1.5);
         
       \node [below, color=black] at (0,0) {$L$};
       \node [above, color=black] at (-1,4.4) {$A_1$};
        \node [above, color=black] at (1,4.4) {$A_2$};
        \node [right, color=black] at (1.5,1.5) {$L_{1}$};
       
    \node[draw,circle, inner sep=1.5pt,color=violet, fill=violet] at (0,1.5){};
    \node[draw,circle, inner sep=1.5pt,color=violet, fill=violet] at (0,3){};

    \node [right, color=blue] at (0,0.75) {\small{$1$}};
    \node [right, color=blue] at (0,2.25) {\small{$1$}};
    \node [left, color=blue] at (-0.5,3.55) {\small{$1$}};
    \node [right, color=blue] at (0.5,3.55) {\small{$1$}};
           \node [left, color=violet] at (0,1.5) {\small{$\frac{3}{2}$}}; 
            \node [left, color=violet] at (0,2.9) {\small{$2$}}; 
            \node [right, color=blue] at (0,2.9) {\small{$Q_{2}$}};
              \node [right, color=blue] at (0,1.3) {\small{$Q_{1}$}};
               \node [left, color=blue] at (-1,4.3) {$2$};
           \node [right, color=blue] at (1,4.3) {$2$};          
              
\end{scope}
  \end{tikzpicture}
\end{center}
 \caption{Let $A := A_1+ A_2 + L_1$ in  Example \ref{ex:diff-trees}.
    On the left is represented the Eggers-Wall tree $\Theta_L (A)$
     obtained from the lotus in Figure \ref{fig:E} by applying 
     Definition \ref{Ew-p-lotus}. On the right is drawn the tree Newton-Puiseux-Eggers-Wall tree $\Theta_L (A)[:2]$ introduced in Definition \ref{def:EWp}
     }
\label{fig:ExE}
   \end{figure}

\begin{example} \label{ex:diff-trees}
     Let $A_{1}$ and $A_{2}$ be the branches defined by $f_{1} :=y^{2}+x^{3}$ and 
     $f_{2} :=y^{2}+x^{3}+  x^{4}$ over an algebraically closed field $\Kk$ of characteristic $2$. Resolving $A:=A_{1}+A_{2}$ by blowups of points, we see that there exists an active constellation of crosses adapted to $A$ whose lotus is as shown in Figure \ref{fig:E}, where $L=Z(x)$ and $L_{1}=Z(y)$.
     Following Definition \ref{Ew-p-lotus}, we get the Eggers-Wall tree of $A$ on the left of Figure \ref{fig:ExE}. The tripod formula holds since  
     $ \ic_{L}(P_{2})=\frac{3}{2}+\frac{1}{2}\left(\frac{5}{2}-\frac{3}{2}\right)=
         \frac{3}{2}+\frac{1}{2}=2$, therefore $\de_L (A_1) \cdot \de_L (A_2)  \cdot \ic_{L}(P_{2})=8 = (A_{1} \cdot A_{2})_O$.

     Both $f_1$ and $f_2$ are irreducible in $\Kk[[x]][y]$ and have double roots in $\Kk[[x^{\frac{1}{2}}]][y]$: 
      \[  y^{2}+x^{3} = (y - x^{\frac{3}{2}})^2, \   y^{2}+x^{3}+  x^{4} = (y - (x^{\frac{3}{2}}+x^{2}))^2. \]
    The order of coincidence $k_{L}(A_1,  A_2)$ of $A_1$ and $A_2$ is the order of the difference of the two Newton-Puiseux roots $x^{\frac{3}{2}}$ and $x^{\frac{3}{2}}+x^{2}$, therefore  $k_{L}(A_1, A_2)=2$. This shows that the tree ${\Theta}_L (A)[:2]$ obtained from Definition \ref{def:EWp} is as shown on the right of Figure \ref{fig:ExE}. We get  ${\ic}_{L}[:2](Q_{2})= 2  = \ic_L(Q_2)$.  Since $ {\ic}_{L}[:2](Q_{2}) \cdot \de_{L}(A_{1}) \cdot  \de_{L}(A_{2})=8$, the tripod formula \eqref{eq:tripod} also holds for this tree.
\end{example}

Example \ref{ex:diff-trees} illustrates the fact that in positive characteristic,
when we have Newton-Puiseux roots, even if the tree $\Theta_L (A)$ of Definition \ref{Ew-p-lotus} and the tree $\Theta_L (A)[:p]$ of Definition \ref{def:EWp} coincide and their contact complexity functions $\ic_L$ and $\ic_L[:p]$ are equal (see Corollary \ref{cor:compartrees}), nevertheless the index functions $\de_L$ and $\de_L[:p]$ and the exponent functions $\ex_L $ and $\ex_L[:p] $ may be different.

\end{document}